\documentclass[]{article}

\usepackage [T2A] {fontenc}
\usepackage [utf8] {inputenc}
\usepackage[russian]{babel}

\usepackage{mathtools}
\mathtoolsset{showonlyrefs}

\usepackage{indentfirst}
\usepackage{amssymb}
\usepackage{amsmath}
\usepackage{amsthm}
\usepackage{pifont}
\usepackage{subfig}
\usepackage{graphicx}
\usepackage{graphics}
\usepackage{makecell}
\usepackage{multirow}
\usepackage{booktabs}
\usepackage[colorlinks=true, allcolors=blue]{hyperref}
\usepackage{algorithm2e}
\usepackage{algpseudocode,changepage}
\usepackage{float}
\usepackage[skins]{tcolorbox}

\usepackage{geometry}
\newcommand{\argmax}{\mathop{\arg\!\max}}
\newcommand{\argmin}{\mathop{\arg\!\min}}

\newcommand{\Argmin}{\mathop{\operatorname{Arg}\!\min}}

\makeatletter
\newenvironment{sqcases}{
  \matrix@check\sqcases\env@sqcases
}{
  \endarray\right.
}
\def\env@sqcases{
  \let\@ifnextchar\new@ifnextchar
  \left\lbrack
  \def\arraystretch{1.2}
  \array{@{}l@{\quad}l@{}}
}
\makeatother

\usepackage{dsfont}

\makeatletter
\newcommand{\newreptheorem}[2]{
\newtheorem*{rep@#1}{\rep@title}
\newenvironment{rep#1}[1]{
\def\rep@title{#2 \ref*{##1}}
\begin{rep@#1}}{\end{rep@#1}}}
\makeatother

\newtheorem{theorem}{Теорема}
\newtheorem{re:theorem}{Теорема}
\newreptheorem{theorem}{Теорема}
\newtheorem{th:corollary}{Следствие}[theorem]
\newtheorem{re:th:corollary}{Следствие}[re:theorem]
\newtheorem{lemma}{Лемма}
\newtheorem{lm:corollary}{Следствие}[lemma]

\newtheorem{assumption}{Предположение}

\renewcommand{\arraystretch}{2.0}

\newcommand\scalemath[2]{\scalebox{#1}{\mbox{\ensuremath{\displaystyle #2}}}}

\begin{document}
\renewcommand{\abstractname}{\vspace{-\baselineskip}}

$$
\\\\
$$

\begin{center}
\Large\textbf{Гибкая модификация метода Гаусса--Ньютона и его стохастическое расширение}
\end{center}

\begin{center}
\textbf{Н.\,Е.~Юдин$^{1,2}$, А.\,В.~Гасников$^{1,3}$}

\textit{
$^{1}$Московский физико-технический институт (национальный исследовательский университет), Долгопрудный, Московская обл., Россия \\
$^{2}$Федеральный исследовательский центр <<Информатика и управление>> Российской академии наук, Москва, Россия
\\
$^{3}$Институт проблем передачи информации им. А.А. Харкевича Российской академии наук, Москва, Россия
\\
{$^*$ e-mail: iudin.ne@phystech.edu}}

\end{center}

\begin{abstract}
\noindent В работе предлагается новая версия метода Гаусса--Ньютона для решения системы нелинейных уравнений, основанная на идеях использования верхней оценки нормы невязки системы уравнений и квадратичной регуляризации. В рамках данного метода получена глобальная сходимость, а при естественных предположениях установлена локальная квадратичная сходимость. В новой версии метода Гаусса--Ньютона разработаны стохастические алгоритмы оптимизации, для данных алгоритмов выведены условия сублинейной и линейной сходимости, основанные на условии слабого роста и условии Поляка--Лоясиевича. Предложенный метод объединяет в себе несколько существующих и часто используемых на практике модификаций метода Гаусса--Ньютона, позволяя получить гибкий и удобный в использовании метод, реализуемый на практике с помощью стандартных техник выпуклой оптимизации.

\textbf{Ключевые слова}: системы нелинейных уравнений, минимизация эмпирического риска, метод Гаусса--Ньютона, методы доверительной области, невыпуклая оптимизация, неточное проксимальное отображение, неточный оракул, стохастическая оптимизация, стохастическая аппроксимация, перепараметризованная модель, условие слабого роста, условие Поляка--Лоясиевича, оценка сложности
\end{abstract}

\section{Введение} \label{section_1}

\subsection{Мотивация}

Задача решения системы нелинейных уравнений является одной из наиболее фундаментальных в численных методах. В различных вариантах данная задача представлена в трудах и монографиях по численным методам и по численным методам в оптимизации \cite{Samarskii1989, NocedalWright2006, NesterovLectures, Gasnikov2018}. В общем виде \textit{система нелинейных уравнений} задаётся через многозначное отображение $F:~\mathbb{R}^{n}\rightarrow\mathbb{R}^{m}$:
\begin{equation}\label{eq:main_det_system}
    F(x) = \mathbf{0}_{m},~\mathbf{0}_{m} = (0,~\dots, 0)^{T}.
\end{equation}
В работе рассматривается следующая релаксация исходной задачи решения системы (не)линейных уравнений с помощью решения задачи безусловной минимизации (здесь и далее по умолчанию подразумевается, что $\|~\|$ --- стандартная евклидова норма в $\mathbb{R}^{m}$, если не оговорено иное, см.  раздел~\ref{euclidian_general} о возможных обобщениях):
\begin{equation}\label{eq:main_det_problem}
    \min\limits_{x\in\mathbb{R}^{n}}\left\{f_{1}(x) \overset{\operatorname{def}}{=} \left\|F(x)\right\|\right\}.
\end{equation}

Обычно введённую задачу минимизации решают через минимизацию $\frac{1}{2}\left(f_{1}(x)\right)^{2}$, дополнительно предположив достаточную гладкость получившейся функции или субдифференцируемость отображения $f_{1}$~\cite{Gratton2007, Le2020, Botev2017, Cai2019, Ren2019}. При таком подходе часто применяются \textit{методы доверительной области} и \textit{квазиньютоновские методы} для минимизации квадрата оптимизируемого функционала в \eqref{eq:main_det_problem} с применением различных стратегий \cite{Le2020, Botev2017, Cai2019, Ren2019, Thiele2020, Gargiani2020, Smietanski2020, Cartis2019}. Однако подобные преобразования потенциально могут увеличить необходимое количество итераций для достижения требуемого $\epsilon$--уровня значения функции $f_{1}$ из произвольного начального приближения $x_{0}\in\mathbb{R}^{n}$ с помощью итерационного метода оптимизации. В частности, для линейного оператора $F$ возведение в квадрат $f_{1}$ приводит к увеличению числа обусловленности задачи и к квадратичному росту количества необходимых итераций для достижения $\epsilon$--оптимального решения задачи  \eqref{eq:main_det_problem}, что крайне ощутимо на практике при больших значениях $mn$. В качестве альтернативного подхода можно рассмотреть прямое решение задачи \eqref{eq:main_det_problem} с помощью \textit{метода Гаусса--Ньютона}, заключающегося в сведении к решению последовательности вспомогательных задач оптимизации при условии гладкости функций $F_{i},~i\in\{1,~\dots, m\}$:
\begin{equation}\label{eq:main_det_relax_problem}
    \begin{aligned}
        &\min\limits_{h\in\mathbb{R}^{n}}\left\{\left\|F(x) + F^{'}(x)h\right\|:~x + h\in\operatorname{D}(x)\right\},\\
        &F^{'}(x) \overset{\operatorname{def}}{=} \left(\frac{\partial F_{i}}{\partial x_{j}}(x)\right)_{i,j = 1}^{m, n}\in\mathbb{R}^{m\times n}\text{ --- матрица Якоби,}
    \end{aligned}
\end{equation}
$\operatorname{D}(x)$ --- соответствующая окрестность точки $x\in\mathbb{R}^{n}$. Среди полезных свойств такого сведения можно отметить отсутствие необходимости вычислять производные второго и более высоких порядков, а также возможность в естественных условиях невырожденности оптимизировать \eqref{eq:main_det_problem} с помощью последовательности задач \eqref{eq:main_det_relax_problem} с локальной квадратичной скоростью сходимости \cite{NocedalWright2006}.

В данной работе предлагается другой подход к решению задачи \eqref{eq:main_det_problem}. Он заключается в замене функции, измеряющей невязку системы уравнений \eqref{eq:main_det_system}, на евклидову норму, поделённую на квадратный корень количества уравнений $m$. Полученная невязка используется в построении метода, решающего задачу \eqref{eq:main_det_problem} с помощью решения последовательности задач, в каждой из которых новое приближение решения \eqref{eq:main_det_problem} вычисляется как результат параметризованного проксимального отображения, в котором выполняется поиск точки минимума суммы верхней оценки линеаризованной по аналогии с \eqref{eq:main_det_relax_problem} невязки на основе евклидовой нормы, поделённой на квадратный корень количества координат, и квадратичного проксимального слагаемого \cite{Nesterov2007, Nesterov2020}. В рамках предложенного подхода при естественных для задач вида \eqref{eq:main_det_problem} и \eqref{eq:main_det_relax_problem} предположениях установлено монотонное уменьшение значения функции $f_{1}$, а также локальная сублинейная и линейная сходимость \cite{Nesterov2020}. В отличие от \cite{Nesterov2007, Nesterov2020} в данной работе рассматриваются оба случая соотношения значений $m$ и $n$: $m\leq n$ и $m > n$. Первый случай соответствует решению системы (не)линейных алгебраических уравнений, второй случай является задачей восстановления регрессии, решаемой методом наименьших квадратов. В данной работе представлен анализ локальной и глобальной сложности решения задачи \eqref{eq:main_det_problem} в рамках предложенного подхода, выведены условия локальной квадратичной сходимости, для случаев $m\leq n$ и $m > n$ разработаны и исследованы \textit{стохастические методы} решения задачи $\eqref{eq:main_det_problem}$.

Выполненный анализ стохастических методов применим для задач \textit{минимизации эмпирического и среднего риска}, имеющих представление в форме \eqref{eq:main_det_problem}. Для таких задач в работе использовано свойство функционалов в форме \eqref{eq:main_det_problem}, называемое \textit{условием слабого роста} \cite{Schmidt2013, Taylor2019, Ajalloeian2020}. В текущей работе оно заключается в мажорировании квадрата нормы градиента функции $\left(f_{1}(x)\right)^{2}$ значением самой функции $\left(f_{1}(x)\right)^{2}$. Условие слабого роста имеет и стохастическую форму, в которой мажорируется средний квадрат нормы стохастического градиента функции $\left(f_{1}(x)\right)^{2}$, оцененного по части функций из системы \eqref{eq:main_det_system}, значением самой функции $\left(f_{1}(x)\right)^{2}$. Наравне с условием слабого роста для задач \eqref{eq:main_det_problem} в данной работе рассматривается выполнение \textit{условия Поляка--Лоясиевича}, состоящее в мажорировании значения $\left(f_{1}(x)\right)^{2}$ квадратом нормы градиента $\left(f_{1}(x)\right)^{2}$~\cite{Polyak1963}. Для условия Поляка--Лоясиевича так же рассматривается стохастическая версия, заключающаяся в мажорировании значения $\left(f_{1}(x)\right)^{2}$ средним квадратом нормы стохастической оценки градиента $\left(f_{1}(x)\right)^{2}$, построенной с помощью части функций из системы \eqref{eq:main_det_system}. Выполнение обозначенных условий в случае $m\leq n$ приводит к совместности системы уравнений \eqref{eq:main_det_system} и позволяет решить задачу \eqref{eq:main_det_problem} с любой наперёд заданной точностью с помощью предложенных в этой работе \textit{стохастических методов оптимизации} с линейной скоростью. В машинном обучении и в статистическом моделировании подобные условия, из которых следует совместность системы уравнений \eqref{eq:main_det_system}, обычно выполняются для \textit{перепараметризованных моделей} ($m\leq n$, если за $m$ принять количество объектов в выборке), а сами условия в терминах \textit{стохастической аппроксимации} часто называются \textit{условиями интерполяции}, из которых как раз следует выполнение тождественного равенства $F(x) = \mathbf{0}_{m}$ для некоторых $x\in\mathbb{R}^{n}$~\cite{Moulines2011, Needell2014, Gower2019, Gorbunov2020, Loizou2020, Ma2018, VaswaniBach2019, Liu2018, Vaswani2019}.

\subsection{Содержание работы}

Раздел \ref{sec:main_results} содержит в себе краткое изложение основных полученных в данной работе результатов. В разделе \ref{sec:modified_gnm} разрабатывается теория модифицированного метода Гаусса--Ньютона, выводится общее правило обновления приближения решения задачи \eqref{eq:main_opt_problem}, задаются алгоритмы метода Гаусса--Ньютона с \textit{неточным оракулом} и адаптивным подбором гипрепараметров, отвечающих за длину шага метода (схема \ref{alg:gen_det_gnm}); выодится алгоритм с адаптивно настраиваемой точностью верхней оценки оптимизируемого функционала (схема \ref{alg:gen_det_flex_gnm}). В разделе \ref{sec:modified_gnm} доказывается в рамках естественных предположений сходимость любого процесса построения приближения решения задачи \eqref{eq:main_opt_problem} к стационарной точке с выполнением условий оптимальности первого порядка в случае неограниченного возрастания точности вычисления приближения решения на каждой итерации. При наличии невырожденности матрицы Якоби установлена локальная квадратичная сходимость. В условиях невырожденной сопряжённой матрицы Якоби установлена локальная линейная сходимость, а в случае достаточно точного оракула невырожденность сопряжённой матрицы Якоби приводит к глобальной линейной сходимости метода. В разделе \ref{sec:stoch_modified_gnm} представлена стохастическая версия модификации метода Гаусса--Ньютона с несколькими стратегиями обновления приближения решения задачи вида \eqref{eq:main_opt_problem}. В разделе \ref{sec:stoch_modified_gnm_analysis} исследуется сходимость предложенного метода со схемой реализации \ref{alg:gen_stoch_gnm}, выводятся условия сходимости к приближённому решению в терминах среднего. При наличии невырожденности сэмплируемой на каждой итерации сопряжённой матрицы Якоби установлена линейная сходимость к приближённому решению в среднем. Также исследуются свойства улучшенной версии схемы \ref{alg:gen_stoch_gnm} с сэмплированием двух батчей на каждой итерации для оценки шага метода (схема \ref{alg:gen_double_stoch_gnm}). Для неточных проксимальных отображений и для неограниченных функционалов с неограниченными якобианами разработана своя версия схемы \ref{alg:gen_stoch_gnm}, представленная в схеме \ref{alg:gen_stoch_unbounded_gnm}. В данном разделе предложен стохастический аналог метода Гаусса--Ньютона с адаптивно настраиваемой точностью стохастической оценки оптимизируемого функционала (схема \ref{alg:gen_stoch_flex_gnm}). В разделе \ref{sec:stoch_modified_gnm_weak_growth} представлена стратегия обновления приближения решения в стохастическом методе Гаусса--Ньютона, позволяющая при наличии условия слабого роста и невырожденности сопряжённой матрицы Якоби решить с любой наперёд заданной точностью задачу \eqref{eq:main_opt_problem} вне зависимости от размера батча сэмплируемых функций. Раздел \ref{sec:practical_considerations} посвящён вопросам реализации предложенных модификаций метода Гаусса--Ньютона на практике, предлагаются способы эффективного вычисления проксимальных отображений, рассматриваются схемы метода Гаусса--Ньютона в произвольных нормированных пространствах, демонстрируется возможность решения возникающих вопросов с помощью стандартных средств линейной алгебры и выпуклой оптимизации. В разделе \ref{sec:comparsion} описываются предложенные модификации метода Гаусса--Ньютона в классе квазиньютоновских методов, демонстрируется связь с методами оптимизации второго порядка, с методом Ньютона. В \hyperref[sec:appendix]{приложении} представлены доказательства выведенных в данной работе утверждений.

\subsection{Условные обозначения}

Введём обозначение конечномерного евклидового пространства с помощью буквы $E$ (наравне с этим обозначением будут использованы обозначения с индексацией), для этого пространства зафиксируем стандартную евклидову норму $\|~\|$. Обозначим евклидовы пространства $E_{1}$ с $\operatorname{dim}(E_{1}) = n$ и $E_{2}$ с $\operatorname{dim}(E_{2}) = m$. Определим сопряжённое евклидово пространство $E^{*}$ для пространства $E$ как пространство линейных функций над $E$. Значение в точке $x\in E$ для функции $u\in E^{*}$ определяется скалярным произведением: $\langle u, x\rangle.$ Для нормы $\|x\|,~x\in E$ имеется классическое соотношение, связывающее с нормой $\|u\|,~u\in E^{*}$:
\begin{equation*}
    \begin{cases}
        \|x\| = \max\limits_{u\in E^{*}}\left\{\langle u, x\rangle:~\|u\|\leq 1\right\};\\[5pt]
        \|u\| = \max\limits_{x\in E}\left\{\langle u, x\rangle:~\|x\|\leq 1\right\}.
    \end{cases}
\end{equation*}
Из соотношения выше следует выполнение неравенства Коши--Буняковского--Шварца: $\langle u, x\rangle\leq\|u\|\|x\|$.

Для гладкой по $x$ функции $f: E_{1}\rightarrow E_{2}$ обозначим вычисленную в точке $x\in E_{1}$ первую и вторую производную по $x$ как $\nabla_{x}f(x)$ и $\nabla_{x}^{2}f(x)$ соответственно (в случае отсутствия неоднозначности при определении переменной дифференцирования индексация у $\nabla$ опускается). Для $E_{2} \equiv \mathbb{R}$ первую и вторую производные будем называть градиентом и гессианом. Заметим, что $\nabla f(x)\in E_{1}^{*}$, $\nabla^{2}f(x): E_{1}\rightarrow E_{1}^{*}$ --- самосопряжённый линейный оператор.

Далее, введённые обозначения позволяют определить сопряжённый оператор $A^{*}:E_{2}^{*}\rightarrow E_{1}^{*}$ для оператора $A: E_{1}\rightarrow E_{2}$:
\begin{equation*}
    \left\langle u, Ax\right\rangle = \left\langle A^{*}u, x\right\rangle,~u\in E_{2}^{*}, x\in E_{1}.
\end{equation*}

Определим операторную норму для линейного оператора $A: E_{1}\rightarrow E_{2}$ как максимальное сингулярное число матрицы оператора $\sigma_{\max}(A)$:
\begin{equation*}
    \left\|A\right\| = \sigma_{\max}(A) = \max\limits_{x\in E_{1}}\left\{\left\|Ax\right\|:~\|x\|\leq 1\right\} = \sqrt{\lambda_{\max}\left(AA^{*}\right)} = \sqrt{\lambda_{\max}\left(A^{*}A\right)},
\end{equation*}
где $\lambda_{\max}(\cdot)$ --- максимальное собственное значение оператора. Дополнительно обозначим с помощью $\left\|A\right\|_{F}$ фробениусову норму оператора $A$ с матрицей $\left(a_{ij}\right)_{i,j = 1}^{m,n}$:
\begin{equation*}
    \left\|A\right\|_{F} = \sqrt{\sum\limits_{i,j = 1}^{m, n}|a_{ij}|^{2}} = \sqrt{\operatorname{Tr}\left(AA^{*}\right)} = \sqrt{\operatorname{Tr}\left(A^{*}A\right)}.
\end{equation*}
Ясно, что $\left\|A\right\|\leq\left\|A\right\|_{F}$, по свойству следа оператора $\operatorname{Tr}(\cdot)$. Также введём минимальное сингулярное число матрицы данного оператора $A$:
\begin{equation*}
    \sigma_{\min}(A) = \min\limits_{x\in E_{1}}\left\{\left\|Ax\right\|:~\|x\|\leq 1\right\}.
\end{equation*}

Для многозначного отображения $F: E_{1}\rightarrow E_{2}$ определим матрицу Якоби $F^{'}(x)$ в точке $x\in E_{1}$ как матрицу линейного оператора из $E_{1}$ в $E_{2}$:
\begin{equation*}
    F^{'}(x)h = \lim\limits_{t\rightarrow 0}\left(\frac{1}{t}\left(F(x + th) - F(x)\right)\right)\in E_{2},~h\in E_{1}.
\end{equation*}

Для линейных операторов задаётся отношение частичного порядка на конусе неотрицательно определённых матриц следующим стандартным образом:
\begin{equation*}
    A\preceq A_{1},~A_{1}\succeq A,~A: E\rightarrow E^{*},~A_{1}: E\rightarrow E^{*}\Leftrightarrow \left\langle (A_{1} - A)x,~x\right\rangle\geq 0,~\forall x\in E.
\end{equation*}
Аналогичное отношение верно и относительно сопряжённого пространства:
\begin{equation*}
    B\preceq B_{1},~B_{1}\succeq B,~B: E^{*}\rightarrow E,~B_{1}: E^{*}\rightarrow E\Leftrightarrow \left\langle u,~(B_{1} - B)u\right\rangle\geq 0,~\forall u\in E^{*}.
\end{equation*}
Заметим, что для линейного оператора $A: E_{1}\rightarrow E_{2}$ верно отношение
\begin{equation*}
    \begin{cases}
        AA^{*}\succeq\sigma_{\min}(A^{*})^{2}I_{\dim(E_{2})};\\[5pt]
        A^{*}A\succeq\sigma_{\min}(A)^{2}I_{\dim(E_{1})}.
    \end{cases}
\end{equation*}

Обозначим за $\overline{1, m}$ множество целых чисел от $1$ до $m$ включительно: $\{1,~\dots, m\}$. Обозначим через $f(x) = \operatorname{O}(h(x))$  оценку сверху функции $f$ функцией $h$ с точностью до константы и, быть может, полилогарифмических факторов. Так же через $f(x) = \Omega(h(x))$ обозначим оценку снизу функции $f$ функцией $h$ с точностью до константы и, быть может, полилогарифмических факторов. Положим также
$$f^{*} = \min\limits_{x\in E_{1}}f(x),~g^{*}(y) = \min\limits_{x\in E_{1}}g(x, y),$$
определив минимальные возможные значения по аргументу $x$ для функций $f$ и $g$ соответственно.

\section{Основные результаты}\label{sec:main_results}

В работе представлена модификация метода Гаусса--Ньютона с сильно выпуклой параметризованной локальной моделью нормы невязки системы нелинейных уравнений, сам прообраз исследуемых методов впервые предложил Юрий Евгеньевич Нестеров в своём препринте~\cite{Nesterov2020}. Используя предложенную локальную модель, построены алгоритмы детерминированной оптимизации и алгоритмы стохастической оптимизации для решения задачи \eqref{eq:main_opt_problem}. Предложены алгоритмы решения задачи \eqref{eq:main_opt_problem} с адаптивной настройкой гиперпараметров локальной модели. В анализе построенных методов заложено понятие неточного оракула, формализованное в виде отличия значения локальной модели в точке очередного приближения решения системы \eqref{eq:smooth_system} от минимального значения локальной модели на текущем шаге. В классе детерминированных методов Гаусса--Ньютона разработан с помощью предложенной локальной модели алгоритм, адаптивно учитывающий произвольное значение погрешности неточного оракула в области квадратичной сходимости. Для стохастических методов Гаусса--Ньютона выведены условия, при которых задача \eqref{eq:smooth_system} разрешима с произвольным размером батча. Среди разработанных вариаций метода Гаусса--Ньютона присутствует версия для неограниченных оптимизируемых функционалов в рамках естественных предположений. Для каждого представленного алгоритма решения задачи \eqref{eq:main_opt_problem} дан анализ сходимости с неасимптотическими оценками относительно потенциала уровня $\epsilon > 0$, используемого в качестве индикатора сходимости итеративного процесса.

Кратко результаты работы по построению и изучению модифицированного метода Гаусса--Ньютона представлены в таблице \ref{tab:methods_summary}. В ней столбец <<Схемы метода>> представляет собой собрание применимых алгоритмов в обозначенных в столбце справа теоремах. Столбец <<Условие сходимости>> содержит ссылки на потенциалы, используемые для измерения сходимости до уровня $\epsilon > 0$. Остальные столбцы описывают основные требуемые условия для успешного применения схем с указанной асимптотикой. В таблице первые четыре строки соответствуют детерминированным методам, остальные --- стохастическим. Прочерки в столбце <<Погрешность оракула>> означают использование точного направления минимизации локальной модели с некоторым масштабом на каждом шаге соответствующего строке с прочерком алгоритма.

\begin{table}[ht]
\centering
\resizebox{\columnwidth}{!}{
\begin{tabular}{c|c|c|c|c|c|c}
     \makecell{Схемы\\метода} & Теоремы & Предположения & \makecell{Условие\\сходимости} & \makecell{Количество\\итераций} &  \makecell{Размер\\батча} & \makecell{Погрешность\\оракула}\\
     \hline
     \ref{alg:gen_det_gnm}, \ref{alg:gen_det_flex_gnm} & \ref{th:DetSublinConvMain}, \ref{th:DetFlexSublinConvMain} & \ref{as:det_hat_F_der_smooth} & \eqref{eq:DetSublinConvMainConvCond} & $\operatorname{O}\left(\frac{1}{\epsilon^{2}}\right)$ & $m$ & $\operatorname{O}\left(\epsilon^{2}\right)$\\
     \hline
     \ref{alg:gen_det_gnm} & \ref{th:DetQuadConvMain} & \ref{as:det_hat_F_der_smooth} & \eqref{eq:DetQuadConvMainConvCond} & $\operatorname{O}\left(\log_{2}\left(\ln\left(\frac{1}{\alpha\epsilon}\right)\right)\right)$ & $m$ & $\geq0$\\
     \hline
     \ref{alg:gen_det_gnm}, \ref{alg:gen_det_flex_gnm} & \ref{th:2_main} & \ref{as:det_hat_F_der_smooth} & \eqref{eq:th2_main_conv_cond} & $\operatorname{O}\left(\frac{1}{\epsilon^{2}}\right)$ & $m$ & ---\\
     \hline
     \ref{alg:gen_det_gnm}, \ref{alg:gen_det_flex_gnm} & \ref{th:1_main}, \ref{lm:1_main} & \ref{as:det_hat_F_der_smooth}, \ref{as:det_hat_F_PL_condition} & \eqref{eq:1_main_conv_cond} & $\operatorname{O}\left(\ln\left(\frac{1}{\epsilon}\right)\right)$ & $m$ & ---\\
     \hline
     \ref{alg:gen_stoch_gnm}, \ref{alg:gen_stoch_unbounded_gnm}, \ref{alg:gen_stoch_flex_gnm} & \ref{th:3_main} & \ref{as:1}, \ref{as:2}, \ref{as:3}, \ref{as:4} & \eqref{eq:mean_grad_conv_cond} & $\operatorname{O}\left(\frac{1}{\epsilon^{2}}\right)$ & $\min\left\{m, \operatorname{O}\left(\frac{1}{\epsilon^{4}}\right)\right\}$ & ---\\
     \hline
     \ref{alg:gen_stoch_gnm}, \ref{alg:gen_stoch_unbounded_gnm}, \ref{alg:gen_stoch_flex_gnm} & \ref{th:4_main} & \ref{as:1}, \ref{as:2}, \ref{as:3}, \ref{as:4}, \ref{as:5} & \eqref{eq:4_main_conv_cond} & $\operatorname{O}\left(\ln\left(\frac{1}{\epsilon}\right)\right)$ & $\min\left\{m, n, \operatorname{O}\left(\frac{1}{\epsilon^{4}}\right)\right\}$ & ---\\
     \hline
     \ref{alg:gen_double_stoch_gnm} & \ref{th:double_stoch_sublin_conv_main} & \ref{as:1}, \ref{as:2}, \ref{as:3}, \ref{as:4} & \eqref{eq:mean_grad_conv_cond} & $\operatorname{O}\left(\frac{1}{\epsilon^{2}}\right)$ & $\min\left\{m, \operatorname{O}\left(\frac{1}{\epsilon^{4}}\right)\right\}$ & ---\\
     \hline
     \ref{alg:gen_double_stoch_gnm} & \ref{th:double_stoch_lin_conv_main} & \ref{as:1}, \ref{as:2}, \ref{as:3}, \ref{as:4}, \ref{as:5} & \eqref{eq:4_main_conv_cond} & $\operatorname{O}\left(\ln\left(\frac{1}{\epsilon}\right)\right)$ & $\min\left\{m, n, \operatorname{O}\left(\frac{1}{\epsilon^{4}}\right)\right\}$ & ---\\
     \hline
     \ref{alg:gen_stoch_gnm}, \ref{alg:gen_stoch_unbounded_gnm}, \ref{alg:gen_stoch_flex_gnm} & \ref{th:5_main} & \ref{as:1}, \ref{as:2}, \ref{as:3}, \ref{as:4} & \eqref{eq:mean_grad_conv_cond} & $\operatorname{O}\left(\frac{1}{\epsilon^{2}}\right)$ & $\min\left\{m, \operatorname{O}\left(\frac{1}{\epsilon^{4}}\right)\right\}$ & \eqref{eq:5_main_regime_1}, \eqref{eq:5_main_regime_2}\\
     \hline
     \ref{alg:gen_stoch_gnm}, \ref{alg:gen_stoch_unbounded_gnm}, \ref{alg:gen_stoch_flex_gnm} & \ref{th:6_main} & \ref{as:1}, \ref{as:2}, \ref{as:3}, \ref{as:4}, \ref{as:5} & \eqref{eq:4_main_conv_cond} & $\operatorname{O}\left(\ln\left(\frac{1}{\epsilon}\right)\right)$ & $\min\left\{m, n, \operatorname{O}\left(\frac{1}{\epsilon^{4}}\right)\right\}$ & \eqref{eq:6_main_regime_1}, \eqref{eq:6_main_regime_2}\\
     \hline
     \ref{alg:gen_stoch_unbounded_gnm}, \ref{alg:gen_stoch_flex_gnm} & \ref{th:gen_stoch_sublin_conv_main} & \ref{as:jacob_smoothness}, \ref{as:bounded_variance_growth} & \eqref{eq:th_stoch_prox_norm} & \eqref{eq:gen_sublin_conv_cond} & \eqref{eq:gen_sublin_conv_cond} & \eqref{eq:gen_sublin_conv_cond}\\
     \hline
     \ref{alg:gen_stoch_unbounded_gnm}, \ref{alg:gen_stoch_flex_gnm} & \ref{th:gen_stoch_lin_conv_main} & \ref{as:jacob_smoothness}, \ref{as:bounded_variance_growth}, \ref{as:prox_PL_condition} & \eqref{eq:th_stoch_f2_val} & \eqref{eq:gen_lin_conv_cond} & \eqref{eq:gen_lin_conv_cond} & \eqref{eq:gen_lin_conv_cond}\\
     \hline
     \ref{alg:gen_double_stoch_gnm} & \ref{th:weak_growth_condition_main} & \ref{as:2}, \ref{as:3}, \ref{as:5}, \ref{as:jacob_smoothness} & \eqref{eq:th_stoch_f2_val} & $\operatorname{O}\left(\ln\left(\frac{1}{\epsilon}\right)\right)$ & $\in\overline{1,~m},~m\leq n$ & ---\\
\end{tabular}
}
\caption{Основные характеристики разработанных модификаций}\label{tab:methods_summary}
\end{table}

\section{Модифицированный метод Гаусса--Ньютона}\label{sec:modified_gnm}

\subsection{Модель оптимизируемого функционала}

Вернёмся к задаче поиска решения $x^{*}\in E_{1}$ гладкой нелинейной системы уравнений:
\begin{equation}\label{eq:smooth_system}
    F(x) = \mathbf{0}_{m},
\end{equation}
где $F: E_{1}\rightarrow E_{2}$ --- гладкое многозначное отображение с матрицей Якоби $F^{'}(x),~x\in E_{1}$. Для оценки близости текущего приближения к решению системы уравнений \eqref{eq:smooth_system} рассмотрим следующую функцию невязки для системы с набором функций $\hat{F}(x) \overset{\operatorname{def}}{=} \frac{1}{\sqrt{m}}F(x)$:
$$\hat{f}_{1}(x) \overset{\operatorname{def}}{=} \frac{1}{\sqrt{m}}\left\|F(x)\right\| = \left\|\hat{F}(x)\right\|.$$

Используя функцию невязки $\hat{f}_{1}(x)$ можно решить задачу \eqref{eq:smooth_system} через сведение к задаче оптимизации без ограничений:
\begin{equation}\label{eq:main_opt_problem}
    \hat{f}_{1}^{*} = \min\limits_{x\in E_{1}}\left\{\hat{f}_{1}(x) = \frac{1}{\sqrt{m}}\left\|F(x)\right\| = \frac{1}{\sqrt{m}}\left\|\left(F_{1}(x),~\dots, F_{m}(x)\right)^{*}\right\|\right\}.
\end{equation}

Существование решения задачи \eqref{eq:smooth_system} равносильно $\hat{f}_{1}^{*} = \hat{f}_{1}(x^{*}) = 0$. В работе рассматривается итеративная процедура решения задачи \eqref{eq:main_opt_problem}, основанная на минимизации \textit{локальной модели} оптимизируемого функционала:
\begin{equation*}
    \phi(x, y) \overset{\operatorname{def}}{=} \left\|\hat{F}(x) + \hat{F}^{'}(x)(y - x)\right\|,~(x, y)\in E_{1}^{2},~\hat{F}^{'}(x) = \frac{1}{\sqrt{m}}F^{'}(x).
\end{equation*}
Для классического метода Гаусса--Ньютона на каждой итерации $k\in\mathbb{Z}_{+}$ очередное приближение решения \eqref{eq:smooth_system} вычисляется поиском точки минимума выпуклой по $y$ функции $\phi(x, y)$:
$$x_{k + 1}\in \Argmin\limits_{y\in E_{1}}\left\{\phi(x_{k}, y)\right\}.$$

Однако добавление регуляризации к классической схеме метода Гаусса--Ньютона позволяет установить свойства локальной и глобальной эффективности всего метода. В данной работе проводится анализ регуляризованного метода Гаусса--Ньютона с модифицированной локальной моделью оптимизируемого функционала, предложенного в \cite{Nesterov2020}. Для этого введём изначальные предположения о решаемой задаче. Рассмотрим $\mathcal{F}\subseteq E_{1}$ --- замкнутое выпуклое множество с непустым подмножеством внутренних точек.

\begin{assumption}\label{as:det_hat_F_der_smooth}
    Пусть многозначное отображение $\hat{F}(x)$ является гладким на $\mathcal{F}$ с Липшиц--непре\-рывной матрицей Якоби:
    \begin{equation}\label{eq:hat_F_der_Lipschitz}
        \exists L_{\hat{F}} > 0:~\left\|\hat{F}^{'}(y) - \hat{F}^{'}(x)\right\|_{F}\leq L_{\hat{F}}\|y - x\|,~\forall (x, y)\in\mathcal{F}^{2}.
    \end{equation}
\end{assumption}
Из предположения \ref{as:det_hat_F_der_smooth} по свойству соотношения операторной нормы и нормы Фробениуса следует неравенство:
$$\left\|\hat{F}^{'}(y) - \hat{F}^{'}(x)\right\|\leq L_{\hat{F}}\|y - x\|,~\forall (x, y)\in\mathcal{F}^{2}.$$
Введём понятие множества уровня $\mathcal{L}(v)$ для функции $\hat{f}_{1}$:
$$\mathcal{L}(v) \overset{\operatorname{def}}{=} \left\{x: \hat{f}_{1}(x)\leq v\right\},$$
предположив, что $$\mathcal{L}(\hat{f}_{1}(x_{0}))\subseteq\mathcal{F},~x_{0}\in\mathcal{F}\text{ --- начальное приближение решения,}$$
то есть для каждого начального приближения $x_{0}$ во всей работе размер $\mathcal{F}$ предполагается достаточно большим, чтобы вся последовательность $\left\{x_{k}:~\hat{f}_{1}(x_{k})\leq\hat{f}_{1}(x_{k - 1})\right\}_{k\in\mathbb{N}}$ принадлежала $\mathcal{F}$.
\begin{assumption}\label{as:det_hat_F_PL_condition}
    Пусть для многозначного отображения выполнено условие Поляка--Лоясиевича:
    \begin{equation}\label{eq:det_hat_F_PL_condition}
        \exists\mu >0,~\sigma_{\min}(\hat{F}^{'}(x)^{*})\geq\sqrt{\mu},~\forall x\in\mathcal{F}.
    \end{equation}
\end{assumption}
Из предположения \ref{as:det_hat_F_PL_condition} неявно следует неравенство $\dim(E_{1})\leq\dim(E_{2})$. Само предположение \ref{as:det_hat_F_PL_condition} называется условием Поляка--Лоясиевича, так как из него следует неравенство Поляка--Лоясиевича для функции $\hat{f}_{2}$:
\begin{equation*}
    \begin{aligned}
        \left\|\nabla \hat{f}_{2}(x)\right\|^{2} &= \left\|2\hat{F}^{'}(x)^{*}\hat{F}(x)\right\|^{2}\geq4\mu\left\|\hat{F}(x)\right\|^{2} = 4\mu\hat{f}_{2}(x),~x\in\mathcal{F};\\
        \hat{f}_{2}(x) &\overset{\operatorname{def}}{=} \left(\hat{f}_{1}(x)\right)^{2},~x\in E_{1}.
    \end{aligned}
\end{equation*}
В работе \cite{Nesterov2020} предложена следующая локальная регуляризованная модель оптимизируемой функции $\hat{f}_{1}(y)$:
\begin{equation}
    \hat{f}_{1}(y)\leq\frac{\hat{f}_{1}(x)}{2} + \frac{\left(\phi(x, y)\right)^{2}}{2\hat{f}_{1}(x)} + \frac{L}{2}\|y - x\|^{2},~L\geq L_{\hat{F}},~(x, y)\in\mathcal{F}^{2}.
\end{equation}
Однако сейчас рассматривается более общая форма данной модели, называемая \textit{общей локальной моделью}, вывод которой представлен в лемме \ref{lm:aux_det_upper_model}:
\begin{equation*}
    \hat{f}_{1}(y)\leq\psi_{x, L, \tau}(y) = \frac{\tau}{2} + \frac{\left(\phi(x, y)\right)^{2}}{2\tau} + \frac{L}{2}\|y - x\|^{2},~L\geq L_{\hat{F}},~\tau > 0,~(x, y)\in\mathcal{F}^{2}.
\end{equation*}
Общая локальная модель позволяет ввести правило точного обновления приближения решения $x$ в итерационной схеме регуляризованного метода Гаусса---Ньютона:
$$T_{L,\tau}(x) \overset{\operatorname{def}}{=} \argmin\limits_{y\in E_{1}}\left\{\psi_{x, L, \tau}(y)\right\}.$$

\subsection{Анализ схемы метода}

Разработанная схема обновления $x$ объединяет в себе ранее предложенные модификации метода Гаусса--Ньютона, различающиеся выбором гиперпараметра $\tau$: $\tau = \phi(x, y)$ \cite{Nesterov2007} и $\tau = \hat{f}_{1}(x)$ \cite{Nesterov2020}. В отличие от случая c $\tau = \phi(x, y)$, модель $\psi_{x, L, \hat{f}_{1}(x)}(y)$ является гладкой по $y$ и позволяет однозначно вычислять $T_{L, \hat{f}_{1}(x)}(x)$ в силу сильной выпуклости по $y$. Также в данной работе для общей локальной модели расширено множество значений $\tau$, при которых имеет место сходимость модифицированного метода Гаусса--Ньютона. Предлагаемая в работе модификация метода Гаусса--Ньютона является дальнейшим развитием метода \textit{нормализованных квадратов}, предложенного Юрием Нестеровым в своём препринте \cite{Nesterov2020}, и описана в схеме \ref{alg:gen_det_gnm}.

\RestyleAlgo{boxruled}
\begin{algorithm}[!ht]{}
\caption{\textbf{Общий метод нормализованных квадратов с неточным проксимальным отображением}}
\label{alg:gen_det_gnm}
\textbf{Вход:}
    \begin{equation*}
        \scalemath{1.0}{
        \begin{cases}
            x_{0}\in E_{1},~\mathcal{L}(\hat{f}_{1}(x_{0}))\subseteq\mathcal{F}\text{ --- начальное приближение},~x_{-1} = x_{0};\\
            \mathcal{E}(\cdot)\text{ --- функция погрешности проксимального отображения};\\
            N\in\mathbb{N}\text{ --- количество итераций метода};\\
            L\text{ --- оценка локальной постоянной Липшица},~L\in (0,~L_{\hat{F}}],~L_{0} = L;\\
            \mathcal{T}(\cdot)\text{ --- функция, определяющая значение $\tau$}.
        \end{cases}
        }
    \end{equation*}
    \vspace{0.2cm}
    \textbf{Повторять для $k = 0, 1,~\dots, N - 1$:}
    \begin{itemize}
        \item[1.] определить $\tau_{k} = \mathcal{T}(x_{k}, L_{k}, \varepsilon_{k})$,~$\varepsilon_{k} = \mathcal{E}(k, x_{k}, x_{k - 1})$;
        \item[2.] вычислить такой $x_{k + 1}\in E_{1}$, что $\psi_{x_{k}, L_{k}, \tau_{k}}(x_{k + 1}) - \psi_{x_{k}, L_{k}, \tau_{k}}(T_{L_{k}, \tau_{k}}(x_{k}))\leq\varepsilon_{k}$ и $\hat{f}_{1}(x_{k}) - \psi_{x_{k}, L_{k}, \tau_{k}}(x_{k + 1})\geq 0$;
        \item[3.] если $\hat{f}_{1}(x_{k + 1}) > \psi_{x_{k}, L_{k}, \tau_{k}}(x_{k + 1})$, то положить $L_{k} := \min\left\{2L_{k}, 2L_{\hat{F}}\right\}$ и\\вернуться к пункту 1;
        \item[4.] $L_{k + 1} = \max\left\{\frac{L_{k}}{2},~L\right\}$.
    \end{itemize}
    \vspace{0.2cm}
    \textbf{Выход:} $x_{N}$.
\end{algorithm}

В предложенной схеме рассматривается так называемый неточный оракул, который на каждой внешней итерации $k$ в качестве $x_{k + 1}$ возвращает не точку минимума локальной модели, а приближение значения $T_{L_{k}, \tau_{k}}(x_{k})$ с погрешностью $\varepsilon_{k}\geq0$. Схема \ref{alg:gen_det_gnm} обладает определённым уровнем общности, выраженным в возможности динамически менять $\tau_{k}$ и $\varepsilon_{k}$ на каждой итерации, формально это описано с помощью отображений $\mathcal{E}(\cdot)$ и $\mathcal{T}(\cdot)$, которые не всегда от всех обозначенных аргументов существенно зависят, однако позволяют обозначить формально частичную произвольность в выборе $\tau_{k}$ и $\varepsilon_{k}$, уточняемую в тексте работы в менее строгой форме, но достаточной для понимания концепции метода. Кроме общности, в разработанной схеме содержится адаптивный подбор локальной постоянной Липшица, построенный по принципу бинарного поиска на отрезке $[L,~2L_{\hat{F}}]$, причём на практике не обязательно знать верхнюю границу отрезка поиска $L_{k}$, так как для $L_{k}\geq L_{\hat{F}}$ локальная модель $\psi_{x_{k}, L_{k}, \tau_{k}}(\cdot)$ всегда корректно определена на $\mathcal{F}$ и неравенство в пункте $3$ схемы \ref{alg:gen_det_gnm} выполнено с противоположным знаком. Стоит отметить, что в схеме \ref{alg:gen_det_gnm} на практике желательно выбирать достаточно малое значение $\varepsilon_{k}\geq0$, чтобы было гарантированное уменьшение (выполнялось $\hat{f}_{1}(x_{k}) > \psi_{x_{k}, L_{k}, \tau_{k}}(x_{k + 1})$) до достижения области неоднозначности, существование которой обусловлено наличием неточности при вычислении $x_{k + 1}$ на $k$--ой итерации, причём не для каждого способа выбора $\tau_{k}$ может быть выполнено гарантированное уменьшение, и в данной работе условия теорем определяют способы, позволяющие добиться обозначенного уменьшения. Например, в схеме \ref{alg:gen_det_gnm} на практике это часто приводит к присвоению $x_{k + 1} = x_{k}$, если на $k$--ом шаге не удалось подобрать $x_{k + 1}$, для которого верно $\hat{f}_{1}(x_{k})\geq\psi_{x_{k}, L_{k}, \tau_{k}}(x_{k + 1})$ при $\tau_{k} = \hat{f}_{1}(x_{k})$. Также на практике вместо правила $\hat{f}_{1}(x_{k})\geq\psi_{x_{k}, L_{k}, \tau_{k}}(x_{k + 1})$ для обеспечения корректности схемы метода оптимизации \ref{alg:gen_det_gnm} могут применяться следующие процедуры:

\begin{itemize}
    \item выбор достаточно малого $\varepsilon_{k}\geq0$ для гарантии выполнения отношения $x_{k + 1}\in\mathcal{F}$: $\mathcal{L}(\hat{f}_{1}(x_{k}) + \varepsilon_{k})\subseteq\mathcal{F}$; 
    \item введение <<процедуры коррекции>>, например, проекции на множество $\mathcal{F}$ для каждого только что вычисленного $x_{k + 1}$.
\end{itemize}

Неточное вычисление $x_{k + 1}$ в данной работе обозначено в виде <<чёрного ящика>> и на практике может быть представлено другим итерационным методом, например, методом градиентного спуска, минимизирующим функционал $\psi_{x_{k}, L_{k}, \tau_{k}}(\cdot)$ на $k$--ом шаге метода Гаусса--Ньютона; контроль за точностью вычисления $x_{k + 1}$ представлен с помощью сравнения значения функции $\psi_{x_{k}, L_{k}, \tau_{k}}(x_{k})$ с минимальным значением $\psi_{x_{k}, L_{k}, \tau_{k}}(T_{L_{k}, \tau_{k}}(x_{k}))$, хотя эквивалентно можно сравнивать норму градиента локальной модели $\psi_{x_{k}, L_{k}, \tau_{k}}(x_{k + 1})$ (см. раздел \ref{subsec:gradient_control}).

Прежде чем перейти к оценке сходимости последовательности $\left\{x_{k}\right\}_{k\in\mathbb{Z}_{+}}$, построенной по предложенной схеме, рассмотрим две величины, оценивающие близость текущего приближения $x_{k}$ к стационароной точке:
\begin{itemize}
    \itemsep=-4pt
    \item норма обобщённого проксимального градиента --- $\left\|L_{k}\left(T_{L_{k}, \tau_{k}}(x_{k}) - x_{k}\right)\right\|$;
    \item приращение локальной модели --- $\Delta_{r}(x_{k}) \overset{\operatorname{def}}{=} \hat{f}_{2}(x_{k}) - \min\limits_{y\in E_{1}}\left\{\left(\phi(x_{k}, y)\right)^{2}:~\|y - x_{k}\|\leq r\right\},~r > 0$.
\end{itemize}
Обе величины позволяют определить множества стационарных точек, причём нетрудно установить эквивалентность данных определений:
\begin{itemize}
    \itemsep=-4pt
    \item $\left\{x^{*}:~x^{*}\in E_{1},~\|L\left(T_{L, \tau}(x^{*}) - x^{*}\right)\| = 0,~\forall L > 0,~\forall\tau > 0\right\}$;
    \item $\left\{x^{*}:~x^{*}\in E_{1},~\Delta_{r}(x^{*}) = 0,~\forall r > 0\right\}$.
\end{itemize}
С помощью введённых величин близости к стационарной точке установлена глобальная сублинейная сходимость к окрестности стационарной точки для метода, реализованного по схеме \ref{alg:gen_det_gnm}.

\begin{theorem}\label{th:DetSublinConvMain}
    Пусть выполнено предположение \ref{as:det_hat_F_der_smooth}, $k\in\mathbb{N},~r > 0$. Тогда для метода Гаусса--Ньютона, реализованного по схеме \ref{alg:gen_det_gnm} с $\tau_{k} = \hat{f}_{1}(x_{k})$, $\varepsilon_{k} = \varepsilon \geq 0$, верны следующие оценки:
    \begin{equation*}
        \begin{cases}
            &\frac{8L_{\hat{F}}^{2}}{L}\left(\varepsilon + \frac{\left(\hat{f}_{1}(x_{0}) - \hat{f}_{1}(x_{k})\right)}{k}\right)\geq\min\limits_{i\in\overline{0, k - 1}}\left\{\left\|2L_{\hat{F}}\left(T_{2L_{\hat{F}}, \hat{f}_{1}(x_{i})}(x_{i}) - x_{i}\right)\right\|^{2}\right\};\\[10pt]
            &L_{\hat{F}}\left(\varepsilon + \frac{\left(\hat{f}_{1}(x_{0}) - \hat{f}_{1}(x_{k})\right)}{k}\right)\geq\min\limits_{i\in\overline{0, k - 1}}\left\{2\left(L_{\hat{F}}r\right)^{2}\varkappa\left(\frac{\Delta_{r}(x_{i})}{4\hat{f}_{1}(x_{i})L_{\hat{F}}r^{2}}\right)\right\};
        \end{cases}
    \end{equation*}
    где $\varkappa(t) = \frac{t^{2}}{2}\mathds{1}_{\left\{t\in[0, 1]\right\}} + \left(t - \frac{1}{2}\right)\mathds{1}_{\left\{t > 1\right\}}.$
\end{theorem}

В теореме \ref{th:DetSublinConv} утверждается аддитивность вклада в оценку сходимости двух факторов: погрешность вычисления проксимального отображения и количество итераций метода, из--за этого для достижения минимальной нормы проксимального градиента на уровне $\epsilon > 0$ необходимо часть от $\epsilon^{2}$ покрыть с помощью достаточно малого $\varepsilon\geq0$, а оставшуюся часть --- с помощью достаточно большого количества итераций. Это неасимптотическое условие сходимости к уровню $\epsilon$ представляет собой следующее выражение:
\begin{equation}\label{eq:DetSublinConvMainConvCond}
    \begin{aligned}
        &\min\limits_{i\in\overline{0, k - 1}}\left\{\left\|2L_{\hat{F}}\left(T_{2L_{\hat{F}},\hat{f}_{1}(x_{i})}(x_{i}) - x_{i}\right)\right\|\right\}\leq\epsilon.
    \end{aligned}
\end{equation}
Формально данное условие описывается с помощью системы неравенств:
\begin{equation*}
    \begin{cases}
        \frac{8L_{\hat{F}}^{2}\varepsilon}{L}\leq r\epsilon^{2},~r\in(0, 1);\\[5pt]
        \frac{8L_{\hat{F}}^{2}\left(\hat{f}_{1}(x_{0}) - \hat{f}_{1}(x_{k})\right)}{Lk}\leq (1 - r)\epsilon^{2}.
    \end{cases}
\end{equation*}
Из неравенств выводятся максимальное значение погрешности $\varepsilon$ и минимальное количество итераций $k$:
$$\varepsilon = \frac{r\epsilon^{2} L}{8L_{\hat{F}}^{2}} = \operatorname{O}\left(\epsilon^{2}\right),~k = \left\lceil\frac{8L_{\hat{F}}^{2}\hat{f}_{1}(x_{0})}{(1 - r)\epsilon^{2}L}\right\rceil = \operatorname{O}\left(\frac{1}{\epsilon^{2}}\right).$$
Также полученные оценки указывают на ускорение метода при сужении отрезка поиска постоянной Липшица на каждой итерации вокруг истинного значения $L_{\hat{F}}$. При выборе адаптивной стратегии вычисления $x_{k + 1}$ с постепенно уменьшающейся погрешностью $\varepsilon_{k}$ до нулевого предельного значения возможно приближение к стационарной точке с любой наперёд заданной точностью (следствие \ref{th:SubLinConvCor1}), более того, при фиксированном $x_{0}$ все получаемые стационарные точки $x^{*}$ принадлежат связному множеству (следствие \ref{th:SubLinConvCor2}), хотя не все из них являются решениями системы \eqref{eq:smooth_system}, даже возможен случай, в котором ни одна из полученных стационарных точек не будет решением системы уравнений \eqref{eq:smooth_system}. Для наличия решений системы уравнений необходима совместность системы, в следующем утверждении содержатся условия локальной сходимости схемы \ref{alg:gen_det_gnm} при наличии разрешимости \eqref{eq:smooth_system}.

\begin{theorem}\label{th:DetQuadConvMain}
    Пусть выполнено предположение \ref{as:det_hat_F_der_smooth}, пусть для метода Гаусса--Ньютона со схемой \ref{alg:gen_det_gnm} существует $x^{*}\in\mathcal{L}(\hat{f}_{1}(x_{0})),~\hat{F}(x^{*}) = \mathbf{0}_{m}$ --- решение с $\sigma_{\min}\left(\hat{F}^{'}(x^{*})\right)\geq\varsigma > 0$. Если выполнено $\varsigma > \frac{2L_{\hat{F}}}{\alpha}$ при некотором фиксированном $\alpha\in\left(0,~1\right)$ для всех $\varepsilon_{k}\geq 0$, $k\in\mathbb{Z}_{+}$ в схеме \ref{alg:gen_det_gnm} с
    \begin{equation*}
        \begin{cases}
            \|x_{k} - x^{*}\| < \frac{\varsigma}{L_{\hat{F}}} - \frac{2}{\alpha};\\[5pt]
            0 < \tau_{k}\leq\frac{\left(\left(\alpha\left(\varsigma - L_{\hat{F}}\|x_{k} - x^{*}\|\right) - \frac{3L_{\hat{F}}}{2}\right)^{2} - \frac{L_{\hat{F}}^{2}}{4}\right)\|x_{k} - x^{*}\|^{4}}{\|x_{k} - x^{*}\|^{2}L_{k} + 2\varepsilon_{k}};
        \end{cases}
    \end{equation*}
    то $x_{k + 1}\in\mathcal{L}(\hat{f}_{1}(x_{0}))$ и
    \begin{equation*}
        \begin{aligned}
            \|x_{k + 1} - x^{*}\|&\leq\frac{\frac{3L_{\hat{F}}\|x_{k} - x^{*}\|^{2}}{2} + \sqrt{\|x_{k} - x^{*}\|^{2}\left(\tau_{k}L_{k} + \frac{L_{\hat{F}}^{2}\|x_{k} - x^{*}\|^{2}}{4}\right) + 2\tau_{k}\varepsilon_{k}}}{\varsigma - L_{\hat{F}}\|x_{k} - x^{*}\|}\leq \alpha\|x_{k} - x^{*}\|^{2}.
        \end{aligned}
    \end{equation*}
Если не существует такого $\alpha\in(0,~1)$, то в схеме \ref{alg:gen_det_gnm} при выборе
\begin{equation*}
    \begin{cases}
        \tau_{k} = c_{1}\|x_{k} - x^{*}\|^{2},~c_{1} > 0;\\[5pt]
        \varepsilon_{k} = c_{2}\|x_{k} - x^{*}\|^{2},~c_{2}\geq0;
    \end{cases}
\end{equation*}
в области
$$\|x_{k} - x^{*}\|\leq\frac{\varsigma}{\frac{5L_{\hat{F}}}{2} + \sqrt{2c_{1}L_{\hat{F}} + \frac{L_{\hat{F}}^{2}}{4} + 2c_{1}c_{2}}},~k\in\mathbb{Z}_{+}$$
выполнена следующая оценка:
\begin{equation*}
    \begin{aligned}
        \|x_{k + 1} - x^{*}\|&\leq\frac{\frac{3L_{\hat{F}}\|x_{k} - x^{*}\|^{2}}{2} + \sqrt{\|x_{k} - x^{*}\|^{2}\left(\tau_{k}L_{k} + \frac{L_{\hat{F}}^{2}\|x_{k} - x^{*}\|^{2}}{4}\right) + 2\tau_{k}\varepsilon_{k}}}{\varsigma - L_{\hat{F}}\|x_{k} - x^{*}\|}\leq\|x_{k} - x^{*}\|,~x_{k + 1}\in\mathcal{L}(\hat{f}_{1}(x_{0})).
    \end{aligned}
\end{equation*}
\end{theorem}
Согласно теореме \ref{th:DetQuadConv}, в условиях невырожденности достаточно близкое нахождение к оптимуму позволяет решить задачу \eqref{eq:main_opt_problem} с квадратичной скоростью, затратив в худшем случае не больше
$$k = \left\lceil\frac{1}{\ln(2)}\left(\ln\left(\ln\left(\frac{1}{\alpha\epsilon}\right)\right) - \ln\left(\ln\left(\frac{L_{\hat{F}}}{\alpha\varsigma}\right)\right)\right)\right\rceil = \operatorname{O}\left(\log_{2}\left(\ln\left(\frac{1}{\alpha\epsilon}\right)\right)\right)$$
итераций для приближения к точке оптимума на расстояние
\begin{equation}\label{eq:DetQuadConvMainConvCond}
    \begin{aligned}
        &\|x_{k} - x^{*}\|\leq\epsilon
    \end{aligned}
\end{equation}
в конце итерационного процесса, если выполнено структурное ограничение $\varsigma > \frac{2L_{\hat{F}}}{\alpha}$. В утверждении теоремы \ref{th:DetQuadConv} при выполнении ограничения $\varsigma > \frac{2L_{\hat{F}}}{\alpha}$ указана произвольность выбора $\varepsilon_{k}\geq0$ и $L_{k} > 0$, которая компенсируется выбором $\tau_{k} > 0$, что означает перераспределение вклада $L_{k}$ в оптимизируемом функционале на $\tau_{k}$, при этом произвольность $\varepsilon_{k}$ для неточного оракула оставляет единственным ограничением необходимость выбора $x_{k + 1}$, для которого $\hat{f}_{1}(x_{k}) - \psi_{x_{k}, L_{k}, \tau_{k}}(x_{k + 1})\geq 0$, хотя это ограничение заложено схемой \ref{alg:gen_det_gnm}, в ходе доказательства теоремы \ref{th:DetQuadConv} не используется и может быть проигнорировано, так как увеличение $\varepsilon_{k}$ ведёт к уменьшению $\tau_{k}$, увеличивая оптимизируемый функционал и компенсируя рост погрешности поиска $x_{k + 1}$, что автоматически ведёт к выполнению $\hat{f}_{1}(x_{k}) - \psi_{x_{k}, L_{k}, \tau_{k}}(x_{k + 1})\geq 0$, однако на практике всё--таки полезно работать с не слишком большими величинами, избегая значительных погрешностей в вычислениях с плавающей точкой. Если структурное ограничение $\varsigma > \frac{2L_{\hat{F}}}{\alpha}$ не выполнено, то суперлинейная сходимость имеет место, но только в более узкой области и при непроизвольном выборе $(\tau_{k}, \varepsilon_{k})$. Условия невырожденности в теореме \ref{th:DetQuadConv} позволяют локально быстро решить задачу оптимизации, однако они неявно требуют выполнения соотношения $\dim(E_{1})\leq\dim(E_{2})$, то есть требуется при выполнении совместности системы наличие количества уравнений, не уступающего количеству параметров (следствие \ref{th:DetQuadConvCor1}). Стоит также заметить, что в случае $\varsigma > \frac{2L_{\hat{F}}}{\alpha}$ для $\tau_{k} = \phi(x_{k}, y)$ радиус локальной квадратичной сходимости в 4 раза меньше (Theorem 3.4, \cite{Nesterov2007}).

Для задач с $\dim(E_{1}) > \dim(E_{2})$ гарантия наличия $\varsigma > 0$ уже пропадает, хотя остаётся возможность совместности системы уравнений, и если предположить невырожденность системы в виде условия \ref{as:det_hat_F_PL_condition}, то имеет место локальная линейная сходимость к решению системы \eqref{eq:smooth_system}, согласно изложенному ниже утверждению.

\begin{theorem}\label{th:glob_sub_lin_and_lin_conv_main}
    Допустим выполнение предположений \ref{as:det_hat_F_der_smooth} и \ref{as:det_hat_F_PL_condition} для метода Гаусса--Ньютона со схемой реализации \ref{alg:gen_det_gnm}, в которой $\tau_{k} = \hat{f}_{1}(x_{k})$. Тогда в схеме \ref{alg:gen_det_gnm} для последовательности $\left\{x_{k}\right\}_{k\in\mathbb{Z}_{+}}$ выполняются следующие соотношения:
    \begin{equation*}
        \hat{f}_{1}(x_{k + 1})\leq\varepsilon_{k} + \begin{sqcases}
            \frac{\hat{f}_{1}(x_{k})}{2} + \frac{L_{\hat{F}}}{\mu}\hat{f}_{2}(x_{k})\leq\frac{3}{4}\hat{f}_{1}(x_{k})\text{, если }\hat{f}_{1}(x_{k})\leq\frac{\mu}{4L_{\hat{F}}};\\[5pt]
            \hat{f}_{1}(x_{k}) - \frac{\mu}{16L_{\hat{F}}}\text{, иначе}.
        \end{sqcases}
    \end{equation*}
    Если при генерации последовательности $\left\{x_{k}\right\}_{k\in\mathbb{Z}_{+}}$ была зафиксирована $L_{k} = L_{\hat{F}}$, то данные соотношения выражаются по--другому:
    \begin{equation*}
        \hat{f}_{1}(x_{k + 1})\leq\varepsilon_{k} + \begin{sqcases}
            \frac{\hat{f}_{1}(x_{k})}{2} + \frac{L_{\hat{F}}}{2\mu}\hat{f}_{2}(x_{k})\leq\frac{3}{4}\hat{f}_{1}(x_{k})\text{, если }\hat{f}_{1}(x_{k})\leq\frac{\mu}{2L_{\hat{F}}};\\[5pt]
            \hat{f}_{1}(x_{k}) - \frac{\mu}{8L_{\hat{F}}}\text{, иначе}.
        \end{sqcases}
    \end{equation*}
\end{theorem}

Теорема \ref{th:glob_sub_lin_and_lin_conv} утверждает глобальную сублинейную сходимость к точке оптимума с локальной линейной сходимостью на старших итерациях при малых значениях оптимизируемого функционала. Данное утверждение фиксирует главную особенность метода Гаусса--Ньютона --- наличие локальной линейной сходимости в условии Поляка--Лоясиевича с коэффициентом линейной сходимости, не зависящим от значения $\sqrt{\mu} > 0$, отделяющего минимальное сингулярное число матрицы $\hat{F}^{'}(x)^{*}$ от нуля. Если сравнить модификацию с $\tau_{k} = \hat{f}_{1}(x_{k})$ и модификацию с $\tau_{k} = \phi(x_{k}, y)$, то можно сделать вывод о том, что упрощение поиска $x_{k + 1}$ на каждой итерации в худшем случае замедлило скорость локальной линейной сходимости метода ($\frac{3}{4}$ при $\tau_{k} = \hat{f}_{1}(x_{k})$ против $\frac{1}{2}$ при $\tau_{k} = \phi(x_{k}, y)$, Theorem 4.4, \cite{Nesterov2007}), однако возможность явно выразить $x_{k + 1}$ при $\tau_{k} = \hat{f}_{1}(x_{k})$ позволяет вывести глобальную линейную сходимость \cite{Nesterov2020}. И, как в случае теоремы \ref{th:DetSublinConv}, подбор монотонно убывающей последовательности $\left\{\varepsilon_{k}\right\}_{k\in\mathbb{Z}_{+}}$ с нулевым пределом позволяет вычислить решение задачи \eqref{eq:main_opt_problem} с любой наперёд заданной точностью (следствие \ref{th:glob_sub_lin_and_lin_conv_cor}).

Кроме гарантии линейной сходимости к решению системы \eqref{eq:smooth_system}, условие Поляка--Лоясиевича позволяет оценить расстояние текущего приближения до решения задачи, неасимптотические границы на данное расстояние представлены в теореме \ref{th:track_glob_bound}.
\begin{theorem}\label{th:track_glob_bound_main}
    Пусть выполнены предположения \ref{as:det_hat_F_der_smooth} и \ref{as:det_hat_F_PL_condition} для метода Гаусса--Ньютона со схемой реализации \ref{alg:gen_det_gnm}, в которой $\tau_{k} = \hat{f}_{1}(x_{k})$, $\varepsilon_{k} = 0$, $k\in\mathbb{Z}_{+}$. Тогда существует решение $x^{*}\in\mathcal{F}$ задачи \eqref{eq:smooth_system}, такое, что $\hat{f}_{1}(x^{*}) = 0$ и $\|x_{0} - x^{*}\|\leq4\hat{f}_{1}(x_{0})\sqrt{\frac{2L_{\hat{F}}}{\mu L}}.$
\end{theorem}
В теореме \ref{th:track_glob_bound}, как и в теореме \ref{th:glob_sub_lin_and_lin_conv}, информация о точном значении постоянной Липшица позволяет в худшем случае быстрее приближаться к решению, находиться ближе к решению на каждой итерации (следствия \ref{th:glob_sub_lin_and_lin_conv_cor_2} и \ref{th:track_glob_bound_cor}). Если дополнительно верны условия теоремы \ref{th:DetQuadConv}, то значение $\hat{f}_{1}(x_{k})$ можно использовать при подборе $\tau_{k}$ для достижения суперлинейной сходимости (лемма \ref{lm:aux_tau_schedule}) в области
$$\|x_{k} - x^{*}\|\leq\min\left\{\frac{2\varsigma}{5L_{\hat{F}}},~\frac{1}{12L_{\hat{F}}}\left(\left(3M_{\hat{F}} + 5\varsigma\right) - \sqrt{\left(3M_{\hat{F}} + 5\varsigma\right)^{2} - 24\varsigma^{2}}\right)\right\},$$
предположив ограниченность нормы матрицы Якоби:
\begin{equation*}
    \begin{aligned}
        &\text{то есть для }\left\|\hat{F}^{'}(x)\right\|\leq M_{\hat{F}},~x\in \mathcal{F},\text{ возьмём }\varepsilon_{k} = 0,~\tau_{k} = \hat{f}_{1}(x_{k})\Rightarrow\\
        &\Rightarrow\frac{2\hat{f}_{1}(x_{k})}{3M_{\hat{F}}}<\|x_{k} - x^{*}\|\leq\min\left\{\frac{2\varsigma}{5L_{\hat{F}}},~\frac{1}{12L_{\hat{F}}}\left(\left(3M_{\hat{F}} + 5\varsigma\right) - \sqrt{\left(3M_{\hat{F}} + 5\varsigma\right)^{2} - 24\varsigma^{2}}\right)\right\}<\frac{\varsigma}{L_{\hat{F}}}.
    \end{aligned}
\end{equation*}
Главным же недостатком данного способа выбора $\tau_{k}$ является необходимость точно вычислять каждое приближение решения $x_{k + 1}$. Выполнение условия Поляка--Лоясиевича также позволяет ограничить значение $\|x_{k} - x^{*}\|$ (лемма \ref{lm:aux_tau_schedule}, теорема \ref{th:track_glob_bound}) с двух сторон в случае ограниченных якобианов $\left\|\hat{F}^{'}(x)\right\|\leq M_{\hat{F}}$ и точного оракула, используя значение $\hat{f}_{1}(x_{k})$:
\begin{equation*}
    \begin{aligned}
        &\hat{f}_{1}(x_{k})\leq M_{\hat{F}}\left\|x_{k} - x^{*}\right\| + \frac{L_{\hat{F}}}{2}\left\|x_{k} - x^{*}\right\|^{2},~\|x_{k} - x^{*}\|\leq4\hat{f}_{1}(x_{k})\sqrt{\frac{2L_{\hat{F}}}{\mu L}}\Rightarrow\\
        &\Rightarrow\frac{\hat{f}_{1}(x_{k})\sqrt{\mu L}}{M_{\hat{F}}\sqrt{\mu L} + 2\hat{f}_{1}(x_{k})L_{\hat{F}}\sqrt{2L_{\hat{F}}}}\leq\|x_{k} - x^{*}\|\leq4\hat{f}_{1}(x_{k})\sqrt{\frac{2L_{\hat{F}}}{\mu L}},\\
        &\varepsilon_{k} = 0,~\tau_{k} = \hat{f}_{1}(x_{k}),~k\in\mathbb{Z}_{+}.
    \end{aligned}
\end{equation*}
При этом одно из важных свойств решения задачи, удовлетворяющей условию Поляка--Лоясиевича, состоит в единственности $x^{*}$ для данного начального приближения $x_{0}\in E_{1}$, $\mathcal{L}(\hat{f}_{1}(x_{0}))\subseteq\mathcal{F}$ с $\tau_{k} = \hat{f}_{1}(x_{k})$, $\varepsilon_{k} = 0$ (следствие \ref{th:track_glob_bound_cor}). Более того, выполнение предположения \ref{as:det_hat_F_PL_condition} для системы уравнений c $\dim(E_{1}) = \dim(E_{2})$ приводит к наличию участков сублинейной, линейной и суперлинейной сходимости при решении задачи \eqref{eq:main_opt_problem}.

В отличие от условий теоремы \ref{th:glob_sub_lin_and_lin_conv}, использование информации о явном выражении $T_{L_{k}, \tau_{k}}(x_{k})$ позволяет получить глобальную линейную сходимость для произвольного начального приближения $x_{0}\in E_{1}$, $\mathcal{L}(\hat{f}_{1}(x_{0}))\subseteq\mathcal{F}$. Согласно следствию \ref{lm:cor_aux_det_local_decrease_1}, $T_{L_{k}, \tau_{k}}(x_{k})$ имеет следующее представление:
\begin{equation*}
    T_{L_{k}, \tau_{k}}(x_{k}) = x_{k} - \left(\hat{F}^{'}(x_{k})^{*}\hat{F}^{'}(x_{k}) + \tau_{k}L_{k}I_{n}\right)^{-1}\hat{F}^{'}(x_{k})^{*}\hat{F}(x_{k}).
\end{equation*}
В данной работе рассматривается более общая форма правила обновления $x_{k + 1}$, использующая явное выражение $T_{L_{k}, \tau_{k}}(x_{k})$:
\begin{equation}\label{eq:det_expl_update_rule}
    x_{k + 1} = x_{k} - \eta_{k}\left(\hat{F}^{'}(x_{k})^{*}\hat{F}^{'}(x_{k}) + \tau_{k}L_{k}I_{n}\right)^{-1}\hat{F}^{'}(x_{k})^{*}\hat{F}(x_{k}),~\eta_{k}\in\mathbb{R}.
\end{equation}
Полученное таким образом значение $x_{k + 1}$ является субоптимальным относительно $T_{L_{k}, \tau_{k}}(x_{k})$, однако позволяет лучше рассмотреть метод Гаусса--Ньютона в классе квазиньютоновских методов. Следующее утверждение расширяет множество оценок на значение $\hat{f}_{1}(x_{k + 1})$, представленное в теореме \ref{th:glob_sub_lin_and_lin_conv}.

\begin{theorem}\label{th:1_main}
    Пусть выполнены предположения \ref{as:det_hat_F_der_smooth} и \ref{as:det_hat_F_PL_condition}. Рассмотрим последовательность $\{x_{k}\}_{k\in\mathbb{Z}_{+}}$, вычисляемую по схеме \ref{alg:gen_det_gnm} с правилом \eqref{eq:det_expl_update_rule}, в котором $\tau_{k} > 0$, $\eta_{k}\in(0, 2)$, $k\in\mathbb{Z}_{+}$. Тогда для $k\in\mathbb{Z}_{+}$:
    \begin{equation*}
        \hat{f}_{1}(x_{k + 1}) \leq \frac{\tau_{k}}{2} + \begin{sqcases}
        \frac{\hat{f}_{2}(x_{k})}{2\tau_{k}}\left(1 - \frac{\eta_{k}(2 - \eta_{k})\mu}{L_{k} \tau_{k} + \mu}\right),~\tau_{k} \geq \frac{\mu}{L_{k}};\\[10pt]
        \frac{\hat{f}_{2}(x_{k})(1 - \eta_{k})^{2}}{2\tau_{k}} + \frac{\eta_{k}(2 - \eta_{k})L_{k}\hat{f}_{2}(x_{k})}{2\mu} - \frac{\eta_{k}(2 - \eta_{k})L_{k}^{2}\hat{f}_{2}(x_{k})\tau_{k}}{2\mu^{2}(1 + \xi)^{3}}\text{, для}\\[5pt]
        \text{~~~~~некоторого }\xi\in(-1, 1]\text{ при }\tau_{k} < \frac{\mu}{L_{k}}.
        \end{sqcases}
    \end{equation*}
    При этом для $\eta_{k} = 1,~k\in\mathbb{Z}_{+}$ верно:
    \begin{equation*}
        \hat{f}_{1}(x_{k + 1}) \leq \frac{\tau_{k}}{2} + \frac{L_{\hat{F}}}{\mu}\hat{f}_{2}(x_{k}).
    \end{equation*}
\end{theorem}

Использование явной формы $T_{L_{k}, \tau_{k}}(x_{k})$ в анализе позволяет устанавливать сходимость не только в терминах проксимальных градиентов, но ещё и относительно градиента функции $\hat{f}_{2}$, как указано в утверждении ниже.

\begin{theorem}\label{th:2_main}
    Пусть выполнено предположение \ref{as:det_hat_F_der_smooth}. Рассмотрим последовательность $\{x_{k}\}_{k\in\mathbb{Z}_{+}}$, вычисляемую по схеме \ref{alg:gen_det_gnm} c правилом \eqref{eq:det_expl_update_rule},  в котором $\tau_{k} > 0$ и $\eta_{k} > 0$, $\eta_{k}(2 - \eta_{k})\geq c > 0$, $k\in\mathbb{Z}_{+}$. Дополнительно предположим ограниченность нормы матрицы Якоби: существует $M_{\hat{F}} > 0$, для которого выполнено $\left\|\hat{F}^{'}(x)\right\|\leq M_{\hat{F}}$ при всех $x\in \mathcal{F}$. Тогда при $k\in\mathbb{N}$:
    \begin{equation*}
        \begin{aligned}
            \min\limits_{i\in\overline{0, k - 1}}\left\{\left\|\nabla\hat{f}_{2}(x_{i})\right\|^{2}\right\}&\leq\frac{8\left(2L_{\hat{F}}\max\limits_{i\in\overline{0, k - 1}}\left\{\tau_{i}\right\} + M_{\hat{F}}^{2}\right)}{\eta(2 - \eta)k}\sum\limits_{i = 0}^{k - 1}\left(\frac{1}{2}\left(\tau_{i} - \hat{f}_{1}(x_{i})\right)^{2} + \tau_{i}\left(\hat{f}_{1}(x_{i}) - \hat{f}_{1}(x_{i + 1})\right)\right),\\
            \eta&\in\Argmin\limits_{k\in\mathbb{Z}_{+}}\left\{\eta_{k}(2 - \eta_{k})\right\}.
        \end{aligned}
    \end{equation*}
\end{theorem}

Представленные теоремы \ref{th:1} и \ref{th:2} полезны для общего понимания того, как устроена процедура оптимизации в методе трёх квадратов, однако далеко не любая последовательность $\left\{\tau_{k}\right\}_{k\in\mathbb{Z}_{+}}$ обеспечивает гарантированную сходимость метода даже к стационарной точке. С целью внести дополнительную ясность в определение $\tau_{k}$ рассмотрим следующее утверждение.
\begin{theorem}\label{lm:1_main}
    Пусть условия теоремы \ref{th:1} верны, определим на каждой итерации метода нормализованных квадратов максимальный коэффициент линейной сходимости $\alpha_{k}$:
    $$\alpha_{k}\in\left[\frac{1}{2} + \frac{L_{k}
        \hat{f}_{1}(x_{k}) + (1 - \eta_{k})^{2}\mu}{2\left(L_{k}\hat{f}_{1}(x_{k}) + \mu\right)},~1\right),~k\in\mathbb{Z}_{+}.$$ Дополнительно предположим, что на каждой итерации $\tau_{k} = c_{k}\hat{f}_{1}(x_{k})$, $k\in\mathbb{Z}_{+}$:
    \begin{equation*}
        \begin{aligned}
            c_{k} \in &\left[\frac{L_{k}\hat{f}_{1}(x_{k}) + (1 - \eta_{k})^{2}\mu}{(2\alpha_{k} - 1)(L_{k}\hat{f}_{1}(x_{k}) + \mu)},~\alpha_{k} - \frac{1}{2} - \frac{\mu}{2L_{k}\hat{f}_{1}(x_{k})} +\right.\\
            &\left.~~+ \sqrt{\alpha_{k}^{2} + \alpha_{k}\left(\frac{\mu}{L_{k}\hat{f}_{1}(x_{k})} - 1\right) + \frac{1}{4}\left(\frac{\mu}{L_{k}\hat{f}_{1}(x_{k})} + 1\right)^{2} - \frac{(1 - \eta_{k})^{2}\mu}{L_{k}\hat{f}_{1}(x_{k})}}~\right].
        \end{aligned}
    \end{equation*}
    Тогда метод нормализованных квадратов с вычислением $x_{k + 1}$ по правилу \eqref{eq:det_expl_update_rule} глобально сходится не хуже, чем линейно к решению задачи \eqref{eq:main_opt_problem} $\lim\limits_{k\rightarrow+\infty}x_{k} = x^{*}: \hat{F}(x^{*}) = \mathbf{0}_{m}$ со следующей оценкой:
    $$\hat{f}_{1}(x_{k})\leq\hat{f}_{1}(x_{0})\prod_{i = 0}^{k - 1}\alpha_{i},~k\in\mathbb{Z}_{+},~\prod\limits_{i = 0}^{-1}\alpha_{i} \overset{\operatorname{def}}{=} 1.$$
\end{theorem}

Утверждение выше указывает на возможные значения $\tau_{k}$, пропорциональные $\hat{f}_{1}(x_{k})$, при которых метод нормализованных квадратов имеет глобально линейную сходимость. Стоит заметить, что наименьшее допустимое значение $\alpha_{k}$ соответствует $c_{k} = 1$. Это представляет метод нормализованных квадратов с $\tau_{k} = \hat{f}_{1}(x_{k})$ как наиболее быстрый в классе методов нормализованных квадратов с гарантией линейной сходимости, отличающихся выбором $c_{k}$. Данный результат является скорее теоретико--иллюстративным, обосновывающим важность и общее удобство случая $\tau_{k} = \hat{f}_{1}(x_{k})$. Также в условии теоремы \ref{lm:1} роль $\eta_{k}$ характерно показана: близость значений $\eta_{k}$ к $0$ и к $2$ симметрично относительно $\eta_{k} = 1$ приводит к замедлению сходимости в терминах $\alpha_{k}$.
Выбор $\tau_{k} = \hat{f}_{1}(x_{k})$ в теореме \ref{th:2} приводит к упрощению оценки на квадрат нормы градиента в силу наличия сходимости ($\hat{f}_{1}(x_{k})\hat{f}_{1}(x_{k + 1})\leq\hat{f}_{2}(x_{k}),~k\in\mathbb{Z}_{+}$):
\begin{equation*}
    \begin{aligned}
        \min\limits_{i\in\overline{0, k - 1}}\left\{\left\|\nabla\hat{f}_{2}(x_{i})\right\|^{2}\right\}&\leq\frac{8\hat{f}_{2}(x_{0})\left(2L_{\hat{F}}\hat{f}_{1}(x_{0}) + M_{\hat{F}}^{2}\right)}{\eta(2 - \eta)k}.
    \end{aligned}
\end{equation*}
И для достижения минимального значения нормы градиента $\hat{f}_{2}$ на уровне $\epsilon > 0$ необходимо затратить $k = \operatorname{O}\left(\frac{1}{\epsilon^{2}}\right)$ итераций метода:
\begin{equation*}
    \begin{aligned}
        &k = \left\lceil\frac{8\hat{f}_{2}(x_{0})\left(2L_{\hat{F}}\hat{f}_{1}(x_{0}) + M_{\hat{F}}^{2}\right)}{\eta(2 - \eta)\epsilon^{2}}\right\rceil,
    \end{aligned}
\end{equation*}
формально для такого количества итераций выполнено неравенство:
\begin{equation}\label{eq:th2_main_conv_cond}
    \begin{aligned}
        &\min\limits_{i\in\overline{0, k - 1}}\left\{\left\|\nabla\hat{f}_{2}(x_{i})\right\|\right\}\leq\epsilon.
    \end{aligned}
\end{equation}
В теореме \ref{lm:1} выбор $\tau_{k} = \hat{f}_{1}(x_{k})$ соответствует следующей оценке:
\begin{equation*}
    \begin{aligned}
        \hat{f}_{1}(x_{k + 1})&\leq\hat{f}_{1}(x_{k})\left(1 - \frac{\eta_{k}(2 - \eta_{k})\mu}{2\left(L_{k}\hat{f}_{1}(x_{k}) + \mu\right)}\right)\leq\hat{f}_{1}(x_{k})\left(1 - \frac{\eta(2 - \eta)\mu}{2\left(2L_{\hat{F}}\hat{f}_{1}(x_{0}) + \mu\right)}\right)\leq\\
        &\leq\hat{f}_{1}(x_{k})\exp\left(-\frac{\eta(2 - \eta)\mu}{2\left(2L_{\hat{F}}\hat{f}_{1}(x_{0}) + \mu\right)}\right)\leq\hat{f}_{1}(x_{0})\exp\left(-\frac{(k + 1)\eta(2 - \eta)\mu}{2\left(2L_{\hat{F}}\hat{f}_{1}(x_{0}) + \mu\right)}\right),\\
        \eta&\in\Argmin\limits_{k\in\mathbb{Z}_{+}}\left\{\eta_{k}(2 - \eta_{k})\right\}.
    \end{aligned}
\end{equation*}
А благодаря строению функции $\hat{f}_{2}$ возможно установить аналогичную оценку на норму градиента:
\begin{equation*}
    \begin{aligned}
        &\left\|\nabla\hat{f}_{2}(x_{k})\right\| = \left\|2\hat{F}^{'}(x_{k})^{*}\hat{F}(x_{k})\right\|\leq2\left\|\hat{F}^{'}(x_{k})^{*}\right\|\left\|\hat{F}(x_{k})\right\| = 2\left\|\hat{F}^{'}(x_{k})\right\|\hat{f}_{1}(x_{k}).
    \end{aligned}
\end{equation*}
Оценка в теореме \ref{th:1} при $\tau_{k} = \hat{f}_{1}(x_{k})$ означает для достижения уровня
\begin{equation}\label{eq:1_main_conv_cond}
    \begin{aligned}
        &\hat{f}_{1}(x_{k})\leq\epsilon
    \end{aligned}
\end{equation}
необходимость затратить $k = \operatorname{O}\left(\ln\left(\frac{1}{\epsilon}\right)\right)$ итераций:
\begin{equation*}
    \begin{aligned}
        &k = \left\lceil\frac{2\left(2L_{\hat{F}}\hat{f}_{1}(x_{0}) + \mu\right)}{\eta(2 - \eta)\mu}\ln\left(\frac{\hat{f}_{1}(x_{0})}{\epsilon}\right)\right\rceil.
    \end{aligned}
\end{equation*}
Таким образом, теорема \ref{lm:1} устанавливает глобальную линейную сходимость при выполнении условия Поляка--Лоясиевича с гарантией отсутствия участков сублинейной сходимости.

Хотя выбор $\tau_{k} = \hat{f}_{1}(x_{k})$ очень удобен на практике, в процессе оптимизации можно значение $\tau_{k}$ на каждой итерации адаптивно настраивать, a при достаточном приближении $\tau_{k}$ к значению $\phi(x_{k}, y)$ общее поведение метода нормализованных квадратов будет похоже на поведение метода Гаусса--Ньютона, рассмотренного в работе \cite{Nesterov2007}. Само значение $\tau = \phi(x, y)$ соответствует ближайшей верхней оценке на $\hat{f}_{1}(y)$ относительно $\tau$ (лемма \ref{lm:aux_det_upper_model}) с $L\geq L_{\hat{F}}$, $\tau > 0$:
\begin{equation*}
    \begin{aligned}
        \hat{f}_{1}(y)&\leq\frac{L}{2}\|y - x\|^{2} + \underbrace{\left\|\hat{F}(x) + \hat{F}^{'}(x)(y - x)\right\|}_{=\phi(x, y)}\leq\frac{L}{2}\|y - x\|^{2} + \frac{\tau}{2} + \frac{1}{2\tau}\left\|\hat{F}(x) + \hat{F}^{'}(x)(y - x)\right\|^{2},~(x, y)\in\mathcal{F}^{2}.
    \end{aligned}
\end{equation*}
Важно заметить, что в силу неравенства для локальных моделей выше, утверждения теорем \ref{th:1} и \ref{th:2} при подстановке $\tau_{k} = \hat{f}_{1}(x_{k}),~k\in\mathbb{Z}_{+}$ в условие выполнены для метода Гаусса--Ньютона с локальной моделью $\psi_{x, L, \phi(x, y)}(y)$, относительно которой в схеме \ref{alg:gen_det_gnm} значение $x_{k + 1}$ является приближением элемента из $\Argmin\limits_{y\in\mathcal{F}}\left\{\psi_{x_{k}, L_{k}, \phi(x_{k}, y)}(y)\right\}$ \cite{Nesterov2020}.

Введём обозначение оптимального значения $\tau$:
$$\mathcal{T}_{L}(x) = \argmin\limits_{\tau > 0}\left\{\psi_{x, L, \tau}(T_{L, \tau}(x))\right\},$$
которое вытекает из свойства строгой выпуклости по $\tau$ общей локальной модели и позволяет упростить вычисление приближения точки минимума по $y$ в случае $\tau = \phi(x, y)$, $\mathcal{L}(\hat{f}_{1}(x))\subseteq\mathcal{F}$:
\begin{equation*}
    \begin{aligned}
        \hat{f}_{1}(T_{L, \mathcal{T}_{L}(x)}(x))&\leq\min\limits_{y\in\mathcal{F}}\left\{\frac{L}{2}\|y - x\|^{2} + \phi(x, y)\right\} = \min\limits_{y\in\mathcal{F}}\min\limits_{\tau > 0}\left\{\frac{L}{2}\|y - x\|^{2} + \frac{\tau}{2} + \frac{\left(\phi(x, y)\right)^{2}}{2\tau}\right\} =\\
        &= \min\limits_{\tau > 0}\left\{\frac{\tau}{2} + \min\limits_{y\in\mathcal{F}}\left\{\frac{L}{2}\|y - x\|^{2} + \frac{\left(\phi(x, y)\right)^{2}}{2\tau}\right\}\right\} = \min\limits_{\tau > 0}\left\{\psi_{x, L, \tau}(T_{L, \tau}(x))\right\}\Rightarrow\\
        &\Rightarrow T_{L, \mathcal{T}_{L}(x)}(x)\in\Argmin\limits_{y\in\mathcal{F}}\left\{\frac{L}{2}\|y - x\|^{2} + \phi(x, y)\right\},~L\geq L_{\hat{F}}.
    \end{aligned}
\end{equation*}
Обозначенная выше справедливость теорем \ref{th:1} и \ref{th:2} для локальной модели с $\tau_{k} = \mathcal{T}_{L_{k}}(x_{k}),~k\in\mathbb{Z}_{+}$ при подстановке в условия данных теорем $\tau_{k} = \hat{f}_{1}(x_{k})$, $k\in\mathbb{Z}_{+}$ наглядно следует из цепочки неравенств ниже:
\begin{equation}\label{eq:local_models_orders}
\scalemath{0.93}{
    \begin{aligned}
        \lefteqn{\underbrace{\phantom{\hat{f}_{1}(T_{L_{k}, \mathcal{T}_{L_{k}}(x_{k})}(x_{k}))\leq\psi_{x_{k}, L_{k}, \mathcal{T}_{L_{k}}(x_{k})}(T_{L_{k}, \mathcal{T}_{L_{k}}(x_{k})}(x_{k}))}}_{\psi_{x_{k}, L_{k}, \mathcal{T}_{L_{k}}(x_{k})}(\cdot)\text{ --- локальная модель }\hat{f}_{1}(\cdot)}\leq\underbrace{\phantom{\psi_{x_{k}, L_{k}, \mathcal{T}_{L_{k}}(x_{k})}(T_{L_{k}, \hat{f}_{1}(x_{k})}(x_{k}))\leq\psi_{x_{k}, L_{k}, \hat{f}_{1}(x_{k})}(T_{L_{k}, \hat{f}_{1}(x_{k})}(x_{k}))}}_{\text{по определению }\mathcal{T}_{L_{k}}(x_{k})\text{ (лемма \ref{lm:aux_det_upper_model})}}}\hat{f}_{1}(T_{L_{k}, \mathcal{T}_{L_{k}}(x_{k})}(x_{k}))&\leq\overbrace{\psi_{x_{k}, L_{k}, \mathcal{T}_{L_{k}}(x_{k})}(T_{L_{k}, \mathcal{T}_{L_{k}}(x_{k})}(x_{k}))\leq\psi_{x_{k}, L_{k}, \mathcal{T}_{L_{k}}(x_{k})}(T_{L_{k}, \hat{f}_{1}(x_{k})}(x_{k}))}^{T_{L_{k}, \mathcal{T}_{L_{k}}(x_{k})}(x_{k})\text{ минимизирует }\psi_{x_{k}, L_{k}, \mathcal{T}_{L_{k}}(x_{k})}(\cdot)}\leq\psi_{x_{k}, L_{k}, \hat{f}_{1}(x_{k})}(T_{L_{k}, \hat{f}_{1}(x_{k})}(x_{k})).
    \end{aligned}
}
\end{equation}
То есть в схеме \ref{alg:gen_det_gnm} для последовательности $\left\{x_{k}\right\}_{k\in\mathbb{Z}_{+}}$ всегда выполнено соотношение
\begin{equation}\label{eq:local_diffs_order}
    \begin{aligned}
        &\hat{f}_{1}(x_{k}) - \psi_{x_{k}, L_{k}, \mathcal{T}_{L_{k}}(x_{k})}(T_{L_{k}, \mathcal{T}_{L_{k}}(x_{k})}(x_{k}))\geq\hat{f}_{1}(x_{k}) - \psi_{x_{k}, L_{k}, \hat{f}_{1}(x_{k})}(T_{L_{k}, \hat{f}_{1}(x_{k})}(x_{k}))\geq0.
    \end{aligned}
\end{equation}

Таким образом, если обозначить процедуру получения $x_{k + 1}$ по известным $(x_{k}, L_{k}, \tau_{k})$ на каждой итерации $k\in\mathbb{Z}_{+}$ через отображение $\mathcal{X}:\mathcal{F}\times\mathbb{R}_{++}^{2}\rightarrow\mathcal{F}$, то схема оптимизации с адаптивным подбором $\tau_{k}$ на каждом шаге $k$ заключается в следующем:
\begin{itemize}
    \itemsep=-4pt
    \item[1.] Вычислить $\tau_{k}^{*}$ --- приближение оптимального значения $\mathcal{T}_{L_{k}}(x_{k})$, $\tau_{k}^{*}$ может быть получено приближённо из задачи поиска элемента множества $\Argmin\limits_{\tau > 0}\left\{\psi_{x_{k}, L_{k}, \tau}(\mathcal{X}(x_{k}, L_{k}, \tau))\right\}$;
    \item[2.] Получить значение $x_{k + 1} = \mathcal{X}(x_{k}, L_{k}, \tau_{k}^{*})$ как приближение $T_{L_{k}, \tau_{k}^{*}}(x_{k})$.
\end{itemize}
Схема \ref{alg:gen_det_flex_gnm} представляет метод Гаусса--Ньютона, в котором используется адаптивный подбор $\tau_{k}$.
\RestyleAlgo{boxruled}
\begin{algorithm}[!ht]{}
\caption{\textbf{Общий метод нормализованных квадратов с неточным проксимальным отображением и адаптивным подбором $\tau$}}
\label{alg:gen_det_flex_gnm}
\textbf{Вход:}
    \begin{equation*}
        \begin{cases}
            x_{0}\in E_{1},~\mathcal{L}(\hat{f}_{1}(x_{0}))\subseteq\mathcal{F}\text{ --- начальное приближение},~x_{-1} = x_{0};\\
            \mathcal{E}(\cdot)\text{ --- функция погрешности подбора }\tau;\\
            N\in\mathbb{N}\text{ --- количество итераций метода};\\
            L\text{ --- оценка локальной постоянной Липшица},~L\in (0,~L_{\hat{F}}],~L_{0} = L;\\
            \mathcal{X}(\cdot)\text{ --- отображение, аппроксимирующее }T_{L_{k}, \tau_{k}}(x_{k}).
        \end{cases}
    \end{equation*}
    \vspace{0.2cm}
    \textbf{Повторять для $k = 0, 1,~\dots, N - 1$:}
    \begin{itemize}
        \item[1.] определить $\varepsilon_{k} = \mathcal{E}(k, x_{k}, x_{k - 1})$;
        \item[2.] вычислить $\tau_{k}^{*} > 0$, для которого выполнено $\psi_{x_{k}, L_{k}, \tau_{k}^{*}}(\mathcal{X}(x_{k}, L_{k}, \tau_{k}^{*})) - \psi_{x_{k}, L_{k}, \mathcal{T}_{L_{k}}(x_{k})}(T_{L_{k}, \mathcal{T}_{L_{k}}(x_{k})}(x_{k}))\leq\varepsilon_{k}$ и\\$\hat{f}_{1}(x_{k}) - \psi_{x_{k}, L_{k}, \tau_{k}^{*}}(\mathcal{X}(x_{k}, L_{k}, \tau_{k}^{*}))\geq 0$;
        \item[3.] вычислить $x_{k + 1} = \mathcal{X}(x_{k}, L_{k}, \tau_{k}^{*})$;
        \item[4.] если $\hat{f}_{1}(x_{k + 1}) > \psi_{x_{k}, L_{k}, \tau_{k}^{*}}(x_{k + 1})$, то положить $L_{k} := \min\left\{2L_{k}, 2L_{\hat{F}}\right\}$ и вернуться к пункту 2;
        \item[5.] $L_{k + 1} = \max\left\{\frac{L_{k}}{2},~L\right\}$.
    \end{itemize}
    \vspace{0.2cm}

    \textbf{Выход:} $x_{N}$.
\end{algorithm}
В этой схеме на первое место выходит задача поиска гиперпараметра $\tau_{k}^{*}$, то есть вместе с поиском $x_{k + 1}$ происходит оптимизация оптимизатора, представленного в виде отображения $\mathcal{X}$, и в таком виде метод Гаусса--Ньютона ведёт себя в режиме, близком к режиму с $\tau_{k} = \phi(x_{k}, y)$, что формально отражено в теоремах \ref{th:DetFlexSublinConv} (следствия \ref{th:DetFlexSublinConvCor1} и \ref{th:DetFlexSublinConvCor2}) и \ref{th:DetFlexlinConv} (следствия \ref{th:DetFlexlinConvCor1} и \ref{th:DetFlexlinConvCor2}).
\begin{theorem}\label{th:DetFlexSublinConvMain}
    Пусть выполнено предположение \ref{as:det_hat_F_der_smooth}, $k\in\mathbb{N},~r > 0$. Рассмотрим функции
    $$\varkappa(t) = \frac{t^{2}}{2}\mathds{1}_{\left\{t\in[0, 1]\right\}} + \left(t - \frac{1}{2}\right)\mathds{1}_{\left\{t > 1\right\}}\text{ и }\tilde{\Delta}_{r}(x) \overset{\operatorname{def}}{=} \hat{f}_{1}(x) - \min\limits_{y\in E_{1}}\left\{\phi(x, y):~\|y - x\|\leq r\right\}.$$
    Тогда для метода Гаусса--Ньютона, реализованного по схеме \ref{alg:gen_det_flex_gnm} с $\varepsilon_{k} = \varepsilon \geq 0$, верны следующие оценки:
    \begin{equation*}
        \begin{cases}
            &\frac{8L_{\hat{F}}^{2}}{L}\left(\varepsilon + \frac{\left(\hat{f}_{1}(x_{0}) - \hat{f}_{1}(x_{k})\right)}{k}\right)\geq\min\limits_{i\in\overline{0, k - 1}}\left\{\left\|2L_{\hat{F}}\left(T_{2L_{\hat{F}}, \mathcal{T}_{2L_{\hat{F}}}(x_{i})}(x_{i}) - x_{i}\right)\right\|^{2}\right\};\\[10pt]
            &L_{\hat{F}}\left(\varepsilon + \frac{\left(\hat{f}_{1}(x_{0}) - \hat{f}_{1}(x_{k})\right)}{k}\right)\geq\min\limits_{i\in\overline{0, k - 1}}\left\{2\left(L_{\hat{F}}r\right)^{2}\varkappa\left(\frac{\tilde{\Delta}_{r}(x_{i})}{2L_{\hat{F}}r^{2}}\right)\right\}.
        \end{cases}
    \end{equation*}
\end{theorem}
\begin{theorem}\label{th:DetFlexlinConvMain}
    Допустим выполнение предположений \ref{as:det_hat_F_der_smooth} и \ref{as:det_hat_F_PL_condition} для метода Гаусса--Ньютона со схемой реализации \ref{alg:gen_det_flex_gnm}. Тогда для последовательности $\left\{x_{k}\right\}_{k\in\mathbb{Z}_{+}}$ выполняются следующие соотношения:
    \begin{equation*}
        \hat{f}_{1}(x_{k + 1})\leq\varepsilon_{k} + \begin{sqcases}
            \frac{L_{\hat{F}}}{\mu}\hat{f}_{2}(x_{k})\leq\frac{1}{2}\hat{f}_{1}(x_{k})\text{, если }\hat{f}_{1}(x_{k})\leq\frac{\mu}{2L_{\hat{F}}};\\[5pt]
            \hat{f}_{1}(x_{k}) - \frac{\mu}{4L_{\hat{F}}}\text{, иначе}.
        \end{sqcases}
    \end{equation*}
    Если при генерации последовательности $\left\{x_{k}\right\}_{k\in\mathbb{Z}_{+}}$ была зафиксирована $L_{k} = L_{\hat{F}}$, то данные соотношения выражаются по--другому:
    \begin{equation*}
        \hat{f}_{1}(x_{k + 1})\leq\varepsilon_{k} + \begin{sqcases}
            \frac{L_{\hat{F}}}{2\mu}\hat{f}_{2}(x_{k})\leq\frac{1}{2}\hat{f}_{1}(x_{k})\text{, если }\hat{f}_{1}(x_{k})\leq\frac{\mu}{L_{\hat{F}}};\\[5pt]
            \hat{f}_{1}(x_{k}) - \frac{\mu}{2L_{\hat{F}}}\text{, иначе}.
        \end{sqcases}
    \end{equation*}
\end{theorem}
В теореме \ref{th:DetFlexSublinConv} скорость сходимости не отличается от результатов в теореме \ref{th:DetSublinConv}. А теорема \ref{th:DetFlexlinConv} в своих следствиях \ref{th:DetFlexlinConvCor1} и \ref{th:DetFlexlinConvCor2} наглядно демонстрирует уменьшение необходимого количества итераций для достижения $\epsilon$--уровня функции $\hat{f}_{1}$ по сравнению с теоремой \ref{th:glob_sub_lin_and_lin_conv}, однако на практике часто это означает усложнение вычисления каждой итерации, что приходится соизмерять для определения наиболее оптимальных гиперпараметров решения задачи, тем не менее, в схеме \ref{alg:gen_det_flex_gnm} представлен универсальный способ учёта сложности каждой итерации для извлечения преимуществ из явного алгоритма вычисления точки минимума локальной модели $\psi_{x, L, \mathcal{T}_{L}(x)}(\cdot)$.

Минимизация величины $\psi_{x_{k}, L_{k}, \tau}(\mathcal{X}(x_{k}, L_{k}, \tau))$ по $\tau$ может быть достаточно трудоёмкой процедурой. Более того, отображение $\mathcal{X}(\cdot)$ может быть негладким по $\tau > 0$, а в случае дифференцируемости по $\tau$ практическая реализация может представлять собой разновидность алгоритма распространения ошибки через итерации метода оптимизации, осуществляющего аппроксимацию отображения $T_{L_{k}, \tau}(x_{k})$. При использовании правила вычисления $x_{k + 1}$ \eqref{eq:det_expl_update_rule} в качестве отображения $\mathcal{X}(\cdot)$ величина $\psi_{x_{k}, L_{k}, \tau}(\mathcal{X}(x_{k}, L_{k}, \tau))$ принимает следующий вид:
\begin{equation*}
    \scalemath{0.98}{
    \begin{aligned}
        &\psi_{x_{k}, L_{k}, \tau}(\mathcal{X}(x_{k}, L_{k}, \tau)) = \frac{\tau}{2} + \frac{\hat{f}_{2}(x_{k})}{2\tau} - \frac{\eta_{k}(2 - \eta_{k})}{2\tau}\left\langle\left(\hat{F}^{'}(x_{k})^{*}\hat{F}^{'}(x_{k}) + \tau L_{k}I_{n}\right)^{-1}\hat{F}^{'}(x_{k})^{*}\hat{F}(x_{k}),~\hat{F}^{'}(x_{k})^{*}\hat{F}(x_{k})\right\rangle.\\
    \end{aligned}
    }
\end{equation*}
Для $\eta_{k} = 1$ данная функция является строго выпуклой по $\tau$, так как локальная модель $\psi_{x, L, \tau}(y)$ строго выпукла по $\tau$ и сильно выпукла по $y$, а $\psi_{x_{k}, L_{k}, \tau}(\mathcal{X}(x_{k}, L_{k}, \tau))$ представляет собой проекцию по $y$ локальной модели $\psi_{x, L, \tau}(y)$ (Theorem 3.1.7, \cite{NesterovLectures}), поэтому в рассматриваемом случае приближение оптимального $\tau^{*}_{k} > 0$ можно эффективно найти с помощью стандартных средств линейной алгебры и выпуклой оптимизации, в частности, процедур одномерного поиска.

\section{Стохастическая модификация метода Гаусса--Ньютона}\label{sec:stoch_modified_gnm}

\subsection{Стохастическая локальная модель}

В данной работе предлагается решение задачи \eqref{eq:main_opt_problem} в режиме стохастической аппроксимации, в котором локальная модель на каждой итерации оценивается по сэмплированному подмножеству функций из системы \eqref{eq:smooth_system}. То есть природа стохастичности, возникающая в задаче \eqref{eq:main_opt_problem} при оптимизации с помощью стохастического оракула, порождена сэмплированием начального приближения $x_{0}\in E_{1}$ и батчированием функций из конечной генеральной совокупности. Как следствие, стохастичность в оценке функции $\hat{f}_{1}(x)$ заложена взятием подмножества функций из $F$ в виде батча $B$ размера $|B| = b\in\{1,~\dots, m\}$ независимо на каждой итерации и из одного и того же распределения:
\begin{equation*}
    \begin{aligned}
        B &= \left\{F_{i_{j}}(x)\middle|~j\in\{1,~\dots, b\},~i_{j}\in \left\{1,~\dots, m\right\}\right\}.
    \end{aligned}
\end{equation*}
Батч $B$ может быть и мультимножеством, если сэмплировать $i_{j}$ с возвращением. Дополнительно зададим $B_{k}$, $k\in\mathbb{Z}_{+}$ --- батч размера $b$, сэмплированный на $k$--ой итерации метода оптимизации. Рассматривается только независимое от $x_{k},~k\in\mathbb{Z}_{+}$ сэмплирование без возвращения в генеральной совокупности $\mathcal{B}$ размера $m$ с равновероятными подмножествами размера $b$:
$$\mathcal{B} = \left\{F_{i}(x)\middle|~i\in\{1,~\dots,m\}\right\}.$$
Определим соответствующие подмножеству функций $B$ вектор $\hat{G}(x, B)$ и матрицу Якоби $\hat{G}^{'}(x, B)$:
\begin{equation*}
    \begin{aligned}
        \hat{G}(x, B) &\overset{\operatorname{def}}{=} \frac{1}{\sqrt{b}}\left(F_{i_{1}}(x),~\dots, F_{i_{b}}(x)\right)^{*};\\
        \hat{G}^{'}(x, B) &\overset{\operatorname{def}}{=} \frac{1}{\sqrt{b}}(\nabla F_{i_{1}}(x),~\dots, \nabla F_{i_{b}}(x))^{*}.
    \end{aligned}
\end{equation*}
Тогда всей выборке функций соответствуют следующий вектор $\hat{F}(x)$ и его матрица Якоби $\hat{F}^{'}(x)$:
\begin{equation*}
    \begin{aligned}
        \hat{F}(x) &= \frac{1}{\sqrt{m}}\left(F_{1}(x),~\dots, F_{m}(x)\right)^{*} \overset{\operatorname{def}}{=} \hat{G}(x, \mathcal{B});\\
        \hat{F}^{'}(x) &= \frac{1}{\sqrt{m}}(\nabla F_{1}(x),~\dots, \nabla F_{m}(x))^{*} \overset{\operatorname{def}}{=} \hat{G}^{'}(x, \mathcal{B}).
    \end{aligned}
\end{equation*}
Фактически таким определением для $\hat{G}(x, \mathcal{B})$ зафиксировали единственно возможный порядок функций, хотя в результате сэмплирования без возвращения он может быть иным, но данный момент не принципиален для дальнейшего анализа в силу определения функции $\hat{f}_{1}$. В связи с этим введём оценки функций $\hat{f}_{1},~\hat{f}_{2}$ по батчу функций $B$:
\begin{equation*}
    \begin{aligned}
        &\hat{g}_{1}(x, B) \overset{\operatorname{def}}{=} \left\|\hat{G}(x, B)\right\|;\\
        &\hat{g}_{2}(x, B) \overset{\operatorname{def}}{=} \left(\hat{g}_{1}(x, B)\right)^{2}.
    \end{aligned}
\end{equation*}

Введённые функции позволяют задать \textit{стохастическую локальную модель} функции $\hat{f}_{1}$, но уже в качестве ограничивающей параболы функции $\hat{g}_{1}$ (лемма \ref{lm:aux_upper_model}):
\begin{equation*}
    \begin{aligned}
        \hat{g}_{1}(y, B)\leq&\hat{\psi}_{x, L, \tau}(y, B) \overset{\operatorname{def}}{=} \frac{\tau}{2} + \frac{L}{2}\left\|y - x\right\|^{2} + \frac{1}{2\tau}\left\|\hat{G}(x, B) + \hat{G}^{'}(x, B)(y - x)\right\|^{2},~L\geq L_{\hat{F}},\\
        &(x, y)\in E_{1}^{2},~\tau>0,~B\subseteq\mathcal{B}.
    \end{aligned}
\end{equation*}

Данная локальная модель предлагает естественный оракул для поиска приближения решения задачи \eqref{eq:main_opt_problem}:
$$x_{k + 1} = \hat{T}_{L_{k}, \tau_{k}}(x_{k}, B_{k}) \overset{\operatorname{def}}{=} \argmin\limits_{y\in E_{1}}\left\{\hat{\psi}_{x_{k}, L_{k}, \tau_{k}}(y, B_{k})\right\},~\tau_{k}>0,~L_{k} > 0,~k\in\mathbb{Z}_{+}.$$

В этой работе рассматривается случай, в котором стохастический оракул отличается от предложенного в Method of Stochastic Squares \cite{Nesterov2020}. В данном случае точка $x_{k + 1}$ целиком оценивается по батчу функций из $\hat{F}$, поэтому рассматриваемый вариант метода стохастических квадратов (Method of Stochastic Squares) и называется \textit{методом трёх стохастических квадратов}, и в этом методе рассматриваются следующие оценки на функции $\hat{f}_{1}$ и $\hat{f}_{2}$:
\begin{equation*}
    \begin{aligned}
        \hat{f}_{1}(x) &= \hat{g}_{1}(x, \mathcal{B}) = \sqrt{\mathbb{E}_{q}\left[\frac{1}{b}\sum\limits_{j = 1}^{b}\left(F_{i_{j}}(x)\right)^{2}\right]}\geq\left\{\text{неравенство Йенсена, $\sqrt{\cdot}$ --- вогнутая}\right\}\geq\\
        &\geq\mathbb{E}_{q}\left[\sqrt{\frac{1}{b}\sum\limits_{j = 1}^{b}\left(F_{i_{j}}(x)\right)^{2}}~\right] = \mathbb{E}_{q}\left[\hat{g}_{1}(x, B)\right];\\
        \hat{f}_{2}(x) &= \hat{g}_{2}(x, \mathcal{B}) =  \mathbb{E}_{q}\left[\frac{1}{b}\sum\limits_{j = 1}^{b}\left(F_{i_{j}}(x)\right)^{2}\right] = \mathbb{E}_{q}\left[\hat{g}_{2}(x, B)\right].
    \end{aligned}
\end{equation*}
Символ $q$ обозначает равномерное распределение на подмножества индексов координат вектор--функции $\hat{F}$ размера $b$. Данную стохастичность можно обобщить и на бесконечные генеральные совокупности, рассмотрев, например, следующие формы риска ($i_{j}$ и $\xi\in\left\{1,~\dots, m\right\}$ --- сэмплы, подчиняющиеся вероятностному закону $q$):
\begin{equation*}
    \begin{aligned}
        \hat{f}_{1}(x) &= \sqrt{\mathbb{E}_{q}\left[\tilde{F}_{\xi}(x)\right]}\geq\left\{\text{неравенство Йенсена}\right\}\geq \mathbb{E}_{q}\left[\sqrt{\frac{1}{b}\sum\limits_{j = 1}^{b}\tilde{F}_{i_{j}}(x)}~\right] = \mathbb{E}_{q}\left[\hat{g}_{1}(x, B)\right];\\
        \hat{f}_{2}(x) &= \mathbb{E}_{q}\left[\tilde{F}_{\xi}(x)\right] = \mathbb{E}_{q}\left[\frac{1}{b}\sum\limits_{j = 1}^{b}\tilde{F}_{i_{j}}(x)\right] = \mathbb{E}_{q}\left[\hat{g}_{2}(x, B)\right],~\tilde{F}_{\xi}(x) \overset{\operatorname{def}}{=} \left(F_{\xi}(x)\right)^{2}.
    \end{aligned}
\end{equation*}
Замена конечной суммы по компонентам $\hat{f}_{1}$ на математическое ожидание приводит задачу \eqref{eq:main_opt_problem} к разновидности задачи минимизации среднего риска с потенциально бесконечной генеральной совокупностью $\mathcal{B}$:
\begin{equation}\label{eq:main_risk_problem}
    \min\limits_{x\in E_{1}}\left\{\hat{f}_{1}(x) = \sqrt{\mathbb{E}_{q}\left[\left(F_{\xi}(x)\right)^{2}\right]}\right\}.
\end{equation}
Однако для задачи \eqref{eq:main_risk_problem} в случае бесконечной генеральной совокупности $\mathcal{B}$ не определён нормированный вектор $\hat{F}(x)$. Для решения подобного вида задач методом трёх стохастических квадратов необходимо сменить вид оптимизируемого функционала, например, на тот, в котором размерность матрицы Якоби не зависит от количества сэмплов в оценке риска.

\subsection{Использованные предположения}

Для стохастической версии метода Гаусса--Ньютона вместо предположений \ref{as:det_hat_F_der_smooth} и \ref{as:det_hat_F_PL_condition} используются следующие.

\begin{assumption}[Липшиц--непрерывность функций из $F$]\label{as:1}
    Существуют конечные $L_{\hat{F}} > 0,~l_{\hat{F}} > 0,$ для которых 
    \begin{equation*}
        \begin{aligned}
            \left\|\nabla F_{i}(x) - \nabla F_{i}(y)\right\|&\leq L_{\hat{F}}\left\|x - y\right\|,~\left|\left(F_{i}(x)\right)^{2} - \left(F_{i}(y)\right)^{2}\right|\leq l_{\hat{F}}\|x - y\|,~\forall (x, y)\in E_{1}^{2},~\forall i\in\overline{1, m}.
        \end{aligned}
    \end{equation*}
\end{assumption}
В отличие от предположений, введённых в работе \cite{Nesterov2020}, в предположении \ref{as:1} липшицевость рассматривается относительно отдельной функции из выборки. Стоит отметить, что в случае бесконечно большой генеральной совокупности $\mathcal{B}$ предположение \ref{as:1} выполняется в смысле почти наверно.

\begin{assumption}[Ограниченность нормы матрицы Якоби $\hat{G}^{'}$]\label{as:2}
    Существует конечное $M_{\hat{G}} > 0,$ для которого $\left\|\hat{G}^{'}(x, B)\right\|\leq M_{
    \hat{G}}$ при любых $x\in E_{1}$ и $B\subseteq\mathcal{B}$, $|B| = b\in\overline{1, m}$. Для $b = m$ и при сэмплировании батча функций без возвращения существует конечное $M_{\hat{F}} > 0,$ для которого $\left\|\hat{F}^{'}(x)\right\|\leq M_{
    \hat{F}}$ при всех $x\in E_{1}$.
\end{assumption}

Предположение \ref{as:2} так же ограничивает максимальное собственное значение матрицы $\hat{G}^{'}(x, B)^{*}\hat{G}^{'}(x, B)$ значением $M_{\hat{G}}^{2}$ по свойству операторной нормы, что ограничивает сверху сингулярные числа матрицы $\hat{G}^{'}(x, B)$ значением $M_{\hat{G}}$ (лемма \ref{lm:aux_matrix_power_order}).

\begin{assumption}[Ограниченность значения функции $\hat{g}_{1}$]\label{as:3}
    Существует конечное $P_{\hat{g}_{1}} > 0,$ для которого\\$\hat{g}_{1}(x, B)\leq P_{
    \hat{g}_{1}}$, при любых $x\in E_{1}$ и $B\subseteq\mathcal{B},~|B| = b\in\overline{1, m}$. Для $b = m$ и при сэмплировании батча функций без возвращения существует конечное $P_{\hat{f}_{1}} > 0,$ для которого $\left\|\hat{f}_{1}(x)\right\|\leq P_{
    \hat{f}_{1}}$ при всех $x\in E_{1}$.
\end{assumption}
Предположение \ref{as:3} явно задаёт ограниченность множества значений в функциях $\hat{f}_{1}$ и $\hat{g}_{1}$, в то время как в работе \cite{Nesterov2020} неявно предполагается, что $\hat{f}_{1}(x_{0}) < +\infty$, из--за чего все использованные значения функций были ограничены сверху $\hat{f}_{1}(x_{0})$ в силу нахождения в множестве уровня $\hat{f}_{1}(x_{0})$. Так же, как и предположение \ref{as:1}, предположение \ref{as:3} в случае бесконечно большой генеральной совокупности $\mathcal{B}$ выполняется в смысле почти наверно. Выполнение предположений \ref{as:2} и \ref{as:3} влечёт за собой липшицевость $\left(F_{i}(x)\right)^{2}$ из предположения \ref{as:1} и липшицевость функции $\hat{g}_{2}(x, B)$ (лемма \ref{lm:aux_lipschitz}). А по свойству Липшиц--непрерывности наилучшее (наименьшее) значение постоянной Липшица функции $\hat{g}_{2}(x, B)$ равно $\sup\limits_{x\in E_{1}}\left\{\left\|\nabla_{x}\hat{g}_{2}(x, B)\right\|\right\}$ \cite{Jordan2020} и ограничено:
\begin{equation}\label{eq:lipschitz_upper_bound}
    \sup\limits_{x\in E_{1}}\left\{\left\|\nabla_{x}\hat{g}_{2}(x, B)\right\|\right\}\leq\min\left\{l_{\hat{F}},~2M_{\hat{G}}P_{\hat{g}_{1}}\right\},~\forall B\subseteq\mathcal{B},
\end{equation}
так как
\begin{equation*}
    \begin{aligned}
        &\left|\hat{g}_{2}(z, B) - \hat{g}_{2}(y, B)\right|\leq\underbrace{\sup\limits_{x\in E_{1}}\left\{\left\|\nabla_{x}\hat{g}_{2}(x, B)\right\|\right\}}_{\leq l_{\hat{F}}\text{ (лемма \ref{lm:aux_lipschitz})}}\left\|z - y\right\|,~\forall (y, z)\in E_{1}^{2},~\forall B\subseteq\mathcal{B}
    \end{aligned}
\end{equation*}
и
\begin{equation*}
    \begin{aligned}
        \sup\limits_{x\in E_{1}}\left\{\left\|\nabla_{x}\hat{g}_{2}(x, B)\right\|\right\} &= \sup\limits_{x\in E_{1}}\left\{\left\|2\hat{G}^{'}(x, B)^{*}\hat{G}(x, B)\right\|\right\}\leq2\sup\limits_{x\in E_{1}}\left\{\left\|\hat{G}^{'}(x, B)\right\|\left\|\hat{G}(x, B)\right\|\right\}\leq2M_{\hat{G}}P_{\hat{g}_{1}},~\forall B\subseteq\mathcal{B},
    \end{aligned}
\end{equation*}
причём норма $\|~\|$ для $z - y$ и норма $\|~\|$ для $\nabla_{x}\hat{g}_{2}(x, B)$ являются сопряжёнными друг к другу, но в данном случае они совпадают в силу того, что они евклидовы.

\begin{assumption}[Ограниченность дисперсии значения $\hat{g}_{2},~b=1$]\label{as:4}
    Существует конечное $\sigma > 0$, для которого $\mathbb{E}_{q}\left[\left|\hat{g}_{2}(x, B) - \hat{f}_{2}(x)\right|^{2}\right]\leq\sigma^{2}$ для всех $x\in E_{1}$ и $B\subseteq\mathcal{B},~|B| = 1$.
\end{assumption}

Предположение \ref{as:4} автоматически выполняется при выполнении предположения \ref{as:3}, поэтому оно в данном случае внесено из соображений удобства для дальнейшего анализа.

\begin{assumption}[Условие Поляка--Лоясиевича для матрицы Якоби $\hat{G}^{'}$]\label{as:5}
    Существует конечное $\mu > 0$, для которого $\hat{G}^{'}(x, B)\hat{G}^{'}(x, B)^{*}\succeq\mu I_{b}$ для всех $x\in E_{1}$ и $B\subseteq\mathcal{B},~|B| = b\leq \min\left\{m,~n\right\}$. Отношение порядка <<$\succeq$>> выполнено на конусе неотрицательно определённых матриц.
\end{assumption}

Предположение \ref{as:5} ограничивает снизу операторную норму матрицы Якоби $\hat{G}^{'}(x, B)^{*}$ (лемма \ref{lm:aux_matrix_power_order}), однако выполнение данного условия не всегда возможно в случае $n < m$, так как можно построить пример, в котором сэмплированием без возвращения в один батч могли бы попасть векторы, совокупно формирующие линейно зависимую систему векторов (например, используя одновременно диагонализуемые матрицы). При этом модельный пример с $m > n,~n = 2$, в котором любая пара векторов формирует базис, позволяет из этих векторов строить матрицы $n\times b,~b\leq n$ с ненулевым минимальным сингулярным числом не меньше некоторого $\sqrt{\mu} > 0$. Важно понимать, что $\hat{G}^{'}(x, B),~\hat{g}_{2}(x, B)$ и $\hat{g}_{2}^{*}(B)$, будучи случайными величинами, построены на одном и том же батче $B$. Сама функция $\hat{g}_{2}$ может быть \textit{невыпуклой}. Также нетрудно заметить, что для выполнения условия Поляка--Лоясиевича необходимо выбирать в батчи функции сэмплированием без возвращения. Иначе у матрицы $\hat{G}^{'}(x, B)^{*}$ была бы обязательно положительная вероятность наличия нулевого сингулярного числа.

\subsection{Схема процесса оптимизации}

Согласно введённой стохастической локальной модели, в работе рассматриваются три основные стратегии выбора нового приближения $x_{k + 1}$ по известному $x_{k}$.
\begin{equation}\label{eq:stoch_direct_update_rule}
    \begin{aligned}
        &x_{k + 1} = x_{k} - \eta_{k}\left(\hat{G}^{'}(x_{k}, B_{k})^{*}\hat{G}^{'}(x_{k}, B_{k}) + \tau_{k}L_{k}I_{n}\right)^{-1}\hat{G}^{'}(x_{k}, B_{k})^{*}\hat{G}(x_{k}, B_{k}),~\eta_{k}\in\mathbb{R}.
    \end{aligned}
\end{equation}
Правило вычисления $x_{k + 1}$ \eqref{eq:stoch_direct_update_rule} выводится из прямой минимизации $\hat{\psi}_{x_{k}, L_{k}, \tau_{k}}(y, B_{k})$ по $y\in E_{1}$ и при $\eta_{k}~=~1$ значение $y = x_{k + 1}$ является точкой минимума для стохастической локальной модели. В следующей стратегии вычисления $x_{k + 1}$ уже используются два батча функций:
\begin{equation}\label{eq:double_stoch_direct_update_rule}
    \begin{aligned}
        x_{k + 1} &= x_{k} - \eta_{k}\left(\hat{G}^{'}(x_{k}, \tilde{B}_{k})^{*}\hat{G}^{'}(x_{k}, \tilde{B}_{k}) + \tilde{\tau}_{k}L_{k}I_{n}\right)^{-1}\hat{G}^{'}(x_{k}, B_{k})^{*}\hat{G}(x_{k}, B_{k}),~\eta_{k}\in\mathbb{R},\\
        &\tilde{B}_{k}\subseteq\mathcal{B}\text{ и }B_{k}\text{ независимо сэмплированы},~\tilde{\tau}_{k} > 0.
    \end{aligned}
\end{equation}
Значение $x_{k + 1}$, оценённое с помощью \eqref{eq:double_stoch_direct_update_rule}, выводится из другого представления стохастической локальной модели, в котором градиент и гессиан оценены на различных батчах, а $x_{k + 1}$ вычисляется через масштабированный в $\eta_{k}$ раз шаг метода Ньютона:
\begin{equation*}
    \begin{aligned}
        \hat{\psi}_{x_{k}, L_{k}, \tilde{\tau}_{k}}(y, \tilde{B}_{k}) &= \frac{\tilde{\tau}_{k}}{2} + \frac{L_{k}}{2}\left\|y - x_{k}\right\|^{2} + \frac{1}{2\tilde{\tau}_{k}}\left\|\hat{G}(x_{k}, \tilde{B}_{k}) + \hat{G}^{'}(x_{k}, \tilde{B}_{k})(y - x_{k})\right\|^{2} = \left(\frac{\tilde{\tau}_{k}}{2} + \frac{\hat{g}_{2}(x_{k}, \tilde{B}_{k})}{2\tilde{\tau}_{k}}\right) +\\
        &+ \left\langle \frac{\hat{G}^{'}(x_{k}, \tilde{B}_{k})^{*}\hat{G}(x_{k}, \tilde{B}_{k})}{\tilde{\tau}_{k}},~y - x_{k}\right\rangle +\\
        &+ \frac{1}{2}\left\langle\left(\frac{\hat{G}^{'}(x_{k}, \tilde{B}_{k})^{*}\hat{G}^{'}(x_{k}, \tilde{B}_{k})}{\tilde{\tau}_{k}} + L_{k}I_{n}\right)(y - x_{k}),~y - x_{k}\right\rangle\Rightarrow
    \end{aligned}
\end{equation*}

\begin{equation*}
    \begin{aligned}
        \Rightarrow\hat{\psi}_{x_{k}, L_{k}, \tilde{\tau}_{k}}(y, \tilde{B}_{k})&\approx\left(\frac{\tilde{\tau}_{k}}{2} + \frac{\hat{g}_{2}(x_{k}, \tilde{B}_{k})}{2\tilde{\tau}_{k}}\right) + \left\langle \frac{\hat{G}^{'}(x_{k}, B_{k})^{*}\hat{G}(x_{k}, B_{k})}{\tilde{\tau}_{k}},~y - x_{k}\right\rangle +\\
        &+ \frac{1}{2}\left\langle\left(\frac{\hat{G}^{'}(x_{k}, \tilde{B}_{k})^{*}\hat{G}^{'}(x_{k}, \tilde{B}_{k})}{\tilde{\tau}_{k}} + L_{k}I_{n}\right)(y - x_{k}),~y - x_{k}\right\rangle.
    \end{aligned}
\end{equation*}
При этом оценка градиента $\nabla_{y}\hat{\psi}_{x_{k}, L_{k}, \tilde{\tau}_{k}}(y, \tilde{B}_{k})$ в точке $y = x_{k}$ по батчу $B_{k}$ не смещена. Будем называть такой шаг метода в \eqref{eq:double_stoch_direct_update_rule} \textit{дважды стохастическим}. Стратегия \eqref{eq:stoch_approx_general_update_rule} является наиболее общей и включает в себя, как минимум, полностью стратегии вида \eqref{eq:stoch_direct_update_rule}:
\begin{equation}\label{eq:stoch_approx_general_update_rule}
    \begin{aligned}
        x_{k + 1}\in E_{1}:~&0\leq\hat{\psi}_{x_{k}, L_{k}, \tau_{k}}(x_{k + 1}, B_{k}) - \hat{\psi}_{x_{k}, L_{k}, \tau_{k}}(\hat{T}_{L_{k}, \tau_{k}}(x_{k}, B_{k}), B_{k})\leq\varepsilon_{k},\\
        &\hat{g}_{1}(x_{k}, B_{k}) - \hat{\psi}_{x_{k}, L_{k}, \tau_{k}}(x_{k + 1}, B_{k})\geq 0.
    \end{aligned}
\end{equation}
Как и в детерминированном случае, в способе \eqref{eq:stoch_approx_general_update_rule} оракул представлен в виде <<чёрного ящика>> и на практике может оказаться другим итерационным методом оптимизации, минимизирующим функционал $\hat{\psi}_{x_{k}, L_{k}, \tau_{k}}(\cdot, B_{k})$.

Предлагаемая в качестве стохастического метода оптимизации для решения задачи \eqref{eq:main_opt_problem} схема стохастической аппроксимации \ref{alg:gen_stoch_gnm} представляет собой прямую модификацию схемы \ref{alg:gen_det_gnm}, разработанной для детерминированного случая. Стоит заметить, что в теории данная схема корректна при $\gamma\geq1$, однако на практике сходимость оценивается при $\gamma\geq2$ в силу особенности поиска локальной постоянной Липшица: если $L_{k}\in\left[\frac{L_{\hat{F}}}{2},~L_{\hat{F}}\right]$, то для $2L_{k}\in\left[L_{\hat{F}},~2L_{\hat{F}}\right]$ локальная модель всегда корректно определена.

\RestyleAlgo{boxruled}
\begin{algorithm}[!ht]{}
\caption{\textbf{Общий метод трёх стохастических квадратов с неточным проксимальным отображением}}
\label{alg:gen_stoch_gnm}
\textbf{Вход:}
    \begin{equation*}
        \begin{cases}
            x_{0}\in E_{1}\text{ --- начальное приближение},~x_{-1} = x_{0};\\
            \mathcal{E}(\cdot)\text{ --- функция погрешности проксимального отображения};\\
            N\in\mathbb{N}\text{ --- количество итераций метода};\\
            \gamma\geq1\text{ --- фактор верхней границы поиска }L_{\hat{F}};\\
            L\text{ --- оценка локальной постоянной Липшица},~L\in(0,~\gamma L_{\hat{F}}],~L_{0} = L;\\
            \mathcal{T}(\cdot)\text{ --- функция, определяющая значение $\tau$};\\
            \mathcal{B}\text{ --- выборка функций};\\
            b\in\overline{1, m}\text{ --- размер $B_{k}\subseteq\mathcal{B},~k\in\mathbb{Z}_{+}$}.
        \end{cases}
    \end{equation*}
    \vspace{0.2cm}
    \textbf{Повторять для $k = 0, 1,~\dots, N - 1$:}
    \begin{itemize}
        \item[1.] сэмплировать батч $B_{k}$ из $\mathcal{B}$ размера $b$;
        \item[2.] определить $\tau_{k} = \mathcal{T}(x_{k}, L_{k}, B_{k})$, $\varepsilon_{k} = \mathcal{E}(k, x_{k}, x_{k - 1}, B_{k})$;
        \item[3.] вычислить $x_{k + 1}\in E_{1}$ согласно одному из выбранных изначально правил: \eqref{eq:stoch_direct_update_rule} или \eqref{eq:stoch_approx_general_update_rule};
        \item[4.] если $\hat{g}_{1}(x_{k + 1}, B_{k}) > \hat{\psi}_{x_{k}, L_{k}, \tau_{k}}(x_{k + 1}, B_{k})$, то положить $L_{k} := \min\left\{2L_{k}, \gamma L_{\hat{F}}\right\}$ и\\вернуться к пункту 2;
        \item[5.] $L_{k + 1} = \max\left\{\frac{L_{k}}{2},~L\right\}$.
    \end{itemize}
    \vspace{0.2cm}

    \textbf{Выход:} $x_{N}$.
\end{algorithm}

\section{Анализ сходимости}\label{sec:stoch_modified_gnm_analysis}

\subsection{Использование масштабированного шага}

Использование правила \eqref{eq:stoch_direct_update_rule} с $\tau_{k} = \hat{g}_{1}(x_{k}, B_{k})$ вместе с линейным поиском $L_{k}\in[L, \gamma L_{\hat{F}}],~L\in(0, \gamma L_{\hat{F}}]$, $\gamma\geq1$ позволяет добиться сходимости в терминах среднего, что формально изложено в теоремах \ref{th:3} и \ref{th:4} ниже.

\begin{theorem}\label{th:3_main}
    Пусть выполнены предположения \ref{as:1}, \ref{as:2}, \ref{as:3}, \ref{as:4}. Рассмотрим метод Гаусса--Ньютона со схемой реализации \ref{alg:gen_stoch_gnm}, в которой последовательность $\{x_{k}\}_{k\in\mathbb{Z}_{+}}$ вычисляется по правилу \eqref{eq:stoch_direct_update_rule} с $\tau_{k} = \hat{g}_{1}(x_{k}, B_{k})$,\\$\eta_{k}\in[\eta, 1]$, $\eta\in(0, 1]$. Тогда:
    \begin{equation}\label{eq:main_stoch_sub_lin_conv_1}
        \begin{aligned}
            \mathbb{E}\left[\min\limits_{i\in\overline{0, k - 1}}\left\|\nabla\hat{f}_{2}(x_{i})\right\|^{2}\right]&\leq\frac{8\left(M_{\hat{G}}^{2} + \gamma P_{\hat{g}_{1}}L_{\hat{F}}\right)}{\eta(2 - \eta)}\left(\frac{\mathbb{E}\left[\hat{f}_{2}(x_{0})\right]}{k} + 2l_{\hat{F}}\min\left\{\sqrt{\frac{2P_{\hat{g}_{1}}}{L}},~\frac{M_{\hat{G}}}{L}\right\}\mathds{1}_{\left\{b < m\right\}} + \tilde{\sigma}\sqrt{\frac{1}{b} - \frac{1}{m}}\right),\\
            k&\in\mathbb{N}.
        \end{aligned}
    \end{equation}
    Оператор математического ожидания $\mathbb{E}\left[\cdot\right]$ усредняет по всей случайности процесса оптимизации.
\end{theorem}

В теореме \ref{th:3} оценка \eqref{eq:main_stoch_sub_lin_conv_1} сразу содержит в себе оба главных гиперпараметра оптимизатора --- количество итераций $k$ и размер батча $b$. При этом они входят в оценку аддитивно, что означает необходимость одновременного подбора всех гиперпараметров. То есть независимо от количества итераций верхняя оценка на средний минимум квадрата нормы градиента не опустится ниже уровня, явно зависящего от размера батча, а формально это означает, что для достижения уровня $\epsilon^{2}$ долю $r\in(0, 1)$ от $\epsilon^{2}$ придётся покрыть выбором подходящего размера батча, а долю $(1 - r)$ от $\epsilon^{2}$ --- подбором достаточного количества итераций. В оценке \eqref{eq:main_stoch_sub_lin_conv_1} присутствует неустранимое в стохастическом режиме слагаемое
$$\frac{16l_{\hat{F}}\left(M_{\hat{G}}^{2} + \gamma P_{\hat{g}_{1}}L_{\hat{F}}\right)}{\eta(2 - \eta)}\min\left\{\sqrt{\frac{2P_{\hat{g}_{1}}}{L}},~\frac{M_{\hat{G}}}{L}\right\}\mathds{1}_{\left\{b < m\right\}},$$
заключающее в себе верхнюю оценку на квадрат нормы градиента функции $\hat{g}_{2}$ \eqref{eq:lipschitz_upper_bound}. Поэтому для гарантии сходимости до уровня $\epsilon$ в среднем удобнее вывести фактор дисперсии за скобки, как было сделано в неравенстве \eqref{eq:deviation_upper_bound}. Этот приём сводит условия сходимости к уровню $\epsilon$ в \eqref{eq:main_stoch_sub_lin_conv_1} к следующей системе неравенств:
\begin{equation}\label{eq:sublin_batch_iter_tradeoff_1}
    \begin{cases}
         \frac{8(M_{\hat{G}}^{2} + \gamma P_{\hat{g}_{1}}L_{\hat{F}})\mathbb{E}\left[\hat{g}_{2}(x_{0}, B_{0})\right]}{k\eta(2 - \eta)}\leq (1 - r)\epsilon^{2};\\[10pt]
         \frac{8(M_{\hat{G}}^{2} + \gamma P_{\hat{g}_{1}}L_{\hat{F}})}{\eta(2 - \eta)}\left(2l_{\hat{F}}\sqrt{m(m - 1)}\min\left\{\sqrt{\frac{2P_{\hat{g}_{1}}}{L}},~\frac{M_{\hat{G}}}{L}\right\}\mathds{1}_{\left\{b < m\right\}} + \tilde{\sigma}\right)\sqrt{\frac{1}{b} - \frac{1}{m}}\leq r\epsilon^{2}.
    \end{cases}
\end{equation}
Само условие сходимости до уровня $\epsilon > 0$ в среднем выражается в виде следующего неравенства:
\begin{equation}\label{eq:mean_grad_conv_cond}
    \begin{aligned}
        &\mathbb{E}\left[\min\limits_{i\in\overline{0, k - 1}}\left\{\left\|\nabla\hat{f}_{2}(x_{i})\right\|^{2}\right\}\right]\leq\epsilon^{2}.
    \end{aligned}
\end{equation}
Неравенства \eqref{eq:sublin_batch_iter_tradeoff_1} накладывают ограничения на минимальный размер батча $b$ и минимальное количество итераций $k$:
\begin{equation}\label{eq:sublin_batch_iter_tradeoff_2}
    \begin{cases}
         k = \left\lceil\frac{8(M_{\hat{G}}^{2} + \gamma P_{\hat{g}_{1}}L_{\hat{F}})\mathbb{E}\left[\hat{g}_{2}(x_{0}, B_{0})\right]}{\epsilon^{2}(1 - r)\eta(2 - \eta)}\right\rceil,~r\in(0, 1);\\[10pt]
         b = \min\left\{m,~\left\lceil\frac{\frac{64(M_{\hat{G}}^{2} + \gamma P_{\hat{g}_{1}}L_{\hat{F}})^{2}}{\eta^{2}(2 - \eta)^{2}}\left(2l_{\hat{F}}\sqrt{m(m - 1)}\min\left\{\sqrt{\frac{2P_{\hat{g}_{1}}}{L}},~\frac{M_{\hat{G}}}{L}\right\} + \tilde{\sigma}\right)^{2}}{\epsilon^{4}r^{2} + \frac{64(M_{\hat{G}}^{2} + \gamma P_{\hat{g}_{1}}L_{\hat{F}})^{2}}{m\eta^{2}(2 - \eta)^{2}}\left(2l_{\hat{F}}\sqrt{m(m - 1)}\min\left\{\sqrt{\frac{2P_{\hat{g}_{1}}}{L}},~\frac{M_{\hat{G}}}{L}\right\} + \tilde{\sigma}\right)^{2}}\right\rceil\right\}.
    \end{cases}
\end{equation}
Выражения в \eqref{eq:sublin_batch_iter_tradeoff_2} означают следующие асимптотики для $b$ и $k$:
$$k = \operatorname{O}\left(\frac{1}{\epsilon^{2}}\right),~b = \min\left\{m,~\operatorname{O}\left(\frac{1}{\epsilon^{4}}\right)\right\},$$
что для подходящего условиям \eqref{eq:sublin_batch_iter_tradeoff_1} размера батча $b$ соответствует сублинейной сходимости в среднем к \mbox{$\epsilon$--стационарной} точке
$$x^{*}:~\left\|\nabla\hat{f}_{2}(x^{*})\right\|\leq\epsilon.$$
Следующее утверждение устанавливает уже линейную сходимость метода трёх стохастических квадратов в случае выполнения условия Поляка--Лоясиевича (предположение \ref{as:5}).

\begin{theorem}\label{th:4_main}
    Пусть выполнены предположения \ref{as:1}, \ref{as:2}, \ref{as:3}, \ref{as:4}, \ref{as:5}. Рассмотрим  метод Гаусса--Ньютона со схемой реализации \ref{alg:gen_stoch_gnm}, в котором последовательность $\{x_{k}\}_{k\in\mathbb{Z}_{+}}$ вычисляется по правилу \eqref{eq:stoch_direct_update_rule} с\\$\tau_{k}~=~\hat{g}_{1}(x_{k}, B_{k})$, $\eta_{k}\in[\eta, 1]$, $\eta\in(0, 1]$. Тогда:
    \begin{equation}\label{eq:main_stoch_lin_conv_1}
        \begin{cases}
            \begin{aligned}
                \mathbb{E}&\left[\left\|\nabla\hat{f}_{2}(x_{k})\right\|^{2}\right]\leq4M_{\hat{G}}^{2}\Delta_{k,b};\\[5pt]
                \mathbb{E}&\left[\hat{f}_{2}(x_{k})\right]\leq\hat{f}_{2}^{*} + \Delta_{k,b};\\
                &\Delta_{k,b} \overset{\operatorname{def}}{=} \mathbb{E}\left[\hat{f}_{2}(x_{0})\right]\exp\left(-\frac{k\eta(2 - \eta)\mu}{2\left(\gamma L_{\hat{F}}P_{\hat{g}_{1}} + \mu\right)}\right) +\\
                &+ 4\left(l_{\hat{F}}\min\left\{\sqrt{\frac{2P_{\hat{g}_{1}}}{L}},~\frac{M_{\hat{G}}}{L}\right\}\mathds{1}_{\left\{b < m\right\}} + \tilde{\sigma}\sqrt{\frac{1}{b} - \frac{1}{m}}\right)\left(\frac{\gamma L_{\hat{F}}P_{\hat{g}_{1}} + \mu}{\eta(2 - \eta)\mu}\right),~k\in\mathbb{Z}_{+},~b\in\overline{1,~\min\{m, n\}}.
            \end{aligned}
        \end{cases}
    \end{equation}
    Оператор математического ожидания $\mathbb{E}\left[\cdot\right]$ усредняет по всей случайности процесса оптимизации.
\end{theorem}

Теорема \ref{th:4} устанавливает линейную сходимость и для случая $m > n$ при условии выполнения предположения \ref{as:5} для размера батча $b\leq n$. В \eqref{eq:main_stoch_lin_conv_1} приведена оценка на среднее значение функции $\hat{f}_{2}$ в добавок к оценке на средний квадрат нормы градиента. Зависимость уровня нормы градиента от $b$ и $k$ представлена аддитивно так, что для достижения уровня $\epsilon > 0$ необходимо подобрать достаточно большой по размеру батч, из этого следует возможность достижения некоторых уровней нормы градиента на практике только при $b = m$; а при сохранении линейной сходимости данное равенство означает выполнение неравенства $m\leq n$. Поэтому условия на допустимые значения $b$ и $k$ выводятся аналогично \eqref{eq:sublin_batch_iter_tradeoff_1}:
\begin{equation}\label{eq:lin_batch_iter_tradeoff_1}
    \begin{cases}
        4M_{\hat{G}}^{2}\mathbb{E}\left[\hat{g}_{2}(x_{0}, B_{0})\right]\exp\left(-\frac{k\eta(2 - \eta)\mu}{2\left(\gamma L_{\hat{F}}P_{\hat{g}_{1}} + \mu\right)}\right)\leq (1 - r)\epsilon^{2},~r\in(0, 1);\\[10pt]
        16M_{\hat{G}}^{2}\sqrt{\frac{1}{b} - \frac{1}{m}}\left(l_{\hat{F}}\sqrt{m(m - 1)}\min\left\{\sqrt{\frac{2P_{\hat{g}_{1}}}{L}},~\frac{M_{\hat{G}}}{L}\right\}\mathds{1}_{\left\{b < m\right\}} + \tilde{\sigma}\right)\left(\frac{\gamma L_{\hat{F}}P_{\hat{g}_{1}} + \mu}{\eta(2 - \eta)\mu}\right)\leq r\epsilon^{2}.
    \end{cases}
\end{equation}
Неравенства выше означают выполнение ограничения на норму градиента:
\begin{equation}\label{eq:4_main_conv_cond}
    \begin{aligned}
        &\mathbb{E}\left[\left\|\nabla\hat{f}_{2}(x_{k})\right\|^{2}\right]\leq\epsilon^{2}.
    \end{aligned}
\end{equation}
Неравенства в \eqref{eq:lin_batch_iter_tradeoff_1} приводят к следующим минимальным значениям для $b$ и $k$:
\begin{equation}\label{eq:lin_batch_iter_tradeoff_2}
    \begin{cases}
        k = \left\lceil\frac{2\left(\gamma L_{\hat{F}}P_{\hat{g}_{1}} + \mu\right)}{\eta(2 - \eta)\mu}\ln\left(\frac{4M_{\hat{G}}^{2}\mathbb{E}\left[\hat{g}_{2}(x_{0}, B_{0})\right]}{\epsilon^{2}(1 - r)}\right)\right\rceil,~r\in(0, 1);\\[10pt]
        b = \min\left\{m,~n,~\left\lceil\frac{256M_{\hat{G}}^{4}\left(l_{\hat{F}}\sqrt{m(m - 1)}\min\left\{\sqrt{\frac{2P_{\hat{g}_{1}}}{L}},~\frac{M_{\hat{G}}}{L}\right\} + \tilde{\sigma}\right)^{2}\left(\frac{\gamma L_{\hat{F}}P_{\hat{g}_{1}} + \mu}{\eta(2 - \eta)\mu}\right)^{2}}{\epsilon^{4}r^{2} + \frac{256M_{\hat{G}}^{4}}{m}\left(l_{\hat{F}}\sqrt{m(m - 1)}\min\left\{\sqrt{\frac{2P_{\hat{g}_{1}}}{L}},~\frac{M_{\hat{G}}}{L}\right\} + \tilde{\sigma}\right)^{2}\left(\frac{\gamma L_{\hat{F}}P_{\hat{g}_{1}} + \mu}{\eta(2 - \eta)\mu}\right)^{2}}\right\rceil\right\}.
    \end{cases}
\end{equation}
Или в сокращённой асимптотической форме: 
$$k = \operatorname{O}\left(\ln\left(\frac{1}{\epsilon}\right)\right),~b = \min\left\{m,~n,~\operatorname{O}\left(\frac{1}{\epsilon^{4}}\right)\right\}.$$
Для размера батча в \eqref{eq:lin_batch_iter_tradeoff_2} ситуация такая же, как и в \eqref{eq:sublin_batch_iter_tradeoff_2}, однако размер батча огранчен ещё и количеством параметров $n$, что означает невозможность оптимизации с любой наперёд заданной точностью без замены линейной скорости сходимости на сублинейную при $n < m$. Зато теперь есть возможность для некоторых $\epsilon > 0$ решить задачу с линейной скоростью сходимости по $k$. Также стоит обратить внимание на тот факт, что выполнение предположения \ref{as:5} приводит, в среднем, к монотонному неувеличению квадрата нормы градиента в \eqref{eq:main_stoch_lin_conv_1}.

\subsection{Использование дважды стохастического шага}

В текущем подразделе представлены основные свойства метода Гаусса--Ньютона, основанном на правиле вычисления приближения решения \eqref{eq:double_stoch_direct_update_rule}. Для правила \eqref{eq:double_stoch_direct_update_rule} разработан алгоритм решения задачи \eqref{eq:main_opt_problem}, описанный в схеме \ref{alg:gen_double_stoch_gnm}. 
\RestyleAlgo{boxruled}
\begin{algorithm}[!ht]{}
\caption{\textbf{Общий метод трёх стохастических квадратов с дважды стохастическим шагом}}
\label{alg:gen_double_stoch_gnm}
\textbf{Вход:}
    \begin{equation*}
        \begin{cases}
            x_{0}\in E_{1}\text{ --- начальное приближение};\\
            N\in\mathbb{N}\text{ --- количество итераций метода};\\
            \gamma\geq1\text{ --- фактор верхней границы поиска }l_{\hat{g}_{2}};\\
            l\text{ --- оценка локальной постоянной Липшица},~l\in(0,~\gamma l_{\hat{g}_{2}}],~l_{0} = l;\\
            \mathcal{T}(\cdot)\text{ --- функция, определяющая значение произведения $\tau L$};\\
            \mathcal{B}\text{ --- выборка функций};\\
            b, \tilde{b}\in\overline{1, m}\text{ --- размеры $B_{k}, \tilde{B}_{k}\subseteq\mathcal{B}\text{ соответственно},~k\in\mathbb{Z}_{+}$}.
        \end{cases}
    \end{equation*}
    \vspace{0.2cm}
    \textbf{Повторять для $k = 0, 1,~\dots, N - 1$:}
    \begin{itemize}
        \item[1.] сэмплировать батчи $B_{k}, \tilde{B}_{k}$ из $\mathcal{B}$ размеров $b, \tilde{b}$ соответственно;
        \item[2.] определить $\tilde{\tau}_{k}L_{k} = \mathcal{T}(x_{k}, l_{k}, \tilde{B}_{k})$;
        \item[3.] вычислить $x_{k + 1}\in E_{1}$ согласно правилу \eqref{eq:double_stoch_direct_update_rule};
        \item[4.] если $\gamma = 1$ и $l = l_{\hat{g}_{2}}$, то перейти к пункту 1;
        \item[5.] если $\hat{g}_{2}(x_{k + 1}, B_{k}) > \hat{\varphi}_{x_{k}, l_{k}}(x_{k + 1}, B_{k})$, то положить $l_{k} := \min\left\{2l_{k}, \gamma l_{\hat{g}_{2}}\right\}$ и вернуться к пункту 2;
        \item[6.] $l_{k + 1} = \max\left\{\frac{l_{k}}{2},~l\right\}$.
    \end{itemize}
    \vspace{0.2cm}

    \textbf{Выход:} $x_{N}$.
\end{algorithm}
Данная схема оптимизационного процесса заключается в минимизации на каждом шаге отличной от использованной в схеме \ref{alg:gen_stoch_gnm} локальной модели $\hat{\varphi}_{x, l}$, вывод которой описан в лемме \ref{lm:aux_g2_local_model}:
\begin{equation*}
    \begin{aligned}
        \hat{g}_{2}(y, B)\leq&\hat{\varphi}_{x, l}(y, B) \overset{\operatorname{def}}{=} \hat{g}_{2}(x, B) + \left\langle\nabla_{x}\hat{g}_{2}(x, B),~y - x\right\rangle + \frac{l}{2}\|y - x\|^{2},~l\geq l_{\hat{g}_{2}} \overset{\operatorname{def}}{=} \underbrace{2\left(M_{\hat{G}}^{2} + L_{\hat{F}}P_{\hat{g}_{1}}\right)}_{\text{согласно лемме \ref{lm:aux_g2_lipschitz_gradient}}},\\
        &(x, y)\in E_{1}^{2},~B\subseteq\mathcal{B}.
    \end{aligned}
\end{equation*}
Так как в схеме \ref{alg:gen_double_stoch_gnm} используется оптимизация локальной модели $\hat{\varphi}_{x, l}$, значение $\tilde{\tau}L_{k}$ уже не играет той ключевой роли, какую могло бы играть в схеме \ref{alg:gen_stoch_gnm}, и в схеме \ref{alg:gen_double_stoch_gnm} производится бинарный поиск $l_{k}$, а от значения $\tilde{\tau}_{k}L_{k}$ требуется только положительность. При явном указании постоянной Липшица градиента функции $\hat{g}_{2}$ в схеме \ref{alg:gen_double_stoch_gnm} адаптивный поиск $l_{k}$ не производится, в отличие от схемы \ref{alg:gen_stoch_gnm}.

Основное достоинство стратегии \eqref{eq:double_stoch_direct_update_rule} по сравнению с правилом \eqref{eq:stoch_direct_update_rule} состоит в более гибких получаемых оценках сходимости, как указано в теоремах \ref{th:double_stoch_sublin_conv} и \ref{th:double_stoch_lin_conv}.

\begin{theorem}\label{th:double_stoch_sublin_conv_main}
    Пусть выполнены предположения \ref{as:1}, \ref{as:2}, \ref{as:3} и \ref{as:4}. Рассмотрим метод Гаусса--Ньютона, реализованный по схеме \ref{alg:gen_double_stoch_gnm} со стратегией вычисления $x_{k + 1}$ \eqref{eq:double_stoch_direct_update_rule}, в которой
    \begin{equation*}
    \begin{aligned}
        \eta_{k} &= \frac{2\left\langle\hat{G}^{'}(x_{k}, B_{k})^{*}\hat{G}(x_{k}, B_{k}),~\left(\hat{G}^{'}(x_{k}, \tilde{B}_{k})^{*}\hat{G}^{'}(x_{k}, \tilde{B}_{k}) + \tilde{\tau}_{k}L_{k}I_{n}\right)^{-1}\hat{G}^{'}(x_{k}, B_{k})^{*}\hat{G}(x_{k}, B_{k})\right\rangle}{l_{k}\left\langle\hat{G}^{'}(x_{k}, B_{k})^{*}\hat{G}(x_{k}, B_{k}),~\left(\hat{G}^{'}(x_{k}, \tilde{B}_{k})^{*}\hat{G}^{'}(x_{k}, \tilde{B}_{k}) + \tilde{\tau}_{k}L_{k}I_{n}\right)^{-2}\hat{G}^{'}(x_{k}, B_{k})^{*}\hat{G}(x_{k}, B_{k})\right\rangle},\\
        &\tilde{\mathcal{T}}\geq\tilde{\tau}_{k}\geq\tilde{\tau} > 0,~L_{k}\geq L > 0,~k\in\mathbb{Z}_{+}.
    \end{aligned}
\end{equation*}
Тогда при независимом сэмплировании $B_{k}$ и $\tilde{B}_{k}$:
\begin{equation}\label{eq:double_stoch_sublin_conv_main_two_batches}
    \begin{aligned}
        \mathbb{E}\left[\min\limits_{i\in\overline{0, k - 1}}\left\{\left\|\nabla\hat{f}_{2}(x_{i})\right\|^{2}\right\}\right]&\leq2\gamma l_{\hat{g}_{2}}\left(\frac{M_{\hat{G}}^{2}}{\tilde{\tau}L} + 1\right)^{2}\left(\frac{\mathbb{E}\left[\hat{f}_{2}(x_{0})\right]}{k} +\right.\\
        &\left.+ \frac{4l_{\hat{F}}M_{\hat{G}}P_{\hat{g}_{1}}}{l}\left(\frac{M_{\hat{G}}^{2}}{\tilde{\tau}L} + 1\right)^{2} + \tilde{\sigma}\sqrt{\frac{1}{b} - \frac{1}{m}}\right),~k\in\mathbb{N}.
    \end{aligned}
\end{equation}
В случае сэмплирования одного батча на каждом шаге ($B_{k}\equiv\tilde{B}_{k}$) оценка сходимости следующая:
\begin{equation}\label{eq:double_stoch_sublin_conv_main_one_batch}
    \begin{aligned}
        \mathbb{E}\left[\min\limits_{i\in\overline{0, k - 1}}\left\{\left\|\nabla\hat{f}_{2}(x_{i})\right\|^{2}\right\}\right]&\leq2\gamma l_{\hat{g}_{2}}\left(\frac{M_{\hat{G}}^{2}}{\tilde{\tau}L} + 1\right)^{2}\left(\frac{\mathbb{E}\left[\hat{f}_{2}(x_{0})\right]}{k} +\right.\\
        &\left.+ 2l_{\hat{F}}\min\left\{\sqrt{\frac{1}{L}\left(\tilde{\mathcal{T}} + \frac{P_{\hat{g}_{1}}^{2}}{\tilde{\tau}}\right)},~\frac{2M_{\hat{G}}P_{\hat{g}_{1}}}{l}\left(\frac{M_{\hat{G}}^{2}}{\tilde{\tau}L} + 1\right)^{2}\right\}\mathds{1}_{\left\{b < m\right\}} + \tilde{\sigma}\sqrt{\frac{1}{b} - \frac{1}{m}}\right),\\
        &k\in\mathbb{N}.
    \end{aligned}
\end{equation}
Оператор математического ожидания $\mathbb{E}\left[\cdot\right]$ усредняет по всей случайности процесса оптимизации.
\end{theorem}

Теорема \ref{th:double_stoch_sublin_conv} требует тонкой настройки шага метода $\eta_{k}$ для достижения указанной оценки сходимости, однако из доказательства леммы \ref{lm:aux_double_stoch_update} выводится допустимый полуинтервал значений фактора шага метода:
$$\eta_{k}\in\left(0,~\frac{2\left\langle\hat{G}^{'}(x_{k}, B_{k})^{*}\hat{G}(x_{k}, B_{k}),~\left(\hat{G}^{'}(x_{k}, \tilde{B}_{k})^{*}\hat{G}^{'}(x_{k}, \tilde{B}_{k}) + \tilde{\tau}_{k}L_{k}I_{n}\right)^{-1}\hat{G}^{'}(x_{k}, B_{k})^{*}\hat{G}(x_{k}, B_{k})\right\rangle}{l_{k}\left\langle\hat{G}^{'}(x_{k}, B_{k})^{*}\hat{G}(x_{k}, B_{k}),~\left(\hat{G}^{'}(x_{k}, \tilde{B}_{k})^{*}\hat{G}^{'}(x_{k}, \tilde{B}_{k}) + \tilde{\tau}_{k}L_{k}I_{n}\right)^{-2}\hat{G}^{'}(x_{k}, B_{k})^{*}\hat{G}(x_{k}, B_{k})\right\rangle}\right],~k\in\mathbb{Z}_{+},$$
при которых сходимость будет наблюдаться, но с уже меньшей скоростью, при этом значение $\tilde{\tau}_{k}$ может быть случайной величиной, для которой выполнено свойство ограниченности: $\tilde{\tau}_{k}\in[\tilde{\tau},~\tilde{\mathcal{T}}~]$, что оправдывается обычно внесением неявной регуляризации в решение задачи \eqref{eq:main_opt_problem}, типично применимое в задачах минимизации риска и статистического обучения. Главное достоинство правила \eqref{eq:double_stoch_direct_update_rule} при сэмплировании одного батча на каждом шаге метода состоит в наличии слагаемого
$$4\gamma l_{\hat{g}_{2}}l_{\hat{F}}\left(\frac{M_{\hat{G}}^{2}}{\tilde{\tau}L} + 1\right)^{2}\min\left\{\sqrt{\frac{1}{L}\left(\tilde{\mathcal{T}} + \frac{P_{\hat{g}_{1}}^{2}}{\tilde{\tau}}\right)},~\frac{2M_{\hat{G}}P_{\hat{g}_{1}}}{l}\left(\frac{M_{\hat{G}}^{2}}{\tilde{\tau}L} + 1\right)^{2}\right\},$$
которое в стохастическом режиме можно сделать сколь угодно малым благодаря увеличению значения $L$, а для достаточно большого значения $L$ это слагаемое становится равным
$$4\gamma l_{\hat{g}_{2}}l_{\hat{F}}\left(\frac{M_{\hat{G}}^{2}}{\tilde{\tau}L} + 1\right)^{2}\sqrt{\frac{1}{L}\left(\tilde{\mathcal{T}} + \frac{P_{\hat{g}_{1}}^{2}}{\tilde{\tau}}\right)}.$$
Условия теоремы \ref{th:double_stoch_sublin_conv} демонстрируют полное превосходство изначальной стратегии сэмплирования по одному батчу на каждом шаге над двойной стохастикой из--за неустранимого фактора
$$\frac{8\gamma l_{\hat{g}_{2}}l_{\hat{F}}M_{\hat{G}}P_{\hat{g}_{1}}}{l}$$
в оценке сходимости дважды стохастического шага, возникающего вследствие наличия шума от сэмплирования батчей. Как и в случае правила \eqref{eq:stoch_direct_update_rule}, оценку сходимости \eqref{eq:double_stoch_sublin_conv_main_one_batch} в режиме сэмплирования по одному батчу на шаге можно разделить на три слагаемых, каждое из которых возможно уменьшить, настраивая один параметр независимо, что сводит условие получения средней минимальной нормы на уровне $\epsilon > 0$ к следующей системе неравенств:
\begin{equation*}
    \begin{cases}
        2\gamma l_{\hat{g}_{2}}\left(\frac{M_{\hat{G}}^{2}}{\tilde{\tau}L} + 1\right)^{2}\frac{\mathbb{E}\left[\hat{g}_{2}(x_{0}, B_{0})\right]}{k}\leq r_{1}\epsilon^{2},~r_{1}\in(0, 1);\\[10pt]
        4\gamma l_{\hat{g}_{2}}l_{\hat{F}}\left(\frac{M_{\hat{G}}^{2}}{\tilde{\tau}L} + 1\right)^{2}\sqrt{\frac{1}{L}\left(\tilde{\mathcal{T}} + \frac{P_{\hat{g}_{1}}^{2}}{\tilde{\tau}}\right)}\mathds{1}_{\left\{b < m\right\}}\leq r_{2}\epsilon^{2},~r_{2}\in(0, 1);\\[10pt]
        2\gamma l_{\hat{g}_{2}}\left(\frac{M_{\hat{G}}^{2}}{\tilde{\tau}L} + 1\right)^{2}\tilde{\sigma}\sqrt{\frac{1}{b} - \frac{1}{m}}\leq r_{3}\epsilon^{2},~r_{3}\in(0, 1);\\[5pt]
        r_{1} + r_{2} + r_{3} = 1;
    \end{cases}
\end{equation*}
из которой выводятся минимальные значения количества итераций, нижней границы поиска постоянной Липшица, размера батча, на котором оценивается градиент функции $\hat{g}_{2}$:
\begin{equation}\label{eq:double_stoch_sublin_conv_min_hyperparams}
    \begin{cases}
        k = \left\lceil2\gamma l_{\hat{g}_{2}}\left(\frac{M_{\hat{G}}^{2}}{\tilde{\tau}L} + 1\right)^{2}\frac{\mathbb{E}\left[\hat{g}_{2}(x_{0}, B_{0})\right]}{r_{1}\epsilon^{2}}\right\rceil,~r_{1}\in(0, 1);\\[10pt]
        L = \min\limits_{c > 1}\max\left\{\frac{M_{\hat{G}}^{2}}{\tilde{\tau}(\sqrt{c} - 1)},~\left(\tilde{\mathcal{T}} + \frac{P_{\hat{g}_{1}}^{2}}{\tilde{\tau}}\right)\left(\frac{4c\gamma l_{\hat{g}_{2}}l_{\hat{F}}}{r_{2}\epsilon^{2}}\right)^{2}\right\},~r_{2}\in(0, 1);\\[10pt]
        b = \min\left\{m,~\left\lceil\frac{\left(2\gamma l_{\hat{g}_{2}}\left(\frac{M_{\hat{G}}^{2}}{\tilde{\tau}L} + 1\right)^{2}\tilde{\sigma}\right)^{2}}{r_{3}^{2}\epsilon^{4} + \frac{1}{m}\left(2\gamma l_{\hat{g}_{2}}\left(\frac{M_{\hat{G}}^{2}}{\tilde{\tau}L} + 1\right)^{2}\tilde{\sigma}\right)^{2}}\right\rceil\right\},~r_{3}\in(0, 1);\\[10pt]
        r_{1} + r_{2} + r_{3} = 1;
    \end{cases}
\end{equation}
в асимптотической форме это означает следующие оценки:
$$k = \operatorname{O}\left(\frac{1}{\epsilon^{2}}\right),~b = \min\left\{m,~\operatorname{O}\left(\frac{1}{\epsilon^{4}}\right)\right\}.$$
Оценка на $L$ получена через введение вспомогательной переменной $c > 1$:
\begin{equation*}
    \begin{aligned}
        &\underbrace{\left(\frac{M_{\hat{G}}^{2}}{\tilde{\tau}L} + 1\right)^{2}}_{\leq c}\underbrace{4\gamma l_{\hat{g}_{2}}l_{\hat{F}}\sqrt{\frac{1}{L}\left(\tilde{\mathcal{T}} + \frac{P_{\hat{g}_{1}}^{2}}{\tilde{\tau}}\right)}}_{\leq\frac{r_{2}\epsilon^{2}}{c}}\mathds{1}_{\left\{b < m\right\}}\leq r_{2}\epsilon^{2}.
    \end{aligned}
\end{equation*}
В режиме дважды стохастического шага минимальные значения количества итераций и размера батча $B_{k}$ для оценки \eqref{eq:double_stoch_sublin_conv_main_two_batches} имеют ту же асимптотику, что и для оценки \eqref{eq:double_stoch_sublin_conv_main_one_batch}, и описываются системой уравнений ниже:
\begin{equation*}
    \begin{cases}
        k = \left\lceil2\gamma l_{\hat{g}_{2}}\left(\frac{M_{\hat{G}}^{2}}{\tilde{\tau}L} + 1\right)^{2}\frac{\mathbb{E}\left[\hat{g}_{2}(x_{0}, B_{0})\right]}{(1 - r)\epsilon^{2}}\right\rceil,~r\in(0, 1);\\[10pt]
        b = \min\left\{m,~\left\lceil\frac{\left(2\gamma l_{\hat{g}_{2}}\left(\frac{M_{\hat{G}}^{2}}{\tilde{\tau}L} + 1\right)^{2}\left(\frac{4l_{\hat{F}}M_{\hat{G}}P_{\hat{g}_{1}}}{l}\sqrt{m(m - 1)}\left(\frac{M_{\hat{G}}^{2}}{\tilde{\tau}L} + 1\right)^{2} + \tilde{\sigma}\right)\right)^{2}}{r^{2}\epsilon^{4} + \frac{1}{m}\left(2\gamma l_{\hat{g}_{2}}\left(\frac{M_{\hat{G}}^{2}}{\tilde{\tau}L} + 1\right)^{2}\left(\frac{4l_{\hat{F}}M_{\hat{G}}P_{\hat{g}_{1}}}{l}\sqrt{m(m - 1)}\left(\frac{M_{\hat{G}}^{2}}{\tilde{\tau}L} + 1\right)^{2} + \tilde{\sigma}\right)\right)^{2}}\right\rceil\right\};\\
        \tilde{b}\in\overline{1,~m}.
    \end{cases}
\end{equation*}
В теореме \ref{th:double_stoch_lin_conv} уже установлены условия линейной сходимости в методе Гаусса--Ньютона с правилом обновления \eqref{eq:double_stoch_direct_update_rule}.
\begin{theorem}\label{th:double_stoch_lin_conv_main}
    Пусть выполнены предположения \ref{as:1}, \ref{as:2}, \ref{as:3}, \ref{as:4} и \ref{as:5}. Рассмотрим метод Гаусса--Ньютона, реализованный по схеме \ref{alg:gen_double_stoch_gnm} со стратегией вычисления $x_{k + 1}$ \eqref{eq:double_stoch_direct_update_rule}, в которой
    \begin{equation*}
    \begin{aligned}
        \eta_{k} &= \frac{2\left\langle\hat{G}^{'}(x_{k}, B_{k})^{*}\hat{G}(x_{k}, B_{k}),~\left(\hat{G}^{'}(x_{k}, \tilde{B}_{k})^{*}\hat{G}^{'}(x_{k}, \tilde{B}_{k}) + \tilde{\tau}_{k}L_{k}I_{n}\right)^{-1}\hat{G}^{'}(x_{k}, B_{k})^{*}\hat{G}(x_{k}, B_{k})\right\rangle}{l_{k}\left\langle\hat{G}^{'}(x_{k}, B_{k})^{*}\hat{G}(x_{k}, B_{k}),~\left(\hat{G}^{'}(x_{k}, \tilde{B}_{k})^{*}\hat{G}^{'}(x_{k}, \tilde{B}_{k}) + \tilde{\tau}_{k}L_{k}I_{n}\right)^{-2}\hat{G}^{'}(x_{k}, B_{k})^{*}\hat{G}(x_{k}, B_{k})\right\rangle},\\
        &\tilde{\mathcal{T}}\geq\tilde{\tau}_{k}\geq\tilde{\tau} > 0,~L_{k}\geq L > 0,~k\in\mathbb{Z}_{+}.
    \end{aligned}
\end{equation*}
Тогда:
\begin{equation*}
    \begin{aligned}
        &\begin{cases}
            \mathbb{E}\left[\left\|\nabla\hat{f}_{2}(x_{k})\right\|^{2}\right]\leq4M_{\hat{G}}^{2}\hat{\Delta}_{k,b};\\[10pt]
            \mathbb{E}\left[\hat{f}_{2}(x_{k})\right]\leq\hat{f}_{2}^{*} + \hat{\Delta}_{k,b};
        \end{cases}
    \end{aligned}
\end{equation*}
где при $k\in\mathbb{Z}_{+}$ и $b\in\overline{1,~\min\{m, n\}}$ оценка $\hat{\Delta}_{k, b}$ определяется следующим образом:
\begin{equation}\label{eq:double_stoch_lin_conv_main_two_batches}
    \begin{aligned}
        \hat{\Delta}_{k, b} &= \mathbb{E}\left[\hat{f}_{2}(x_{0})\right]\exp\left(-\frac{2\mu k}{\gamma l_{\hat{g}_{2}}}\left(\frac{\tilde{\tau}L}{M_{\hat{G}}^{2} + \tilde{\tau}L}\right)^{2}\right) +\\
        &+ \frac{\gamma l_{\hat{g}_{2}}}{\mu}\left(\tilde{\sigma}\sqrt{\frac{1}{b} - \frac{1}{m}} + \frac{2l_{\hat{F}}M_{\hat{G}}P_{\hat{g}_{1}}}{l}\left(\frac{M_{\hat{G}}^{2}}{\tilde{\tau}L} + 1\right)^{2}\mathds{1}_{\left\{b < m\right\}}\right)\left(\frac{M_{\hat{G}}^{2}}{\tilde{\tau}L} + 1\right)^{2}
    \end{aligned}
\end{equation}
при независимом сэмплировании $B_{k}$ и $\tilde{B}_{k}$ и
\begin{equation}\label{eq:double_stoch_lin_conv_main_one_batch}
    \begin{aligned}
        \hat{\Delta}_{k, b} &= \mathbb{E}\left[\hat{f}_{2}(x_{0})\right]\exp\left(-\frac{2\mu k}{\gamma l_{\hat{g}_{2}}}\left(\frac{\tilde{\tau}L}{M_{\hat{G}}^{2} + \tilde{\tau}L}\right)^{2}\right) +\\
        &+ \frac{\gamma l_{\hat{g}_{2}}}{\mu}\left(\tilde{\sigma}\sqrt{\frac{1}{b} - \frac{1}{m}} + l_{\hat{F}}\min\left\{\sqrt{\frac{1}{L}\left(\tilde{\mathcal{T}} + \frac{P_{\hat{g}_{1}}^{2}}{\tilde{\tau}}\right)},~\frac{2M_{\hat{G}}P_{\hat{g}_{1}}}{l}\left(\frac{M_{\hat{G}}^{2}}{\tilde{\tau}L} + 1\right)^{2}\right\}\mathds{1}_{\left\{b < m\right\}}\right)\left(\frac{M_{\hat{G}}^{2}}{\tilde{\tau}L} + 1\right)^{2}
    \end{aligned}
\end{equation}
в случае сэмплирования на каждом шаге одного батча ($B_{k}\equiv\tilde{B}_{k}$). Оператор математического ожидания $\mathbb{E}\left[\cdot\right]$ усредняет по всей случайности процесса оптимизации.
\end{theorem}

Как и в теореме \ref{th:double_stoch_sublin_conv}, замена локальной модели $\hat{\psi}_{x_{k}, L_{k}, \tau_{k}}$ на модель $\hat{\varphi}_{x_{k}, l_{k}}$ при сэмплировании одного батча на шаге метода позволяет избавиться от слагаемого, пропорционального верхней оценке квадрата нормы градиента функции $\hat{g}_{2}$:
$$\frac{4\gamma l_{\hat{g}_{2}}l_{\hat{F}}M_{\hat{G}}^{2}}{\mu}\sqrt{\frac{1}{L}\left(\tilde{\mathcal{T}} + \frac{P_{\hat{g}_{1}}^{2}}{\tilde{\tau}}\right)}\left(\frac{M_{\hat{G}}^{2}}{\tilde{\tau}L} + 1\right)^{2}\mathds{1}_{\left\{b < m\right\}}.$$
Соответственно, аналогичным образом для оценки \eqref{eq:double_stoch_lin_conv_main_one_batch} сходимость средней нормы градиента к уровню $\epsilon > 0$ задаётся следующей системой неравенств:
\begin{equation*}
    \begin{cases}
        4M_{\hat{G}}^{2}\mathbb{E}\left[\hat{g}_{2}(x_{0}, B_{0})\right]\exp\left(-\frac{2\mu k}{\gamma l_{\hat{g}_{2}}}\left(\frac{\tilde{\tau}L}{M_{\hat{G}}^{2} + \tilde{\tau}L}\right)^{2}\right)\leq r_{1}\epsilon^{2},~r_{1}\in(0, 1);\\[10pt]
        \frac{4\gamma l_{\hat{g}_{2}}l_{\hat{F}}M_{\hat{G}}^{2}}{\mu}\sqrt{\frac{1}{L}\left(\tilde{\mathcal{T}} + \frac{P_{\hat{g}_{1}}^{2}}{\tilde{\tau}}\right)}\left(\frac{M_{\hat{G}}^{2}}{\tilde{\tau}L} + 1\right)^{2}\mathds{1}_{\left\{b < m\right\}}\leq r_{2}\epsilon^{2},~r_{2}\in(0, 1);\\[10pt]
        \frac{4\gamma l_{\hat{g}_{2}}M_{\hat{G}}^{2}}{\mu}\left(\frac{M_{\hat{G}}^{2}}{\tilde{\tau}L} + 1\right)^{2}\tilde{\sigma}\sqrt{\frac{1}{b} - \frac{1}{m}}\leq r_{3}\epsilon^{2},~r_{3}\in(0, 1);\\[5pt]
        r_{1} + r_{2} + r_{3} = 1;
    \end{cases}
\end{equation*}
порождая минимальные значения количества итераций, нижней границы поиска постоянной Липшица и размера батча, на котором оценивается градиент функции $\hat{g}_{2}$:
\begin{equation}\label{eq:double_stoch_lin_conv_min_hyperparams}
    \begin{cases}
        k = \left\lceil\frac{\gamma l_{\hat{g}_{2}}}{2\mu}\left(\frac{M_{\hat{G}}^{2}}{\tilde{\tau}L} + 1\right)^{2}\ln\left(\frac{4M_{\hat{G}}^{2}\mathbb{E}\left[\hat{g}_{2}(x_{0}, B_{0})\right]}{r_{1}\epsilon^{2}}\right)\right\rceil,~r_{1}\in(0, 1);\\[10pt]
        L = \min\limits_{c > 1}\max\left\{\frac{M_{\hat{G}}^{2}}{\tilde{\tau}(\sqrt{c} - 1)},~\left(\tilde{\mathcal{T}} + \frac{P_{\hat{g}_{1}}^{2}}{\tilde{\tau}}\right)\left(\frac{4c\gamma l_{\hat{g}_{2}}l_{\hat{F}}M_{\hat{G}}^{2}}{\mu r_{2}\epsilon^{2}}\right)^{2}\right\},~r_{2}\in(0, 1);\\[10pt]
        b = \min\left\{m,~n,~\left\lceil\frac{\left(\frac{4\gamma l_{\hat{g}_{2}}M_{\hat{G}}^{2}}{\mu}\left(\frac{M_{\hat{G}}^{2}}{\tilde{\tau}L} + 1\right)^{2}\tilde{\sigma}\right)^{2}}{r_{3}^{2}\epsilon^{4} + \frac{1}{m}\left(\frac{4\gamma l_{\hat{g}_{2}}M_{\hat{G}}^{2}}{\mu}\left(\frac{M_{\hat{G}}^{2}}{\tilde{\tau}L} + 1\right)^{2}\tilde{\sigma}\right)^{2}}\right\rceil\right\},~r_{3}\in(0, 1);\\[10pt]
        r_{1} + r_{2} + r_{3} = 1.
    \end{cases}
\end{equation}
Оценка на $L$ так же получена через введение вспомогательной переменной $c > 1$, как и в случае теоремы \ref{th:double_stoch_sublin_conv}:
\begin{equation*}
    \begin{aligned}
        &\underbrace{\left(\frac{M_{\hat{G}}^{2}}{\tilde{\tau}L} + 1\right)^{2}}_{\leq c}\underbrace{\frac{4\gamma l_{\hat{g}_{2}}l_{\hat{F}}M_{\hat{G}}^{2}}{\mu}\sqrt{\frac{1}{L}\left(\tilde{\mathcal{T}} + \frac{P_{\hat{g}_{1}}^{2}}{\tilde{\tau}}\right)}}_{\leq\frac{r_{2}\epsilon^{2}}{c}}\mathds{1}_{\left\{b < m\right\}}\leq r_{2}\epsilon^{2}.
    \end{aligned}
\end{equation*}
В асимптотической форме ясно видна линейная сходимость по количеству итераций, при этом есть зависимость от размера батча, на котором оценивается градиент локальной модели, но нет зависимости от размера батча, на котором оценивается гессиан локальной модели:
$$k = \operatorname{O}\left(\ln\left(\frac{1}{\epsilon}\right)\right),~b = \min\left\{m,~n,~\operatorname{O}\left(\frac{1}{\epsilon^{4}}\right)\right\}.$$
Для оценки сходимости с дважды стохастическим шагом \eqref{eq:double_stoch_lin_conv_main_two_batches} минимальное количество итераций и минимальный размер батча имеют аналогичную асимптотику:
\begin{equation*}
    \begin{cases}
        k = \left\lceil\frac{\gamma l_{\hat{g}_{2}}}{2\mu}\left(\frac{M_{\hat{G}}^{2}}{\tilde{\tau}L} + 1\right)^{2}\ln\left(\frac{4M_{\hat{G}}^{2}\mathbb{E}\left[\hat{g}_{2}(x_{0}, B_{0})\right]}{(1 - r)\epsilon^{2}}\right)\right\rceil,~r\in(0, 1);\\[10pt]
        b = \min\left\{m,~n,~\left\lceil\frac{\left(\frac{\gamma l_{\hat{g}_{2}}}{\mu}\left(\frac{2l_{\hat{F}}M_{\hat{G}}P_{\hat{g}_{1}}}{l}\sqrt{m(m - 1)}\left(\frac{M_{\hat{G}}^{2}}{\tilde{\tau}L} + 1\right)^{2} + \tilde{\sigma}\right)\left(\frac{M_{\hat{G}}^{2}}{\tilde{\tau}L} + 1\right)^{2}\right)^{2}}{r^{2}\epsilon^{4} + \frac{1}{m}\left(\frac{\gamma l_{\hat{g}_{2}}}{\mu}\left(\frac{2l_{\hat{F}}M_{\hat{G}}P_{\hat{g}_{1}}}{l}\sqrt{m(m - 1)}\left(\frac{M_{\hat{G}}^{2}}{\tilde{\tau}L} + 1\right)^{2} + \tilde{\sigma}\right)\left(\frac{M_{\hat{G}}^{2}}{\tilde{\tau}L} + 1\right)^{2}\right)^{2}}\right\rceil\right\};\\
        \tilde{b}\in\overline{1,~m}.
    \end{cases}
\end{equation*}

В условиях теорем \ref{th:double_stoch_sublin_conv} и \ref{th:double_stoch_lin_conv} значение размера батча, по которому оценивается гессиан стохастической локальной модели может быть произвольным, так как только границы спектра гессиана вносят вклад в оценку сходимости, а в экстремальные свойства спектра матрицы Якоби $\hat{G}^{'}(x, B)$ в предположениях данной работы не заложен явно размер батча $\tilde{b}$ и оценка
$$\left\|\hat{G}^{'}(x, B)\right\|\leq M_{\hat{G}}$$
предполагается выполненной для произвольного размера батча $|B|\in\overline{1, m}$. Оценки \eqref{eq:double_stoch_sublin_conv_min_hyperparams}, \eqref{eq:double_stoch_lin_conv_min_hyperparams} явно указывают на то, что метод будет только быстрее сходиться, если неограниченно увеличивать $\tilde{\tau}L$, сохраняя выполнение всех неравенств в условиях теорем \ref{th:double_stoch_sublin_conv} и \ref{th:double_stoch_lin_conv}. Данное действие приводит к замене правила \eqref{eq:double_stoch_direct_update_rule} на обновление градиентного спуска (следствия \ref{th:double_stoch_sublin_conv_cor1} и \ref{th:double_stoch_lin_conv_cor1}):
\begin{equation*}
    \begin{aligned}
        &x_{k + 1} = x_{k} - \frac{1}{l_{k}}\nabla_{x_{k}}\hat{g}_{2}(x_{k}, B_{k}),~k\in\mathbb{Z}_{+},
    \end{aligned}
\end{equation*}
что представляет метод стохастического градиентного спуска в оценках среднего как наиболее быстрый и предельный в стратегии обновления \eqref{eq:double_stoch_direct_update_rule} при выполнении тождества $B_{k}\equiv\tilde{B}_{k}$. Более того, двойная стохастика в теоремах \ref{th:double_stoch_sublin_conv} и \ref{th:double_stoch_lin_conv} хуже сходится относительно шума батча, и поэтому спокойно можно оценивать $x_{k + 1}$ по схеме \ref{alg:gen_double_stoch_gnm}, используя только один батч на каждом шаге метода, то есть технически положить $B_{k}\equiv\tilde{B}_{k}$ всё так же из--за того, что только границы спектра матрицы Якоби $\hat{G}^{'}(x_{k}, \tilde{B}_{k})^{*}$ используются в оценках сходимости.

Однако предельный случай  \eqref{eq:double_stoch_direct_update_rule} с $B_{k}\equiv\tilde{B}_{k}$ всё-таки не всегда быстрее правила \eqref{eq:stoch_direct_update_rule} сходится в терминах итераций при известной постоянной Липшица ($\gamma = 1$), если сравнить с соответствующими оценками из теорем \ref{th:3} и \ref{th:4}, взяв единичный масштаб шага $\eta = 1$:
\begin{equation*}
    \begin{aligned}
        \underbrace{4\left(M_{\hat{G}}^{2} + P_{\hat{g}_{1}}L_{\hat{F}}\right)\left(\frac{M_{\hat{G}}^{2}}{\tilde{\tau}L} + 1\right)^{2}\frac{\mathbb{E}\left[\hat{g}_{2}(x_{0}, B_{0})\right]}{k}}_{\text{теорема \ref{th:double_stoch_sublin_conv},}~\gamma = 1,~\tilde{\tau}L\geq\frac{M_{\hat{G}}^{2}}{\sqrt{2}- 1}}&\leq\underbrace{\frac{8(M_{\hat{G}}^{2} + P_{\hat{g}_{1}}L_{\hat{F}})\mathbb{E}\left[\hat{g}_{2}(x_{0}, B_{0})\right]}{k}}_{\text{теорема \ref{th:3},~}\gamma = 1,~\eta = 1}\leq\\
        &\leq\underbrace{4\left(M_{\hat{G}}^{2} + P_{\hat{g}_{1}}L_{\hat{F}}\right)\left(\frac{M_{\hat{G}}^{2}}{\tilde{\tau}L} + 1\right)^{2}\frac{\mathbb{E}\left[\hat{g}_{2}(x_{0}, B_{0})\right]}{k}}_{\text{теорема \ref{th:double_stoch_sublin_conv},}~\gamma = 1,~0 < \tilde{\tau}L\leq\frac{M_{\hat{G}}^{2}}{\sqrt{2}- 1}}.
    \end{aligned}
\end{equation*}
В случае сублинейной сходимости одного увеличения $\tilde{\tau}L$ достаточно для улучшения сходимости по схеме \ref{alg:gen_double_stoch_gnm} по сравнению со схемой \ref{alg:gen_stoch_gnm}, если известны постоянные Липшица $L_{\hat{F}}$ и $l_{\hat{g}_{2}}$ для задачи \eqref{eq:main_opt_problem}. Предельное значение правила \eqref{eq:double_stoch_direct_update_rule} в схеме \ref{alg:gen_double_stoch_gnm} в случае линейной сходимости может быть эффективнее правила \eqref{eq:stoch_direct_update_rule} в схеме \ref{alg:gen_stoch_gnm} только для тех задач \eqref{eq:main_opt_problem}, в которых выполнено следующее структурное свойство $M_{\hat{G}}^{2}\leq2\mu + L_{\hat{F}}P_{\hat{g}_{1}}$:
\begin{equation*}
    \begin{aligned}
        4M_{\hat{G}}^{2}\mathbb{E}\left[\hat{g}_{2}(x_{0}, B_{0})\right]\underbrace{\exp\left(-\frac{\mu k}{L_{\hat{F}}P_{\hat{g}_{1}} + M_{\hat{G}}^{2}}\right)}_{\substack{\text{теорема \ref{th:double_stoch_lin_conv},}~\gamma = 1,~\tilde{\tau}L\rightarrow+\infty,\\\mu\leq M_{\hat{G}}^{2}\leq2\mu + L_{\hat{F}}P_{\hat{g}_{1}}}}&\leq4M_{\hat{G}}^{2}\mathbb{E}\left[\hat{g}_{2}(x_{0}, B_{0})\right]\underbrace{\exp\left(-\frac{\mu k}{2\left(L_{\hat{F}}P_{\hat{g}_{1}} + \mu\right)}\right)}_{\text{теорема \ref{th:4},~}\gamma = 1,~\eta = 1}\leq\\
        &\leq4M_{\hat{G}}^{2}\mathbb{E}\left[\hat{g}_{2}(x_{0}, B_{0})\right]\underbrace{\exp\left(-\frac{\mu k}{L_{\hat{F}}P_{\hat{g}_{1}} + M_{\hat{G}}^{2}}\left(\frac{\tilde{\tau}L}{M_{\hat{G}}^{2} + \tilde{\tau}L}\right)^{2}\right)}_{\text{теорема \ref{th:double_stoch_lin_conv},}~\gamma = 1,~0 < \tilde{\tau}L\leq\frac{M_{\hat{G}}^{2}}{\sqrt{2} - 1}}.
    \end{aligned}
\end{equation*}
Это же структурное свойство важно для эффективности правила обновления \eqref{eq:double_stoch_direct_update_rule} в схеме \ref{alg:gen_double_stoch_gnm} по сравнению с правилом \eqref{eq:stoch_direct_update_rule} в схеме \ref{alg:gen_stoch_gnm} относительно дисперсии батча при выполнении условия Поляка--Лоясиевича:
\begin{equation*}
    \begin{aligned}
        \underbrace{\frac{8\tilde{\sigma}M_{\hat{G}}^{2}\left(L_{\hat{F}}P_{\hat{g}_{1}} + M_{\hat{G}}^{2}\right)}{\mu}\sqrt{\frac{1}{b} - \frac{1}{m}}}_{\substack{\text{теорема \ref{th:double_stoch_lin_conv},}~\gamma = 1,~\tilde{\tau}L\rightarrow+\infty,\\\mu\leq M_{\hat{G}}^{2}\leq 2\mu + L_{\hat{F}}P_{\hat{g}_{1}}}}&\leq\underbrace{16\tilde{\sigma}M_{\hat{G}}^{2}\left(\frac{L_{\hat{F}}P_{\hat{g}_{1}} + \mu}{\mu}\right)\sqrt{\frac{1}{b} - \frac{1}{m}}}_{\text{теорема \ref{th:4},~}\gamma = 1,~\eta = 1}\leq\\
        &\leq\underbrace{\frac{8\tilde{\sigma}M_{\hat{G}}^{2}\left(L_{\hat{F}}P_{\hat{g}_{1}} + M_{\hat{G}}^{2}\right)}{\mu}\left(\frac{M_{\hat{G}}^{2}}{\tilde{\tau}L} + 1\right)^{2}\sqrt{\frac{1}{b} - \frac{1}{m}}}_{\text{теорема \ref{th:double_stoch_lin_conv},}~\gamma = 1,~0 < \tilde{\tau}L\leq\frac{M_{\hat{G}}^{2}}{\sqrt{2} - 1}}.
    \end{aligned}
\end{equation*}
И даже при отсутствии выполнения предположения \ref{as:5} эффективность схемы \ref{alg:gen_double_stoch_gnm} по сравнению со схемой \ref{alg:gen_stoch_gnm} относительно размера батча может иметь место:
\begin{equation*}
    \begin{aligned}
        \underbrace{4\left(M_{\hat{G}}^{2} + P_{\hat{g}_{1}}L_{\hat{F}}\right)\left(\frac{M_{\hat{G}}^{2}}{\tilde{\tau}L} + 1\right)^{2}\tilde{\sigma}\sqrt{\frac{1}{b} - \frac{1}{m}}}_{\text{теорема \ref{th:double_stoch_sublin_conv},}~\gamma = 1,~\tilde{\tau}L\geq\frac{M_{\hat{G}}^{2}}{\sqrt{2} - 1}}&\leq\underbrace{8(M_{\hat{G}}^{2} + P_{\hat{g}_{1}}L_{\hat{F}})\tilde{\sigma}\sqrt{\frac{1}{b} - \frac{1}{m}}}_{\text{теорема \ref{th:3},~}\gamma = 1,~\eta = 1}\leq\\
        &\leq\underbrace{4\left(M_{\hat{G}}^{2} + P_{\hat{g}_{1}}L_{\hat{F}}\right)\left(\frac{M_{\hat{G}}^{2}}{\tilde{\tau}L} + 1\right)^{2}\tilde{\sigma}\sqrt{\frac{1}{b} - \frac{1}{m}}}_{\text{теорема \ref{th:double_stoch_sublin_conv},}~\gamma = 1,~0 < \tilde{\tau}L\leq\frac{M_{\hat{G}}^{2}}{\sqrt{2} - 1}}.
    \end{aligned}
\end{equation*}

\subsection{Использование неточного шага}\label{subsec:approximate_step}

В данном подразделе рассматривается неточный поиск $x_{k + 1}$ при вычислении последовательности
$$\left\{x_{k}\right\}_{k\in\mathbb{Z}_{+}},$$
но, в отличие от правила \eqref{eq:stoch_direct_update_rule}, вычисление $x_{k + 1}$ выполняется с некоторой погрешностью $\varepsilon_{k} > 0$ согласно правилу \eqref{eq:stoch_approx_general_update_rule}. Ясно, что при $\varepsilon_{k} = 0$ воспроизводится правило \eqref{eq:stoch_direct_update_rule} с $\eta_{k} = 1$. Стратегия \eqref{eq:stoch_approx_general_update_rule} может быть полезной в случае наличия практических трудностей с обращением матриц большого размера или в случае использования неточной оптимизации в качестве аналога механизма регуляризации при настройке робастных моделей с повышенной устойчивостью к переобучению в теории статистического обучения. Вычисление $x_{k + 1}$ также может обладать стохастической природой относительно текущей итерации, однако в условиях данного раздела $\varepsilon_{k}$--оптимальность из \eqref{eq:stoch_approx_general_update_rule} полагается выполненной в смысле почти наверно.

Для вывода схемы, обладающей сходимостью в терминах среднего, рассмотрим переменный интервал поиска приближения локальной постоянной Липшица:
\begin{equation*}
    \begin{aligned}
        L_{k}\in&\left[\max\left\{L,~\frac{L}{s_{k}}\right\},~\max\left\{\tilde{\gamma}L_{\hat{F}},~\frac{\gamma L_{\hat{F}}}{s_{k}}\right\}\right],~\gamma\geq\tilde{\gamma}\geq1,~s_{k} > 0,~L\in\left(0,~\tilde{\gamma}L_{\hat{F}}\right],~k\in\mathbb{Z}_{+},
    \end{aligned}
\end{equation*}
причём на каждом шаге метода Гаусса--Ньютона количество итераций поиска наименьшего подходящего $L_{k}$ сверху ограничено величиной
$$\operatorname{O}\left(\left\lceil\log_{2}\left(\frac{\gamma L_{\hat{F}}}{L}\right)\right\rceil + 1\right)$$
по лемме \ref{lm:aux_L_search_time_complexity}, данная оценка верна и для схемы \ref{alg:gen_stoch_gnm}. Этой же величиной ограничено количество итераций линейного поиска $L_{k}$ в схемах \ref{alg:gen_det_gnm} и \ref{alg:gen_det_flex_gnm} при $\gamma = 2$. В силу единообразия рассматриваемой стратегии бинарного поиска локальной постоянной Липшица для схемы \ref{alg:gen_double_stoch_gnm} количество итераций линейного поиска ограничено сверху величиной
$$\operatorname{O}\left(\left\lceil\log_{2}\left(\frac{\gamma l_{\hat{g}_{2}}}{l}\right)\right\rceil + 1\right).$$

Соответственно, для введённой стратегии подбора $L_{k}$ схема метода является модификацией схемы \ref{alg:gen_stoch_gnm} и представлена в листинге \ref{alg:gen_stoch_unbounded_gnm}.
\RestyleAlgo{boxruled}
\begin{algorithm}[!ht]{}
\caption{\textbf{Общий метод трёх стохастических квадратов с неточным проксимальным отображением и переменным отрезком линейного поиска}}
\label{alg:gen_stoch_unbounded_gnm}
\textbf{Вход:}
    \begin{equation*}
        \begin{cases}
            x_{0}\in E_{1}\text{ --- начальное приближение},~x_{-1} = x_{0};\\
            \mathcal{E}(\cdot)\text{ --- функция погрешности проксимального отображения};\\
            N\in\mathbb{N}\text{ --- количество итераций метода};\\
            \gamma\geq\tilde{\gamma}\geq1\text{ --- факторы верхней границы поиска }L_{\hat{F}};\\
            L\text{ --- оценка локальной постоянной Липшица},~L\in(0,~\tilde{\gamma}L_{\hat{F}}];\\
            \mathcal{T}(\cdot)\text{ --- функция, определяющая значение $\tau$};\\
            \mathcal{B}\text{ --- выборка функций};\\
            b\in\overline{1, m}\text{ --- размер $B_{k}\subseteq\mathcal{B},~k\in\mathbb{Z}_{+}$}.
        \end{cases}
    \end{equation*}
    \vspace{0.2cm}
    \textbf{Повторять для $k = 0, 1,~\dots, N - 1$:}
    \begin{itemize}
        \item[1.] сэмплировать батч $B_{k}$ из $\mathcal{B}$ размера $b$;
        \item[2.] определить $\tau_{k} = \mathcal{T}(x_{k}, B_{k})$, $\varepsilon_{k} = \mathcal{E}(k, x_{k}, x_{k - 1}, B_{k})$, $L_{0} = \max\left\{L,~\frac{L}{\tau_{0}}\right\}$;
        \item[3.] вычислить $x_{k + 1}\in E_{1}$ согласно одному из выбранных изначально правил: \eqref{eq:stoch_direct_update_rule} или \eqref{eq:stoch_approx_general_update_rule};
        \item[4.] если $\hat{g}_{1}(x_{k + 1}, B_{k}) > \hat{\psi}_{x_{k}, L_{k}, \tau_{k}}(x_{k + 1}, B_{k})$, то положить $L_{k} := \min\left\{2L_{k},~\max\left\{\tilde{\gamma}L_{\hat{F}},~\frac{\gamma L_{\hat{F}}}{\tau_{k}}\right\}\right\}$ и вернуться к пункту 3;
        \item[5.] $L_{k + 1} = \max\left\{\frac{L_{k}}{2},~\max\left\{L,~\frac{L}{\tau_{k}}\right\}\right\}$.
    \end{itemize}
    \vspace{0.2cm}

    \textbf{Выход:} $x_{N}$.
\end{algorithm}
Следующие утверждения устанавливают сохранение сходимости при использовании стратегии \eqref{eq:stoch_approx_general_update_rule} с такой же асимптотикой, что и при использовании стратегии \eqref{eq:stoch_direct_update_rule}, если на каждой итерации метода трёх стохастических квадратов осуществлять поиск $x_{k + 1}$ с достаточно малой погрешностью $\varepsilon_{k} > 0$.

\begin{theorem}\label{th:5_main}
    Пусть выполнены предположения \ref{as:1}, \ref{as:2}, \ref{as:3}, \ref{as:4}. Рассмотрим метод Гаусса--Ньютона со схемой реализации \ref{alg:gen_stoch_unbounded_gnm}, в котором последовательность $\{x_{k}\}_{k\in\mathbb{Z}_{+}}$ вычисляется по правилу \eqref{eq:stoch_approx_general_update_rule} с $\tau_{k} = \hat{g}_{1}(x_{k}, B_{k})$. Если в схеме \ref{alg:gen_stoch_unbounded_gnm} выбрать следующий отрезок погрешностей $\varepsilon_{k}$:
    \begin{equation}\label{eq:5_main_regime_1}
        \begin{aligned}
            &0\leq\varepsilon_{k}\leq\frac{\varepsilon}{\hat{g}_{1}(x_{k}, B_{k})},
        \end{aligned}
    \end{equation}
    то выполнено:
    \begin{equation}\label{eq:main_stoch_approx_sub_lin_conv}
        \begin{aligned}
            \mathbb{E}\left[\min\limits_{i\in\overline{0, k - 1}}\left\{\left\|\nabla\hat{f}_{2}(x_{i})\right\|^{2}\right\}\right]&\leq8\left(M_{\hat{G}}^{2} + \max\left\{\tilde{\gamma}P_{\hat{g}_{1}}L_{\hat{F}},~\gamma L_{\hat{F}}\right\}\right)\left(\frac{\mathbb{E}\left[\hat{f}_{2}(x_{0})\right]}{k} + \varepsilon +\right.\\
            &\left.+ 2l_{\hat{F}}\left(\sqrt{\frac{2\varepsilon}{L}} + \sqrt{\frac{2P_{\hat{g}_{1}}}{L}}\right)\mathds{1}_{\left\{b < m\right\}} + \tilde{\sigma}\sqrt{\frac{1}{b} - \frac{1}{m}}\right),~k\in\mathbb{N}.
        \end{aligned}
    \end{equation}
    Если в схеме \ref{alg:gen_stoch_unbounded_gnm} ограничения на $\varepsilon_{k}$ выбрать такие:
    \begin{equation}\label{eq:5_main_regime_2}
        \begin{aligned}
            &0\leq\varepsilon_{k}\leq\frac{\delta\left\|\nabla_{x_{k}}\hat{g}_{2}(x_{k}, B_{k})\right\|^{2}}{8\hat{g}_{1}(x_{k}, B_{k})\left(M_{\hat{G}}^{2} + \hat{g}_{1}(x_{k}, B_{k})L_{k}\right)},~\delta\in[0, 1),
        \end{aligned}
    \end{equation}
    то оценка сходимости будет следующей:
    \begin{equation}\label{eq:main_stoch_approx_sub_lin_conv_1}
        \begin{aligned}
            \mathbb{E}\left[\min\limits_{i\in\overline{0, k - 1}}\left\|\nabla\hat{f}_{2}(x_{i})\right\|^{2}\right]&\leq\frac{8\left(M_{\hat{G}}^{2} + \max\left\{\tilde{\gamma}P_{\hat{g}_{1}}L_{\hat{F}},~\gamma L_{\hat{F}}\right\}\right)}{(1 - \delta)}\left(\frac{\mathbb{E}\left[\hat{f}_{2}(x_{0})\right]}{k} +\right.\\
            &\left.+ 2l_{\hat{F}}\left(\sqrt{\frac{\delta P_{\hat{g}_{1}}}{L}} + \sqrt{\frac{2P_{\hat{g}_{1}}}{L}}\right)\mathds{1}_{\left\{b < m\right\}} + \tilde{\sigma}\sqrt{\frac{1}{b} - \frac{1}{m}}\right),~k\in\mathbb{N}.
        \end{aligned}
    \end{equation}
    Оператор математического ожидания $\mathbb{E}\left[\cdot\right]$ усредняет по всей случайности процесса оптимизации.
\end{theorem}

В теореме \ref{th:5} представлены два режима неточного поиска $x_{k + 1}$. В рамках первого режима при уменьшении значения сэмпла оптимизируемой функции точность поиска $x_{k + 1}$ может уменьшиться в силу увеличения верхней границы допустимого значения $\varepsilon_{k}$ ($0\leq\varepsilon_{k}\leq\frac{\varepsilon}{\hat{g}_{1}(x_{k}, B_{k})}$), а условия оптимальности для оценки \eqref{eq:main_stoch_approx_sub_lin_conv} в терминах средней минимальной нормы градиента на уровне $\epsilon$ выглядят следующим образом:
\begin{equation*}
    \begin{cases}
        8\left(M_{\hat{G}}^{2} + \max\left\{\tilde{\gamma}P_{\hat{g}_{1}}L_{\hat{F}},~\gamma L_{\hat{F}}\right\}\right)\frac{\mathbb{E}\left[\hat{g}_{2}(x_{0}, B_{0})\right]}{k}\leq r_{1}\epsilon^{2},~r_{1}\in(0, 1);\\[10pt]
        8\left(M_{\hat{G}}^{2} + \max\left\{\tilde{\gamma}P_{\hat{g}_{1}}L_{\hat{F}},~\gamma L_{\hat{F}}\right\}\right)\left(\varepsilon + 2l_{\hat{F}}\sqrt{\frac{2\varepsilon}{L}}\mathds{1}_{\{b < m\}}\right)\leq r_{2}\epsilon^{2},~r_{2}\in(0, 1);\\[10pt]
        8\left(M_{\hat{G}}^{2} + \max\left\{\tilde{\gamma}P_{\hat{g}_{1}}L_{\hat{F}},~\gamma L_{\hat{F}}\right\}\right)\left(2l_{\hat{F}}\sqrt{\frac{2m(m - 1)P_{\hat{g}_{1}}}{L}}\mathds{1}_{\left\{b < m\right\}} + \tilde{\sigma}\right)\sqrt{\frac{1}{b} - \frac{1}{m}}\leq r_{3}\epsilon^{2},~r_{3}\in(0, 1);\\[5pt]
        r_{1} + r_{2} + r_{3} = 1.
    \end{cases}
\end{equation*}
Из системы неравенств выводятся минимальное количество итераций, минимальный размер батча и максимальная погрешность поиска $x_{k + 1}$:
\begin{equation*}
    \begin{cases}
        k = \left\lceil8\left(M_{\hat{G}}^{2} + \max\left\{\tilde{\gamma}P_{\hat{g}_{1}}L_{\hat{F}},~\gamma L_{\hat{F}}\right\}\right)\frac{\mathbb{E}\left[\hat{g}_{2}(x_{0}, B_{0})\right]}{r_{1}\epsilon^{2}}\right\rceil,~r_{1}\in(0, 1);\\[10pt]
        \varepsilon = \left(-l_{\hat{F}}\sqrt{\frac{2}{L}}\mathds{1}_{\{b < m\}} + \frac{1}{2}\sqrt{\left(\frac{8l_{\hat{F}}^{2}}{L}\right)\mathds{1}_{\{b < m\}} + \frac{r_{2}\epsilon^{2}}{2\left(M_{\hat{G}}^{2} + \max\left\{\tilde{\gamma}P_{\hat{g}_{1}}L_{\hat{F}},~\gamma L_{\hat{F}}\right\}\right)}}\right)^{2},~r_{2}\in(0, 1);\\[10pt]
        b = \min\left\{m,~\left\lceil\frac{64\left(M_{\hat{G}}^{2} + \max\left\{\tilde{\gamma}P_{\hat{g}_{1}}L_{\hat{F}},~\gamma L_{\hat{F}}\right\}\right)^{2}\left(2l_{\hat{F}}\sqrt{\frac{2m(m - 1)P_{\hat{g}_{1}}}{L}} + \tilde{\sigma}\right)^{2}}{r_{3}^{2}\epsilon^{4} + \frac{64}{m}\left(M_{\hat{G}}^{2} + \max\left\{\tilde{\gamma}P_{\hat{g}_{1}}L_{\hat{F}},~\gamma L_{\hat{F}}\right\}\right)^{2}\left(2l_{\hat{F}}\sqrt{\frac{2m(m - 1)P_{\hat{g}_{1}}}{L}} + \tilde{\sigma}\right)^{2}}\right\rceil\right\},~r_{3}\in(0, 1);\\[5pt]
        r_{1} + r_{2} + r_{3} = 1.
    \end{cases}
\end{equation*}
Те же самые величины в асимптотической форме демонстрируют необходимость на каждой итерации решать задачу поиска $x_{k + 1}$ асимптотически точнее, чем требуемый уровень нормы градиента функции $\hat{f}_{2}$:
$$k = \operatorname{O}\left(\frac{1}{\epsilon^{2}}\right),~\varepsilon = \operatorname{O}\left(\epsilon^{2}\right),~b = \min\left\{m,~\operatorname{O}\left(\frac{1}{\epsilon^{4}}\right)\right\}.$$
Во втором режиме неточного поиска $x_{k + 1}$ точность вычисления на каждой итерации может быть потенциально сколь угодно малой в силу пропорциональности верхней грани квадрата нормы градиента значению функции:
\begin{equation*}
    \begin{aligned}
        &0\leq\varepsilon_{k}\leq\frac{\delta\left\|\nabla_{x_{k}}\hat{g}_{2}(x_{k}, B_{k})\right\|^{2}}{8\hat{g}_{1}(x_{k}, B_{k})\left(M_{\hat{G}}^{2} + \hat{g}_{1}(x_{k}, B_{k})L_{k}\right)}\leq\\
        &\leq\left\{\left\|\nabla_{x_{k}}\hat{g}_{2}(x_{k}, B_{k})\right\|^{2}\leq4\underbrace{\left\|\hat{G}^{'}(x_{k}, B_{k})\right\|^{2}}_{\leq M_{\hat{G}}^{2}}\underbrace{\left\|\hat{G}(x_{k}, B_{k})\right\|^{2}}_{=\hat{g}_{2}(x_{k}, B_{k})}\right\}\leq\frac{\delta M_{\hat{G}}^{2}\hat{g}_{1}(x_{k}, B_{k})}{2\left(M_{\hat{G}}^{2} + \hat{g}_{1}(x_{k}, B_{k})L_{k}\right)}.
    \end{aligned}
\end{equation*}
Соответственно, в этом режиме нет необходимости отдельно вычислять максимальное значение $\varepsilon_{k}$, однако вместо $\varepsilon_{k}$ требуется определить $\delta$:
\begin{equation*}
    \begin{cases}
        \frac{8\left(M_{\hat{G}}^{2} + \max\left\{\tilde{\gamma}P_{\hat{g}_{1}}L_{\hat{F}},~\gamma L_{\hat{F}}\right\}\right)\mathbb{E}\left[\hat{g}_{2}(x_{0}, B_{0})\right]}{(1 - \delta)k}\leq r_{1}\epsilon^{2},~r_{1}\in(0, 1);\\[5pt]
        \frac{16l_{\hat{F}}\left(M_{\hat{G}}^{2} + \max\left\{\tilde{\gamma}P_{\hat{g}_{1}}L_{\hat{F}},~\gamma L_{\hat{F}}\right\}\right)}{1 - \delta}\sqrt{\frac{\delta P_{\hat{g}_{1}}}{L}}\mathds{1}_{\{b < m\}}\leq r_{2}\epsilon^{2},~r_{2}\in(0, 1);\\
        \frac{8\left(M_{\hat{G}}^{2} + \max\left\{\tilde{\gamma}P_{\hat{g}_{1}}L_{\hat{F}},~\gamma L_{\hat{F}}\right\}\right)}{1 - \delta}\left(2l_{\hat{F}}\sqrt{\frac{2m(m - 1)P_{\hat{g}_{1}}}{L}}\mathds{1}_{\left\{b < m\right\}} + \tilde{\sigma}\right)\sqrt{\frac{1}{b} - \frac{1}{m}}\leq r_{3}\epsilon^{2},~r_{3}\in(0, 1);\\
        r_{1} + r_{2} + r_{3} = 1.
    \end{cases}
\end{equation*}
Получаются следующие минимальные значения количества итераций и размера батча:
\begin{equation*}
    \begin{cases}
        k = \left\lceil\frac{8\left(M_{\hat{G}}^{2} + \max\left\{\tilde{\gamma}P_{\hat{g}_{1}}L_{\hat{F}},~\gamma L_{\hat{F}}\right\}\right)\mathbb{E}\left[\hat{g}_{2}(x_{0}, B_{0})\right]}{(1 - \delta)r_{1}\epsilon^{2}}\right\rceil,~r_{1}\in(0, 1);\\[10pt]
        \scalemath{0.88}{
        \begin{aligned}
            \delta &= \min\left\{\alpha,~\left(\frac{8l_{\hat{F}}\left(M_{\hat{G}}^{2} + \max\left\{\tilde{\gamma}P_{\hat{g}_{1}}L_{\hat{F}},~\gamma L_{\hat{F}}\right\}\right)}{r_{2}\epsilon^{2}}\sqrt{\frac{P_{\hat{g}_{1}}}{L}} - \sqrt{\underbrace{\left(\frac{8l_{\hat{F}}\left(M_{\hat{G}}^{2} + \max\left\{\tilde{\gamma}P_{\hat{g}_{1}}L_{\hat{F}},~\gamma L_{\hat{F}}\right\}\right)}{r_{2}\epsilon^{2}}\right)^{2}\frac{P_{\hat{g}_{1}}}{L}}_{\text{с достаточно малым $r_{2}\epsilon^{2}$, чтобы дробь была не меньше }1} - 1}\right)^{2}\right\},\\
            &\alpha\in[0, 1),~r_{2}\in(0, 1);\\
            \vspace{1pt}
        \end{aligned}
        }\\
        b = \min\left\{m,~\left\lceil\frac{\frac{64\left(M_{\hat{G}}^{2} + \max\left\{\tilde{\gamma}P_{\hat{g}_{1}}L_{\hat{F}},~\gamma L_{\hat{F}}\right\}\right)^{2}}{(1 - \delta)^{2}}\left(2l_{\hat{F}}\sqrt{\frac{2m(m - 1)P_{\hat{g}_{1}}}{L}} + \tilde{\sigma}\right)^{2}}{r_{3}^{2}\epsilon^{4} + \frac{64\left(M_{\hat{G}}^{2} + \max\left\{\tilde{\gamma}P_{\hat{g}_{1}}L_{\hat{F}},~\gamma L_{\hat{F}}\right\}\right)^{2}}{m(1 - \delta)^{2}}\left(2l_{\hat{F}}\sqrt{\frac{2m(m - 1)P_{\hat{g}_{1}}}{L}} + \tilde{\sigma}\right)^{2}}\right\rceil\right\},~r_{3}\in(0, 1);\\[10pt]
        r_{1} + r_{2} + r_{3} = 1.
    \end{cases}
\end{equation*}
или в сокращённой асимптотической форме:
$$k = \operatorname{O}\left(\frac{1}{\epsilon^{2}}\right),~b = \min\left\{m,~\operatorname{O}\left(\frac{1}{\epsilon^{4}}\right)\right\}.$$
Если же дополнительно выполнено предположение \ref{as:5}, то необходимое количество итераций значительно уменьшается, как указано в теореме \ref{th:6}.
\begin{theorem}\label{th:6_main}
    Пусть выполнены предположения \ref{as:1}, \ref{as:2}, \ref{as:3}, \ref{as:4}, \ref{as:5}. Рассмотрим метод Гаусса--Ньютона со схемой реализации \ref{alg:gen_stoch_unbounded_gnm}, в котором последовательность $\{x_{k}\}_{k\in\mathbb{Z}_{+}}$ вычисляется по правилу \eqref{eq:stoch_approx_general_update_rule} с\\$\tau_{k} = \hat{g}_{1}(x_{k}, B_{k})$. Тогда:
    \begin{equation*}
        \begin{cases}
            \begin{aligned}
                &\mathbb{E}\left[\left\|\nabla\hat{f}_{2}(x_{k})\right\|^{2}\right]\leq4M_{\hat{G}}^{2}\tilde{\Delta}_{k,b};\\[5pt]
                &\mathbb{E}\left[\hat{f}_{2}(x_{k})\right]\leq\hat{f}_{2}^{*} + \tilde{\Delta}_{k,b};
            \end{aligned}
        \end{cases}
    \end{equation*}
    где 
    \begin{equation}\label{eq:main_stoch_approx_lin_conv}
        \begin{aligned}
            \tilde{\Delta}_{k,b} &= \mathbb{E}\left[\hat{f}_{2}(x_{0})\right]\exp\left(-\frac{k\mu}{2\left(\max\left\{\tilde{\gamma}L_{\hat{F}}P_{\hat{g}_{1}},~\gamma L_{\hat{F}}\right\} + \mu\right)}\right) + 2\left(\varepsilon + 2\left(l_{\hat{F}}\left(\sqrt{\frac{2\varepsilon}{L}} + \sqrt{\frac{2P_{\hat{g}_{1}}}{L}}\right)\mathds{1}_{\left\{b < m\right\}} +\right.\right.\\
            &\left.\left.+ \tilde{\sigma}\sqrt{\frac{1}{b} - \frac{1}{m}}\right)\right)\left(\frac{\max\left\{\tilde{\gamma}L_{\hat{F}}P_{\hat{g}_{1}},~\gamma L_{\hat{F}}\right\}}{\mu} + 1\right),~k\in\mathbb{Z}_{+},~b\in\overline{1,~\min\{m, n\}}
        \end{aligned}
    \end{equation}
    в случае 
    \begin{equation}\label{eq:6_main_regime_1}
        \begin{aligned}
            &0\leq\varepsilon_{k}\leq\frac{\varepsilon}{\hat{g}_{1}(x_{k}, B_{k})},
        \end{aligned}
    \end{equation}
    и
    \begin{equation}\label{eq:main_stoch_approx_lin_conv_1}
        \begin{aligned}
            \tilde{\Delta}_{k,b} &= \mathbb{E}\left[\hat{f}_{2}(x_{0})\right]\exp\left(-\frac{k(1 - \delta)\mu}{2\left(\max\left\{\tilde{\gamma}L_{\hat{F}}P_{\hat{g}_{1}},~\gamma L_{\hat{F}}\right\} + \mu\right)}\right) + 4\left(l_{\hat{F}}\left(\sqrt{\frac{\delta P_{\hat{g}_{1}}}{L}} + \sqrt{\frac{2P_{\hat{g}_{1}}}{L}}\right)\mathds{1}_{\left\{b < m\right\}} +\right.\\
            &\left.+ \tilde{\sigma}\sqrt{\frac{1}{b} - \frac{1}{m}}\right)\left(\frac{\max\left\{\tilde{\gamma}L_{\hat{F}}P_{\hat{g}_{1}},~\gamma L_{\hat{F}}\right\} + \mu}{(1 - \delta)\mu}\right),~k\in\mathbb{Z}_{+},~b\in\overline{1,~\min\{m, n\}}
        \end{aligned}
    \end{equation}
    в случае
    \begin{equation}\label{eq:6_main_regime_2}
        \begin{aligned}
            &0\leq\varepsilon_{k}\leq\frac{\delta\hat{g}_{1}(x_{k}, B_{k})\mu}{2\left(L_{k}\hat{g}_{1}(x_{k}, B_{k}) + \mu\right)},~\delta\in[0, 1).
        \end{aligned}
    \end{equation}
    Оператор математического ожидания $\mathbb{E}\left[\cdot\right]$ усредняет по всей случайности процесса оптимизации.
\end{theorem}
Оценка сходимости \eqref{eq:main_stoch_approx_lin_conv} из теоремы выше для первого режима вычисления $x_{k + 1}$ порождает следующие условия сходимости к уровню $\epsilon$ средней нормы градиента функции $\hat{f}_{2}$:
\begin{equation*}
    \begin{cases}
        4M_{\hat{G}}^{2}\mathbb{E}\left[\hat{g}_{2}(x_{0}, B_{0})\right]\exp\left(-\frac{k\mu}{2\left(\max\left\{\tilde{\gamma}L_{\hat{F}}P_{\hat{g}_{1}},~\gamma L_{\hat{F}}\right\} + \mu\right)}\right)\leq r_{1}\epsilon^{2},~r_{1}\in(0, 1);\\[10pt]
        8M_{\hat{G}}^{2}\left(\frac{\max\left\{\tilde{\gamma}L_{\hat{F}}P_{\hat{g}_{1}},~\gamma L_{\hat{F}}\right\} + \mu}{\mu}\right)\left(\varepsilon + 2l_{\hat{F}}\sqrt{\frac{2\varepsilon}{L}}\mathds{1}_{\{b < m\}}\right)\leq r_{2}\epsilon^{2},~r_{2}\in(0, 1);\\[10pt]
        16M_{\hat{G}}^{2}\left(l_{\hat{F}}\sqrt{\frac{2m(m - 1)P_{\hat{g}_{1}}}{L}}\mathds{1}_{\left\{b < m\right\}} + \tilde{\sigma}\right)\left(\frac{\max\left\{\tilde{\gamma}L_{\hat{F}}P_{\hat{g}_{1}},~\gamma L_{\hat{F}}\right\}}{\mu} + 1\right)\sqrt{\frac{1}{b} - \frac{1}{m}}\leq r_{3}\epsilon^{2},~r_{3}\in(0, 1);\\[5pt]
        r_{1} + r_{2} + r_{3} = 1.
    \end{cases}
\end{equation*}
Условия сходимости задают минимальное количество итераций, максимальную погрешность вычисления и минимальный размер батча в точной форме:
\begin{equation*}
    \begin{cases}
        k = \left\lceil\frac{2\left(\max\left\{\tilde{\gamma}L_{\hat{F}}P_{\hat{g}_{1}},~\gamma L_{\hat{F}}\right\} + \mu\right)}{\mu}\ln\left(\frac{4M_{\hat{G}}^{2}\mathbb{E}\left[\hat{g}_{2}(x_{0}, B_{0})\right]}{r_{1}\epsilon^{2}}\right)\right\rceil,~r_{1}\in(0, 1);\\[10pt]
        \varepsilon = \left(-l_{\hat{F}}\sqrt{\frac{2}{L}}\mathds{1}_{\{b < m\}} + \frac{1}{2}\sqrt{\left(\frac{8l_{\hat{F}}^{2}}{L}\right)\mathds{1}_{\{b < m\}} + \frac{\mu r_{2}\epsilon^{2}}{2M_{\hat{G}}^{2}\left(\max\left\{\tilde{\gamma}P_{\hat{g}_{1}}L_{\hat{F}},~\gamma L_{\hat{F}}\right\} + \mu\right)}}\right)^{2},~r_{2}\in(0, 1);\\[10pt]
        b = \min\left\{m,~n,~\left\lceil\frac{256M_{\hat{G}}^{4}\left(l_{\hat{F}}\sqrt{\frac{2m(m - 1)P_{\hat{g}_{1}}}{L}} + \tilde{\sigma}\right)^{2}\left(\frac{\max\left\{\tilde{\gamma}L_{\hat{F}}P_{\hat{g}_{1}},~\gamma L_{\hat{F}}\right\}}{\mu} + 1\right)^{2}}{r_{3}^{2}\epsilon^{4} + \frac{256M_{\hat{G}}^{4}}{m}\left(l_{\hat{F}}\sqrt{\frac{2m(m - 1)P_{\hat{g}_{1}}}{L}} + \tilde{\sigma}\right)^{2}\left(\frac{\max\left\{\tilde{\gamma}L_{\hat{F}}P_{\hat{g}_{1}},~\gamma L_{\hat{F}}\right\}}{\mu} + 1\right)^{2}}\right\rceil\right\},~r_{3}\in(0, 1);\\[5pt]
        r_{1} + r_{2} + r_{3} = 1;
    \end{cases}
\end{equation*}
и в асимптотической форме:
$$k = \operatorname{O}\left(\ln\left(\frac{1}{\epsilon}\right)\right),~\varepsilon = \operatorname{O}\left(\epsilon^{2}\right),~b = \min\left\{m,~n,~\operatorname{O}\left(\frac{1}{\epsilon^{4}}\right)\right\}.$$
Для второго режима условия задаются аналогичным образом:
\begin{equation*}
    \begin{cases}
        4M_{\hat{G}}^{2}\mathbb{E}\left[\hat{g}_{2}(x_{0}, B_{0})\right]\exp\left(-\frac{k(1 - \delta)\mu}{2\left(\max\left\{\tilde{\gamma}L_{\hat{F}}P_{\hat{g}_{1}},~\gamma L_{\hat{F}}\right\} + \mu\right)}\right)\leq r_{1}\epsilon^{2},~r_{1}\in(0, 1);\\[10pt]
        16 M_{\hat{G}}^{2}\left(l_{\hat{F}}\sqrt{\frac{\delta P_{\hat{g}_{1}}}{L}}\mathds{1}_{\{b < m\}}\right)\left(\frac{\max\left\{\tilde{\gamma}L_{\hat{F}}P_{\hat{g}_{1}},~\gamma L_{\hat{F}}\right\} + \mu}{(1 - \delta)\mu}\right)\leq r_{2}\epsilon^{2},~r_{2}\in(0, 1);\\[10pt]
        16M_{\hat{G}}^{2}\left(l_{\hat{F}}\sqrt{\frac{2m(m - 1)P_{\hat{g}_{1}}}{L}}\mathds{1}_{\left\{b < m\right\}} + \tilde{\sigma}\right)\left(\frac{\max\left\{\tilde{\gamma}L_{\hat{F}}P_{\hat{g}_{1}},~\gamma L_{\hat{F}}\right\} + \mu}{(1 - \delta)\mu}\right)\sqrt{\frac{1}{b} - \frac{1}{m}}\leq r_{3}\epsilon^{2},~r_{3}\in(0, 1);\\
        r_{1} + r_{2} + r_{3} = 1.
    \end{cases}
\end{equation*}
порождая минимальные значения количества итераций и размера батча:
\begin{equation*}
    \begin{cases}
        k = \left\lceil\frac{2\left(\max\left\{\tilde{\gamma}L_{\hat{F}}P_{\hat{g}_{1}},~\gamma L_{\hat{F}}\right\} + \mu\right)}{(1 - \delta)\mu}\ln\left(\frac{4M_{\hat{G}}^{2}\mathbb{E}\left[\hat{g}_{2}(x_{0}, B_{0})\right]}{r_{1}\epsilon^{2}}\right)\right\rceil,~r_{1}\in(0, 1);\\[10pt]
        \scalemath{0.88}{
        \begin{aligned}
            \delta &= \min\left\{\alpha,~\left(\frac{8M_{\hat{G}}^{2}l_{\hat{F}}\left(\max\left\{\tilde{\gamma}L_{\hat{F}}P_{\hat{g}_{1}},~\gamma L_{\hat{F}}\right\} + \mu\right)}{\mu r_{2}\epsilon^{2}}\sqrt{\frac{P_{\hat{g}_{1}}}{L}} - \sqrt{\underbrace{\left(\frac{8M_{\hat{G}}^{2}l_{\hat{F}}\left(\max\left\{\tilde{\gamma}L_{\hat{F}}P_{\hat{g}_{1}},~\gamma L_{\hat{F}}\right\} + \mu\right)}{\mu r_{2}\epsilon^{2}}\right)^{2}\frac{P_{\hat{g}_{1}}}{L}}_{\text{с достаточно малым $r_{2}\epsilon^{2}$, чтобы дробь была не меньше }1} - 1}\right)^{2}\right\},\\
            &\alpha\in[0, 1),~r_{2}\in(0, 1);\\
            \vspace{1pt}
        \end{aligned}
        }\\
        b = \min\left\{m,~n,~\left\lceil\frac{256M_{\hat{G}}^{4}\left(l_{\hat{F}}\sqrt{\frac{2m(m - 1)P_{\hat{g}_{1}}}{L}} + \tilde{\sigma}\right)^{2}\left(\frac{\max\left\{\tilde{\gamma}L_{\hat{F}}P_{\hat{g}_{1}},~\gamma L_{\hat{F}}\right\} + \mu}{(1 - \delta)\mu}\right)^{2}}{r_{3}^{2}\epsilon^{4} + \frac{256M_{\hat{G}}^{4}}{m}\left(l_{\hat{F}}\sqrt{\frac{2m(m - 1)P_{\hat{g}_{1}}}{L}} + \tilde{\sigma}\right)^{2}\left(\frac{\max\left\{\tilde{\gamma}L_{\hat{F}}P_{\hat{g}_{1}},~\gamma L_{\hat{F}}\right\} + \mu}{(1 - \delta)\mu}\right)^{2}}\right\rceil\right\},~r_{3}\in(0, 1);\\[10pt]
        r_{1} + r_{2} + r_{3} = 1.
    \end{cases}
\end{equation*}
что асимптотически означает:
$$k = \operatorname{O}\left(\ln\left(\frac{1}{\epsilon}\right)\right),~b = \min\left\{m,~n,~\operatorname{O}\left(\frac{1}{\epsilon^{4}}\right)\right\}.$$

Оценки \eqref{eq:main_stoch_approx_sub_lin_conv_1}, \eqref{eq:main_stoch_approx_lin_conv_1} в теоремах \ref{th:5} и \ref{th:6} принципиально отличаются от аналогичных в теоремах \ref{th:3} и \ref{th:4} именно заменой $\eta(2 - \eta)$ на $(1 - \delta)$ и добавления слагаемого, зависящего от $\delta$, в оценку шума батча. Этот факт качественно описывает общность модели неточного вычисления $x_{k + 1}$ по правилу \eqref{eq:stoch_direct_update_rule} и модели неточного вычисления $x_{k + 1}$ \eqref{eq:stoch_approx_general_update_rule}, подтверждая тот факт, что специализированный способ \eqref{eq:stoch_direct_update_rule} не хуже общего \eqref{eq:stoch_approx_general_update_rule} решает поставленную задачу. Поэтому оценки на минимальные значения $b$ и $k$ для получения в среднем уровня $\epsilon$ у нормы градиента функции $\hat{f}_{2}$ вычисляются способами, похожими на использованные при вычислении \eqref{eq:sublin_batch_iter_tradeoff_2} и \eqref{eq:lin_batch_iter_tradeoff_2}.

В приведённых схемах \ref{alg:gen_stoch_gnm} и \ref{alg:gen_stoch_unbounded_gnm} вместо явного указания точности вычисления $x_{k + 1}$ можно задать количество необходимых внутренних итераций для случая, в котором стратегия \eqref{eq:stoch_approx_general_update_rule} имеет представление в виде итерационного метода минимизации $\hat{\psi}_{x_{k}, L_{k}, \hat{g}_{1}(x_{k}, B_{k})}(x, B_{k})$ по $x$, использующего одинаковое количество шагов на каждой внешней итерации $k$. Как было замечено в лемме \ref{lm:aux_bounded_variation} (следствие \ref{lm:aux_bounded_variation_term_cond}), вместо прогона метода через все $N$ внешних итераций можно добавить в качестве шестого шага в схемах \ref{alg:gen_stoch_gnm} и \ref{alg:gen_stoch_unbounded_gnm} условие раннего останова по значению нормы разности приближений оптимальных параметров с соседних итераций.

\subsection{Использование произвольной стратегии выбора шага}

В этом подразделе представлена стохастическая версия метода Гаусса--Ньютона для неограниченных функционалов с неограниченными якобианами. Вместо предположений \ref{as:1}, \ref{as:2}, \ref{as:3}, \ref{as:4}, \ref{as:5} рассматриваются новые три предположения, описывающие липшицевость якобианов, ограниченность дисперсии функции и новое условие класса Поляка--Лоясиевича.
\begin{assumption}[Липшиц--непрерывность отображения $F^{'}$]\label{as:jacob_smoothness}
    Существует конечная $L_{\hat{F}} > 0$, для которой
    \begin{equation*}
        \begin{aligned}
            \left\|\nabla F_{i}(x) - \nabla F_{i}(y)\right\|&\leq L_{\hat{F}}\left\|x - y\right\|,~\forall (x, y)\in E_{1}^{2},~\forall i\in\overline{1, m}.
        \end{aligned}
    \end{equation*}
\end{assumption}
В отличие от предположения \ref{as:1}, в данном случае рассмотрена липшицевость только производных отдельных функций.
\begin{assumption}[Ограниченность отклонения]\label{as:bounded_variance_growth}
    Для любых $\gamma\geq\tilde{\gamma}\geq1$ существуют $c_{i}\geq0$, $i\in\overline{1, 3}$, при которых выполнено
    \begin{equation*}
        \begin{aligned}
            \mathbb{E}\left[\left|\hat{g}_{2}(x_{k + 1}, B_{k}) - \hat{f}_{2}(x_{k + 1})\right|\right]&\leq\left.\sqrt{\frac{1}{b} - \frac{1}{m}}\right(c_{1} + c_{2}\mathbb{E}\left[\hat{g}_{1}(x_{k}, B_{k})\varepsilon_{k}\right] +\\
            &\left.+ c_{3}\mathbb{E}\left[\left\|\hat{T}_{\max\left\{\tilde{\gamma}L_{\hat{F}},~\frac{\gamma L_{\hat{F}}}{\hat{g}_{1}(x_{k}, B_{k})}\right\}, \hat{g}_{1}(x_{k}, B_{k})}(x_{k}, B_{k}) - x_{k}\right\|^{2}\right]\right)
        \end{aligned}
    \end{equation*}
    для всех $\varepsilon_{k}\geq0,~k\in\mathbb{Z}_{+}$ и $b\in\overline{1,~m}$ в рамках схемы \ref{alg:gen_stoch_unbounded_gnm} со стратегией \eqref{eq:stoch_approx_general_update_rule} и схемы \ref{alg:gen_stoch_flex_gnm}. Оператор математического ожидания $\mathbb{E}\left[\cdot\right]$ усредняет по всей случайности процесса оптимизации.
\end{assumption}
Предположение \ref{as:bounded_variance_growth} обобщает сразу несколько предположений, ограничивающих рост значения
$$\mathbb{E}\left[\left|\hat{g}_{2}(x_{k + 1}, B_{k}) - \hat{f}_{2}(x_{k + 1})\right|\right],$$
в частности, в класс оценок, который описывается данным предположением, входит результат леммы \ref{lm:aux_bounded_deviation} для ограниченных функционалов \eqref{eq:deviation_upper_bound} при $x_{k + 1} = \hat{T}_{\max\left\{\tilde{\gamma}L_{\hat{F}},~\frac{\gamma L_{\hat{F}}}{\hat{g}_{1}(x_{k}, B_{k})}\right\}, \hat{g}_{1}(x_{k}, B_{k})}(x_{k}, B_{k}):$
\begin{equation}\label{eq:bounded_g2_bounded_growth}
    \begin{aligned}
        \mathbb{E}\left[\left|\hat{f}_{2}(x_{k + 1}) - \hat{g}_{2}(x_{k + 1}, B_{k})\right|\right]&\leq\sqrt{\frac{1}{b} - \frac{1}{m}}\left(2l_{\hat{F}}\sqrt{m(m - 1)}\mathds{1}_{\left\{b < m\right\}}\mathbb{E}\left[\left\|x_{k + 1} - x_{k}\right\|\right] + \tilde{\sigma}\right)\leq\\
        &\leq\left\{\|v\|\leq\frac{\|v\|^{2}}{2a} + \frac{a}{2},~a > 0,~v\in E_{1}\Rightarrow\mathbb{E}\left[\|v\|\right]\leq\frac{1}{2a}\mathbb{E}\left[\|v\|^{2}\right] + \frac{a}{2}\right\}\leq\\
        &\leq\sqrt{\frac{1}{b} - \frac{1}{m}}\left(\underbrace{\frac{l_{\hat{F}}\sqrt{m(m - 1)}\mathds{1}_{\left\{b < m\right\}}}{a}}_{=c_{3}}\mathbb{E}\left[\left\|x_{k + 1} - x_{k}\right\|^{2}\right] +\right.\\
        &\left.+ \underbrace{\tilde{\sigma} + al_{\hat{F}}\sqrt{m(m - 1)}\mathds{1}_{\left\{b < m\right\}}}_{=c_{1}}\right),~c_{2} = 0,~a > 0,~k\in\mathbb{Z}_{+}.
    \end{aligned}
\end{equation}
Что характерно, в данном примере $c_{i},~i\in\overline{1,~3}$ не зависят от значений $\tilde{\gamma}$ и $\gamma$.
\begin{assumption}[Условие Поляка--Лоясиевича для проксимального градиента]\label{as:prox_PL_condition}
    Для любых конечных $\gamma\geq\tilde{\gamma}\geq1$ существует $\nu > 0$, при котором выполнено
    \begin{equation*}
        \begin{aligned}
            &\mathbb{E}\left[\left\|\gamma L_{\hat{F}}\left(\hat{T}_{\max\left\{\tilde{\gamma}L_{\hat{F}},~\frac{\gamma L_{\hat{F}}}{\hat{g}_{1}(x_{k}, B_{k})}\right\}, \hat{g}_{1}(x_{k}, B_{k})}(x_{k}, B_{k}) - x_{k}\right)\right\|^{2}\right]\geq\nu\left(\mathbb{E}\left[\hat{f}_{2}(x_{k})\right] - \hat{f}_{2}^{*}\right)
        \end{aligned}
    \end{equation*}
    для всех $k\in\mathbb{Z}_{+}$ в рамках схемы \ref{alg:gen_stoch_unbounded_gnm} со стратегией \eqref{eq:stoch_approx_general_update_rule} и схемы \ref{alg:gen_stoch_flex_gnm}. Оператор математического ожидания $\mathbb{E}\left[\cdot\right]$ усредняет по всей случайности процесса оптимизации.
\end{assumption}
В предположении \ref{as:prox_PL_condition} для ограниченных по предположениям \ref{as:2} и \ref{as:3} функционалов возможно выразить $\nu > 0$ явно, согласно лемме \ref{lm:aux_bounded_variation} при $\eta_{k} = 1$, $L_{k} = \max\left\{\tilde{\gamma}L_{\hat{F}},~\frac{\gamma L_{\hat{F}}}{\hat{g}_{1}(x_{k}, B_{k})}\right\}$ и\\$x_{k + 1} = \hat{T}_{\max\left\{\tilde{\gamma}L_{\hat{F}},~\frac{\gamma L_{\hat{F}}}{\hat{g}_{1}(x_{k}, B_{k})}\right\}, \hat{g}_{1}(x_{k}, B_{k})}(x_{k}, B_{k})$:
\begin{equation*}
    \begin{aligned}
        \mathbb{E}\left[\left\|\gamma L_{\hat{F}}(x_{k + 1} - x_{k})\right\|^{2}\right]&\geq\mathbb{E}\left[\left(\frac{\gamma L_{\hat{F}}\left\|\nabla_{x}\hat{g}_{2}(x_{k}, B_{k})\right\|}{2\left(M_{\hat{G}}^{2} + \hat{g}_{1}(x_{k}, B_{k})\max\left\{\tilde{\gamma}L_{\hat{F}},~\frac{\gamma L_{\hat{F}}}{\hat{g}_{1}(x_{k}, B_{k})}\right\}\right)}\right)^{2}\right]\geq\left\{\text{предположение \ref{as:5}}\right\}\geq\\
        &\geq\left(\frac{\gamma L_{\hat{F}}\sqrt{\mu}}{M_{\hat{G}}^{2} + \max\left\{\tilde{\gamma}L_{\hat{F}}P_{\hat{g}_{1}},~\gamma L_{\hat{F}}\right\}}\right)^{2}\mathbb{E}\left[\hat{f}_{2}(x_{k})\right]\geq\\
        &\geq\underbrace{\left(\frac{\gamma L_{\hat{F}}\sqrt{\mu}}{M_{\hat{G}}^{2} + \max\left\{\tilde{\gamma}L_{\hat{F}}P_{\hat{g}_{1}},~\gamma L_{\hat{F}}\right\}}\right)^{2}}_{=\nu\text{,~может быть }> 1}\left(\mathbb{E}\left[\hat{f}_{2}(x_{k})\right] - \hat{f}_{2}^{*}\right)\Rightarrow
    \end{aligned}
\end{equation*}
\begin{equation}\label{eq:nu_value}
    \begin{aligned}
        &\Rightarrow\lim\limits_{\gamma\rightarrow+\infty}\mathbb{E}\left[\left\|\gamma L_{\hat{F}}(x_{k + 1} - x_{k})\right\|^{2}\right] = \frac{1}{4}\mathbb{E}\left[\left\|\nabla_{x_{k}}\hat{g}_{2}(x_{k}, B_{k})\right\|^{2}\right]\geq\mu\left(\mathbb{E}\left[\hat{f}_{2}(x_{k})\right] - \hat{f}_{2}^{*}\right),
    \end{aligned}
\end{equation}
так как проксимальное отображение выглядит следующим образом:
\begin{equation*}
    \begin{aligned}
        \hat{T}_{\max\left\{\tilde{\gamma}L_{\hat{F}},~\frac{\gamma L_{\hat{F}}}{\hat{g}_{1}(x_{k}, B_{k})}\right\}, \hat{g}_{1}(x_{k}, B_{k})}(x_{k}, B_{k}) &= x_{k} - \left(\hat{G}^{'}(x_{k}, B_{k})^{*}\hat{G}^{'}(x_{k}, B_{k}) +\right.\\
        &\left.+ \max\left\{\tilde{\gamma}L_{\hat{F}}\hat{g}_{1}(x_{k}, B_{k}),~\gamma L_{\hat{F}}\right\}I_{n}\right)^{-1}\hat{G}^{'}(x_{k}, B_{k})^{*}\hat{G}(x_{k}, B_{k}).
    \end{aligned}
\end{equation*}
То есть выполнение предположения \ref{as:5} приводит к тому, что любая, в среднем, стационарная точка устанавливает уровень глобального минимума $\hat{f}_{2}^{*} = 0$, а это эквивалентно совместности системы $\eqref{eq:smooth_system}$.

В следующем утверждении установлена сходимость в терминах среднего к стационарной точке для последовательности $\left\{x_{k}\right\}_{k\in\mathbb{Z}_{+}}$, построенной по схеме \ref{alg:gen_stoch_unbounded_gnm} с правилом \eqref{eq:stoch_approx_general_update_rule}.

\begin{theorem}\label{th:gen_stoch_sublin_conv_main}
Допустим выполнение предположений \ref{as:jacob_smoothness} и \ref{as:bounded_variance_growth}. Рассмотрим метод Гаусса--Ньютона, реализованный по схеме \ref{alg:gen_stoch_unbounded_gnm} со стратегией обновления приближения решения \eqref{eq:stoch_approx_general_update_rule}, $\tau_{k} = \hat{g}_{1}(x_{k}, B_{k})$, $k\in\mathbb{Z}_{+}$. Тогда выполнена следующая оценка:
\begin{equation*}
    \begin{aligned}
        \left(\gamma L_{\hat{F}}\right)^{2}&\left(\frac{L}{2} - c_{3}\sqrt{\frac{1}{b} - \frac{1}{m}}\right)^{-1}\left(\frac{\mathbb{E}\left[\hat{f}_{2}(x_{0})\right]}{k} + \frac{1}{k}\sum\limits_{i = 0}^{k - 1}\mathbb{E}\left[\varepsilon_{i}\hat{g}_{1}(x_{i}, B_{i})\right]\left(1 + c_{2}\sqrt{\frac{1}{b} - \frac{1}{m}}\right) + c_{1}\sqrt{\frac{1}{b} - \frac{1}{m}}\right)\geq\\
        &\geq\mathbb{E}\left[\min\limits_{i\in\overline{0, k - 1}}\left\{\left\|\gamma L_{\hat{F}}\left(\hat{T}_{\max\left\{\tilde{\gamma}L_{\hat{F}},\frac{\gamma L_{\hat{F}}}{\hat{g}_{1}(x_{i}, B_{i})}\right\}, \hat{g}_{1}(x_{i}, B_{i})}(x_{i}, B_{i}) - x_{i}\right)\right\|^{2}\right\}\right],\\[10pt]
        b&\in\overline{\min\left\{m, \left\lceil\frac{4c_{3}^{2}m}{L^{2}m + 4c_{3}^{2}}\right\rceil + 1\right\},~m},~\tilde{\gamma}L_{\hat{F}}\geq L > 2c_{3}\sqrt{\frac{1}{b} - \frac{1}{m}},~\tilde{\gamma}>\max\left\{1,~\frac{2c_{3}}{L_{\hat{F}}}\sqrt{\frac{1}{b} - \frac{1}{m}}\right\},~\gamma\geq\tilde{\gamma}.
    \end{aligned}
\end{equation*}
\end{theorem}

Согласно теореме \ref{th:gen_stoch_sublin_conv} условие сходимости для достижения уровня $\epsilon > 0$ минимальной нормы проксимального градиента в терминах среднего при ограничении значения погрешности вычисления $x_{k + 1}$ неравенством $0\leq\varepsilon_{k}\leq\frac{\varepsilon}{\hat{g}_{1}(x_{k}, B_{k})}$ задаётся следующей системой неравенств:
\begin{equation*}
    \begin{cases}
        \left(\gamma L_{\hat{F}}\right)^{2}\left(\frac{L}{2} - c_{3}\sqrt{\frac{1}{b} - \frac{1}{m}}\right)^{-1}\frac{\mathbb{E}\left[\hat{g}_{2}(x_{0}, B_{0})\right]}{k}\leq r_{1}\epsilon^{2},~r_{1}\in(0, 1);\\[5pt]
        \left(\gamma L_{\hat{F}}\right)^{2}\left(\frac{L}{2} - c_{3}\sqrt{\frac{1}{b} - \frac{1}{m}}\right)^{-1}\left(1 + c_{2}\sqrt{\frac{1}{b} - \frac{1}{m}}\right)\varepsilon\leq r_{2}\epsilon^{2},~r_{2}\in(0, 1);\\[5pt]
        \left(\gamma L_{\hat{F}}\right)^{2}\left(\frac{L}{2} - c_{3}\sqrt{\frac{1}{b} - \frac{1}{m}}\right)^{-1}c_{1}\sqrt{\frac{1}{b} - \frac{1}{m}}\leq r_{3}\epsilon^{2},~r_{3}\in(0, 1);\\[5pt]
        r_{1} + r_{2} + r_{3} = 1;
    \end{cases}
\end{equation*}
из которой выводятся минимальное количество итераций, максимальное значение величины $\varepsilon$ и минимальный размер батча:
\begin{equation}\label{eq:gen_sublin_conv_cond}
    \begin{cases}
        k = \left\lceil\left(\frac{L}{2} - c_{3}\sqrt{\frac{1}{b} - \frac{1}{m}}\right)^{-1}\frac{\left(\gamma L_{\hat{F}}\right)^{2}\mathbb{E}\left[\hat{g}_{2}(x_{0}, B_{0})\right]}{r_{1}\epsilon^{2}}\right\rceil,~r_{1}\in(0, 1);\\[10pt]
        \varepsilon = \frac{r_{2}\epsilon^{2}}{\left(\gamma L_{\hat{F}}\right)^{2}}\left(\frac{L}{2} - c_{3}\sqrt{\frac{1}{b} - \frac{1}{m}}\right)\left(1 + c_{2}\sqrt{\frac{1}{b} - \frac{1}{m}}\right)^{-1},~r_{2}\in(0, 1);\\[10pt]
        b = \min\left\{m,~\left\lceil\frac{4c_{3}^{2}m\left(\frac{c_{1}}{c_{3}r_{3}}\left(\frac{\gamma L_{\hat{F}}}{\epsilon}\right)^{2} + 1\right)^{2}}{L^{2}m + 4c_{3}^{2}\left(\frac{c_{1}}{c_{3}r_{3}}\left(\frac{\gamma L_{\hat{F}}}{\epsilon}\right)^{2} + 1\right)^{2}}\right\rceil\right\},~r_{3}\in(0, 1);\\[5pt]
        r_{1} + r_{2} + r_{3} = 1.
    \end{cases}
\end{equation}
Неравенства из \eqref{eq:gen_sublin_conv_cond} задают ограничение на среднюю минимальную норму проксимального градиента:
\begin{equation}\label{eq:th_stoch_prox_norm}
    \begin{aligned}
        &\mathbb{E}\left[\min\limits_{i\in\overline{0, k - 1}}\left\{\left\|\gamma L_{\hat{F}}\left(\hat{T}_{\max\left\{\tilde{\gamma}L_{\hat{F}},\frac{\gamma L_{\hat{F}}}{\hat{g}_{1}(x_{i}, B_{i})}\right\}, \hat{g}_{1}(x_{i}, B_{i})}(x_{i}, B_{i}) - x_{i}\right)\right\|^{2}\right\}\right]\leq\epsilon^{2}.
    \end{aligned}
\end{equation}
Среди гиперпараметров оценки сходимости в теореме \ref{th:gen_stoch_sublin_conv} наиболее существенными на практике являются размер батча $b$ и нижняя оценка локальной постоянной Липшица $L$, так как их требуется явно задать при решении задачи \eqref{eq:main_opt_problem} по схеме \ref{alg:gen_stoch_unbounded_gnm}, в то время как значения гиперпараметров $\gamma$ и $\tilde{\gamma}$ можно более свободно варьировать, чтобы все неравенства в условии теоремы выполнялись, однако всё равно остаётся зависимость постоянных $c_{i} \geq 0,~i\in\overline{1,~3}$ и $\nu > 0$ от значений $\tilde{\gamma}$ и $\gamma$, и важность выбора значений $b$ и $L$ дополнительно обосновывается наличием этих зависимостей в предположениях \ref{as:bounded_variance_growth} и \ref{as:prox_PL_condition}. В самом условии \eqref{eq:gen_sublin_conv_cond} возможно увеличенем отношения $\frac{L}{\gamma L_{\hat{F}}}$ до $1$ добиться снижения нижней границы на размер батча, увеличения значения $\varepsilon$ и уменьшения $k$. Следующее утверждение уже демонстрирует, как выполнение предположения \ref{as:prox_PL_condition} вдобавок к условиям теоремы \ref{th:gen_stoch_sublin_conv} приводит к линейной сходимости.
\begin{theorem}\label{th:gen_stoch_lin_conv_main}
Предположим выполнение предположений \ref{as:jacob_smoothness}, \ref{as:bounded_variance_growth} и \ref{as:prox_PL_condition}. Рассмотрим метод Гаусса--Ньютона, реализованный по схеме \ref{alg:gen_stoch_unbounded_gnm} со стратегией обновления приближения решения \eqref{eq:stoch_approx_general_update_rule}, $\tau_{k} = \hat{g}_{1}(x_{k}, B_{k})$, $k\in\mathbb{Z}_{+}$. Тогда выполнена следующая оценка
\begin{equation*}
    \begin{aligned}
        \frac{\left(\gamma L_{\hat{F}}\right)^{2}}{\nu}&\left(\frac{L}{2} - c_{3}\sqrt{\frac{1}{b} - \frac{1}{m}}\right)^{-1}\left(c_{1}\sqrt{\frac{1}{b} - \frac{1}{m}} + \left(1 + c_{2}\sqrt{\frac{1}{b} - \frac{1}{m}}\right)\mathbb{E}\left[\max\limits_{i\in\overline{0, k - 1}}\left\{\varepsilon_{i}\hat{g}_{1}(x_{i}, B_{i})\right\}\right]\right) +\\
        &+ \exp\left(-\frac{k\nu}{\left(\gamma L_{\hat{F}}\right)^{2}}\left(\frac{L}{2} - c_{3}\sqrt{\frac{1}{b} - \frac{1}{m}}\right)\right)\mathbb{E}\left[\hat{f}_{2}(x_{0}) - \hat{f}_{2}^{*}\right]\geq\mathbb{E}\left[\hat{f}_{2}(x_{k}) - \hat{f}_{2}^{*}\right],\\[10pt]
        b&\in\overline{\min\left\{m,~\left\lceil\frac{4c_{3}^{2}m}{L^{2}m + 4c_{3}^{2}}\right\rceil + 1\right\},~m},~\tilde{\gamma}L_{\hat{F}}\geq L>2c_{3}\sqrt{\frac{1}{b} - \frac{1}{m}},\\
        \gamma&\geq\max\left\{\tilde{\gamma},~\frac{1}{L_{\hat{F}}}\sqrt{\nu\left(\frac{L}{2} - c_{3}\sqrt{\frac{1}{b} - \frac{1}{m}}\right)}\right\},~\tilde{\gamma} > \max\left\{1,~\frac{2c_{3}}{L_{\hat{F}}}\sqrt{\frac{1}{b} - \frac{1}{m}}\right\}.
    \end{aligned}
\end{equation*}
\end{theorem}

В случае $0\leq\varepsilon_{k}\leq\frac{\varepsilon}{\hat{g}_{1}(x_{k}, B_{k})}$ для теоремы \ref{th:gen_stoch_lin_conv} условие сходимости к уровню $\epsilon^{2} > 0$ функции $\hat{f}_{2}(x) - \hat{f}_{2}^{*}$ в терминах среднего выглядит следующим образом:
\begin{equation*}
    \begin{cases}
        \exp\left(-\frac{k\nu}{\left(\gamma L_{\hat{F}}\right)^{2}}\left(\frac{L}{2} - c_{3}\sqrt{\frac{1}{b} - \frac{1}{m}}\right)\right)\mathbb{E}\left[\hat{g}_{2}(x_{0}, B_{0})\right]\leq r_{1}\epsilon^{2},~r_{1}\in(0, 1);\\[10pt]
        \frac{\left(\gamma L_{\hat{F}}\right)^{2}}{\nu}\left(\frac{L}{2} - c_{3}\sqrt{\frac{1}{b} - \frac{1}{m}}\right)^{-1}\left(1 + c_{2}\sqrt{\frac{1}{b} - \frac{1}{m}}\right)\varepsilon\leq r_{2}\epsilon^{2},~r_{2}\in(0, 1);\\[10pt]
        \frac{\left(\gamma L_{\hat{F}}\right)^{2}}{\nu}\left(\frac{L}{2} - c_{3}\sqrt{\frac{1}{b} - \frac{1}{m}}\right)^{-1}c_{1}\sqrt{\frac{1}{b} - \frac{1}{m}}\leq r_{3}\epsilon^{2},~r_{3}\in(0, 1);\\[5pt]
        r_{1} + r_{2} + r_{3} = 1.
    \end{cases}
\end{equation*}
Неравенства выше определяют минимальное количество итераций, максимальное значение $\varepsilon$, минимальный размер батча:
\begin{equation}\label{eq:gen_lin_conv_cond}
    \begin{cases}
        k = \left\lceil\frac{\left(\gamma L_{\hat{F}}\right)^{2}}{\nu}\left(\frac{L}{2} - c_{3}\sqrt{\frac{1}{b} - \frac{1}{m}}\right)^{-1}\ln\left(\frac{\mathbb{E}\left[\hat{g}_{2}(x_{0}, B_{0})\right]}{r_{1}\epsilon^{2}}\right)\right\rceil,~r_{1}\in(0, 1);\\
        \varepsilon = \frac{\nu r_{2}\epsilon^{2}}{\left(\gamma L_{\hat{F}}\right)^{2}}\left(\frac{L}{2} - c_{3}\sqrt{\frac{1}{b} - \frac{1}{m}}\right)\left(1 + c_{2}\sqrt{\frac{1}{b} - \frac{1}{m}}\right)^{-1},~r_{2}\in(0, 1);\\
        b = \min\left\{m,~\left\lceil\frac{4c_{3}^{2}m\left(\frac{c_{1}}{c_{3}\nu r_{3}}\left(\frac{\gamma L_{\hat{F}}}{\epsilon}\right)^{2} + 1\right)^{2}}{L^{2}m + 4c_{3}^{2}\left(\frac{c_{1}}{c_{3}\nu r_{3}}\left(\frac{\gamma L_{\hat{F}}}{\epsilon}\right)^{2} + 1\right)^{2}}\right\rceil\right\},~r_{3}\in(0, 1);\\
        r_{1} + r_{2} + r_{3} = 1.
    \end{cases}
\end{equation}
Система неравенств \eqref{eq:gen_lin_conv_cond} ограничивает среднее значение $\hat{f}_{2}(x_{k})$:
\begin{equation}\label{eq:th_stoch_f2_val}
    \begin{aligned}
        &\mathbb{E}\left[\hat{f}_{2}(x_{k}) - \hat{f}_{2}^{*}\right]\leq\epsilon^{2}.
    \end{aligned}
\end{equation}
Как и в условии \eqref{eq:gen_sublin_conv_cond} увеличение соотношения $\frac{L}{\gamma L_{\hat{F}}}\in\left(0,~1\right]$ в \eqref{eq:gen_lin_conv_cond} приводит к ускорению сходимости, ослаблению требований на точность вычисления $x_{k + 1}$ и расширению отрезка значений размера батча. В полученной оценке особое место занимает случай $b = m$, в котором оценка стохастического метода становится следующей оценкой детерминированного метода:
\begin{equation*}
    \begin{aligned}
        &\frac{2\left(\gamma L_{\hat{F}}\right)^{2}}{\nu L}\left(\max\limits_{i\in\overline{0, k - 1}}\left\{\varepsilon_{i}\hat{f}_{1}(x_{i})\right\}\right) + \exp\left(-\frac{k\nu L}{2\left(\gamma L_{\hat{F}}\right)^{2}}\right)\left(\hat{f}_{2}(x_{0}) - \hat{f}_{2}^{*}\right) + \hat{f}_{2}^{*}\geq\hat{f}_{2}(x_{k}).
    \end{aligned}
\end{equation*}
Предположив точное вычисление $x_{k + 1}$ ($\varepsilon_{k} = 0$), в случае ограниченных по предположениям \ref{as:2} и \ref{as:3} функционалов получаем явное выражение $\nu$ \eqref{eq:nu_value} (в рамках предположения \ref{as:5}) и уточнение оценки выше:
\begin{equation*}
    \begin{aligned}
        &\exp\left(-\frac{k\mu L}{2}\left(\frac{1}{M_{\hat{G}}^{2} + \max\left\{\tilde{\gamma}L_{\hat{F}}P_{\hat{g}_{1}},~\gamma L_{\hat{F}}\right\}}\right)^{2}\right)\left(\hat{f}_{2}(x_{0}) - \hat{f}_{2}^{*}\right) + \hat{f}_{2}^{*}\geq\hat{f}_{2}(x_{k}).\\
    \end{aligned}
\end{equation*}
А так как в \eqref{eq:nu_value} и в доказательстве оценки сходимости из теоремы \ref{th:gen_stoch_lin_conv} значение $\hat{f}_{2}^{*}$ возможно заменить на произвольное неотрицательное число, то предположение \ref{as:5} для ограниченных функционалов означает разрешимость системы уравнений \eqref{eq:smooth_system}:
\begin{equation*}
    \begin{aligned}
        &\exp\left(-\frac{k\mu L}{2}\left(\frac{1}{M_{\hat{G}}^{2} + \max\left\{\tilde{\gamma}L_{\hat{F}}P_{\hat{g}_{1}},~\gamma L_{\hat{F}}\right\}}\right)^{2}\right)\hat{f}_{2}(x_{0})\geq\hat{f}_{2}(x_{k}),
    \end{aligned}
\end{equation*}
которая обычно выполняется при $m\leq n,~\dim(E_{2})\leq\dim(E_{1})$.

Согласно выводам подраздела \ref{subsec:approximate_step}, схема \ref{alg:gen_stoch_unbounded_gnm} также может быть использована и для ограниченных по предположениям \ref{as:2} и \ref{as:3} функционалов. Более того, можно для данной схемы доказать соответствующие варианты теорем \ref{th:3}, \ref{th:4}. В схемах доказательства нижняя оценка на значения $L_{k}$ не изменится, так как $L\leq\max\left\{L,~\frac{L}{\tilde{\tau}_{k}}\right\}$, зато изменится верхняя оценка на $L_{k}$ и в оценках сходимости произойдёт следующая замена:
\begin{equation*}
    \begin{aligned}
        &\hat{g}_{1}(x_{k}, B_{k})L_{k}\leq\max\left\{\tilde{\gamma}P_{\hat{g}_{1}}L_{\hat{F}},\gamma L_{\hat{F}}\right\}\text{ вместо }\hat{g}_{1}(x_{k}, B_{k})L_{k}\leq\gamma P_{\hat{g}_{1}}L_{\hat{F}}.
    \end{aligned}
\end{equation*}

По аналогии с детерминированным случаем можно ввести адаптивную настройку $\tau_{k}$ на каждой итерации, одна из возможных реализаций которой представлена на схеме \ref{alg:gen_stoch_flex_gnm}.
\RestyleAlgo{boxruled}
\begin{algorithm}[!ht]{}
\caption{\textbf{Общий метод трёх стохастических квадратов с неточным проксимальным отображением и адаптивным поиском $\tau$}}
\label{alg:gen_stoch_flex_gnm}
\textbf{Вход:}
    \begin{equation*}
        \begin{cases}
            x_{0}\in E_{1}\text{ --- начальное приближение},~x_{-1} = x_{0};\\
            \mathcal{E}(\cdot)\text{ --- функция погрешности подбора }\tau;\\
            N\in\mathbb{N}\text{ --- количество итераций метода};\\
            \gamma\geq\tilde{\gamma}\geq1\text{ --- факторы верхней границы поиска }L_{\hat{F}};\\
            L\text{ --- оценка локальной постоянной Липшица},~L\in(0,~\tilde{\gamma} L_{\hat{F}}];\\
            \mathcal{B}\text{ --- выборка функций},~b\in\overline{1, m}\text{ --- размер $B_{k}\subseteq\mathcal{B},~k\in\mathbb{Z}_{+}$};\\
            \hat{\mathcal{X}}(\cdot)\text{ --- отображение, аппроксимирующее }\hat{T}_{L_{k}, \tau_{k}}(x_{k}, B_{k}).
        \end{cases}
    \end{equation*}
    \vspace{0.2cm}
    \textbf{Повторять для $k = 0, 1,~\dots, N - 1$:}
    \begin{itemize}
        \item[1.] сэмплировать батч $B_{k}$ из $\mathcal{B}$ размера $b$;
        \item[2.] определить $\varepsilon_{k} = \mathcal{E}(k, x_{k}, x_{k - 1}, B_{k})$, $L_{0} = \max\left\{L, \frac{L}{\hat{g}_{1}(x_{0}, B_{0})}\right\}$;
        \item[3.] вычислить $\tau_{k}^{*} > 0$, для которого выполнено $\hat{g}_{1}(x_{k}, B_{k})\geq\hat{\psi}_{x_{k}, L_{k}, \tau_{k}^{*}}(\hat{\mathcal{X}}(x_{k}, B_{k}, L_{k}, \tau_{k}^{*}), B_{k})$ и $\hat{\psi}_{x_{k}, L_{k}, \tau_{k}^{*}}(\hat{\mathcal{X}}(x_{k}, B_{k}, L_{k}, \tau_{k}^{*}), B_{k}) - \hat{\psi}_{x_{k}, L_{k}, \hat{\mathcal{T}}_{L_{k}}(x_{k}, B_{k})}(\hat{T}_{L_{k}, \hat{\mathcal{T}}_{L_{k}}(x_{k}, B_{k})}(x_{k}, B_{k}), B_{k})\leq\varepsilon_{k}$;
        \item[4.] вычислить $x_{k + 1} = \hat{\mathcal{X}}(x_{k}, B_{k}, L_{k}, \tau_{k}^{*})$;
        \item[5.] если $\hat{g}_{1}(x_{k + 1}, B_{k}) > \hat{\psi}_{x_{k}, L_{k}, \tau_{k}^{*}}(x_{k + 1}, B_{k})$, то положить $L_{k} := \min\left\{2L_{k},~\max\left\{\tilde{\gamma}L_{\hat{F}},~\frac{\gamma L_{\hat{F}}}{\hat{g}_{1}(x_{k}, B_{k})}\right\}\right\}$ и вернуться к пункту 3;
        \item[6.] $L_{k + 1} = \max\left\{\frac{L_{k}}{2},~\max\left\{L,~\frac{L}{\hat{g}_{1}(x_{k}, B_{k})}\right\}\right\}$.
    \end{itemize}
    \vspace{0.2cm}
    
    \textbf{Выход:} $x_{N}$.
\end{algorithm}
В данной схеме похожим на детерминированный случай способом введено оптимальное значение $\tau$:
$$\hat{\mathcal{T}}_{L}(x, B) = \argmin\limits_{\tau > 0}\left\{\hat{\psi}_{x, L, \tau}(\hat{T}_{L, \tau}(x, B), B)\right\},~x\in E_{1},~L > 0,~B\subseteq\mathcal{B},$$
которое вытекает из свойства строгой выпуклости по $\tau$ стохастической локальной модели и позволяет получить представление точки минимума по $y$ при $L\geq L_{\hat{F}}$ для негладкой локальной модели:
\begin{equation*}
    \begin{aligned}
        \hat{g}_{1}(\hat{T}_{L, \hat{\mathcal{T}}_{L}(x, B)}(x, B), B)&\leq\min\limits_{y\in E_{1}}\left\{\frac{L}{2}\|y - x\|^{2} + \left\|\hat{G}(x, B) + \hat{G}^{'}(x, B)(y - x)\right\|\right\} =\\
        &= \min\limits_{y\in E_{1}}\min\limits_{\tau > 0}\left\{\frac{L}{2}\|y - x\|^{2} + \frac{\tau}{2} + \frac{1}{2\tau}\left\|\hat{G}(x, B) + \hat{G}^{'}(x, B)(y - x)\right\|^{2}\right\} =\\
        &= \min\limits_{\tau > 0}\left\{\frac{\tau}{2} + \min\limits_{y\in E_{1}}\left\{\frac{L}{2}\|y - x\|^{2} + \frac{1}{2\tau}\left\|\hat{G}(x, B) + \hat{G}^{'}(x, B)(y - x)\right\|^{2}\right\}\right\} =\\
        &= \min\limits_{\tau > 0}\left\{\hat{\psi}_{x, L, \tau}(\hat{T}_{L, \tau}(x, B), B)\right\}\Rightarrow\\
        &\Rightarrow \hat{T}_{L, \hat{\mathcal{T}}_{L}(x, B)}(x, B)\in\Argmin\limits_{y\in E_{1}}\left\{\frac{L}{2}\|y - x\|^{2} + \left\|\hat{G}(x, B) + \hat{G}^{'}(x, B)(y - x)\right\|\right\}.
    \end{aligned}
\end{equation*}
Как следствие, цепочка неравенств ниже позволяет утверждать справедливость теорем \ref{th:3}, \ref{th:4}, \ref{th:5}, \ref{th:6}, \ref{th:gen_stoch_sublin_conv}, \ref{th:gen_stoch_lin_conv} при выборе $\tau_{k} \approx \hat{\mathcal{T}}_{L_{k}}(x_{k}, B_{k})$ в условии и замене схем \ref{alg:gen_stoch_gnm}, \ref{alg:gen_stoch_unbounded_gnm} на схему \ref{alg:gen_stoch_flex_gnm}, как и в детерминированном случае по аналогии с неравенствами \eqref{eq:local_models_orders} и \eqref{eq:local_diffs_order}:
\begin{equation*}
    \begin{aligned}
        \hat{g}_{1}(x_{k}, B_{k}) - \hat{g}_{1}(\hat{T}_{L_{k}, \hat{\mathcal{T}}_{L_{k}}(x_{k}, B_{k})}(x_{k}, B_{k}), B_{k})&\geq\hat{g}_{1}(x_{k}, B_{k}) - \hat{\psi}_{x_{k}, L_{k}, \hat{\mathcal{T}}_{L_{k}}(x_{k}, B_{k})}(\hat{T}_{L_{k}, \hat{\mathcal{T}}_{L_{k}}(x_{k}, B_{k})}(x_{k}, B_{k}), B_{k})\geq\\
        &\geq\hat{g}_{1}(x_{k}, B_{k}) - \hat{\psi}_{x_{k}, L_{k}, \hat{g}_{1}(x_{k}, B_{k})}(\hat{T}_{L_{k}, \hat{g}_{1}(x_{k}, B_{k})}(x_{k}, B_{k}), B_{k})\geq0.
    \end{aligned}
\end{equation*}
В этой цепочке неравенств последовательность $\left\{x_{k}\right\}_{k\in\mathbb{Z}_{+}}$ порождена по схеме \ref{alg:gen_stoch_flex_gnm}. В данной схеме процедура получения $x_{k + 1}$ по известным $(x_{k}, B_{k}, L_{k}, \tau_{k})$ на каждой итерации $k\in\mathbb{Z}_{+}$ обозначена через отображение
$$\hat{\mathcal{X}}:E_{1}\times\sigma(\mathcal{B})\times\mathbb{R}_{++}^{2}\rightarrow E_{1},$$
где $\sigma(\mathcal{B})$ --- сигма--алгебра конечного набора функций $F_{i},~i\in\overline{1,~m}$, а процесс оптимизации с адаптивным подбором $\tau_{k}$ на каждом шаге $k$ заключается в следующем:
\begin{itemize}
    \itemsep=-4pt
    \item[1.] Вычислить $\tau_{k}^{*}$ --- приближение оптимального значения $\hat{\mathcal{T}}_{L_{k}}(x_{k} ,B_{k})$, $\tau_{k}^{*}$ может быть получено приближённо из задачи поиска элемента множества $\Argmin\limits_{\tau > 0}\left\{\hat{\psi}_{x_{k}, L_{k}, \tau}(\hat{\mathcal{X}}(x_{k}, B_{k}, L_{k}, \tau), B_{k})\right\}$;
    \item[2.] Получить значение $x_{k + 1} = \hat{\mathcal{X}}(x_{k}, B_{k}, L_{k}, \tau_{k}^{*})$ как приближение $\hat{T}_{L_{k}, \tau_{k}^{*}}(x_{k}, B_{k})$.
\end{itemize}
Как и в детерминированном случае, минимизация по $\tau_{k}$ в схеме \ref{alg:gen_stoch_flex_gnm} может быть такой же трудоёмкой процедурой. И при использовании правила \eqref{eq:stoch_direct_update_rule} в качестве представления $\hat{\mathcal{X}}$ значение\\$\hat{\psi}_{x_{k}, L_{k}, \tau}(\hat{\mathcal{X}}(x_{k}, B_{k}, L_{k}, \tau), B_{k})$ выражается явно:
\begin{equation*}
    \begin{aligned}
        \hat{\psi}_{x_{k}, L_{k}, \tau}(&\hat{\mathcal{X}}(x_{k}, B_{k}, L_{k}, \tau), B_{k}) = \frac{\tau}{2} + \frac{\hat{g}_{2}(x_{k}, B_{k})}{2\tau} -\\
        &- \frac{\eta_{k}(2 - \eta_{k})}{2\tau}\left\langle\left(\hat{G}^{'}(x_{k}, B_{k})^{*}\hat{G}^{'}(x_{k}, B_{k}) + \tau L_{k}I_{n}\right)^{-1}\hat{G}^{'}(x_{k}, B_{k})^{*}\hat{G}(x_{k}, B_{k}),~\hat{G}^{'}(x_{k}, B_{k})^{*}\hat{G}(x_{k}, B_{k})\right\rangle.
    \end{aligned}
\end{equation*}
Установка $\eta_{k} = 1$ делает функцию $\hat{\psi}_{x_{k}, L_{k}, \tau}(\hat{\mathcal{X}}(x_{k}, B_{k}, L_{k}, \tau), B_{k})$ выпуклой по $\tau$ по тем же причинам, что и в детерминированном случае, позволяя эффективно вычислять приближение точки минимума $\tau_{k}^{*}$ с помощью стандартных методов выпуклой оптимизации.

\section{Условие слабого роста в методе Гаусса--Ньютона}\label{sec:stoch_modified_gnm_weak_growth}

Специфика решаемой задачи \eqref{eq:main_opt_problem} позволяет сверху оценить квадрат нормы градиента функции $\hat{f}_{2}$ с помощью значения функции $\hat{f}_{2}$:
$$\left\|\nabla\hat{f}_{2}(x)\right\|^{2} = \left\|2\hat{F}^{'}(x)^{*}\hat{F}(x)\right\|^{2}\leq4\left\|\hat{F}^{'}(x)\right\|^{2}\hat{f}_{2}(x).$$
Если предположить ограниченность матрицы Якоби в стохастическом режиме оптимизации согласно предположению \ref{as:2}, то для функции $\hat{g}_{2}$ квадрат нормы градиента можно представить как $\operatorname{O}\left(\hat{g}_{2}(x, B)\right)$:
$$\left\|\nabla_{x}\hat{g}_{2}(x, B)\right\|^{2} = \left\|2\hat{G}^{'}(x, B)^{*}\hat{G}(x, B)\right\|^{2}\leq4 M_{\hat{G}}^{2}\hat{g}_{2}(x, B).$$
Полученное неравенство называется условием слабого роста. В текущем разделе рассматривается соотношение размерностей задачи \eqref{eq:smooth_system}, типично возникающее в системах нелинейных уравнений:
$$\dim(E_{1})\geq\dim(E_{2}),~n\geq m,$$
то есть когда параметров в задаче больше, чем условий, что в терминах теории статистического обучения соответствует перепараметризованным моделям. Дополнительно к условию слабого роста предположим доминирование квадрата нормы градиента над значением функции по предположению \ref{as:5}:
$$\left\|\nabla_{x}\hat{g}_{2}(x, B)\right\|^{2} = 4\left\|\hat{G}^{'}(x, B)^{*}\hat{G}(x, B)\right\|^{2}\geq4\mu\hat{g}_{2}(x, B)$$
по свойству сингулярного спектра матрицы $\hat{G}^{'}(x, B)^{*}$. Таким образом, в текущих условиях квадрат нормы градиента ограничен значением функции с двух сторон:
$$4\mu\hat{g}_{2}(x, B)\leq\left\|\nabla_{x}\hat{g}_{2}(x, B)\right\|^{2}\leq4 M_{\hat{G}}^{2}\hat{g}_{2}(x, B).$$
При усреднении по $B\subseteq\mathcal{B}$ границы уже зависят от $\hat{f}_{2}$:
\begin{equation}\label{eq:bounded_stochastic_gradients}
    4\mu\hat{f}_{2}(x)\leq\mathbb{E}_{B}\left[\left\|\nabla_{x}\hat{g}_{2}(x, B)\right\|^{2}\right]\leq4 M_{\hat{G}}^{2}\hat{f}_{2}(x).
\end{equation}
Это означает тождественное равенство нулю сэмплированных градиентов $\nabla_{x^{*}}\hat{g}_{2}(x^{*}, B) = \mathbf{0}_{n},~B\subseteq\mathcal{B}$ в точке минимума $x^{*}:~F(x^{*}) = \mathbf{0}_{m}$. Более того, одновременное выполнение условия слабого роста и условия Поляка--Лоясиевича в случае $m\leq n$ приводит к выполнению \textit{условия сильного роста}, заключающегося в доминировании квадрата нормы градиента функции $\hat{f}_{2}$ над средним квадратом нормы стохастической оценки градиента функции $\hat{f}_{2}$:
\begin{equation*}
    \begin{aligned}
        \left\|\nabla\hat{f}_{2}(x)\right\|^{2} &= \left\|2\hat{F}^{'}(x)^{*}\hat{F}(x)\right\|^{2}\geq\left\{\text{предположение \ref{as:5}}\right\}\geq4\mu\hat{f}_{2}(x)\Rightarrow\left\{\text{\eqref{eq:bounded_stochastic_gradients}}\right\}\Rightarrow\\
        &\Rightarrow\mathbb{E}_{B}\left[\left\|\nabla_{x}\hat{g}_{2}(x, B)\right\|^{2}\right]\leq \frac{M_{\hat{G}}^{2}}{\mu}\left\|\nabla\hat{f}_{2}(x)\right\|^{2}.
    \end{aligned}
\end{equation*}
Из такого условия сильного роста следует, что любая стационарная точка $x^{*}\in E_{1}:~\left\|\nabla\hat{f}_{2}(x^{*})\right\| = 0$ является также точкой глобального минимума.

В данной работе дополнительно установлено, что выполнение условия слабого роста вместе с условием Поляка--Лоясиевича позволяет решить задачу \eqref{eq:main_opt_problem} с любой наперёд заданной точностью и при любом размере батча сэмплированных на каждой итерации функций, используя метод Гаусса--Ньютона по схеме \ref{alg:gen_double_stoch_gnm} с правилом \eqref{eq:double_stoch_direct_update_rule}. В самой схеме \ref{alg:gen_double_stoch_gnm} значение $l_{\hat{g}_{2}}$ заменяется на $l_{\hat{f}_{2}}$, а батчи $B_{k}$ и $\tilde{B}_{k}$ принципиально независимо сэмплируются при оценке сходимости по теореме \ref{th:weak_growth_condition}, при этом пункты 5 и 6 в схеме \ref{alg:gen_double_stoch_gnm} пропускаются, так как условие теоремы \ref{th:weak_growth_condition} означает выбор $l = l_{\hat{g}_{2}}$ и $\gamma = 1$.
\begin{theorem}\label{th:weak_growth_condition_main}
Пусть выполнены предположения \ref{as:2}, \ref{as:3}, \ref{as:5}, \ref{as:jacob_smoothness}. Рассмотрим метод Гаусса--Ньютона со схемой реализации \ref{alg:gen_double_stoch_gnm}, в котором последовательность $\{x_{k}\}_{k\in\mathbb{Z}_{+}}$ вычисляется по правилу \eqref{eq:double_stoch_direct_update_rule} с $\tilde{\tau}_{k}\geq\tilde{\tau} > 0$,\\$L_{k} \geq L > 0$ и $l_{k} \equiv l_{\hat{f}_{2}} = 2\left(L_{\hat{F}}P_{\hat{f}_{1}} + M_{\hat{F}}^{2}\right)$. Тогда для последовательности
\begin{equation*}
    \begin{aligned}
        &\eta_{k} = \frac{\mu\left(\tilde{\tau}_{k}L_{k}\right)^{2}}{\left(M_{\hat{G}}^{2} + \tilde{\tau}_{k}L_{k}\right)\left(L_{\hat{F}}P_{\hat{f}_{1}} + M_{\hat{F}}^{2}\right)M_{\hat{G}}^{2}},~k\in\mathbb{Z}_{+}
    \end{aligned}
\end{equation*}
    верна следующая оценка
    \begin{equation*}
        \begin{aligned}
            &\mathbb{E}\left[\hat{f}_{2}(x_{k})\right]\leq\mathbb{E}\left[\hat{f}_{2}(x_{0})\right]\exp\left(-\frac{k}{\left(L_{\hat{F}}P_{\hat{f}_{1}} + M_{\hat{F}}^{2}\right)M_{\hat{G}}^{2}}\left(\frac{\mu\tilde{\tau}L}{M_{\hat{G}}^{2} + \tilde{\tau}L}\right)^{2}\right),~k\in\mathbb{Z}_{+}.
        \end{aligned}
    \end{equation*}
    В случае $\eta_{k} = 1,~k\in\mathbb{Z}_{+}$ оценка сходимости при использовании правила \eqref{eq:double_stoch_direct_update_rule} в условиях данной теоремы не лучше
    \begin{equation*}
        \begin{cases}
            \mathbb{E}\left[\hat{f}_{2}(x_{k})\right]&\leq\mathbb{E}\left[\hat{f}_{2}(x_{0})\right]\exp\left(-\frac{k\mu^{2}}{M_{\hat{G}}^{2}}\left(\frac{2}{\mu + \left(L_{\hat{F}}P_{\hat{f}_{1}} + M_{\hat{F}}^{2}\right)c} - \frac{1}{\left(L_{\hat{F}}P_{\hat{f}_{1}} + M_{\hat{F}}^{2}\right)c^{2}}\right)\right);\\[10pt]
            &c \overset{\operatorname{def}}{=} \frac{1}{3}\left(1 + 7\sqrt[3]{\frac{2}{47 + 3\sqrt{93}}} + \sqrt[3]{\frac{47 + 3\sqrt{93}}{2}}\right),~k\in\mathbb{Z}_{+}.
        \end{cases}
    \end{equation*}
    Оператор математического ожидания $\mathbb{E}\left[\cdot\right]$ усредняет по всей случайности процесса оптимизации.    
\end{theorem}

В теореме \ref{th:weak_growth_condition} установлена линейная скорость сходимости при выборе настраиваемого шага $\eta_{k}$, и минимальное количество итераций, необходимое для достижения уровня $\epsilon > 0$ функции $\hat{f}_{1}$ в среднем асимптотически выражается как $\operatorname{O}\left(\ln\left(\frac{1}{\epsilon}\right)\right)$:
\begin{equation*}
    \begin{aligned}
        k = \left\lceil\frac{M_{\hat{G}}^{2}\left(L_{\hat{F}}P_{\hat{f}_{1}} + M_{\hat{F}}^{2}\right)}{\mu^{2}}\left(\frac{M_{\hat{G}}^{2}}{\tilde{\tau}L} + 1\right)^{2}\ln\left(\frac{\mathbb{E}\left[\hat{g}_{2}(x_{0}, B_{0})\right]}{\epsilon^{2}}\right)\right\rceil.
    \end{aligned}
\end{equation*}
Как и в теореме \ref{th:double_stoch_sublin_conv}, данная оценка демонстрирует оптимальность выбора $\tilde{\tau}L\rightarrow+\infty$, приближая шаг метода трёх стохастических квадратов к шагу метода стохастического градиентного спуска. Если же зафиксировать $\eta_{k} = 1$, то в доказательстве теоремы \ref{th:weak_growth_condition} вычислено оптимальное значение $\tilde{\tau}_{k}L_{k}$, при котором будет линейная сходимость, но уже более медленная, чем в случае стохастического градиентного спуска, со следующим количеством необходимых итераций для достижения уровня $\epsilon > 0$ в среднем:
\begin{equation*}
    \begin{cases}
        k = \left\lceil\frac{M_{\hat{G}}^{2}}{\mu^{2}}\left(\frac{2}{\mu + \left(L_{\hat{F}}P_{\hat{f}_{1}} + M_{\hat{F}}^{2}\right)c} - \frac{1}{\left(L_{\hat{F}}P_{\hat{f}_{1}} + M_{\hat{F}}^{2}\right)c^{2}}\right)^{-1}\ln\left(\frac{\mathbb{E}\left[\hat{g}_{2}(x_{0}, B_{0})\right]}{\epsilon^{2}}\right)\right\rceil;\\
        c = \frac{1}{3}\left(1 + 7\sqrt[3]{\frac{2}{47 + 3\sqrt{93}}} + \sqrt[3]{\frac{47 + 3\sqrt{93}}{2}}\right).
    \end{cases}
\end{equation*}
И в случае оптимального выбора $\eta_{k}$ и при $\eta_{k} = 1$ в роли главного ограничителя в применении теоремы \ref{th:weak_growth_condition} на практике выступает необходимость знать заранее постоянные $M_{\hat{G}}$, $P_{\hat{f}_{1}}$, $L_{\hat{F}}$ или их верхние оценки. Если же рассматривать метод Гаусса--Ньютона с правилами обновления \eqref{eq:stoch_direct_update_rule} и \eqref{eq:stoch_approx_general_update_rule} в схеме \ref{alg:gen_stoch_gnm}, то при неограниченном уменьшении $\epsilon\rightarrow 0$ в рамках предположений данной работы стохастический режим оптимизации перейдёт в детерминированный, то есть условие слабого роста для этих стратегий не позволяет гарантированно решить задачу с любой наперёд заданной точностью без использования всей генеральной совокупности функций $\mathcal{B}$.

Теорема \ref{th:weak_growth_condition} устанавливает в случае $m\leq n$ условия для получения наименьшей сложности решения задачи относительно количества вызовов оракула, так как в теореме \ref{th:weak_growth_condition} можно положить $b = \operatorname{O}\left(1\right)$ и $\tilde{b} = \operatorname{O}\left(1\right)$ в силу произвольности размеров батчей:

\begin{itemize}
    \item $bk = \min\left\{\operatorname{O}\left(\frac{m}{\epsilon^{2}}\right),~\operatorname{O}\left(\frac{1}{\epsilon^{6}}\right)\right\}$ согласно теореме \ref{th:3};
    \item $bk = \min\left\{\operatorname{O}\left(m\ln\left(\frac{1}{\epsilon}\right)\right),~\operatorname{O}\left(\frac{1}{\epsilon^{4}}\ln\left(\frac{1}{\epsilon}\right)\right)\right\}$ согласно теореме \ref{th:4};
    \item $\left(\tilde{b} + b\right)k = \operatorname{O}\left(\ln\left(\frac{1}{\epsilon}\right)\right)$ согласно теореме \ref{th:weak_growth_condition}.
\end{itemize}

При использовании схем \ref{alg:gen_stoch_unbounded_gnm} и \ref{alg:gen_stoch_flex_gnm} в рамках предположений \ref{as:jacob_smoothness}, \ref{as:bounded_variance_growth} и \ref{as:prox_PL_condition} для решения задач \eqref{eq:main_opt_problem}, в которых выполнено структурное ограничение из предположения \ref{as:bounded_variance_growth} $c_{1} = 0$, гипотетически существует возможность решить систему \eqref{eq:smooth_system} с помощью метода Гаусса--Ньютона с произвольной точностью полностью в режиме стохастической аппроксимации, согласно следствиям \ref{th:gen_stoch_sublin_conv_cor2} и \ref{th:gen_stoch_lin_conv_cor2}. Эта гипотеза дополнительно мотивируется примером \eqref{eq:bounded_g2_bounded_growth}, в котором значения постоянных из предположения \ref{as:bounded_variance_growth} не зависят от гиперпараметров $\tilde{\gamma}$ и $\gamma$.

\section{Особенности реализации на практике}\label{sec:practical_considerations}

\subsection{Эффективное вычисление производных}

При решении задачи \eqref{eq:main_opt_problem} с помощью процедуры неточного поиска приближения решения $x_{k + 1}$ на каждой итерации полезно достаточно быстро производить любой этап данной процедуры. В случае решения вспомогательной задачи
$$\argmin\limits_{y\in E_{1}}\left\{\psi_{x_{k}, L_{k}, \tau_{k}}(y)\right\}$$
градиентными методами оптимизации возможно затратить линейное от размерности задачи количество памяти при вычислении производной локальной модели по параметрам. Ниже представлен один из возможных способов осуществления такого эффективного вычисления значений $\psi_{x_{k}, L_{k}, \tau_{k}}(y)$ и $\nabla_{y}\psi_{x_{k}, L_{k}, \tau_{k}}(y)$:
\begin{equation*}
    \begin{aligned}
        \psi_{x_{k}, L_{k}, \tau_{k}}(y) &= \left(\frac{\tau_{k}}{2} + \frac{\hat{f}_{2}(x_{k})}{2\tau_{k}}\right) + \left\langle\frac{\hat{F}^{'}(x_{k})^{*}\hat{F}(x_{k})}{\tau_{k}},~y - x_{k}\right\rangle +
    \end{aligned}
\end{equation*}
\begin{equation*}
    \begin{aligned}
        &+ \frac{1}{2}\left\langle\left(\frac{\hat{F}^{'}(x_{k})^{*}\hat{F}^{'}(x_{k})}{\tau_{k}} + L_{k} I_{n}\right)(y - x_{k}),~y - x_{k}\right\rangle = \left\{\text{введём }\tilde{x} := x_{k},~\hat{x} := x_{k}\right\} =\\
        &= \left(\frac{\tau_{k}}{2} + \frac{\hat{f}_{2}(x_{k})}{2\tau_{k}}\right) + \frac{1}{\tau_{k}}\left\langle\nabla_{\tilde{x}}\left\langle\hat{F}(x_{k}),~\hat{F}(\tilde{x})\right\rangle,~y - x_{k}\right\rangle +\\
        &+ \frac{1}{2\tau_{k}}\left\langle\nabla_{\hat{x}}\left\langle\nabla_{\tilde{x}}\left\langle\hat{F}(\hat{x}),~\hat{F}(\tilde{x})\right\rangle,~y - x_{k}\right\rangle,~y - x_{k}\right\rangle + \frac{L_{k}}{2}\|y - x_{k}\|^{2}\Rightarrow\\
        &\Rightarrow\nabla_{y}\psi_{x_{k}, L_{k}, \tau_{k}}(y) = \frac{1}{\tau_{k}}\nabla_{\tilde{x}}\left\langle\hat{F}(x_{k}),~\hat{F}(\tilde{x})\right\rangle + \frac{1}{\tau_{k}}\nabla_{\hat{x}}\left\langle\nabla_{\tilde{x}}\left\langle\hat{F}(\hat{x}),~\hat{F}(\tilde{x})\right\rangle,~y - x_{k}\right\rangle + L(y - x_{k}).
    \end{aligned}
\end{equation*}
Сразу видно, в преобразованиях выше не требуется в явном виде хранить матрицу $\hat{F}^{'}(x_{k})$, что приводит к значительной экономии по памяти: $\operatorname{O}(m + n)$ хранимых элементов вместо $\operatorname{O}(mn + m + n)$. Само вычисление $\psi_{x_{k}, L_{k}, \tau_{k}}(y)$ и его производной состоит из суперпозиции произведений матрицы Якоби на вектор, реализуемое за $\operatorname{O}(m + n)$ операций вместо $\operatorname{O}(mn)$ операций прямого вычисления матричного произведения. Дальнейшая экономия времени вычисления заключается в представлении $\psi_{x_{k}, L_{k}, \tau_{k}}(y)$ через $\nabla_{y}\psi_{x_{k}, L_{k}, \tau_{k}}(y)$:
$$\psi_{x_{k}, L_{k}, \tau_{k}}(y) = \frac{\tau_{k}}{2} + \frac{\hat{f}_{2}(x_{k})}{2\tau_{k}} + \frac{1}{2}\left\langle\nabla_{y}\psi_{x_{k}, L_{k}, \tau_{k}}(y) + \frac{1}{\tau_{k}}\nabla_{\tilde{x}}\left\langle\hat{F}(x_{k}),~\hat{F}(\tilde{x})\right\rangle,~y - x_{k}\right\rangle.$$
Такие условия эффективного вычисления обычно выполнены в вычислительных системах с автоматическим дифференцированием, в которых при вычислении самого значения функции одновременно производится вычисление производной функции без хранения избыточного количества переменных \cite{Baydin2017}, применение неточного поиска с автоматическим дифференцированием тем более является оправданным, если на каждом шаге метода Гаусса--Ньютона количество итераций поиска $(x_{k + 1}, L_{k})$ не превышает $\min\left\{m ,~n\right\}$. Описанный выше способ эффективного вычисления производной локальной модели работает и в стохастическом случае, для локальной модели $\hat{\psi}_{x_{k}, L_{k}, \tau_{k}}(y, B_{k})$; аппарат автоматического дифференцирования непосредственно применим в схемах \ref{alg:gen_det_gnm}, \ref{alg:gen_det_flex_gnm}, \ref{alg:gen_stoch_flex_gnm} и в схемах \ref{alg:gen_stoch_gnm}, \ref{alg:gen_stoch_unbounded_gnm} со стратегией поиска $x_{k + 1}$ \eqref{eq:stoch_approx_general_update_rule}.

\subsection{Линейный поиск в локальной модели}

Если же решать задачу \eqref{eq:main_opt_problem} по схеме \ref{alg:gen_det_gnm}, \ref{alg:gen_stoch_gnm}, \ref{alg:gen_double_stoch_gnm}, \ref{alg:gen_stoch_unbounded_gnm} или даже по схемам \ref{alg:gen_det_flex_gnm} и \ref{alg:gen_stoch_flex_gnm} с использованием стратегий \eqref{eq:det_expl_update_rule}, \eqref{eq:stoch_direct_update_rule}, \eqref{eq:double_stoch_direct_update_rule}, то возникает вопрос об эффективности вычисления каждой итерации поиска подходящей оценки локальной постоянной Липшица. Наиболее дорогой операцией при вычислении $x_{k + 1}$ является обращение матрицы размера $n\times n$. Для наиболее оптимальной организации итерации необходимо рассмотреть два случая соотношения размеров задачи $\dim(E_{1}) = n$ и $\dim(E_{2}) = m$. В целях удобства выкладки представлены для детерминированного метода Гаусса--Ньютона, однако они непосредственно перекладываются на стохастический случай при замене локальной модели $\psi_{x_{k}, L_{k}, \tau_{k}}(y)$ на $\hat{\psi}_{x_{k}, L_{k}, \tau_{k}}(y, B_{k})$ и значения размерности $m$ на $b$.

Первый случай соответствует ситуации, возникающей в задаче восстановления регрессии, когда данных больше, чем параметров в задаче: $m > n$. Для того, чтобы при обновлении $L$ в задаче вычисления
$$\argmin\limits_{y\in E_{1}}\left\{\psi_{x_{k}, L_{k}, \tau_{k}}(y)\right\}$$
значение приближения точки минимума быстрее обновлялось, выполним спектральное разложение матрицы $\hat{F}^{'}(x_{k})^{*}\hat{F}^{'}(x_{k})$:
\begin{equation*}
    \begin{aligned}
        \hat{F}^{'}(x_{k})^{*}\hat{F}^{'}(x_{k}) &= Q_{n}\Lambda_{n} Q_{n}^{*},~Q_{n}^{*}Q_{n} = I_{n},~\Lambda_{n}\text{ --- диагональная матрица};\\
        &\hat{F}^{'}(x_{k})^{*}\hat{F}^{'}(x_{k}) + \tau_{k} L_{k} I_{n} = Q_{n}\Lambda_{n} Q_{n}^{*} + \tau_{k} L_{k} I_{n} = Q_{n}\left(\Lambda_{n} + \tau_{k} L_{k} I_{n}\right)Q_{n}^{*}\Rightarrow\\
        &\Rightarrow\left(\hat{F}^{'}(x_{k})^{*}\hat{F}^{'}(x_{k}) + \tau_{k} L_{k} I_{n}\right)^{-1} = Q_{n}\left(\Lambda_{n} + \tau_{k} L_{k} I_{n}\right)^{-1}Q_{n}^{*},~m > n\Rightarrow\\
        &\Rightarrow x_{k + 1} = x_{k} - \eta_{k} Q_{n}\left(\Lambda_{n} + \tau_{k} L_{k} I_{n}\right)^{-1}Q_{n}^{*}\hat{F}^{'}(x_{k})^{*}\hat{F}(x_{k}).
    \end{aligned}
\end{equation*}
Если выбрать $\eta_{k} = 1$, то $x_{k + 1}$ будет точкой минимума локальной модели $\psi_{x_{k}, L_{k}, \tau_{k}}(y)$. В представленном выражении для $x_{k + 1}$ обновление значения $L_{k}$ на каждой итерации за $\operatorname{O}(n)$ операций записывается в обратную матрицу. Само спектральное разложение требует $\operatorname{O}(n^{3})$ операций и $\operatorname{O}(n^{2} + n)$ дополнительно хранимых элементов в виде ортогональной матрицы $Q_{n}$ и диагональной матрицы $\Lambda_{n}$, разложение с данной асимптотикой осуществимо через тридиагонализацию с помощью принципа <<разделяй и властвуй>> и последующее спектральное разложение трёхдиагональной матрицы~\cite{Householder1958, Cuppen1980, BurdenFaires1989}. Неизменный на каждой итерации вектор из $n$ элементов $Q_{n}^{*}\hat{F}^{'}(x_{k})^{*}\hat{F}(x_{k})$ можно вычислить за $\operatorname{O}(n^{2} + nm)$ операций, на каждой итерации диагональная матрица $\left(\Lambda_{n} + \tau_{k} L_{k} I_{n}\right)^{-1}$ обращается за $\operatorname{O}(n)$ операций, а произведение матрицы $Q_{n}$ на вектор $\left(\Lambda_{n} + \tau_{k} L_{k} I_{n}\right)^{-1}Q_{n}^{*}\hat{F}^{'}(x_{k})^{*}\hat{F}(x_{k})$ вычислимо за $\operatorname{O}(n^{2})$ операций. Таким образом, оценивая максимальное число итераций подбора $L_{k}$ сверху как $$\left\lceil\log_{2}\left(\frac{\gamma L_{\hat{F}}}{L}\right)\right\rceil + 1,$$ получаем следующую стоимость каждого шага метода Гаусса--Ньютона:
$$\operatorname{O}\left(n^{3} + n^{2} + mn + n\left(n + 1\right)\left(\left\lceil\log_{2}\left(\frac{\gamma L_{\hat{F}}}{L}\right)\right\rceil + 1\right)\right).$$
Количество итераций подбора $L_{k}$ оценивается из неравенства $\gamma L_{\hat{F}}\geq 2^{i}L$, $i\in\mathbb{Z}_{+}$.
Хотя на практике после $(k - 1)$--го шага с
$$\operatorname{O}\left(\left\lceil\log_{2}\left(\frac{\gamma L_{\hat{F}}}{L}\right)\right\rceil + 1\right)$$
итерациями поиска локальной постоянной Липшица при $\gamma = 2$ и $L\in(0,~L_{\hat{F}}]$ получается $L_{k}:~2L_{k}\geq L_{\hat{F}}$, что эффективно означает не больше двух итераций поиска $L_{k}$. Полученная оценка сложности внешней итерации обобщается на стохастический случай заменой $m$ на $b$ и на $\max\{\tilde{b},~b\}$ для дважды стохастического шага.

В случае  $m\leq n$, соответствующем ситуации, обычно возникающей при решении системы нелинейных уравнений, воспользуемся тождеством Шермана--Моррисона--Вудбери для обращения матрицы меньшего размера и спектральным разложением матрицы $\hat{F}^{'}(x_{k})\hat{F}^{'}(x_{k})^{*}$ для обращения диагональной матрицы $\Lambda_{m}$ на каждой итерации за $\operatorname{O}(m)$ операций с предварительным вычислением матриц $\Lambda_{m}$ и $Q_{m}$ за $\operatorname{O}\left(m^{3}\right)$ операций, затратив дополнительную память для $\operatorname{O}\left(m^{2} + m\right)$ элементов:
\begin{equation*}
    \begin{aligned}
        \hat{F}^{'}(x_{k})\hat{F}^{'}(x_{k})^{*} &= Q_{m}\Lambda_{m} Q_{m}^{*},~Q_{m}^{*}Q_{m} = I_{m},~\Lambda_{m}\text{ --- диагональная матрица};\\
        &\left(\hat{F}^{'}(x_{k})^{*}\hat{F}^{'}(x_{k}) + \tau_{k} L_{k} I_{n}\right)^{-1} = \frac{1}{\tau_{k} L_{k}}I_{n} - \frac{1}{\tau_{k} L_{k}}\hat{F}^{'}(x_{k})^{*}\left(\tau_{k} L_{k} I_{m} + \hat{F}^{'}(x_{k})\hat{F}^{'}(x_{k})^{*}\right)^{-1}\hat{F}^{'}(x_{k}) =\\
        &= \frac{1}{\tau_{k} L_{k}}I_{n} - \frac{1}{\tau_{k} L_{k}}\hat{F}^{'}(x_{k})^{*}Q_{m}\left(\tau_{k} L_{k} I_{m} + \Lambda_{m}\right)^{-1}Q_{m}^{*}\hat{F}^{'}(x_{k}),~m\leq n\Rightarrow\\
        \Rightarrow x_{k + 1} &= x_{k} - \eta_{k}\left(\frac{1}{\tau_{k} L_{k}}I_{n} - \frac{1}{\tau_{k} L_{k}}\hat{F}^{'}(x_{k})^{*}Q_{m}\left(\tau_{k} L_{k} I_{m} + \Lambda_{m}\right)^{-1}Q_{m}^{*}\hat{F}^{'}(x_{k})\right)\hat{F}^{'}(x_{k})^{*}\hat{F}(x_{k}) = \\
        &= x_{k} - \frac{\eta_{k}}{\tau_{k} L_{k}}\left(\hat{F}^{'}(x_{k})^{*}\hat{F}(x_{k}) - \hat{F}^{'}(x_{k})^{*}Q_{m}\left(\tau_{k} L_{k} I_{m} + \Lambda_{m}\right)^{-1}\Lambda_{m} Q_{m}^{*}\hat{F}(x_{k})\right) = \\
        &= x_{k} - \frac{\eta_{k}}{\tau_{k} L_{k}}\hat{F}^{'}(x_{k})^{*}\left(\hat{F}(x_{k}) - Q_{m}\left(\tau_{k} L_{k} I_{m} + \Lambda_{m}\right)^{-1}\Lambda_{m} Q_{m}^{*}\hat{F}(x_{k})\right).
    \end{aligned}
\end{equation*}
Перед началом поиска подходящей $L_{k}$ производится вычисление $m$--мерного вектора $\Lambda_{m} Q_{m}^{*}\hat{F}(x_{k})$ за\\$\operatorname{O}\left(m^{2} + m\right)$ операций. На каждой итерации подбора $L_{k}$ диагональная матрица $\left(\tau_{k} L_{k} I_{m} + \Lambda_{m}\right)^{-1}$ обновляется за $\operatorname{O}\left(m\right)$ операций, сам вектор $x_{k + 1}$ вычисляется суммарно за $\operatorname{O}\left(m^{2} + mn + m + n\right)$ операций. Таким образом, каждый шаг метода Гаусса--Ньютона использует следующее количество операций:
$$\operatorname{O}\left(m^{3} + m^{2} + m + \left(m^{2} + mn + m + n\right)\left(\left\lceil\log_{2}\left(\frac{\gamma L_{\hat{F}}}{L}\right)\right\rceil + 1\right)\right).$$

Для схемы \ref{alg:gen_double_stoch_gnm} стоимость каждого шага в худшем случае будет такой же, только вместо отношения $\frac{\gamma L_{\hat{F}}}{L}$ в асимптотических оценках будет величина $\frac{\gamma l_{\hat{g}_{2}}}{l}$. Полученная оценка сложности расширяется на стохастический случай заменой $m$ на $b$. Но для дважды стохастического шага правило эффективного вычисления $x_{k + 1}$ видоизменяется:
\begin{equation*}
    \begin{aligned}
        x_{k + 1} &= x_{k} - \eta_{k}\left(\frac{1}{\tilde{\tau}_{k} L_{k}}I_{n} -\right.\\
        &\left.- \frac{1}{\tilde{\tau}_{k} L_{k}}\hat{G}^{'}(x_{k}, \tilde{B}_{k})^{*}Q_{\tilde{b}}\left(\tilde{\tau}_{k} L_{k} I_{\tilde{b}} + \Lambda_{\tilde{b}}\right)^{-1}Q_{\tilde{b}}^{*}\hat{G}^{'}(x_{k}, \tilde{B}_{k})\right)\hat{G}^{'}(x_{k}, B_{k})^{*}\hat{G}(x_{k}, B_{k}) =\\
        &= x_{k} - \frac{\eta_{k}}{\tilde{\tau}_{k}L_{k}}\left(\hat{G}^{'}(x_{k}, B_{k})^{*}\hat{G}(x_{k}, B_{k}) -\right.\\
        &\left.- \hat{G}^{'}(x_{k}, \tilde{B}_{k})^{*}Q_{\tilde{b}}\left(\tilde{\tau}_{k} L_{k} I_{\tilde{b}} + \Lambda_{\tilde{b}}\right)^{-1}Q_{\tilde{b}}^{*}\hat{G}^{'}(x_{k}, \tilde{B}_{k})\hat{G}^{'}(x_{k}, B_{k})^{*}\hat{G}(x_{k}, B_{k})\right);\\
        &\hat{G}^{'}(x_{k}, \tilde{B}_{k})\hat{G}^{'}(x_{k}, \tilde{B}_{k})^{*} = Q_{\tilde{b}}\Lambda_{\tilde{b}} Q_{\tilde{b}}^{*},~Q_{\tilde{b}}^{*}Q_{\tilde{b}} = I_{\tilde{b}},~\Lambda_{\tilde{b}}\text{ диагональная матрица},
    \end{aligned}
\end{equation*}
для этого правила сложность внешней итерации следующая:
$$\operatorname{O}\left(\tilde{b}^{3} + \tilde{b}^{2} + (\tilde{b} + b)n + \tilde{b} + n\left(\left\lceil\log_{2}\left(\frac{\gamma l_{\hat{g}_{2}}}{l}\right)\right\rceil + 1\right)\right),$$
при использовании бинарного поиска $l_{k}$, в случае отсутствия процедуры поиска $l_{k}$ сложность итерации меньше:
$$\operatorname{O}\left(\tilde{b}^{3} + \tilde{b}^{2} + (\tilde{b} + b)n + \tilde{b} + n\right),$$
также эта оценка для $n\gg1$ меньше, чем получаемая аналогичная оценка при непосредственном вычислении $x_{k + 1}$ без применения процедур оптимизации, так как $\max\{\tilde{b},~b\}\leq n$. В обоих случаях можно использовать более вычислительно устойчивую процедуру тридиагонализации с помощью преобразований Хаусхолдера, сохраняя прежнюю вычислительную асимптотику \cite{Householder1958}.

\subsection{Неточный поиск минимума локальной модели}\label{subsec:gradient_control}

В схемах \ref{alg:gen_det_gnm}, \ref{alg:gen_stoch_gnm} и \ref{alg:gen_stoch_unbounded_gnm} можно рассмотреть более удобный функционал контроля за точностью вычисления $x_{k + 1}$, основанный на норме градиента локальной модели. А так как структурно в детерминированном и в стохастическом случае локальные модели строятся на похожих принципах, то достаточно рассмотреть детерминированный случай в рамках предположений \ref{as:det_hat_F_der_smooth} и \ref{as:2}, потому что он обобщается на стохастический заменой локальной модели $\psi_{x, L, \tau}(y)$ на стохастическую локальную модель $\hat{\psi}_{x, L, \tau}(y, B)$. И прежде чем вывести критерий, основанный на норме градиента локальной модели, рассмотрим функцию $\psi_{x_{k}, L_{k}, \tau_{k}}(\cdot),~k\in\mathbb{Z}_{+}$ и её производные:
\begin{equation*}
    \begin{aligned}
         \nabla_{x}\psi_{x_{k}, L_{k}, \tau_{k}}(x) &= L_{k}(x - x_{k}) + \frac{1}{\tau_{k}}\hat{F}^{'}(x_{k})^{*}\left(\hat{F}(x_{k}) + \hat{F}^{'}(x_{k})(x - x_{k})\right);\\
         \nabla_{x}^{2}\psi_{x_{k}, L_{k}, \tau_{k}}(x) &= L_{k} + \frac{1}{\tau_{k}}\hat{F}^{'}(x_{k})^{*}\hat{F}^{'}(x_{k})\Rightarrow
    \end{aligned}
\end{equation*}
\begin{equation*}
    \begin{aligned}
         &\Rightarrow\left\langle\nabla_{x}^{2}\psi_{x_{k}, L_{k}, \tau_{k}}(x)v,~v\right\rangle\leq\left(L_{k} + \frac{M_{\hat{F}}^{2}}{\tau_{k}}\right)\left\|v\right\|^{2},~\forall v\in E_{1},~v\neq\mathbf{0}_{n}\Rightarrow\\
         &\Rightarrow L_{\psi} \overset{\operatorname{def}}{=} L_{k} + \frac{M_{\hat{F}}^{2}}{\tau_{k}}\text{ --- постоянная Липшица для }\nabla_{x}\psi_{x_{k}, L_{k}, \tau_{k}}(x).
    \end{aligned}
\end{equation*}
Поэтому для $\psi_{x_{k}, L_{k}, \tau_{k}}(\cdot)$ верно выражение \eqref{eq:lm_aux_det_model_bound} из леммы \ref{lm:aux_det_upper_model}:
\begin{equation*}
    \begin{aligned}
        &\left|\psi_{x_{k}, L_{k}, \tau_{k}}(y) - \psi_{x_{k}, L_{k}, \tau_{k}}(x) - \left\langle\nabla_{x}\psi_{x_{k}, L_{k}, \tau_{k}}(x),~y - x\right\rangle\right|\leq\frac{L_{\psi}}{2}\left\|y - x\right\|^{2}\Rightarrow\\
        &\Rightarrow\left\{\text{оценка снизу на выражение под модулем}\right\}\Rightarrow\\
        &\Rightarrow\psi_{x_{k}, L_{k}, \tau_{k}}(y)\geq\psi_{x_{k}, L_{k}, \tau_{k}}(x) + \left\langle\nabla_{x}\psi_{x_{k}, L_{k}, \tau_{k}}(x),~y - x\right\rangle - \frac{L_{\psi}}{2}\left\|y - x\right\|^{2}.
    \end{aligned}
\end{equation*}
Возьмём в качестве $y\in E_{1}$ точку минимума $\psi_{x_{k}, L_{k}, \tau_{k}}(\cdot)$, а в качестве $x\in E_{1}$ достаточно близкую к $y$ точку:
\begin{equation*}
    x \overset{\operatorname{def}}{=} y + \beta\nabla_{x}\psi_{x_{k}, L_{k}, \tau_{k}}(x),~\text{для некоторого }\beta\in\left[0,~\frac{1}{L_{\psi}}\right],
\end{equation*}
где максимальное значение $\beta$ соответствует максимальной длине шага в методе градиентного спуска для функции с липшицевым градиентом. Получается, квадрат нормы градиента функции оценивает сверху разность значений функции:
\begin{equation*}
    \begin{aligned}
        &\psi_{x_{k}, L_{k}, \tau_{k}}(y)\geq\psi_{x_{k}, L_{k}, \tau_{k}}(x) - \left(\beta + \frac{\beta^{2}L_{\psi}}{2}\right)\left\|\nabla_{x}\psi_{x_{k}, L_{k}, \tau_{k}}(x)\right\|^{2}\Rightarrow\\
        &\Rightarrow\left\|\nabla_{x}\psi_{x_{k}, L_{k}, \tau_{k}}(x)\right\|^{2}\geq\left(\beta + \frac{\beta^{2}L_{\psi}}{2}\right)^{-1}\left(\psi_{x_{k}, L_{k}, \tau_{k}}(x) - \min\limits_{y\in E_{1}}\left(\psi_{x_{k}, L_{k}, \tau_{k}}(y)\right)\right).
    \end{aligned}
\end{equation*}
То есть для $\varepsilon_{k} > 0$ стало удобнее контролировать точность вычисления $x = x_{k + 1}$:
\begin{equation*}
    \begin{aligned}
        &\left\|\nabla_{x}\psi_{x_{k}, L_{k}, \tau_{k}}(x)\right\|^{2}\leq\varepsilon_{k}\left(\beta + \frac{\beta^{2}L_{\psi}}{2}\right)^{-1}\Rightarrow0\leq\psi_{x_{k}, L_{k}, \tau_{k}}(x) - \min\limits_{y\in E_{1}}\left(\psi_{x_{k}, L_{k}, \tau_{k}}(y)\right)\leq\varepsilon_{k}.
    \end{aligned}
\end{equation*}
И если вывод выше позволяет взять наперёд заданное $\beta$, и, вообще говоря, $\varepsilon_{k}$ зависит от $\beta$ в силу определения $x$, то всё равно на практике необходимо оценить $M_{\hat{F}}$ для вычисления $L_{\psi}$, что не всегда просто осуществить. То же верно и для других констант из предположений. Однако зависимость $\varepsilon_{k}$ от $\beta$ не мешает на практике зафиксировать одно из значений $\varepsilon_{k}$, чтобы найденная точка $x_{k + 1}$ была $\varepsilon_{k}$--оптимальной по разности значений $\psi_{x_{k}, L_{k}, \tau_{k}}(x) - \min\limits_{y\in E_{1}}\left(\psi_{x_{k}, L_{k}, \tau_{k}}(y)\right)$, обладая $\operatorname{O}\left(\sqrt{\varepsilon_{k}}\right)$ значением нормы градиента $\left\|\nabla_{x}\psi_{x_{k}, L_{k}, \tau_{k}}(x)\right\|$.

\subsection{Использование произвольной евклидовой нормы}\label{euclidian_general}

Метод нормализованных квадратов можно расширить на произвольные евклидовы нормы:
\begin{equation*}
    \begin{aligned}
        \|x\|_{W_{1}} = &\sqrt{\left\langle W_{1}x,~x\right\rangle},~x\in E_{1},~W_{1}: E_{1}\rightarrow E_{1}^{*}\text{ --- линейный оператор},~W_{1}\succ \mathbf{0}_{\dim(E_{1})\times\dim(E_{1})},~W_{1} = W_{1}^{*};\\
        \|u\|_{W_{2}} = &\sqrt{\left\langle W_{2}u,~u\right\rangle},~u\in E_{2},~W_{2}: E_{2}\rightarrow E_{2}^{*}\text{ --- линейный оператор},~W_{2}\succ \mathbf{0}_{\dim(E_{2})\times\dim(E_{2})},~W_{2} = W_{2}^{*}.
    \end{aligned}
\end{equation*}
Для таких норм общая локальная модель выглядит следующим образом:
\begin{equation*}
    \begin{aligned}
        \psi_{x, L, \tau}(y) &= \frac{\tau}{2} + \frac{1}{2\tau}\left\|\hat{F}(x) + \hat{F}^{'}(x)(y - x)\right\|_{W_{2}}^{2} + \frac{L}{2}\left\|y - x\right\|_{W_{1}}^{2},~(x, y)\in E_{1}^{2},~L = \Omega(L_{\hat{F}}) > 0,~\tau>0.
    \end{aligned}
\end{equation*}
Нижняя граница для $L$ связана с необходимостью выполнения всех предположений для вывода формулы модели $\psi_{x, L, \tau}(y)$. Соответственно, для данной модели точка минимума имеет явное выражение:
\begin{equation*}
    \begin{aligned}
        &y = x - \left(\hat{F}^{'}(x)^{*}W_{2}\hat{F}^{'}(x) + \tau LW_{1}\right)^{-1}\hat{F}^{'}(x)^{*}W_{2}\hat{F}(x).
    \end{aligned}
\end{equation*}

Представленный в работе метод Гаусса--Ньютона применим также для невязок, основанных не на евклидовых нормах, и наиболее удобная форма локальной модели для неевклидовых норм соответствует $\tau = \phi(x, y)$ \cite{Nesterov2007}:
\begin{equation*}
    \begin{aligned}
        &\psi_{x, L}(y) \overset{\operatorname{def}}{=} \underbrace{\left\|\hat{F}(x) + \hat{F}^{'}(x)(y - x)\right\|_{\kappa}}_{\text{неевклидова норма}} + \frac{L}{2}\underbrace{\left\|y - x\right\|^{2}}_{\substack{\text{евклидова}\\\text{норма}}},~L = \Omega\left(L_{\kappa}\right) > 0.
    \end{aligned}
\end{equation*}
Стоит заметить, что в модели $\psi_{x, L}(y)$ необходима липшицевость матрицы Якоби в терминах $\kappa$--нормы и евклидовой нормы, чтобы установить по лемме \ref{lm:aux_det_upper_model} формулу модели $\psi_{x, L}(y)$:
$$\exists L_{\kappa} > 0:~\left\|\hat{F}^{'}(y) - \hat{F}^{'}(x)\right\|_{\kappa}\leq L_{\kappa}\|y - x\|,~\forall (x, y)\in\mathcal{F}^{2}.$$
При этом нижняя граница значений $L$ в модели $\psi_{x, L}(y)$ отмасштабирована относительно $L_{\kappa}$, так как требуется согласование $\kappa$--нормы и евклидовой нормы для выполнения свойства субмультипликативности норм, использованного в доказательстве леммы \ref{lm:aux_det_upper_model}. Для локальной модели $\psi_{x, L}(y)$ сам минимум и его точка определяются через решение следующей задачи оптимизации:
\begin{equation*}
    \begin{aligned}
        \min\limits_{y\in E_{1}}&\left\{\left\|\hat{F}(x) + \hat{F}^{'}(x)(y - x)\right\|_{\kappa} + \frac{L}{2}\left\|y - x\right\|^{2}\right\} = \min\limits_{y\in E_{1}}\max\limits_{\substack{u\in E_{2}^{*}\\\left\|u\right\|_{\kappa^{*}}\leq1}}\left\{\left\langle u,~\hat{F}(x) + \hat{F}^{'}(x)(y - x)\right\rangle + \frac{L}{2}\left\|y - x\right\|^{2}\right\} =\\
        &= \left\{h := y - x\right\} = \max\limits_{\substack{u\in E_{2}^{*}\\\left\|u\right\|_{\kappa^{*}}\leq1}}\min\limits_{h\in E_{1}}\left\{\left\langle u,~\hat{F}(x) + \hat{F}^{'}(x)h\right\rangle + \frac{L}{2}\left\|h\right\|^{2}\right\} =\\
        &= \max\limits_{\substack{u\in E_{2}^{*}}}\left\{\left\langle u,~\hat{F}(x)\right\rangle - \frac{1}{2L}\left\|\hat{F}^{'}(x)^{*}u\right\|^{2}:~\left\|u\right\|_{\kappa^{*}}\leq1\right\},
    \end{aligned}
\end{equation*}
где $\|\cdot\|_{\kappa^{*}}$ --- норма, сопряжённая к норме $\|\cdot\|_{\kappa}$, сама задача поиска $u\in E_{2}^{*}$ является задачей максимизации квадратичного функционала с простым ограничением, эффективно решаемой стандартными методами выпуклой оптимизации. Точка минимума модели $\psi_{x, L}(y)$ явно выражается через решение задачи поиска $u\in E_{2}^{*}$:
\begin{equation*}
    \begin{cases}
        y = x - \frac{1}{L}\hat{F}^{'}(x)^{*}u_{\operatorname{opt}};\\
        u_{\operatorname{opt}} = \argmax\limits_{u\in E_{2}^{*}}\left\{\left\langle u,~\hat{F}(x)\right\rangle - \frac{1}{2L}\left\|\hat{F}^{'}(x)^{*}u\right\|^{2}:~\left\|u\right\|_{\kappa^{*}}\leq1\right\}.
    \end{cases}
\end{equation*}
Локальная модель $\psi_{x, L}(y)$ имеет аналогичное стохастическое расширение, получаемое заменой отображения $\hat{F}$ на $\hat{G}$ в приведённых выкладках, связанных с моделью $\psi_{x, L}(y)$. Более того, представленные в данной работе утверждения можно переформулировать относительно неевклидовых норм согласно принципам, изложенным в текущем подразделе. А сам грамотный подбор нормированных пространств позволяет улучшить коэффициенты в оценке сходимости для ускорения поиска решения поставленной задачи \eqref{eq:main_opt_problem} уже на этапе вывода структурных свойств задачи.

\section{Модифицированный метод Гаусса--Ньютона в классе квазиньютоновских методов}\label{sec:comparsion}

Используя при решении задачи минимизации нормы невязки информацию о первых производных, представленный в этой работе метод Гаусса--Ньютона принадлежит классу квазиньютоновских методов, аппроксимирующих на каждом шаге метода шаг метода Ньютона, об этом нам явно говорят формы вычисления приближений решения, использующие точное вычисление точки минимума общей локальной модели:
\begin{equation*}
    \begin{aligned}
        &x_{k + 1} = x_{k} - \left(\hat{F}^{'}(x_{k})^{*}\hat{F}^{'}(x_{k}) + \tau_{k}L_{k}I_{n}\right)^{-1}\underbrace{\hat{F}^{'}(x_{k})^{*}\hat{F}(x_{k})}_{ = \frac{1}{2}\nabla\hat{f}_{2}(x_{k})},~k\in\mathbb{Z}_{+}.
    \end{aligned}
\end{equation*}
При этом матрицу $\hat{F}^{'}(x_{k})^{*}\hat{F}^{'}(x_{k}) + \tau_{k}L_{k}I_{n}$ можно рассматривать как приближение гессиана функции $\hat{f}_{2}$, однако лучше всего данная матрица приближает гессиан $\hat{f}_{2}$ в областях выпуклости функции $\hat{f}_{2}$, возникновение которых типично около точки минимума. Для большей наглядности связи рассматриваемого метода с методом Ньютона рассмотрим следующую верхнюю оценку локальной модели $\psi_{x, L, \tau}(y)$:
\begin{equation*}
    \begin{aligned}
        &\psi_{x, L, \tau}^{W_{1}}(y) \overset{\operatorname{def}}{=} \frac{\tau}{2} + \frac{1}{2\tau}\left\|\hat{F}(x) + \hat{F}^{'}(x)(y - x)\right\|^{2} + \frac{L}{2}\|y - x\|^{2}_{W_{1}},~W_{1}\succeq I_{n},~W_{1} = W_{1}^{*}.
    \end{aligned}
\end{equation*}
Для модели $\psi_{x, L, \tau}^{W_{1}}(y)$ точка минимума имеет следующее выражение:
\begin{equation}\label{eq:modified_gauss_newton_step}
    \begin{aligned}
        &y = x - \left(\hat{F}^{'}(x)^{*}\hat{F}^{'}(x) + \tau L W_{1}\right)^{-1}\hat{F}^{'}(x)^{*}\hat{F}(x).
    \end{aligned}
\end{equation}
И если при решении задачи \eqref{eq:main_opt_problem} использовать модель $\psi_{x_{k}, L_{k}, \tau_{k}}^{W_{1}^{k}}(x_{k + 1})$, то для строго (сильно) выпуклой функции $\hat{f}_{2}$ можно шаг модифицированного метода Гаусса--Ньютона свести к шагу метода Ньютона:
\begin{equation*}
    \begin{aligned}
        &x_{k + 1} = x_{k} - \underbrace{\left(\hat{F}^{'}(x_{k})^{*}\hat{F}^{'}(x_{k}) + \sum\limits_{i = 1}^{m}\hat{F}_{i}(x_{k})\nabla^{2}\hat{F}_{i}(x_{k})\right)^{-1}}_{=\left(\frac{1}{2}\nabla^{2}\hat{f}_{2}(x_{k})\right)^{-1}}\underbrace{\hat{F}^{'}(x_{k})^{*}\hat{F}(x_{k})}_{=\frac{1}{2}\nabla\hat{f}_{2}(x_{k})},~k\in\mathbb{Z}_{+},\\
        &\tau_{k}L_{k} = \lambda_{\min}\left(\sum\limits_{i = 1}^{m}\hat{F}_{i}(x_{k})\nabla^{2}\hat{F}_{i}(x_{k})\right) > 0,~W_{1}^{k} = \frac{1}{\tau_{k}L_{k}}\sum\limits_{i = 1}^{m}\hat{F}_{i}(x_{k})\nabla^{2}\hat{F}_{i}(x_{k})\succeq I_{n}.
    \end{aligned}
\end{equation*}
При этом выпуклость не обязательна для каждой функции $\hat{F}_{i},~i\in\overline{1, m}$, этот факт скорее интересен с точки зрения теории, так как свойство выпуклости для функции $\hat{f}_{2}$ является сильно ограничивающим, а вычисление модели $\psi_{x, L, \tau}^{W_{1}}(y)$ более трудоёмкое, нежели чем $\psi_{x, L, \tau}(y)$. Приведённые рассуждения верны и для стохастической версии метода Гаусса--Ньютона с локальной моделью $\hat{\psi}_{x, L, \tau}(y, B)$.

Для предложенного нестохастического метода Гаусса--Ньютона с точным проксимальным отображением в данной работе установлена  локальная суперлинейная сходимость, и в худшем случае по теореме \ref{th:glob_sub_lin_and_lin_conv} для достижения области суперлинейной сходимости необходимо затратить $\left\lceil\frac{16L_{\hat{F}}}{\mu}\hat{f}_{1}(x_{0})\right\rceil$ итераций. Введём следующую характеристику близости к решению задачи \eqref{eq:smooth_system} в случае разрешимости системы уравнений, называемую радиусом множества уровня $\hat{f}_{1}(x_{0})$:
$$\operatorname{R} = \min\limits_{x\in E_{1}}\left\{\|x_{0} - x\|:~x\in\mathcal{L}(\hat{f}_{1}(x_{0})),~F(x) = \mathbf{0}_{m}\right\}.$$
По лемме \ref{lm:aux_det_psi_upper_bound} значение $\hat{f}_{1}(x_{1})$ ограничено для метода в теореме \ref{th:glob_sub_lin_and_lin_conv}:
\begin{equation*}
    \begin{aligned}
        \hat{f}_{1}(x_{1})&\leq\frac{\hat{f}_{1}(x_{0})}{2} + \frac{L_{0}\|x_{0} - x^{*}\|^{2}}{2} + \frac{L_{\hat{F}}^{2}\|x_{0} - x^{*}\|^{4}}{8\hat{f}_{1}(x_{0})}\leq\frac{\hat{f}_{1}(x_{0})}{2} + L_{\hat{F}}\operatorname{R}^{2} + \frac{L_{\hat{F}}^{2}\operatorname{R}^{4}}{8\hat{f}_{1}(x_{0})}.
    \end{aligned}
\end{equation*}
Также есть другая оценка значения $\hat{f}_{1}(x_{1})$ по лемме \ref{lm:aux_tau_schedule}:
\begin{equation*}
    \begin{aligned}
        &\hat{f}_{1}(x_{1})\leq\hat{f}_{1}(x_{0})\leq M_{\hat{F}}\left\|x_{0} - x^{*}\right\| + \frac{L_{\hat{F}}}{2}\left\|x_{0} - x^{*}\right\|^{2} = M_{\hat{F}}\operatorname{R} + \frac{L_{\hat{F}}\operatorname{R}^{2}}{2}.
    \end{aligned}
\end{equation*}
Объединим обе оценки:
\begin{equation*}
    \begin{aligned}
        &\hat{f}_{1}(x_{1})\leq\min\left\{\frac{\hat{f}_{1}(x_{0})}{2} + L_{\hat{F}}\operatorname{R}^{2} + \frac{L_{\hat{F}}^{2}\operatorname{R}^{4}}{8\hat{f}_{1}(x_{0})},~M_{\hat{F}}\operatorname{R} + \frac{L_{\hat{F}}\operatorname{R}^{2}}{2}\right\}.
    \end{aligned}
\end{equation*}
То есть необходимое количество итераций для достижения области суперлинейной сходимости ограничено сверху:
\begin{equation}\label{eq:upper_quad_iter_bound_gnm}
    \begin{aligned}
        &\left\lceil1 + \frac{16L_{\hat{F}}}{\mu}\min\left\{\frac{\hat{f}_{1}(x_{0})}{2} + L_{\hat{F}}\operatorname{R}^{2} + \frac{L_{\hat{F}}^{2}\operatorname{R}^{4}}{8\hat{f}_{1}(x_{0})},~M_{\hat{F}}\operatorname{R} + \frac{L_{\hat{F}}\operatorname{R}^{2}}{2}\right\}\right\rceil.
    \end{aligned}
\end{equation}
Представленная оценка для задач в рамках предположений \ref{as:det_hat_F_der_smooth} и \ref{as:det_hat_F_PL_condition} является улучшаемой и в следующем случае на примере метода Ньютона с кубической регуляризацией \cite{Nesterov2006} имеется улучшение. Рассмотрим задачу выпуклой оптимизации без ограничений:
\begin{equation*}
    \begin{aligned}
        \min\limits_{x\in E_{1}}f(x),
    \end{aligned}
\end{equation*}
$f(\cdot)$ --- дважды дифференцируемая и сильно выпуклая функция с липшицевым гессианом. А так как в данном примере можно ввести $\nabla f(x) \equiv \hat{F}(x)$ по свойству эквивалентности решения системы уравнений через оптимизацию нормы невязки и решения задачи оптимизации без ограничений через решение системы уравнений, описывающей условия оптимальности первого порядка, то формально липшицевость и сильную выпуклость можно описать следующим образом \cite{NesterovLectures}:
\begin{equation*}
    \begin{aligned}
        &\left\|\nabla^{2}f(y) - \nabla^{2}f(x)\right\|\leq L_{\hat{F}}\left\|y - x\right\|,~\forall (x, y)\in E_{1}^{2};\\
        &\left\langle\nabla^{2}f(x)y, y\right\rangle\geq\sqrt{\mu}\left\|y\right\|^{2},~\forall (x, y)\in E_{1}^{2}.
    \end{aligned}
\end{equation*}
Тогда для метода Ньютона с кубической регуляризацией количество необходимых итераций для достижения области квадратичной сходимости ограничено сверху величиной $\left\lceil6.25\sqrt{\frac{L_{\hat{F}}\operatorname{R}}{\sqrt{\mu}}}~\right\rceil$ \cite{Nesterov2006, NesterovLectures}. Данная оценка улучшает представленную выше оценку для метода Гаусса--Ньютона \eqref{eq:upper_quad_iter_bound_gnm}, что как раз подтверждает часто применимый на практике тезис: специализированный метод справляется не хуже общего. Однако представленные оценки являются верхними, а для метода Гаусса--Ньютона в классе задач \eqref{eq:main_opt_problem}, порождённом предположениями \ref{as:det_hat_F_der_smooth} и \ref{as:det_hat_F_PL_condition}, нижние \textit{оценки сложности} решения ещё не получены. И, по сравнению с методом Ньютона с кубической регуляризацей, представленные вариации метода Гаусса--Ньютона корректны для произвольных нормированных пространств, а не только в случае евклидова пространства.

Разработанные для евклидовых норм стратегии вычисления приближения решения \eqref{eq:det_expl_update_rule}, \eqref{eq:stoch_direct_update_rule}, \eqref{eq:double_stoch_direct_update_rule} по форме напоминают обновление параметров в алгоритме Левенберга--Марквардта \cite{Levenberg1944, Marquardt1963, More1978}. Однако, в отличие от алгоритма Левенберга--Марквардта, изложенный метод Гаусса--Ньютона автоматически обобщается на случай неевклидовых произвольных норм, и в этих условиях метод всё так же однозначно интерпретируется, позволяя устанавливать глобальные и локальные свойства процесса построения $\left\{x_{k}\right\}_{k\in\mathbb{Z}_{+}}$.

\section*{Благодарности}

Мы выражаем особую благодарность Юрию Евгеньевичу Нестерову за постановку проблемы, благодаря которой появилась на свет данная работа. Также выражаем огромную благодарность Дмитрию Камзолову и Павлу Двуреченскому за ценные замечания, сделанные на разных этапах написания статьи. Наша работа приурочена к 65--летнему юбилею Юрия Евгеньевича Нестерова.

\section*{{\,\,\,\,\,\,\,\,\,\,\,\,\,\,\,\,\,\,\,\,\,\,\,\,\,\,\,\,\,\,\,\,\,\,\,\,\,\,\,\,\,\,\,\,\,\,\,\,\,\,\,\,\,\,\,\,\,\,\,\,\,\,\,\,\,\,\,\,Приложение} }\label{sec:appendix}

\subsection*{Модифицированный метод Гаусса--Ньютона}

\subsubsection*{Вспомогательные утверждения}

В лемме \ref{lm:aux_det_upper_model} выводится формула локальной модели для оптимизируемого функционала в задаче \eqref{eq:main_opt_problem}.
\begin{lemma}[\cite{Nesterov2020}]\label{lm:aux_det_upper_model}
    Пусть $(x, y)\in\mathcal{F}^{2}, L\geq L_{\hat{F}}, \tau > 0$ и выполнено предположение \ref{as:det_hat_F_der_smooth}. Тогда $\hat{f}_{1}(y)\leq\psi_{x, L, \tau}(y).$
\end{lemma}
\begin{proof}
    Выведем неравенство для $\left\|\hat{F}(y) - \hat{F}(x) - \hat{F}^{'}(x)(y - x)\right\|$:
    \begin{equation}\label{eq:lm_aux_det_model_bound}
        \begin{aligned}
            &\left\|\hat{F}(y) - \hat{F}(x) - \hat{F}^{'}(x)(y - x)\right\| = \left\{\hat{F}(y) = \hat{F}(x) + \int\limits_{0}^{1}\hat{F}^{'}(x + t(y - x))(y - x)\operatorname{d}t\right\} = \\
            &=\left\|\int\limits_{0}^{1}\left(\hat{F}^{'}(x + t(y - x)) - \hat{F}^{'}(x)\right)(y - x)\operatorname{d}t\right\|\leq\left\{\text{$\|\cdot\|$ --- выпукла, неравенство Йенсена}\right\}\leq\\
            &\leq\int\limits_{0}^{1}\left\|\left(\hat{F}^{'}(x + t(y - x)) - \hat{F}^{'}(x)\right)(y - x)\right\|\operatorname{d}t\leq\int\limits_{0}^{1}\left\|\hat{F}^{'}(x + t(y - x)) - \hat{F}^{'}(x)\right\|\left\|y - x\right\|\operatorname{d}t\leq\\
            &\leq\left\{\text{предположение \ref{as:det_hat_F_der_smooth}}\right\}\leq\int\limits_{0}^{1}L_{\hat{F}}\|y - x\|^{2}t\operatorname{d}t = \frac{L_{\hat{F}}}{2}\|y - x\|^{2}.
        \end{aligned}
    \end{equation}
    Рассмотрим вспомогательное неравенство:
    \begin{equation}\label{eq:lm_aux_det_model_square}
        \begin{aligned}
            \left(\sqrt{\frac{\tau}{2}} - \frac{1}{\sqrt{2\tau}}\left\|\hat{F}(x) + \hat{F}^{'}(x)(y - x)\right\|\right)^{2} &= \frac{\tau}{2} + \frac{1}{2\tau}\left\|\hat{F}(x) + \hat{F}^{'}(x)(y - x)\right\|^{2} - \left\|\hat{F}(x) + \hat{F}^{'}(x)(y - x)\right\|\geq 0\Rightarrow\\
            &\Rightarrow\frac{\tau}{2} + \frac{1}{2\tau}\left\|\hat{F}(x) + \hat{F}^{'}(x)(y - x)\right\|^{2}\geq\left\|\hat{F}(x) + \hat{F}^{'}(x)(y - x)\right\|.
        \end{aligned}
    \end{equation}
    Тогда для $\hat{f}_{1}$ выполнено
    \begin{equation*}
        \begin{aligned}
            \hat{f}_{1}(y) &= \left\|\hat{F}(y)\right\| = \left\|\hat{F}(y) - \hat{F}(x) - \hat{F}^{'}(x)(y - x) + \hat{F}(x) + \hat{F}^{'}(x)(y - x)\right\|\leq\left\|\hat{F}(y) - \hat{F}(x) - \hat{F}^{'}(x)(y - x)\right\| +\\
            &+ \left\|\hat{F}(x) + \hat{F}^{'}(x)(y - x)\right\|\leq\left\{\text{неравенство из \eqref{eq:lm_aux_det_model_bound}}\right\}\leq\frac{L_{\hat{F}}}{2}\|y - x\|^{2} + \left\|\hat{F}(x) + \hat{F}^{'}(x)(y - x)\right\|\leq\\
            &\leq\left\{\text{неравенство из \eqref{eq:lm_aux_det_model_square}}\right\}\leq\frac{\tau}{2} + \frac{L_{\hat{F}}}{2}\|y - x\|^{2} + \frac{1}{2\tau}\left\|\hat{F}(x) + \hat{F}^{'}(x)(y - x)\right\|^{2} = \psi_{x,L_{\hat{F}},\tau}(y)\leq \psi_{x,L,\tau}(y).
        \end{aligned}
    \end{equation*}
\end{proof}
\begin{lm:corollary}
При $\tau = \phi(x, y)$ зазор между $\hat{f}_{1}(y)$ и $\psi_{x, L, \tau}(y)$ минимален, согласно неравенству из \eqref{eq:lm_aux_det_model_square}.
\end{lm:corollary}

Следующая лемма задаёт общую формулу для измерения убывания значения оптимизируемого функционала в \eqref{eq:main_opt_problem} при минимизации локальной модели. Величина убывания оценивается с помощью нормы проксимального градиента, в следствиях указаны основные свойства введённого проксимального отображения.

\begin{lemma}\label{lm:aux_det_local_decrease_1}
    Пусть выполнено предположение \ref{as:det_hat_F_der_smooth} и $x\in\mathcal{F},~T_{L, \tau}(x)\in\mathcal{F}$, $\tau > 0,~L\geq L_{\hat{F}}$. Тогда выполняется соотношение $$\frac{\tau}{2} + \frac{\hat{f}_{2}(x)}{2\tau} - \hat{f}_{1}(T_{L, \tau}(x))\geq\frac{L}{2}\left\|T_{L, \tau}(x) - x\right\|^{2}.$$
\end{lemma}
\begin{proof}
    Рассмотрим функцию $$h(t) = \min\limits_{y\in E_{1}}\left\{\frac{\tau}{2} + \frac{1}{2\tau}\left\|\hat{F}(x) + \hat{F}^{'}(x)(y - x)\right\|^{2} + \frac{1}{2t}\|y - x\|^{2}\right\}.$$
    Находящаяся под $\min$ локальная модель $\psi_{x, t^{-1}, \tau}(y)$ выпукла по $(\tau, y, t)$ на выпуклом множестве
    $$\left\{(y, \tau, t, \alpha)\in E_{1}\times\mathbb{R}_{+}^{3}: \|y - x\|^{2}\leq\alpha t\right\}.$$
    То есть у функции $h(t)$ выпуклый надграфик, так как он получен с помощью проектирования выпуклого множества, что влечёт выпуклость $h(t)$ (Theorem 3.1.7, \cite{NesterovLectures}). Для выпуклой функции верно следующее представление:
    \begin{equation*}
        \begin{aligned}
            h(0)&\geq h(t) + h^{'}(t)(0 - t) = h(t) - h^{'}(t)t;\\
            h^{'}(t) &= \left\langle\underbrace{\frac{1}{\tau}\hat{F}^{'}(x)^{*}\left(\hat{F}(x) + \hat{F}^{'}(x)\left(T_{t^{-1}, \tau}(x) - x\right)\right) + \frac{1}{t}\left(T_{t^{-1}, \tau}(x) - x\right)}\limits_{=\nabla_{y}\psi_{x, t^{-1}, \tau}(y) = \mathbf{0}_{n}\text{ из--за взятия минимума по }y}, \frac{\partial T_{t^{-1}, \tau}(x)}{\partial t}\right\rangle - \frac{1}{2t^{2}}\left\|T_{t^{-1}, \tau}(x) - x\right\|^{2} =\\
            &= -\frac{1}{2t^{2}}\left\|T_{t^{-1}, \tau}(x) - x\right\|^{2}.
        \end{aligned}
    \end{equation*}
    По свойству проксимального отображения $\lim\limits_{t\rightarrow 0}\argmin\limits_{y\in E_{1}}\left\{\psi_{x, t^{-1}, \tau}(y)\right\} = x\Rightarrow h(0) = \frac{\tau}{2} + \frac{\left\|\hat{F}(x)\right\|^{2}}{2\tau} = \frac{\tau}{2} + \frac{\hat{f}_{2}(x)}{2\tau}$. Значит,
    \begin{equation*}
        \begin{aligned}
            \frac{\tau}{2} + \frac{\hat{f}_{2}(x)}{2\tau}&\geq\psi_{x, t^{-1}, \tau}(T_{t^{-1}, \tau}(x)) + \frac{1}{2t}\left\|T_{t^{-1}, \tau}(x) - x\right\|^{2}\geq\left\{\text{лемма \ref{lm:aux_det_upper_model}}\right\}\geq\hat{f}_{1}(T_{t^{-1}, \tau}(x)) + \frac{1}{2t}\left\|T_{t^{-1}, \tau}(x) - x\right\|^{2}\Rightarrow\\
            &\Rightarrow\left\{t^{-1} = L\right\}\Rightarrow\frac{\tau}{2} + \frac{\hat{f}_{2}(x)}{2\tau} - \hat{f}_{1}(T_{L, \tau}(x))\geq\frac{L}{2}\left\|T_{L, \tau}(x) - x\right\|^{2}.
        \end{aligned}
    \end{equation*}
\end{proof}
\begin{lm:corollary}
Результаты утверждения содержат в себе неравенство
$$\frac{\tau}{2} + \frac{\hat{f}_{2}(x)}{2\tau} - \psi_{x, L, \tau}(T_{L, \tau}(x))\geq\frac{L}{2}\left\|T_{L, \tau}(x) - x\right\|^{2},$$
которое верно для $L > 0$ и $x\in E_{1}$.
\end{lm:corollary}
\begin{lm:corollary}\label{lm:cor_aux_det_local_decrease_1}
$T_{L, \tau}(x) = \argmin\limits_{y\in E_{1}}\left\{\psi_{x, L, \tau}(y)\right\}$ имеет явное выражение при $L > 0$:
$$T_{L, \tau}(x) = x - \left(\hat{F}^{'}(x)^{*}\hat{F}^{'}(x) + \tau L I_{n}\right)^{-1}\hat{F}^{'}(x)^{*}\hat{F}(x).$$
Поэтому $\lim\limits_{L\rightarrow+\infty}T_{L, \tau}(x) = x$ и
\begin{equation*}
    \begin{aligned}
        \frac{\tau}{2} + \frac{\hat{f}_{2}(x)}{2\tau} - \hat{f}_{1}(x)&\geq\frac{1}{2}\lim\limits_{L\rightarrow+\infty}\left(L\left\|T_{L, \tau}(x) - x\right\|^{2}\right) = \frac{1}{2}\lim\limits_{L\rightarrow+\infty}\left(L\left\|\left(\hat{F}^{'}(x)^{*}\hat{F}^{'}(x) + \tau L I_{n}\right)^{-1}\hat{F}^{'}(x)^{*}\hat{F}(x)\right\|^{2}\right) =\\
        &= \frac{1}{2}\lim\limits_{L\rightarrow+\infty}\left\|\left(\frac{1}{\sqrt{L}}\hat{F}^{'}(x)^{*}\hat{F}^{'}(x) + \tau \sqrt{L} I_{n}\right)^{-1}\hat{F}^{'}(x)^{*}\hat{F}(x)\right\|^{2} = 0.
    \end{aligned}
\end{equation*}
Однако величина $\left\|L\left(T_{L,\tau}(x) - x\right)\right\|$ сходится к норме градиента $\psi_{x, L, \tau}(y)$ по $y$ в точке $y = x$ при $L\rightarrow+\infty$:
\begin{equation*}
    \begin{aligned}
        \lim\limits_{L\rightarrow+\infty}\left\|L\left(T_{L, \tau}(x) - x\right)\right\| &= \lim\limits_{L\rightarrow+\infty}\left\|L\left(\hat{F}^{'}(x)^{*}\hat{F}^{'}(x) + \tau L I_{n}\right)^{-1}\hat{F}^{'}(x)^{*}\hat{F}(x)\right\| =\\
        &= \lim\limits_{L\rightarrow+\infty}\left\|\left(\frac{1}{L}\hat{F}^{'}(x)^{*}\hat{F}^{'}(x) + \tau I_{n}\right)^{-1}\hat{F}^{'}(x)^{*}\hat{F}(x)\right\| = \left\|\frac{1}{\tau}\hat{F}^{'}(x)^{*}\hat{F}(x)\right\|.
    \end{aligned}
\end{equation*}
Сама функция $\left\|T_{L, \tau}(x) - x\right\|^{2}$ является монотонно убывающей по $L$ и по $\tau$.
\end{lm:corollary}
\begin{lm:corollary}\label{lm:cor_2_aux_det_local_decrease_1}
При выборе $\tau = \hat{f}_{1}(x) > 0$ из полученной оценки следует, что если $x\in\mathcal{L}(\hat{f}_{1}(x))\subseteq\mathcal{F}$, то и $T_{L, \hat{f}_{1}(x)}(x)\in\mathcal{L}(\hat{f}_{1}(x))$:
\begin{equation*}
    \begin{aligned}
        \frac{\tau}{2} + \frac{\hat{f}_{2}(x)}{2\tau} - \hat{f}_{1}(T_{L, \tau}(x))&\geq\frac{L}{2}\left\|T_{L, \tau}(x) - x\right\|^{2}\Rightarrow\left\{\tau = \hat{f}_{1}(x)\right\}\Rightarrow\hat{f}_{1}(x) - \hat{f}_{1}(T_{L, \hat{f}_{1}(x)}(x))\geq\\
        &\geq\frac{L}{2}\left\|T_{L, \hat{f}_{1}(x)}(x) - x\right\|^{2}\geq 0\Rightarrow\hat{f}_{1}(x)\geq\hat{f}_{1}(T_{L, \hat{f}_{1}(x)}(x))\Rightarrow\\
        &\Rightarrow T_{L, \hat{f}_{1}(x)}(x)\in\mathcal{L}(\hat{f}_{1}(T_{L, \hat{f}_{1}(x)}(x)))\subseteq\mathcal{L}(\hat{f}_{1}(x)).
    \end{aligned}
\end{equation*}
\end{lm:corollary}

В лемме ниже оценивается убывание оптимизируемого функционала $\hat{f}_{1}$ при решении задачи \eqref{eq:main_opt_problem} уже с помощью приращения локальной модели $\Delta_{r}(x)$, вводится полезная для дальнейшего анализа вспомогательная функция $\varkappa(\cdot)$.

\begin{lemma}\label{lm:aux_det_local_decrease_2}
    Пусть выполнено предположение  \ref{as:det_hat_F_der_smooth} и $x\in\mathcal{F},~T_{L, \tau}(x)\in\mathcal{F}$, $\tau > 0,~L\geq L_{\hat{F}}$. Тогда для любого $r > 0$ выполняется соотношение
    $$\frac{\tau}{2} + \frac{\hat{f}_{2}(x)}{2\tau} - \hat{f}_{1}(T_{L, \tau}(x))\geq Lr^{2}\varkappa\left(\frac{\Delta_{r}(x)}{2\tau Lr^{2}}\right),$$
    где
    \begin{equation*}
        \begin{cases}
            \Delta_{r}(x) \overset{\operatorname{def}}{=} \hat{f}_{2}(x) - \min\limits_{y\in E_{1}}\left\{\left(\phi(x, y)\right)^{2}: \|y - x\|\leq r\right\};\\
            \varkappa(t) \overset{\operatorname{def}}{=} \begin{sqcases}
                \frac{t^{2}}{2},~t\in[0,~1];\\
                t - \frac{1}{2},~t > 1.
            \end{sqcases}
        \end{cases}
    \end{equation*}
\end{lemma}
\begin{proof}
Введём $h_{r} = \argmin\limits_{h\in E_{1}}\left\{\left(\phi(x, x + h)\right)^{2}: \|h\|\leq r\right\}$. Распишем локальную модель в точке $T_{L, \tau}(x)$:
\begin{equation*}
    \begin{aligned}
        &\hat{f}_{1}(T_{L, \tau}(x))\leq\left\{\text{лемма \ref{lm:aux_det_upper_model}}\right\}\leq\psi_{x, L, \tau}(T_{L, \tau}(x))\leq\\
        &\leq\min\limits_{t\in[0,~1]}\left\{\frac{\tau}{2} + \frac{1}{2\tau}\left\|\hat{F}(x) + t\hat{F}^{'}(x)h_{r}\right\|^{2} + \frac{L}{2}(tr)^{2}\right\}= \frac{\tau}{2} +\\
        &+ \min\limits_{t\in[0,~1]}\left\{\frac{1}{2\tau}\left\|(1 - t)\hat{F}(x) + t\left(\hat{F}(x) + \hat{F}^{'}(x)h_{r}\right)\right\|^{2} + \frac{L}{2}(tr)^{2}\right\}\leq\left\{\|\cdot\|^{2}\text{ выпукла}\right\}\leq\frac{\tau}{2} +\\
        &+ \min\limits_{t\in[0,~1]}\left\{\frac{(1 - t)}{2\tau}\hat{f}_{2}(x) + \frac{t}{2\tau}\left(\phi(x, x + h_{r})\right)^{2} + \frac{L}{2}(tr)^{2}\right\}=\frac{\tau}{2} + \frac{\hat{f}_{2}(x)}{2\tau} + \min\limits_{t\in[0,~1]}\left\{\frac{-t}{2\tau}\Delta_{r}(x) + \frac{L}{2}(tr)^{2}\right\}\Rightarrow\\
        &\Rightarrow\frac{\tau}{2} + \frac{\hat{f}_{2}(x)}{2\tau} - \hat{f}_{1}(T_{L,\tau}(x))\geq Lr^{2}\max\limits_{t\in[0,~1]}\left\{\frac{\Delta_{r}(x)}{2\tau Lr^{2}}t - \frac{1}{2}t^{2}\right\}.
    \end{aligned}
\end{equation*}
Выражение в правой части получившегося неравенства представляет собой полином второй степени с отрицательным коэффициентом у старшей степени и с корнями $t\in\left\{0,~\frac{\Delta_{r}(x)}{\tau Lr^{2}}\right\}$, что означает для выражения точки условного максимума $t^{*}$ необходимость рассмотреть два случая: $\frac{\Delta_{r}(x)}{2r^{2}\tau L}\leq 1$ и $\frac{\Delta_{r}(x)}{2r^{2}\tau L} > 1$. В первом случае $t^{*} = \frac{\Delta_{r}(x)}{2\tau Lr^{2}}$, во втором $t^{*} = 1$. Полученная оценка имеет следующее представление:
\begin{equation}\label{eq:aux_det_local_decrease_2_eq_1}
    \frac{\tau}{2} + \frac{\hat{f}_{2}(x)}{2\tau} - \hat{f}_{1}(T_{L,\tau}(x))\geq Lr^{2}\cdot\begin{sqcases}
        \frac{1}{2}\left(\frac{\Delta_{r}(x)}{2\tau Lr^{2}}\right)^{2},~\text{при }t^{*} = \frac{\Delta_{r}(x)}{2\tau Lr^{2}};\\[10pt]
        \frac{\Delta_{r}(x)}{2\tau Lr^{2}} - \frac{1}{2},~\text{при }t^{*} = 1.
    \end{sqcases}
\end{equation}
Введём функцию $\varkappa(t) = \begin{sqcases}
    \frac{t^{2}}{2},~t\in[0, 1];\\
    t - \frac{1}{2},~t > 1.
\end{sqcases}$ Перепишем с её помощью оценку \eqref{eq:aux_det_local_decrease_2_eq_1}:
$$\frac{\tau}{2} + \frac{\hat{f}_{2}(x)}{2\tau} - \hat{f}_{1}(T_{L, \tau}(x))\geq Lr^{2}\varkappa\left(\frac{\Delta_{r}(x)}{2\tau Lr^{2}}\right).$$
В этой оценке $\hat{f}_{2}(x)\geq\Delta_{\infty}(x)\geq\Delta_{r}(x)\geq\Delta_{0}(x) = 0$ и $\varkappa(t)\geq 0$ по построению.
\end{proof}
\begin{lm:corollary}
Результаты утверждения содержат в себе неравенство
$$\frac{\tau}{2} + \frac{\hat{f}_{2}(x)}{2\tau} - \psi_{x, L, \tau}(T_{L, \tau}(x))\geq Lr^{2}\varkappa\left(\frac{\Delta_{r}(x)}{2\tau Lr^{2}}\right),$$
которое верно для $L > 0$ и $x\in E_{1}$. Более того, для достаточно малых значений $Lr^{2}$, таких, что $\frac{\Delta_{r}(x)}{2\tau Lr^{2}}\geq 1$, оценка следующая:
$$\frac{\tau}{2} + \frac{\hat{f}_{2}(x)}{2\tau} - \psi_{x, L, \tau}(T_{L, \tau}(x))\geq\frac{\Delta_{r}(x)}{2\tau} - \frac{Lr^{2}}{2}.$$
Для больших значений $Lr^{2}$, при которых $\frac{\Delta_{r}(x)}{2\tau Lr^{2}}\leq 1$, эта оценка имеет другую форму:
$$\frac{\tau}{2} + \frac{\hat{f}_{2}(x)}{2\tau} - \psi_{x, L, \tau}(T_{L, \tau}(x))\geq\frac{\left(\Delta_{r}(x)\right)^{2}}{8\tau^{2} Lr^{2}}.$$
\end{lm:corollary}
\begin{lm:corollary}\label{lm:cor_aux_det_local_decrease_2}
Для достаточно больших значений $Lr^{2}$, таких, что $\frac{\Delta_{r}(x)}{2\tau Lr^{2}}\leq 1$, выведенная оценка упрощается:
$$\frac{\tau}{2} + \frac{\hat{f}_{2}(x)}{2\tau} - \hat{f}_{1}(T_{L, \tau}(x))\geq\frac{\left(\Delta_{r}(x)\right)^{2}}{8\tau^{2} Lr^{2}}.$$
При достаточно малых значениях $r$, для которых $\frac{\Delta_{r}(x)}{2\tau Lr^{2}}\geq 1$, верно другое неравенство:
$$\frac{\tau}{2} + \frac{\hat{f}_{2}(x)}{2\tau} - \hat{f}_{1}(T_{L, \tau}(x))\geq\frac{\Delta_{r}(x)}{2\tau} - \frac{Lr^{2}}{2}.$$
В полученных оценках функция $Lr^{2}\varkappa\left(\frac{\Delta_{r}(x)}{2\tau Lr^{2}}\right)$ является монотонно убывающей по $L$ и по $\tau$.
\end{lm:corollary}
\begin{lm:corollary}\label{lm:cor_2_aux_det_local_decrease_2}
При выборе $\tau = \hat{f}_{1}(x) > 0$ из полученной оценки следует, что если $x\in\mathcal{L}(\hat{f}_{1}(x))\subseteq\mathcal{F}$, то и $T_{L, \hat{f}_{1}(x)}(x)\in\mathcal{L}(\hat{f}_{1}(x))$:
\begin{equation*}
    \begin{aligned}
        \frac{\tau}{2} + \frac{\hat{f}_{2}(x)}{2\tau} - \hat{f}_{1}(T_{L, \tau}(x))&\geq Lr^{2}\varkappa\left(\frac{\Delta_{r}(x)}{2\tau Lr^{2}}\right)\Rightarrow\left\{\tau = \hat{f}_{1}(x)\right\}\Rightarrow\hat{f}_{1}(x) - \hat{f}_{1}(T_{L, \hat{f}_{1}(x)}(x))\geq\\
        &\geq Lr^{2}\varkappa\left(\frac{\Delta_{r}(x)}{2\hat{f}_{1}(x) Lr^{2}}\right)\geq 0\Rightarrow\hat{f}_{1}(x)\geq\hat{f}_{1}(T_{L, \hat{f}_{1}(x)}(x))\Rightarrow\\
        &\Rightarrow T_{L, \hat{f}_{1}(x)}(x)\in\mathcal{L}(\hat{f}_{1}(T_{L, \hat{f}_{1}(x)}(x)))\subseteq\mathcal{L}(\hat{f}_{1}(x)).
    \end{aligned}
\end{equation*}
\end{lm:corollary}

В лемме \ref{lm:aux_det_psi_upper_bound} вводится верхняя оценка уже самой локальной модели, не только функции $\hat{f}_{1}$, данная оценка понадобится для установки в теореме \ref{th:DetQuadConv} локальной квадратичной сходимости метода Гаусса--Ньютона, описанного схемой \ref{alg:gen_det_gnm}.

\begin{lemma}\label{lm:aux_det_psi_upper_bound}
    Пусть $x\in\mathcal{F},~T_{L, \tau}(x)\in\mathcal{F},~L > 0,~\tau > 0$. Тогда
    \begin{equation*}
        \begin{aligned}
            \psi_{x, L, \tau}(T_{L, \tau}(x))&\leq\min\limits_{y\in\mathcal{F}}\left\{\frac{\tau}{2} + \frac{L\|y - x\|^{2}}{2} + \frac{\hat{f}_{2}(y)}{2\tau} + \frac{\hat{f}_{1}(y)L_{\hat{F}}\|y - x\|^{2}}{2\tau} + \frac{L_{\hat{F}}^{2}\|y - x\|^{4}}{8\tau}\right\}.
        \end{aligned}
    \end{equation*}
\end{lemma}
\begin{proof}
По определению $\psi_{x, L, \tau}(\cdot)$:
\begin{equation*}
\begin{aligned}
    \psi_{x, L, \tau}(T_{L, \tau}(x)) &= \min\limits_{y\in\mathcal{F}}\left\{\frac{\tau}{2} + \frac{1}{2\tau}\left\|\hat{F}(x) + \hat{F}^{'}(x)(y - x)\right\|^{2} + \frac{L}{2}\|y - x\|^{2}\right\}=\frac{\tau}{2} +\\
    &+ \min\limits_{y\in\mathcal{F}}\left\{\frac{1}{2\tau}\left(\left\|\hat{F}(y) - \left(\hat{F}(y) - \hat{F}(x) - \hat{F}^{'}(x)(y - x)\right)\right\|\right)^{2} + \frac{L}{2}\|y - x\|^{2}\right\}\leq\frac{\tau}{2} +\\
    &+ \min\limits_{y\in\mathcal{F}}\left\{\frac{1}{2\tau}\left(\hat{f}_{1}(y) + \left\|\hat{F}(y) - \hat{F}(x) - \hat{F}^{'}(x)(y - x)\right\|\right)^{2} + \frac{L}{2}\|y - x\|^{2}\right\}\leq\left\{\text{неравенство \eqref{eq:lm_aux_det_model_bound}}\right\}\leq\\
    &\leq\frac{\tau}{2} + \min\limits_{y\in\mathcal{F}}\left\{\frac{1}{2\tau}\left(\hat{f}_{1}(y) + \frac{L_{\hat{F}}}{2}\|y - x\|^{2}\right)^{2} + \frac{L}{2}\|y - x\|^{2}\right\}\leq\frac{\tau}{2} +\\
    &+ \min\limits_{y\in\mathcal{F}}\left\{\frac{L\|y - x\|^{2}}{2} + \frac{\hat{f}_{2}(y)}{2\tau} + \frac{\hat{f}_{1}(y)L_{\hat{F}}\|y - x\|^{2}}{2\tau} + \frac{L_{\hat{F}}^{2}\|y - x\|^{4}}{8\tau}\right\}.
\end{aligned}
\end{equation*}
\end{proof}
\begin{lm:corollary}\label{lm:cor_aux_det_psi_upper_bound}
Пусть $x^{*}\in\mathcal{F}$ --- решение задачи \eqref{eq:smooth_system}: $\hat{F}(x^{*}) = \mathbf{0}_{m}$, $\mathcal{L}(\hat{f}_{1}(x))\subseteq\mathcal{F}$. Тогда
\begin{equation*}
    \begin{aligned}
        \psi_{x, L, \tau}(T_{L, \tau}(x))&\leq\min\limits_{y\in\mathcal{F}}\left\{\frac{\tau}{2} + \frac{L\|y - x\|^{2}}{2} + \frac{\hat{f}_{2}(y)}{2\tau} + \frac{\hat{f}_{1}(y)L_{\hat{F}}\|y - x\|^{2}}{2\tau} + \frac{L_{\hat{F}}^{2}\|y - x\|^{4}}{8\tau}\right\}\leq\frac{\tau}{2} + \frac{L\|y - x\|^{2}}{2} +\\
        &+ \frac{\hat{f}_{2}(y)}{2\tau} + \frac{\hat{f}_{1}(y)L_{\hat{F}}\|y - x\|^{2}}{2\tau} + \frac{L_{\hat{F}}^{2}\|y - x\|^{4}}{8\tau} = \left\{y = x^{*}\right\} = \frac{\tau}{2} + \frac{L\|x - x^{*}\|^{2}}{2} + \frac{L_{\hat{F}}^{2}\|x - x^{*}\|^{4}}{8\tau}.
    \end{aligned}
\end{equation*}
\end{lm:corollary}

Лемма \ref{lm:aux_matrix_order} устанавливает важное матричное отношение частичного порядка для вывода линейной сходимости в условии Поляка--Лоясиевича.

\begin{lemma}[\cite{Nesterov2020}]\label{lm:aux_matrix_order}
    Пусть линейный оператор $A: E_{1}\rightarrow E_{2}$, с $\dim(E_{1}) = n$, $\dim(E_{2}) = m$, $m\leq n$ обладает матрицей, не вырожденной по строкам:
    $$AA^{*}\succeq\mu I_{m}$$
    для определённого $\mu > 0$. Тогда для любого $\xi > 0$ выполнено
    $$A\left(\xi I_{n} + A^{*}A\right)^{-t}A^{*}\succeq\frac{\mu}{\left(\xi + \mu\right)^{t}}I_{m},~t\in[0,~1].$$
\end{lemma}
Отношение порядка <<$\succeq$>> выполнено на конусе неотрицательно определённых матриц.
\begin{proof}
Рассмотрим сингулярное разложение матрицы оператора $A$:
\begin{equation*}
    A = U\Lambda V^{*},~U^{*}U = I_{m},~V^{*}V = I_{m},
\end{equation*}
где $\Lambda$ --- диагональная матрица, $\Lambda\succeq\sqrt{\mu}I_{m}$ (по условию). Введём матрицу $W$ со столбцами, ортогонально дополняющими столбцы из $V$ до полного базиса в $E_{1}$:
\begin{equation*}
    VV^{*} + WW^{*} = I_{n},~W^{*}V = \mathbf{0}_{(n - m)\times m}.
\end{equation*}
Пользуясь блочной структурой из-за $W^{*}V = \mathbf{0}_{(n - m)\times m}$, получаем:
\begin{equation*}
    \begin{aligned}
        A\left(\xi I_{n} + A^{*}A\right)^{-t}A^{*} &= U\Lambda V^{*}\left(\xi\left(VV^{*} + WW^{*}\right) + V\Lambda^{2}V^{*}\right)^{-t}V\Lambda U^{*} =\\
        &= U\Lambda V^{*}\left(V\left(\xi I_{m} + \Lambda^{2}\right)V^{*} + \xi WW^{*}\right)^{-t}V\Lambda U^{*} =\\
        &= U\Lambda V^{*}\left(V\left(\xi I_{m} + \Lambda^{2}\right)^{-t}V^{*} + \frac{1}{\xi^{t}}WW^{*}\right)V\Lambda U^{*} = U\Lambda\left(\xi I_{m} + \Lambda^{2}\right)^{-t}\Lambda U^{*} =\\
        &= U\left(\xi\Lambda^{-\frac{2}{t}} + \Lambda^{2 - \frac{2}{t}}\right)^{-t}U^{*}\succeq\frac{1}{\left(\xi\mu^{-\frac{1}{t}} + \mu^{1 - \frac{1}{t}}\right)^{t}}I_{m} = \frac{\mu}{\left(\xi + \mu\right)^{t}}I_{m}.
    \end{aligned}
\end{equation*}
\end{proof}

\subsubsection*{Основные утверждения}

В теореме \ref{th:DetSublinConvMain} выведена сходимость к стационарной точке в терминах нормы проксимального градиента и приращения локальной модели.

\begin{re:theorem}\label{th:DetSublinConv}
    Пусть выполнено предположение \ref{as:det_hat_F_der_smooth}, $k\in\mathbb{N},~r > 0$. Тогда для метода Гаусса--Ньютона, реализованного по схеме \ref{alg:gen_det_gnm} с $\tau_{k} = \hat{f}_{1}(x_{k})$, $\varepsilon_{k} = \varepsilon \geq 0$, верны следующие оценки:
    \begin{equation*}
        \begin{cases}
            &\frac{8L_{\hat{F}}^{2}}{L}\left(\varepsilon + \frac{\left(\hat{f}_{1}(x_{0}) - \hat{f}_{1}(x_{k})\right)}{k}\right)\geq\min\limits_{i\in\overline{0, k - 1}}\left\{\left\|2L_{\hat{F}}\left(T_{2L_{\hat{F}}, \hat{f}_{1}(x_{i})}(x_{i}) - x_{i}\right)\right\|^{2}\right\};\\[10pt]
            &L_{\hat{F}}\left(\varepsilon + \frac{\left(\hat{f}_{1}(x_{0}) - \hat{f}_{1}(x_{k})\right)}{k}\right)\geq\min\limits_{i\in\overline{0, k - 1}}\left\{2\left(L_{\hat{F}}r\right)^{2}\varkappa\left(\frac{\Delta_{r}(x_{i})}{4\hat{f}_{1}(x_{i})L_{\hat{F}}r^{2}}\right)\right\};
        \end{cases}
    \end{equation*}
    где $\varkappa(t) = \frac{t^{2}}{2}\mathds{1}_{\left\{t\in[0, 1]\right\}} + \left(t - \frac{1}{2}\right)\mathds{1}_{\left\{t > 1\right\}}.$
\end{re:theorem}
\begin{proof}
Согласно леммам \ref{lm:aux_det_local_decrease_1}, \ref{lm:aux_det_local_decrease_2} и следствиям \ref{lm:cor_2_aux_det_local_decrease_1}, \ref{lm:cor_2_aux_det_local_decrease_2} для $\tau = \hat{f}_{1}(x_{k}),~L = L_{k},~x = x_{k}$ имеем следующее:
\begin{equation*}
    \begin{cases}
        \hat{f}_{1}(x_{k}) - \psi_{x_{k}, L_{k}, \hat{f}_{1}(x_{k})}(T_{L_{k}, \hat{f}_{1}(x_{k})}(x_{k}))\geq\frac{L_{k}}{2}\left\|T_{L_{k}, \hat{f}_{1}(x_{k})}(x_{k}) - x_{k}\right\|^{2};\\[10pt]
        \hat{f}_{1}(x_{k}) - \psi_{x_{k}, L_{k}, \hat{f}_{1}(x_{k})}(T_{L_{k}, \hat{f}_{1}(x_{k})}(x_{k}))\geq L_{k}r^{2}\varkappa\left(\frac{\Delta_{r}(x_{k})}{2\hat{f}_{1}(x_{k})L_{k}r^{2}}\right).
    \end{cases}
\end{equation*}
Добавим и вычтем $\psi_{x_{k}, L_{k}, \hat{f}_{1}(x_{k})}(x_{k + 1})$:
\begin{equation*}
    \begin{cases}
        \begin{aligned}
            \hat{f}_{1}(x_{k}) &+ \left(\psi_{x_{k}, L_{k}, \hat{f}_{1}(x_{k})}(x_{k + 1}) - \psi_{x_{k}, L_{k}, \hat{f}_{1}(x_{k})}(T_{L_{k}, \hat{f}_{1}(x_{k})}(x_{k}))\right) - \psi_{x_{k}, L_{k}, \hat{f}_{1}(x_{k})}(x_{k + 1})\geq\\
            &\geq\frac{L_{k}}{2}\left\|T_{L_{k}, \hat{f}_{1}(x_{k})}(x_{k}) - x_{k}\right\|^{2};\\
            \vspace{1pt}
        \end{aligned}\\
        \begin{aligned}
            \hat{f}_{1}(x_{k}) &+ \left(\psi_{x_{k}, L_{k}, \hat{f}_{1}(x_{k})}(x_{k + 1}) - \psi_{x_{k}, L_{k}, \hat{f}_{1}(x_{k})}(T_{L_{k}, \hat{f}_{1}(x_{k})}(x_{k}))\right) - \psi_{x_{k}, L_{k}, \hat{f}_{1}(x_{k})}(x_{k + 1})\geq\\
            &\geq L_{k}r^{2}\varkappa\left(\frac{\Delta_{r}(x_{k})}{2\hat{f}_{1}(x_{k})L_{k}r^{2}}\right).
        \end{aligned}
    \end{cases}
\end{equation*}
Используем условия $\psi_{x_{k}, L_{k}, \tau_{k}}(x_{k + 1}) - \psi_{x_{k}, L_{k}, \tau_{k}}(T_{L_{k}, \tau_{k}}(x_{k}))\leq\varepsilon_{k} = \varepsilon$ и $-\psi_{x_{k}, L_{k}, \hat{f}_{1}(x_{k})}(x_{k + 1})\leq -\hat{f}_{1}(x_{k + 1})$:
\begin{equation*}
    \begin{cases}
        \begin{aligned}
            \hat{f}_{1}(x_{k}) + \varepsilon - \hat{f}_{1}(x_{k + 1})\geq\frac{L_{k}}{2}\left\|T_{L_{k}, \hat{f}_{1}(x_{k})}(x_{k}) - x_{k}\right\|^{2};
        \end{aligned}\\[10pt]
        \begin{aligned}
            \hat{f}_{1}(x_{k}) + \varepsilon - \hat{f}_{1}(x_{k + 1})\geq L_{k}r^{2}\varkappa\left(\frac{\Delta_{r}(x_{k})}{2\hat{f}_{1}(x_{k})L_{k}r^{2}}\right).
        \end{aligned}
    \end{cases}
\end{equation*}
Усредним обе части неравенств по первым $k$ итерациям:
\begin{equation}\label{eq:det_sub_lin_approx_conv_1}
    \begin{cases}
        \begin{aligned}
            \varepsilon + \frac{\hat{f}_{1}(x_{0}) - \hat{f}_{1}(x_{k})}{k}\geq\frac{1}{k}\sum\limits_{i = 0}^{k - 1}\frac{L_{i}}{2}\left\|T_{L_{i}, \hat{f}_{1}(x_{i})}(x_{i}) - x_{i}\right\|^{2};
        \end{aligned}\\[10pt]
        \begin{aligned}
            \varepsilon + \frac{\hat{f}_{1}(x_{0}) - \hat{f}_{1}(x_{k})}{k}\geq\frac{1}{k}\sum\limits_{i = 0}^{k - 1}L_{i}r^{2}\varkappa\left(\frac{\Delta_{r}(x_{i})}{2\hat{f}_{1}(x_{i})L_{i}r^{2}}\right).
        \end{aligned}
    \end{cases}
\end{equation}
Воспользуемся тем, что в схеме \ref{alg:gen_det_gnm} $L_{k}\geq L$, функции $\left\|T_{L_{i}, \hat{f}_{1}(x_{i})}(x_{i}) - x_{i}\right\|^{2}$ и $L_{i}r^{2}\varkappa\left(\frac{\Delta_{r}(x_{i})}{2\hat{f}_{1}(x_{i})L_{i}r^{2}}\right)$ монотонно убывают по $L_{i}$ (следствия \ref{lm:cor_aux_det_local_decrease_1} и \ref{lm:cor_aux_det_local_decrease_2}):
\begin{equation*}
    \begin{cases}
        \begin{aligned}
            \varepsilon &+ \frac{\hat{f}_{1}(x_{0}) - \hat{f}_{1}(x_{k})}{k}\geq\frac{1}{k}\sum\limits_{i = 0}^{k - 1}\frac{L_{i}}{2}\left\|T_{L_{i}, \hat{f}_{1}(x_{i})}(x_{i}) - x_{i}\right\|^{2}\geq\\
            &\geq\frac{1}{k}\sum\limits_{i = 0}^{k - 1}\frac{L}{2}\left\|T_{2L_{\hat{F}}, \hat{f}_{1}(x_{i})}(x_{i}) - x_{i}\right\|^{2}\geq\min\limits_{i\in\overline{0, k - 1}}\left\{\frac{L}{2}\left\|T_{2L_{\hat{F}}, \hat{f}_{1}(x_{i})}(x_{i}) - x_{i}\right\|^{2}\right\};
        \end{aligned}\\[10pt]
        \begin{aligned}
            \varepsilon &+ \frac{\hat{f}_{1}(x_{0}) - \hat{f}_{1}(x_{k})}{k}\geq\frac{1}{k}\sum\limits_{i = 0}^{k - 1}L_{i}r^{2}\varkappa\left(\frac{\Delta_{r}(x_{i})}{2\hat{f}_{1}(x_{i})L_{i}r^{2}}\right)\geq\\
            &\geq\frac{1}{k}\sum\limits_{i = 0}^{k - 1}2L_{\hat{F}}r^{2}\varkappa\left(\frac{\Delta_{r}(x_{i})}{4\hat{f}_{1}(x_{i})L_{\hat{F}}r^{2}}\right)\geq\min\limits_{i\in\overline{0, k - 1}}\left\{2L_{\hat{F}}r^{2}\varkappa\left(\frac{\Delta_{r}(x_{i})}{4\hat{f}_{1}(x_{i})L_{\hat{F}}r^{2}}\right)\right\}.
        \end{aligned}
    \end{cases}
\end{equation*}
Приведём домножением на константы правые части к формату \textit{обобщённых проксимальных градиентов}:
\begin{equation*}
        \begin{cases}
            &\frac{8L_{\hat{F}}^{2}}{L}\left(\varepsilon + \frac{\left(\hat{f}_{1}(x_{0}) - \hat{f}_{1}(x_{k})\right)}{k}\right)\geq\min\limits_{i\in\overline{0, k - 1}}\left\{\left\|2L_{\hat{F}}\left(T_{2L_{\hat{F}}, \hat{f}_{1}(x_{i})}(x_{i}) - x_{i}\right)\right\|^{2}\right\};\\
            &L_{\hat{F}}\left(\varepsilon + \frac{\left(\hat{f}_{1}(x_{0}) - \hat{f}_{1}(x_{k})\right)}{k}\right)\geq\min\limits_{i\in\overline{0, k - 1}}\left\{2\left(L_{\hat{F}}r\right)^{2}\varkappa\left(\frac{\Delta_{r}(x_{i})}{4\hat{f}_{1}(x_{i})L_{\hat{F}}r^{2}}\right)\right\}.
        \end{cases}
    \end{equation*}
\end{proof}
\begin{re:th:corollary}\label{th:SubLinConvCor1}
В случае такой адаптивной стратегии подбора точности вычисления $x_{k + 1}$, как\\$\varepsilon_{0} = \varepsilon\hat{f}_{1}(x_{0})$, $\varepsilon_{k} = \varepsilon\left(\hat{f}_{1}(x_{k - 1}) - \hat{f}_{1}(x_{k})\right),~k\in\mathbb{N}$, $\varepsilon\geq0$, возможно получить приближение решения задачи \eqref{eq:smooth_system} с любой наперёд заданной точностью при условии неограниченного количества итераций. Для доказательства этого факта рассмотрим \eqref{eq:det_sub_lin_approx_conv_1} и применим обозначенное правило вычисления $\varepsilon_{k}$:
\begin{equation*}
    \begin{cases}
        \begin{aligned}
            \frac{\varepsilon\left(2\hat{f}_{1}(x_{0}) - \hat{f}_{1}(x_{k - 1})\right)}{k} + \frac{\hat{f}_{1}(x_{0}) - \hat{f}_{1}(x_{k})}{k}\geq\frac{1}{k}\sum\limits_{i = 0}^{k - 1}\frac{L_{i}}{2}\left\|T_{L_{i}, \hat{f}_{1}(x_{i})}(x_{i}) - x_{i}\right\|^{2};
        \end{aligned}\\[10pt]
        \begin{aligned}
            \frac{\varepsilon\left(2\hat{f}_{1}(x_{0}) - \hat{f}_{1}(x_{k - 1})\right)}{k} + \frac{\hat{f}_{1}(x_{0}) - \hat{f}_{1}(x_{k})}{k}\geq\frac{1}{k}\sum\limits_{i = 0}^{k - 1}L_{i}r^{2}\varkappa\left(\frac{\Delta_{r}(x_{i})}{2\hat{f}_{1}(x_{i})L_{i}r^{2}}\right).
        \end{aligned}
    \end{cases}
\end{equation*}
Применяя весь дальнейший ход доказательства теоремы получаем:
\begin{equation*}
    \begin{cases}
        \begin{aligned}
            \frac{8L_{\hat{F}}^{2}}{kL}&\left((1 + 2\varepsilon)\hat{f}_{1}(x_{0}) -\varepsilon\hat{f}_{1}(x_{k - 1}) - \hat{f}_{1}(x_{k})\right)\geq\min\limits_{i\in\overline{0, k - 1}}\left\{\left\|2L_{\hat{F}}\left(T_{2L_{\hat{F}}, \hat{f}_{1}(x_{i})}(x_{i}) - x_{i}\right)\right\|^{2}\right\};
        \end{aligned}\\[10pt]
        \begin{aligned}
            \frac{L_{\hat{F}}}{k}&\left((1 + 2\varepsilon)\hat{f}_{1}(x_{0}) - \varepsilon\hat{f}_{1}(x_{k - 1}) - \hat{f}_{1}(x_{k})\right)\geq\min\limits_{i\in\overline{0, k - 1}}\left\{2\left(L_{\hat{F}}r\right)^{2}\varkappa\left(\frac{\Delta_{r}(x_{i})}{4\hat{f}_{1}(x_{i})L_{\hat{F}}r^{2}}\right)\right\}.
        \end{aligned}
    \end{cases}
\end{equation*}
\end{re:th:corollary}
\begin{re:th:corollary}\label{th:SubLinConvCor2}
Из оценок в предыдущем следствии заменим нулевую итерацию на $k$--ую, а $k$--ую итерацию сдвинем на $(N + 1)\in\mathbb{N}$ итерацию вперёд, получим оценку на хвост суммы неравенств, $k\in\mathbb{Z}_{+}$:
\begin{equation*}
    \begin{cases}
        \begin{aligned}
            \frac{8L_{\hat{F}}^{2}}{(N + 1)L}&\left(\varepsilon\left(\hat{f}_{1}(x_{k - 1}) - \hat{f}_{1}(x_{k + N})\right) + \hat{f}_{1}(x_{k}) - \hat{f}_{1}(x_{k + N + 1})\right)\geq\\
            &\geq\min\limits_{i\in\overline{k, k + N}}\left\{\left\|2L_{\hat{F}}\left(T_{2L_{\hat{F}}, \hat{f}_{1}(x_{i})}(x_{i}) - x_{i}\right)\right\|^{2}\right\};
        \end{aligned}\\[10pt]
        \begin{aligned}
            \frac{L_{\hat{F}}}{N + 1}&\left(\varepsilon\left(\hat{f}_{1}(x_{k - 1}) - \hat{f}_{1}(x_{k + N})\right) + \hat{f}_{1}(x_{k}) - \hat{f}_{1}(x_{k + N + 1})\right)\geq\\
            &\geq\min\limits_{i\in\overline{k, k + N}}\left\{2\left(L_{\hat{F}}r\right)^{2}\varkappa\left(\frac{\Delta_{r}(x_{i})}{4\hat{f}_{1}(x_{i})L_{\hat{F}}r^{2}}\right)\right\}.
        \end{aligned}
    \end{cases}
\end{equation*}
Разворачивая заново цепочку доказательства теоремы для начальной итерации $k > 0$ и финальной итерации $k + N$ получаем следующую оценку на суммы неравенств в \eqref{eq:det_sub_lin_approx_conv_1}:
\begin{equation*}
    \begin{cases}
        \begin{aligned}
            \hat{f}_{1}(x_{k}) &- \hat{f}_{1}(x_{k + N + 1}) + \varepsilon\left(\hat{f}_{1}(x_{k - 1}) - \hat{f}_{1}(x_{k + N})\right)\geq\frac{L}{2}\sum\limits_{i = k}^{k + N}\left\|T_{2L_{\hat{F}}, \hat{f}_{1}(x_{i})}(x_{i}) - x_{i}\right\|^{2}\geq\\
            &\geq\frac{L}{2}\left\|T_{2L_{\hat{F}}, \hat{f}_{1}(x_{k})}(x_{k}) - x_{k}\right\|^{2};\\
            \vspace{1pt}
        \end{aligned}\\
        \begin{aligned}
            \hat{f}_{1}(x_{k}) &- \hat{f}_{1}(x_{k + N + 1}) + \varepsilon\left(\hat{f}_{1}(x_{k - 1}) - \hat{f}_{1}(x_{k + N})\right)\geq\sum\limits_{i = k}^{k + N}2L_{\hat{F}}r^{2}\varkappa\left(\frac{\Delta_{r}(x_{i})}{4\hat{f}_{1}(x_{i})L_{\hat{F}}r^{2}}\right)\geq\\
            &\geq2L_{\hat{F}}r^{2}\varkappa\left(\frac{\Delta_{r}(x_{k})}{4\hat{f}_{1}(x_{k})L_{\hat{F}}r^{2}}\right).
        \end{aligned}
    \end{cases}
\end{equation*}
В пределе при устремлении $N\rightarrow+\infty$ получаем следующее:
\begin{equation}\label{eq:det_sub_lin_approx_conv_2}
    \begin{cases}
        \hat{f}_{1}(x_{k}) - \hat{f}_{1}^{*} + \varepsilon\left(\hat{f}_{1}(x_{k - 1}) - \hat{f}_{1}^{*}\right)\geq\frac{L}{2}\left\|T_{2L_{\hat{F}}, \hat{f}_{1}(x_{k})}(x_{k}) - x_{k}\right\|^{2};\\[10pt]
        \hat{f}_{1}(x_{k}) - \hat{f}_{1}^{*} + \varepsilon\left(\hat{f}_{1}(x_{k - 1}) - \hat{f}_{1}^{*}\right)\geq2L_{\hat{F}}r^{2}\varkappa\left(\frac{\Delta_{r}(x_{k})}{4\hat{f}_{1}(x_{k})L_{\hat{F}}r^{2}}\right).
    \end{cases}
\end{equation}
Неравенства в \eqref{eq:det_sub_lin_approx_conv_2} при условии $\lim\limits_{k\rightarrow+\infty}\varepsilon_{k} = \lim\limits_{k\rightarrow+\infty}\varepsilon\left(\hat{f}_{1}(x_{k - 1}) - \hat{f}_{1}(x_{k})\right) = 0$ означают 
$$\lim\limits_{k\rightarrow+\infty}x_{k + 1} = \lim\limits_{k\rightarrow+\infty}T_{2L_{\hat{F}}, \hat{f}_{1}(x_{k})}(x_{k}) = x^{*}$$
и
\begin{equation}\label{eq:det_sub_lin_approx_conv_3}
    \begin{cases}
        \lim\limits_{k\rightarrow+\infty}\left\|x_{k + 1} - x_{k}\right\| = 0;\\
        \lim\limits_{k\rightarrow+\infty}\frac{\Delta_{r}(x_{k})}{\hat{f}_{1}(x_{k})} = 0.
    \end{cases}
\end{equation}
Пределы в \eqref{eq:det_sub_lin_approx_conv_3} получены как следствие рассмотрения предельных значений неравенств в \eqref{eq:det_sub_lin_approx_conv_2} при\\$k\rightarrow+\infty$, они наглядно демонстрируют ограниченность вариации последовательности $\left\{x_{k}\right\}_{k\in\mathbb{Z}_{+}}$ и связность множества стационарных точек $\left\{x^{*}:~x^{*}\in E_{1},~\Delta_{r}(x^{*}) = 0\right\}$ для данной последовательности.
\end{re:th:corollary}

В теореме \ref{th:DetQuadConvMain} устанавливается локальная квадратичная сходимость и определяются условия размерности задачи \eqref{eq:main_opt_problem}, при которых обычно возможна квадратичная сходимость метода Гаусса--Ньютона со схемой \ref{alg:gen_det_gnm}.

\begin{re:theorem}\label{th:DetQuadConv}
    Пусть выполнено предположение \ref{as:det_hat_F_der_smooth}, пусть для метода Гаусса--Ньютона со схемой \ref{alg:gen_det_gnm} существует $x^{*}\in\mathcal{L}(\hat{f}_{1}(x_{0})),~\hat{F}(x^{*}) = \mathbf{0}_{m}$ --- решение с $\sigma_{\min}\left(\hat{F}^{'}(x^{*})\right)\geq\varsigma > 0$. Если выполнено $\varsigma > \frac{2L_{\hat{F}}}{\alpha}$ при некотором фиксированном $\alpha\in\left(0,~1\right)$ для всех $\varepsilon_{k}\geq 0$, $k\in\mathbb{Z}_{+}$ в схеме \ref{alg:gen_det_gnm} с
    \begin{equation*}
        \begin{cases}
            \|x_{k} - x^{*}\| < \frac{\varsigma}{L_{\hat{F}}} - \frac{2}{\alpha};\\[5pt]
            0 < \tau_{k}\leq\frac{\left(\left(\alpha\left(\varsigma - L_{\hat{F}}\|x_{k} - x^{*}\|\right) - \frac{3L_{\hat{F}}}{2}\right)^{2} - \frac{L_{\hat{F}}^{2}}{4}\right)\|x_{k} - x^{*}\|^{4}}{\|x_{k} - x^{*}\|^{2}L_{k} + 2\varepsilon_{k}};
        \end{cases}
    \end{equation*}
    то $x_{k + 1}\in\mathcal{L}(\hat{f}_{1}(x_{0}))$ и
    \begin{equation*}
        \begin{aligned}
            \|x_{k + 1} - x^{*}\|&\leq\frac{\frac{3L_{\hat{F}}\|x_{k} - x^{*}\|^{2}}{2} + \sqrt{\|x_{k} - x^{*}\|^{2}\left(\tau_{k}L_{k} + \frac{L_{\hat{F}}^{2}\|x_{k} - x^{*}\|^{2}}{4}\right) + 2\tau_{k}\varepsilon_{k}}}{\varsigma - L_{\hat{F}}\|x_{k} - x^{*}\|}\leq \alpha\|x_{k} - x^{*}\|^{2}.
        \end{aligned}
    \end{equation*}
Если не существует такого $\alpha\in(0,~1)$, то в схеме \ref{alg:gen_det_gnm} при выборе
\begin{equation*}
    \begin{cases}
        \tau_{k} = c_{1}\|x_{k} - x^{*}\|^{2},~c_{1} > 0;\\[5pt]
        \varepsilon_{k} = c_{2}\|x_{k} - x^{*}\|^{2},~c_{2}\geq0;
    \end{cases}
\end{equation*}
в области
$$\|x_{k} - x^{*}\|\leq\frac{\varsigma}{\frac{5L_{\hat{F}}}{2} + \sqrt{2c_{1}L_{\hat{F}} + \frac{L_{\hat{F}}^{2}}{4} + 2c_{1}c_{2}}},~k\in\mathbb{Z}_{+}$$
выполнена следующая оценка:
\begin{equation*}
    \begin{aligned}
        \|x_{k + 1} - x^{*}\|&\leq\frac{\frac{3L_{\hat{F}}\|x_{k} - x^{*}\|^{2}}{2} + \sqrt{\|x_{k} - x^{*}\|^{2}\left(\tau_{k}L_{k} + \frac{L_{\hat{F}}^{2}\|x_{k} - x^{*}\|^{2}}{4}\right) + 2\tau_{k}\varepsilon_{k}}}{\varsigma - L_{\hat{F}}\|x_{k} - x^{*}\|}\leq\|x_{k} - x^{*}\|,~x_{k + 1}\in\mathcal{L}(\hat{f}_{1}(x_{0})).
    \end{aligned}
\end{equation*}
\end{re:theorem}
\begin{proof}
Согласно лемме \ref{lm:aux_det_psi_upper_bound} (следствие \ref{lm:cor_aux_det_psi_upper_bound}) $\psi_{x_{k}, L_{k}, \tau_{k}}(T_{L_{k}, \tau_{k}}(x_{k}))$ имеет оценку сверху:
\begin{equation*}
    \begin{aligned}
        \psi_{x_{k}, L_{k}, \tau_{k}}(T_{L_{k}, \tau_{k}}(x_{k}))&\leq\frac{\tau_{k}}{2} + \frac{L_{k}\|x_{k} - x^{*}\|^{2}}{2} + \frac{L_{\hat{F}}^{2}\|x_{k} - x^{*}\|^{4}}{8\tau_{k}}\Rightarrow\\
        &\Rightarrow\left\{\text{прибавим }\psi_{x_{k}, L_{k}, \tau_{k}}(x_{k + 1}) - \psi_{x_{k}, L_{k}, \tau_{k}}(T_{L_{k}, \tau_{k}}(x_{k}))\leq\varepsilon_{k}\right\}\Rightarrow\\
        &\Rightarrow\psi_{x_{k}, L_{k}, \tau_{k}}(x_{k + 1})\leq\frac{\tau_{k}}{2} + \frac{L_{k}\|x_{k} - x^{*}\|^{2}}{2} + \frac{L_{\hat{F}}^{2}\|x_{k} - x^{*}\|^{4}}{8\tau_{k}} + \varepsilon_{k}\Rightarrow\\
        \Rightarrow\psi_{x_{k}, L_{k}, \tau_{k}}(x_{k + 1})& = \frac{\tau_{k}}{2} + \frac{\left(\phi(x_{k}, x_{k + 1})\right)^{2}}{2\tau_{k}} + \frac{L_{k}}{2}\|x_{k + 1} - x_{k}\|^{2}\leq\\
        &\leq\frac{\tau_{k}}{2} + \frac{L_{k}\|x_{k} - x^{*}\|^{2}}{2} + \frac{L_{\hat{F}}^{2}\|x_{k} - x^{*}\|^{4}}{8\tau_{k}} + \varepsilon_{k}\Rightarrow\\
        \Rightarrow\frac{\left(\phi(x_{k}, x_{k + 1})\right)^{2}}{2\tau_{k}}&\leq\frac{L_{k}\|x_{k} - x^{*}\|^{2}}{2} + \frac{L_{\hat{F}}^{2}\|x_{k} - x^{*}\|^{4}}{8\tau_{k}} + \varepsilon_{k}\Rightarrow\sqrt{\tau_{k}L_{k}\|x_{k} - x^{*}\|^{2} + \frac{L_{\hat{F}}^{2}\|x_{k} - x^{*}\|^{4}}{4} + 2\tau_{k}\varepsilon_{k}}\geq\\
        &\geq\phi(x_{k}, x_{k + 1})\Rightarrow\sqrt{\|x_{k} - x^{*}\|^{2}\left(\tau_{k}L_{k} + \frac{L_{\hat{F}}^{2}\|x_{k} - x^{*}\|^{2}}{4}\right) + 2\tau_{k}\varepsilon_{k}}\geq\\
        &\geq\left\|\hat{F}(x_{k}) + \hat{F}^{'}(x_{k})(x_{k + 1} - x_{k})\right\|.
    \end{aligned}
\end{equation*}
Перепишем $\phi(x_{k}, x_{k + 1})$ по--другому:
\begin{equation*}
    \begin{aligned}
        \left\|\hat{F}(x_{k}) + \hat{F}^{'}(x_{k})(x_{k + 1} - x_{k})\right\|& = \left\|\underbrace{\hat{F}^{'}(x^{*})\left(x_{k + 1} - x^{*}\right)}_{\overset{\operatorname{def}}{=} A} + \underbrace{\left(\hat{F}(x_{k}) - \hat{F}(x^{*}) - \hat{F}^{'}(x^{*})(x_{k} - x^{*})\right)}_{\overset{\operatorname{def}}{=} B} +\right.\\
        &\left.+ \underbrace{\left(\hat{F}^{'}(x_{k}) - \hat{F}^{'}(x^{*})\right)(x_{k + 1} - x_{k})}_{\overset{\operatorname{def}}{=} C}\right\|.
    \end{aligned}
\end{equation*}
По неравенству треугольника для нормы $\|\cdot\|$:
\begin{equation*}
    \begin{aligned}
        \left\|A\right\| &= \left\|A + B + C - B - C\right\|\leq\left\|A + B + C\right\| + \left\|-B\right\| + \left\|-C\right\|\Rightarrow\left\|A + B + C\right\|\geq\left\|A\right\| - \left\|B\right\| - \left\|C\right\|;\\
        \left\|A\right\|&\geq\left\{\text{из определения минимального сингулярного числа}\right\}\geq\varsigma\|x_{k + 1} - x^{*}\|;\\
        \left\|B\right\|&\leq\left\{\text{неравенство \eqref{eq:lm_aux_det_model_bound}}\right\}\leq\frac{L_{\hat{F}}}{2}\|x_{k} - x^{*}\|^{2};\\
        \left\|C\right\|&\leq\left\{\text{субмультипликативность нормы}\right\}\leq\left\|\hat{F}^{'}(x_{k}) - \hat{F}^{'}(x^{*})\right\|\left\|x_{k + 1} - x_{k}\right\|\leq\left\{\text{предположение \ref{as:det_hat_F_der_smooth}}\right\}\leq\\
        &\leq L_{\hat{F}}\|x_{k} - x^{*}\|\|x_{k + 1} - x^{*} + x^{*} - x_{k}\|\leq L_{\hat{F}}\|x_{k} - x^{*}\|^{2} + L_{\hat{F}}\|x_{k} - x^{*}\|\|x_{k + 1} - x^{*}\|.
    \end{aligned}
\end{equation*}
Собрав вместе неравенства, получаем оценку снизу на $\phi(x_{k}, x_{k + 1}):$
\begin{equation*}
    \phi(x_{k}, x_{k + 1})\geq\left(\varsigma - L_{\hat{F}}\|x_{k} - x^{*}\|\right)\|x_{k + 1} - x^{*}\| - \frac{3L_{\hat{F}}}{2}\|x_{k} - x^{*}\|^{2}.
\end{equation*}
Свяжем нижнюю и верхнюю оценку на $\phi(x_{k}, x_{k + 1})$ в единое неравенство:
\begin{equation*}
    \begin{aligned}
        &\sqrt{\|x_{k} - x^{*}\|^{2}\left(\tau_{k}L_{k} + \frac{L_{\hat{F}}^{2}\|x_{k} - x^{*}\|^{2}}{4}\right) + 2\tau_{k}\varepsilon_{k}}\geq\left(\varsigma - L_{\hat{F}}\|x_{k} - x^{*}\|\right)\|x_{k + 1} - x^{*}\| - \frac{3L_{\hat{F}}\|x_{k} - x^{*}\|^{2}}{2}\Rightarrow\\
    \end{aligned}
\end{equation*}
\begin{equation*}
    \begin{aligned}
        &\Rightarrow\frac{\sqrt{\|x_{k} - x^{*}\|^{2}\left(\tau_{k}L_{k} + \frac{L_{\hat{F}}^{2}\|x_{k} - x^{*}\|^{2}}{4}\right) + 2\tau_{k}\varepsilon_{k}} + \frac{3L_{\hat{F}}\|x_{k} - x^{*}\|^{2}}{2}}{\varsigma - L_{\hat{F}}\|x_{k} - x^{*}\|}\geq\|x_{k + 1} - x^{*}\|.
    \end{aligned}
\end{equation*}
Предположим существование $\alpha$ из условия. Оценим получившееся неравенство сверху выражением\\\mbox{$\alpha\|x_{k} - x^{*}\|^{2}$} и выведем условия на $\tau_{k}$ и $\|x_{k} - x^{*}\|$, при которых данная оценка выполняется всегда.\\Найдём допустимые значения $\tau_{k}$:
\begin{equation*}
    \begin{aligned}
        &\alpha\left(\varsigma - L_{\hat{F}}\|x_{k} - x^{*}\|\right)\|x_{k} - x^{*}\|^{2}\geq\frac{3L_{\hat{F}}\|x_{k} - x^{*}\|^{2}}{2} + \sqrt{\|x_{k} - x^{*}\|^{2}\left(\tau_{k}L_{k} + \frac{L_{\hat{F}}^{2}\|x_{k} - x^{*}\|^{2}}{4}\right) + 2\tau_{k}\varepsilon_{k}}\Rightarrow\\
        &\Rightarrow\left(\underbrace{\alpha\left(\varsigma - L_{\hat{F}}\|x_{k} - x^{*}\|\right) - \frac{3L_{\hat{F}}}{2}}_{\geq0\text{ по условию на }\varsigma}\right)^{2}\|x_{k} - x^{*}\|^{4}\geq\tau_{k}\left(\|x_{k} - x^{*}\|^{2}L_{k} + 2\varepsilon_{k}\right) + \frac{L_{\hat{F}}^{2}\|x_{k} - x^{*}\|^{4}}{4}\Rightarrow\\
    \end{aligned}
\end{equation*}
\begin{equation*}
    \begin{aligned}
        &\Rightarrow0 < \tau_{k}\leq\frac{\left(\left(\alpha\left(\varsigma - L_{\hat{F}}\|x_{k} - x^{*}\|\right) - \frac{3L_{\hat{F}}}{2}\right)^{2} - \frac{L_{\hat{F}}^{2}}{4}\right)\|x_{k} - x^{*}\|^{4}}{\|x_{k} - x^{*}\|^{2}L_{k} + 2\varepsilon_{k}}.
    \end{aligned}
\end{equation*}
Выведем ограничения на $\|x_{k} - x^{*}\|$, чтобы выполнялось $\tau_{k} > 0$:
\begin{equation*}
    \begin{aligned}
        &\left(\alpha\left(\varsigma - L_{\hat{F}}\|x_{k} - x^{*}\|\right) - \frac{3L_{\hat{F}}}{2}\right)^{2} - \frac{L_{\hat{F}}^{2}}{4} > 0\text{ по условию на }\varsigma\Rightarrow\left\|x_{k} - x^{*}\right\| < \frac{\varsigma}{L_{\hat{F}}} - \frac{2}{\alpha}.
    \end{aligned}
\end{equation*}
Докажем оставшуюся часть теоремы. Введём переменную $t_{k} = \|x_{k} - x^{*}\|$, перепишем оценку на $t_{k + 1}$ с подстановкой $\tau_{k}$ и $\varepsilon_{k}$:
\begin{equation*}
    \begin{aligned}
        t_{k + 1}&\leq\left(\frac{\frac{1}{t_{k}}\sqrt{t_{k}^{2}\left(c_{1}t_{k}^{2}L_{k} + \frac{L_{\hat{F}}^{2}}{4}t_{k}^{2}\right) + 2c_{1}c_{2}t_{k}^{4}} + \frac{3L_{\hat{F}}}{2}t_{k}}{\varsigma - L_{\hat{F}}t_{k}}\right)t_{k}\leq\left\{L_{k}\leq2L_{\hat{F}}\right\}\leq\\
        &\leq\underbrace{\left(\frac{\sqrt{2c_{1}L_{\hat{F}} + \frac{L_{\hat{F}}^{2}}{4} + 2c_{1}c_{2}} + \frac{3L_{\hat{F}}}{2}}{\varsigma - L_{\hat{F}}t_{k}}t_{k}\right)}_{\in[0,~1]\text{ --- необходимое условие}}t_{k}\Rightarrow t_{k} = \|x_{k} - x^{*}\|\leq\frac{\varsigma}{\frac{5L_{\hat{F}}}{2} + \sqrt{2c_{1}L_{\hat{F}} + \frac{L_{\hat{F}}^{2}}{4} + 2c_{1}c_{2}}}.
    \end{aligned}
\end{equation*}
\end{proof}
\begin{re:th:corollary}\label{th:DetQuadConvCor1}
Условия теоремы неявно задают ограничения на размерность задачи:
\begin{itemize}
    \itemsep=-4pt
    \item невырожденность системы уравнений \eqref{eq:smooth_system} в точке минимума $\sigma_{\min}\left(\hat{F}^{'}(x^{*})\right)\geq\varsigma > 0$ означает $\dim(E_{2})\geq\dim(E_{1})$;
    \item совместность системы \eqref{eq:smooth_system} $\hat{F}(x^{*}) = \mathbf{0}_{m}$ обычно выполняется в системах при доминировании количества параметров над количеством условий: $\dim(E_{2})\leq\dim(E_{1})$.
\end{itemize}
Таким образом, локальная квадратичная сходимость в условии теоремы обычно возможна на системах с $\dim(E_{1}) = \dim(E_{2})$.
\end{re:th:corollary}

В теореме \ref{th:glob_sub_lin_and_lin_conv_main} выводятся оценки сходимости для метода нормализованных квадратов с выбором\\$\tau_{k} = \hat{f}_{1}(x_{k})$, в оценках содержится явное разделение на область сублинейной сходимости и область линейной сходимости. 

\begin{re:theorem}\label{th:glob_sub_lin_and_lin_conv}
    Допустим выполнение предположений \ref{as:det_hat_F_der_smooth} и \ref{as:det_hat_F_PL_condition} для метода Гаусса--Ньютона со схемой реализации \ref{alg:gen_det_gnm}, в которой $\tau_{k} = \hat{f}_{1}(x_{k})$. Тогда в схеме \ref{alg:gen_det_gnm} для последовательности $\left\{x_{k}\right\}_{k\in\mathbb{Z}_{+}}$ выполняются следующие соотношения:
    \begin{equation*}
        \hat{f}_{1}(x_{k + 1})\leq\varepsilon_{k} + \begin{sqcases}
            \frac{\hat{f}_{1}(x_{k})}{2} + \frac{L_{\hat{F}}}{\mu}\hat{f}_{2}(x_{k})\leq\frac{3}{4}\hat{f}_{1}(x_{k})\text{, если }\hat{f}_{1}(x_{k})\leq\frac{\mu}{4L_{\hat{F}}};\\[5pt]
            \hat{f}_{1}(x_{k}) - \frac{\mu}{16L_{\hat{F}}}\text{, иначе}.
        \end{sqcases}
    \end{equation*}
    Если при генерации последовательности $\left\{x_{k}\right\}_{k\in\mathbb{Z}_{+}}$ была зафиксирована $L_{k} = L_{\hat{F}}$, то данные соотношения выражаются по--другому:
    \begin{equation*}
        \hat{f}_{1}(x_{k + 1})\leq\varepsilon_{k} + \begin{sqcases}
            \frac{\hat{f}_{1}(x_{k})}{2} + \frac{L_{\hat{F}}}{2\mu}\hat{f}_{2}(x_{k})\leq\frac{3}{4}\hat{f}_{1}(x_{k})\text{, если }\hat{f}_{1}(x_{k})\leq\frac{\mu}{2L_{\hat{F}}};\\[5pt]
            \hat{f}_{1}(x_{k}) - \frac{\mu}{8L_{\hat{F}}}\text{, иначе}.
        \end{sqcases}
    \end{equation*}
\end{re:theorem}
\begin{proof}
Рассмотрим систему линейных уравнений $\hat{F}(x) + \hat{F}^{'}(x)h = 0,~x\in\mathcal{F}$. В условиях данной теоремы существует $h\in E_{1}$: $\hat{F}(x) + \hat{F}^{'}(x)h = 0$, $x\in\mathcal{F}$ в силу выполнения условия Поляка--Лоясиевича, при этом
$$h = -\hat{F}^{'}(x)^{*}\left(\hat{F}^{'}(x)\hat{F}^{'}(x)^{*}\right)^{-1}\hat{F}(x).$$
Тогда, согласно предположению \ref{as:det_hat_F_PL_condition},
\begin{equation}\label{eq:th_glob_lin_conv_eq_1}
    \begin{aligned}
        \|h\| &= \left\|\hat{F}^{'}(x)^{*}\left(\hat{F}^{'}(x)\hat{F}^{'}(x)^{*}\right)^{-1}\hat{F}(x)\right\| = \sqrt{\left\langle \left(\hat{F}^{'}(x)\hat{F}^{'}(x)^{*}\right)^{-1}\hat{F}(x),~\hat{F}(x)\right\rangle}\leq\frac{\left\|\hat{F}(x)\right\|}{\sqrt{\mu}} = \frac{\hat{f}_{1}(x)}{\sqrt{\mu}}.
    \end{aligned}
\end{equation}
По определению локальной модели для $x_{k + 1},~k\in\mathbb{Z}_{+}$:
\begin{equation*}
    \begin{aligned}
        \hat{f}_{1}(x_{k + 1})&\leq\psi_{x_{k}, L_{k}, \hat{f}_{1}(x_{k})}(x_{k + 1}) = \psi_{x_{k}, L_{k}, \hat{f}_{1}(x_{k})}(T_{L_{k}, \hat{f}_{1}(x_{k})}(x_{k})) + \left(\psi_{x_{k}, L_{k}, \hat{f}_{1}(x_{k})}(x_{k + 1}) -\right.\\
        &\left.- \psi_{x_{k}, L_{k}, \hat{f}_{1}(x_{k})}(T_{L_{k}, \hat{f}_{1}(x_{k})}(x_{k}))\right)\leq\varepsilon_{k} + \psi_{x_{k}, L_{k}, \hat{f}_{1}(x_{k})}(T_{L_{k}, \hat{f}_{1}(x_{k})}(x_{k})) = \varepsilon_{k} +\\
        &+ \min\limits_{y\in E_{1}}\left\{\frac{\hat{f}_{1}(x_{k})}{2} + \frac{\left(\phi(x_{k}, x_{k} + y)\right)^{2}}{2\hat{f}_{1}(x_{k})} + \frac{L_{k}}{2}\|y\|^{2}\right\}\leq\\
        &\leq\left\{\text{вместо $y$ подставим }th_{k} = -t\hat{F}^{'}(x_{k})^{*}\left(\hat{F}^{'}(x_{k})\hat{F}^{'}(x_{k})^{*}\right)^{-1}\hat{F}(x_{k}),~t\in[0,~1]\right\}\leq\varepsilon_{k} + \frac{\hat{f}_{1}(x_{k})}{2} +\\
        &+ \min\limits_{t\in [0,~1]}\left\{\frac{1}{2\hat{f}_{1}(x_{k})}\left\|\hat{F}(x_{k}) + t\hat{F}^{'}(x_{k})h_{k}\right\|^{2} + \frac{t^{2}L_{k}}{2}\|h_{k}\|^{2}\right\}\leq\left\{\text{неравенство \eqref{eq:th_glob_lin_conv_eq_1}}\right\}\leq\varepsilon_{k} + \frac{\hat{f}_{1}(x_{k})}{2} +\\
        &+ \min\limits_{t\in[0,~1]}\left\{\frac{\left\|(1 - t)\hat{F}(x_{k})\right\|^{2}}{2\hat{f}_{1}(x_{k})} + \frac{t^{2}L_{k}}{2\mu}\hat{f}_{2}(x_{k})\right\}\leq\left\{\|\cdot\|^{2}\text{ --- выпуклая}\right\}\leq\varepsilon_{k} + \frac{\hat{f}_{1}(x_{k})}{2} +\\
        &+ \min\limits_{t\in[0,~1]}\left\{\frac{1 - t}{2}\hat{f}_{1}(x_{k}) + \frac{t^{2}L_{k}}{2\mu}\hat{f}_{2}(x_{k})\right\} = \varepsilon_{k} + \hat{f}_{1}(x_{k}) + \frac{\hat{f}_{2}(x_{k})L_{k}}{\mu}\min\limits_{t\in[0,~1]}\left\{\frac{-t\mu}{2\hat{f}_{1}(x_{k})L_{k}} + \frac{t^{2}}{2}\right\} = \varepsilon_{k} +\\
        &+ \hat{f}_{1}(x_{k}) - \frac{\hat{f}_{2}(x_{k})L_{k}}{\mu}\max\limits_{t\in[0,~1]}\left\{\frac{t\mu}{2\hat{f}_{1}(x_{k})L_{k}} - \frac{t^{2}}{2}\right\} = \left\{\text{\eqref{eq:aux_det_local_decrease_2_eq_1}, лемма \ref{lm:aux_det_local_decrease_2}}\right\} = \varepsilon_{k} + \hat{f}_{1}(x_{k}) -\\
        &- \frac{\hat{f}_{2}(x_{k})L_{k}}{\mu}\varkappa\left(\frac{\mu}{2\hat{f}_{1}(x_{k})L_{k}}\right)\leq\left\{\text{монотонное убывание по }L_{k}\right\}\leq\varepsilon_{k} + \hat{f}_{1}(x_{k}) -\\
        &- \frac{2\hat{f}_{2}(x_{k})L_{\hat{F}}}{\mu}\varkappa\left(\frac{\mu}{4\hat{f}_{1}(x_{k})L_{\hat{F}}}\right).
    \end{aligned}
\end{equation*}
Явно запишем получившееся неравенство в зависимости от $\varkappa(\cdot)$, учитывая монотонное убывание\\$\frac{\hat{f}_{2}(x_{k})L_{k}}{\mu}\varkappa\left(\frac{\mu}{2\hat{f}_{1}(x_{k})L_{k}}\right)$ по $L_{k}$ (следствие \ref{lm:cor_aux_det_local_decrease_2}):
\begin{equation*}
    \begin{aligned}
        &\hat{f}_{1}(x_{k + 1})\leq\varepsilon_{k} + \begin{sqcases}
            \hat{f}_{1}(x_{k}) - \frac{\mu}{16L_{\hat{F}}}\text{, если }\hat{f}_{1}(x_{k})\geq\frac{\mu}{4L_{\hat{F}}};\\[10pt]
            \frac{\hat{f}_{1}(x_{k})}{2} + \frac{\hat{f}_{2}(x_{k})L_{\hat{F}}}{\mu}\leq\frac{3}{4}\hat{f}_{1}(x_{k})\text{, если }\hat{f}_{1}(x_{k})\leq\frac{\mu}{4L_{\hat{F}}}.
        \end{sqcases}
    \end{aligned}
\end{equation*}
Для ограничения на $\hat{f}_{1}(x_{k + 1})$ при $L_{k}\equiv L_{\hat{F}}$ представление в зависимости от $\varkappa(\cdot)$ задаётся иначе:
$$\hat{f}_{1}(x_{k + 1})\leq\varepsilon_{k} + \hat{f}_{1}(x_{k}) - \frac{\hat{f}_{2}(x_{k})L_{\hat{F}}}{\mu}\varkappa\left(\frac{\mu}{2\hat{f}_{1}(x_{k})L_{\hat{F}}}\right).$$
В явном виде это означает следующее:
\begin{equation*}
    \begin{aligned}
        &\hat{f}_{1}(x_{k + 1})\leq\varepsilon_{k} + \begin{sqcases}
            \hat{f}_{1}(x_{k}) - \frac{\mu}{8L_{\hat{F}}}\text{, если }\hat{f}_{1}(x_{k})\geq\frac{\mu}{2L_{\hat{F}}};\\[10pt]
            \frac{\hat{f}_{1}(x_{k})}{2} + \frac{\hat{f}_{2}(x_{k})L_{\hat{F}}}{2\mu}\leq\frac{3}{4}\hat{f}_{1}(x_{k})\text{, если }\hat{f}_{1}(x_{k})\leq\frac{\mu}{2L_{\hat{F}}}.
        \end{sqcases}
    \end{aligned}
\end{equation*}
\end{proof}
\begin{re:th:corollary}\label{th:glob_sub_lin_and_lin_conv_cor}
Адаптивный подбор $\varepsilon_{k}\geq 0$ позволяет точно решить задачу \eqref{eq:main_opt_problem}.
Для этого введём определяющую погрешности последовательность величин $\left\{\delta_{k}\right\}_{k\in\mathbb{Z}_{+}}:~\frac{3}{4}\delta_{k} > \delta_{k + 1} > 0$, $\delta_{-1} = \frac{8}{3}\delta_{0}$, $\lim\limits_{k\rightarrow+\infty}\delta_{k} = 0$.
Дополнительно определим:
\begin{equation*}
    \begin{aligned}
        &\begin{sqcases}
            1.~\hat{f}_{1}(x_{-1}) \overset{\operatorname{def}}{=} \frac{\mu}{4L_{\hat{F}}},~d = 16\text{, для }L_{k}\in[L,~L_{\hat{F}}];\\[10pt]
            2.~\hat{f}_{1}(x_{-1}) \overset{\operatorname{def}}{=} \frac{\mu}{2L_{\hat{F}}},~d = 8\text{, для }L_{k}\equiv L_{\hat{F}}.
        \end{sqcases}
    \end{aligned}
\end{equation*}
Обозначим за $N\in\mathbb{Z}_{+}\cup\left\{-1\right\}$ минимальный номер итерации, на которой выполнена одна из двух цепочек неравенств (положим $N = -1$ в случае отсутствия такой итерации):
\begin{equation*}
    \begin{aligned}
        &\begin{sqcases}
            1.~\hat{f}_{1}(x_{N})\geq\frac{\mu}{4L_{\hat{F}}}\geq\hat{f}_{1}(x_{N + 1})\text{, для }L_{k}\in[L,~L_{\hat{F}}];\\[10pt]
            2.~\hat{f}_{1}(x_{N})\geq\frac{\mu}{2L_{\hat{F}}}\geq\hat{f}_{1}(x_{N + 1})\text{, для }L_{k}\equiv L_{\hat{F}}.
        \end{sqcases}
    \end{aligned}
\end{equation*}
Следующая стратегия выбора $\varepsilon_{k}$ позволяет получить сколь угодно точное приближение решения \eqref{eq:smooth_system}:
\begin{equation*}
    \begin{aligned}
        &\varepsilon_{k} = \begin{sqcases}
            \delta_{0}:~\delta_{0} < \frac{\mu}{dL_{\hat{F}}}\text{, при }k = 0;\\
            \delta_{k - 1} - \delta_{k}\text{, если }0 < k\leq N + 1;\\
            \frac{3}{4}\delta_{k - 1} - \delta_{k}\text{, если }k > N + 1.
        \end{sqcases}
    \end{aligned}
\end{equation*}
То есть с увеличением номера итерации погрешность поиска $x_{k + 1}$ убывает:
\begin{equation*}
    \begin{aligned}
        &\begin{sqcases}
            \hat{f}_{1}(x_{k})\leq2\delta_{0} - \delta_{k - 1} + \hat{f}_{1}(x_{0}) - \frac{k\mu}{dL_{\hat{F}}},\text{ если }0 < k\leq N + 1;\\[15pt]
            \hat{f}_{1}(x_{k})\leq\left(\frac{3}{4}\right)^{k - N - 1}\hat{f}_{1}(x_{N + 1}) + \delta_{N}\left(\frac{3}{4}\right)^{k - N - 1} - \delta_{k - 1},\text{ если }k > N + 1.
        \end{sqcases}
    \end{aligned}
\end{equation*}
Данные оценки выведены с помощью сложения соответствующих неравенств друг с другом из условия теоремы для $k\in\overline{0, N + 1}$ и раскрытия рекуррентной зависимости для $k > N + 1$, к полученным выражениям применены значения $\varepsilon_{k}$.
\end{re:th:corollary}
\begin{re:th:corollary}\label{th:glob_sub_lin_and_lin_conv_cor_2}
В случае постоянной погрешности $\varepsilon_{k} = \varepsilon > 0$ для достижения уровня функции\\$\hat{f}_{1}(x_{k})\leq\epsilon$ количество необходимых в худшем случае итераций и максимальное значение погрешности зависят от стратегии поиска $L_{k}$. Для $L_{k}\in[L,~2L_{\hat{F}}]$ условия следующие:
\begin{itemize}
    \item если $\epsilon\geq\frac{\mu}{4L_{\hat{F}}}$, то $k\geq\left\lceil\left(\frac{\mu}{16L_{\hat{F}}} - \varepsilon\right)^{-1}\left(\hat{f}_{1}(x_{0}) - \epsilon\right)\mathds{1}_{\left\{\hat{f}_{1}(x_{0}) > \epsilon\right\}}\right\rceil$, $\varepsilon < \frac{\mu}{16L_{\hat{F}}}$;
    \item если $\epsilon < \frac{\mu}{4L_{\hat{F}}}$, то $k\geq\left\lceil\left(\frac{\mu}{16L_{\hat{F}}} - \varepsilon\right)^{-1}\left(\hat{f}_{1}(x_{0}) - \frac{\mu}{4L_{\hat{F}}}\right)\mathds{1}_{\left\{\hat{f}_{1}(x_{0}) > \frac{\mu}{4L_{\hat{F}}}\right\}} + \log_{\frac{4}{3}}\left(\frac{\mu}{4r\epsilon L_{\hat{F}}}\right)\right\rceil$, $\varepsilon\leq\frac{(1 - r)\epsilon}{4}$,\\$r\in(0, 1)$.
\end{itemize}
Для точно известного значения $L_{k} = L_{\hat{F}}$ количество необходимых итераций меньше и допустимая погрешность больше:
\begin{itemize}
    \item если $\epsilon\geq\frac{\mu}{2L_{\hat{F}}}$, то $k\geq\left\lceil\left(\frac{\mu}{8L_{\hat{F}}} - \varepsilon\right)^{-1}\left(\hat{f}_{1}(x_{0}) - \epsilon\right)\mathds{1}_{\left\{\hat{f}_{1}(x_{0}) > \epsilon\right\}}\right\rceil$, $\varepsilon < \frac{\mu}{8L_{\hat{F}}}$;
    \item если $\epsilon < \frac{\mu}{2L_{\hat{F}}}$, то $k\geq\left\lceil\left(\frac{\mu}{8L_{\hat{F}}} - \varepsilon\right)^{-1}\left(\hat{f}_{1}(x_{0}) - \frac{\mu}{2L_{\hat{F}}}\right)\mathds{1}_{\left\{\hat{f}_{1}(x_{0}) > \frac{\mu}{2L_{\hat{F}}}\right\}} + \log_{\frac{4}{3}}\left(\frac{\mu}{2r\epsilon L_{\hat{F}}}\right)\right\rceil$, $\varepsilon\leq\frac{(1 - r)\epsilon}{4}$,\\$r\in(0, 1)$.
\end{itemize}
\end{re:th:corollary}

В теореме \ref{th:track_glob_bound_main} рассматривается точный оракул, выдающий при каждом вызове $x_{k + 1}$ с нулевой погрешностью, для которого установлена оценка на радиус множества уровня $\hat{f}_{1}(x_{k})$ в случае решения совместной системы уравнений \eqref{eq:smooth_system}.

\begin{re:theorem}\label{th:track_glob_bound}
    Пусть выполнены предположения \ref{as:det_hat_F_der_smooth} и \ref{as:det_hat_F_PL_condition} для метода Гаусса--Ньютона со схемой реализации \ref{alg:gen_det_gnm}, в которой $\tau_{k} = \hat{f}_{1}(x_{k})$, $\varepsilon_{k} = 0$, $k\in\mathbb{Z}_{+}$. Тогда существует решение $x^{*}\in\mathcal{F}$ задачи \eqref{eq:smooth_system}, такое, что $\hat{f}_{1}(x^{*}) = 0$ и $\|x_{0} - x^{*}\|\leq4\hat{f}_{1}(x_{0})\sqrt{\frac{2L_{\hat{F}}}{\mu L}}.$
\end{re:theorem}
\begin{proof}
Предположим, что $\hat{f}_{1}(x_{0}) > \frac{\mu}{4L_{\hat{F}}}$. Это означает, что $\hat{f}_{1}(x_{k + 1})\leq\hat{f}_{1}(x_{k}) - \frac{\mu}{16L_{\hat{F}}}$, ровно до той итерации, на которой $\hat{f}_{1}(x_{k}) \leq \frac{\mu}{4L_{\hat{F}}}$. Обозначим наименьшую по номеру такую итерацию как $N\in\mathbb{Z}_{+}$:
$$\hat{f}_{1}(x_{N})\geq\frac{\mu}{4L_{\hat{F}}}\geq\hat{f}_{1}(x_{N + 1}).$$
По теореме \ref{th:glob_sub_lin_and_lin_conv}:
\begin{equation}\label{eq:track_glob_bound_eq1}
    \hat{f}_{1}(x_{0}) - \hat{f}_{1}(x_{N + 1})\geq\frac{(N + 1)\mu}{16L_{\hat{F}}}\Rightarrow\frac{16L_{\hat{F}}}{\mu}\left(\hat{f}_{1}(x_{0}) - \hat{f}_{1}(x_{N + 1})\right)\geq N + 1.
\end{equation}
По лемме \ref{lm:aux_det_local_decrease_1} (следствие \ref{lm:cor_2_aux_det_local_decrease_1}):
\begin{equation}\label{eq:track_glob_bound_eq3}
    \begin{aligned}
        \hat{f}_{1}(x_{k}) - \hat{f}_{1}(x_{k + 1})&\geq\frac{L_{k}}{2}\left\|x_{k + 1} - x_{k}\right\|^{2}\geq\frac{L}{2}\left\|x_{k + 1} - x_{k}\right\|^{2}\Rightarrow\hat{f}_{1}(x_{0}) - \hat{f}_{1}(x_{N + 1})\geq\frac{L}{2}\sum\limits_{k = 0}^{N}\|x_{k + 1} - x_{k}\|^{2}\geq\\
        &\geq\left\{\text{неравенство Йенсена}\right\}\geq\frac{L}{2(N + 1)}\left(\sum\limits_{k = 0}^{N}\|x_{k + 1} - x_{k}\|\right)^{2}\geq\frac{L}{2(N + 1)}\|x_{0} - x_{N + 1}\|^{2}\Rightarrow\\
        &\Rightarrow\|x_{0} - x_{N + 1}\|\leq\sqrt{\frac{2(N + 1)}{L}\left(\hat{f}_{1}(x_{0}) - \hat{f}_{1}(x_{N + 1})\right)}\leq\left\{\text{\eqref{eq:track_glob_bound_eq1}}\right\}\leq\\
        &\leq4\sqrt{\frac{2L_{\hat{F}}}{\mu L}}\left(\hat{f}_{1}(x_{0}) - \hat{f}_{1}(x_{N + 1})\right).
    \end{aligned}
\end{equation}
Соответственно, для $N + 1$ и более поздних итераций выполняется неравенство:
\begin{equation}\label{eq:track_glob_bound_eq2}
    \begin{aligned}
        &\hat{f}_{1}(x_{N + k + 2})\leq\frac{\hat{f}_{1}(x_{N + k + 1})}{2} + \frac{L_{\hat{F}}\hat{f}_{2}(x_{N + k + 1})}{\mu}\leq\left(\frac{3}{4}\right)^{k + 1}\hat{f}_{1}(x_{N + 1}),~k\in\mathbb{Z}_{+}.
    \end{aligned}
\end{equation}
Рассмотрим локальную модель для точки $x_{k + 1} = T_{L_{k}, \hat{f}_{1}(x_{k})}(x_{k}),~k\in\mathbb{Z}_{+}$:
\begin{equation}\label{eq:track_glob_bound_eq4}
    \begin{aligned}
        \frac{\hat{f}_{1}(x_{k})}{2} + \frac{L}{2}\|x_{k + 1} - x_{k}\|^{2}&\leq\frac{\hat{f}_{1}(x_{k})}{2} + \frac{L_{k}}{2}\|x_{k + 1} - x_{k}\|^{2}\leq\psi_{x_{k}, L_{k}, \hat{f}_{1}(x_{k})}(x_{k + 1})\leq\\
        &\leq\left\{\text{заменим }T_{L_{k}, \hat{f}_{1}(x_{k})}(x_{k}) - x_{k}\text{ на }h = -\hat{F}^{'}(x_{k})^{*}\left(\hat{F}^{'}(x_{k})\hat{F}^{'}(x_{k})^{*}\right)^{-1}\hat{F}(x_{k})\right\}\leq\\
        &\leq\frac{\hat{f}_{1}(x_{k})}{2} + \frac{L_{k}}{2}\|h\|^{2}\leq\left\{\text{неравенство \eqref{eq:th_glob_lin_conv_eq_1}},~L_{k}\leq 2L_{\hat{F}}\right\}\leq\frac{\hat{f}_{1}(x_{k})}{2} + \frac{L_{\hat{F}}\hat{f}_{2}(x_{k})}{\mu}\Rightarrow\\
        &\Rightarrow\left\{\text{\eqref{eq:track_glob_bound_eq2}},~k := N + k + 1\right\}\Rightarrow\|x_{N + k + 2} - x_{N + k + 1}\|\leq\hat{f}_{1}(x_{N + k + 1})\sqrt{\frac{2L_{\hat{F}}}{\mu L}}\leq\\
        &\leq\left(\frac{3}{4}\right)^{k}\hat{f}_{1}(x_{N + 1})\sqrt{\frac{2L_{\hat{F}}}{\mu L}}.
    \end{aligned}
\end{equation}
Объединим результаты в \eqref{eq:track_glob_bound_eq3} и \eqref{eq:track_glob_bound_eq4}:
\begin{equation*}
    \begin{aligned}
        \|x_{0} - x^{*}\|&\leq\|x_{0} - x_{N + 1}\| + \sum\limits_{k = 0}^{+\infty}\|x_{N + k + 1} - x_{N + k + 2}\|\leq\\
        &\leq4\sqrt{\frac{2L_{\hat{F}}}{\mu L}}\left(\hat{f}_{1}(x_{0}) - \hat{f}_{1}(x_{N + 1}) + \frac{\hat{f}_{1}(x_{N + 1})}{4}\sum\limits_{k = 0}^{+\infty}\left(\frac{3}{4}\right)^{k}\right)= 4\hat{f}_{1}(x_{0})\sqrt{\frac{2L_{\hat{F}}}{\mu L}}.
    \end{aligned}
\end{equation*}
Если же изначально было так, что $\hat{f}_{1}(x_{0})\leq\frac{\mu}{4L_{\hat{F}}}$, то положим $N = -1$ и оценим сверху $\|x_{0} - x^{*}\|$:
\begin{equation*}
    \begin{aligned}
        &\|x_{0} - x^{*}\|\leq\sum\limits_{k = 0}^{+\infty}\|x_{k} - x_{k + 1}\|\leq\left\{\text{\eqref{eq:track_glob_bound_eq4}}\right\}\leq\hat{f}_{1}(x_{0})\sqrt{\frac{2L_{\hat{F}}}{\mu L}}\sum\limits_{k = 0}^{+\infty}\left(\frac{3}{4}\right)^{k} = 4\hat{f}_{1}(x_{0})\sqrt{\frac{2L_{\hat{F}}}{\mu L}}.
    \end{aligned}
\end{equation*}
\end{proof}
\begin{re:th:corollary}\label{th:track_glob_bound_cor}
В условии теоремы начальную итерацию можно заменить на $k$--ую, что позволяет установить единственность $x^{*}$ для данной $\left\{x_{k}\right\}_{k\in\mathbb{Z}_{+}}$: пусть существуют $x_{1}^{*}$ и $x_{2}^{*}$, такие, что\\$\hat{F}(x_{1}^{*}) = \mathbf{0}_{m}$ и $\hat{F}(x_{2}^{*}) = \mathbf{0}_{m}$, но $x_{1}^{*}\neq x_{2}^{*}$. Имеем для $\left\{x_{k}\right\}_{k\in\mathbb{Z}_{+}}$ в силу доказанной сходимости (теоремы \ref{th:DetSublinConv} и \ref{th:glob_sub_lin_and_lin_conv})
\begin{equation*}
    \begin{cases}
        &\lim\limits_{k\rightarrow+\infty}\|x_{k} - x_{1}^{*}\|\leq\lim\limits_{k\rightarrow+\infty}4\hat{f}_{1}(x_{k})\sqrt{\frac{2L_{\hat{F}}}{\mu L}} = 0;\\[10pt]
        &\lim\limits_{k\rightarrow+\infty}\|x_{k} - x_{2}^{*}\|\leq\lim\limits_{k\rightarrow+\infty}4\hat{f}_{1}(x_{k})\sqrt{\frac{2L_{\hat{F}}}{\mu L}} = 0.
    \end{cases}
\end{equation*}
Однако выражения выше устанавливают равномерную ограниченность расстояния до $x_{1}^{*}$ и $x_{2}^{*}$ от $x_{k}$ с монотонно невозрастающим значением расстояния, в пределе равном $0$, что означает $x_{1}^{*} = x_{2}^{*}$. Более того, в силу произвольности последовательности $\left\{x_{k}\right\}_{k\in\mathbb{Z}_{+}}$ для любого $x_{0}\in\mathcal{F}$ в пределе будет единственное решение при использовании метода Гаусса--Ньютона по схеме \ref{alg:gen_det_gnm} с $\tau_{k} = \hat{f}_{1}(x_{k})$, $\varepsilon_{k} = 0$, но для каждого $x_{0}$ это решение $x^{*}$ своё.
\end{re:th:corollary}

Следующая лемма предлагает алгоритм подбора значения $\tau_{k}$ в области суперлинейной сходимости, не требующий явно знание расстояния до решения, согласно теореме \ref{th:DetQuadConv}.

\begin{lemma}\label{lm:aux_tau_schedule}
    Пусть выполнены условия теоремы \ref{th:DetQuadConv}, дополнительно предположим ограниченность нормы матрицы Якоби: существует $ M_{\hat{F}} > 0$, для которого выполнено $\left\|\hat{F}^{'}(x)\right\|\leq M_{\hat{F}}$ при всех $x\in \mathcal{F}$. Тогда для
    $$\varepsilon_{k} = 0,~\tau_{k} = \hat{f}_{1}(x_{k}),~k\in\mathbb{Z}_{+}$$
    в области
    \begin{equation*}
        \begin{aligned}
            \|x_{k} - x^{*}\|\leq\min\left\{\frac{2\varsigma}{5L_{\hat{F}}},~\frac{1}{12L_{\hat{F}}}\left(\left(3M_{\hat{F}} + 5\varsigma\right) - \sqrt{\left(3M_{\hat{F}} + 5\varsigma\right)^{2} - 24\varsigma^{2}}\right)\right\}
        \end{aligned}
    \end{equation*}
    выполнена оценка:
    \begin{equation*}
        \begin{aligned}
            \|x_{k + 1} - x^{*}\|&\leq\frac{\frac{3L_{\hat{F}}\|x_{k} - x^{*}\|^{2}}{2} + \|x_{k} - x^{*}\|\sqrt{\tau_{k}L_{k} + \frac{L_{\hat{F}}^{2}\|x_{k} - x^{*}\|^{2}}{4}}}{\varsigma - L_{\hat{F}}\|x_{k} - x^{*}\|} < \|x_{k} - x^{*}\|.
        \end{aligned}
    \end{equation*}
\end{lemma}
\begin{proof}
Для начала выведем оценку сверху у значения функции $\hat{f}_{1}(x_{k})$ с помощью локальной модели $\psi_{x^{*}, L_{\hat{F}}, \phi(x^{*}, x_{k})}(x_{k})$:
\begin{equation}\label{eq:distance_lower_bound}
    \begin{aligned}
        \hat{f}_{1}(x_{k})&\leq\left\{\text{лемма \ref{lm:aux_det_upper_model}}\right\}\leq\left\|\underbrace{\hat{F}(x^{*})}_{=0} + \hat{F}^{'}(x^{*})(x_{k} - x^{*})\right\| + \frac{L_{\hat{F}}}{2}\left\|x_{k} - x^{*}\right\|^{2}\leq\underbrace{\left\|\hat{F}^{'}(x^{*})(x_{k} - x^{*})\right\|}_{\leq M_{\hat{F}}\|x_{k} - x^{*}\|} + \frac{L_{\hat{F}}}{2}\|x_{k} - x^{*}\|^{2}\leq\\
        &\leq M_{\hat{F}}\|x_{k} - x^{*}\| + \frac{L_{\hat{F}}}{2}\|x_{k} - x^{*}\|^{2}<\left\{\substack{\text{верхняя граница  области}\\\text{суперлинейной сходимости:}}~\|x_{k} - x^{*}\| < \frac{\varsigma}{L_{\hat{F}}}\right\}<\\
        &< \left(M_{\hat{F}} + \frac{\varsigma}{2}\right)\|x_{k} - x^{*}\|\leq\left\{\varsigma\leq\sigma_{\min}(\hat{F}^{'}(x^{*}))\leq\sigma_{\max}(\hat{F}^{'}(x^{*}))\leq M_{\hat{F}}\right\}\leq\frac{3M_{\hat{F}}}{2}\|x_{k} - x^{*}\|.
    \end{aligned}
\end{equation}
Подставим значения $\tau_{k}$ и $\varepsilon_{k}$ в оценку сходимости, положив $t_{k} = \|x_{k} - x^{*}\|$:
\begin{equation*}
    \begin{aligned}
        t_{k + 1}&\leq\frac{\frac{3L_{\hat{F}}t_{k}^{2}}{2} + t_{k}\sqrt{\hat{f}_{1}(x_{k})L_{k} + \frac{L_{\hat{F}}^{2}t_{k}^{2}}{4}}}{\varsigma - L_{\hat{F}}\|x_{k} - x^{*}\|}<\left\{L_{k}\leq2L_{\hat{F}},~\text{оценка \eqref{eq:distance_lower_bound}}\right\}<\\
        &<t_{k}\underbrace{\left(\frac{\frac{3L_{\hat{F}}t_{k}}{2} + \sqrt{3M_{\hat{F}}L_{\hat{F}}t_{k} + \frac{L_{\hat{F}}^{2}t_{k}^{2}}{4}}}{\varsigma - L_{\hat{F}}t_{k}}\right)}_{\in[0, 1]\text{ --- по условию леммы}}\leq t_{k}.
    \end{aligned}
\end{equation*}
Выведем границы допустимых значений $t_{k}$ из ограничения на дробь выше:
\begin{equation*}
    \begin{aligned}
        &0\leq\frac{3L_{\hat{F}}t_{k}}{2} + \sqrt{3M_{\hat{F}}L_{\hat{F}}t_{k} + \frac{L_{\hat{F}}^{2}t_{k}^{2}}{4}}\leq\varsigma - L_{\hat{F}}t_{k}\Rightarrow0\leq\sqrt{3M_{\hat{F}}L_{\hat{F}}t_{k} + \frac{L_{\hat{F}}^{2}t_{k}^{2}}{4}}\leq\varsigma - \frac{5L_{\hat{F}}t_{k}}{2}\Rightarrow t_{k}\leq\frac{2\varsigma}{5L_{\hat{F}}},
    \end{aligned}
\end{equation*}
получено первое ограничение. Для вывода оставшихся ограничений возведём в квадрат неравенство выше:
\begin{equation*}
    \begin{aligned}
        &3M_{\hat{F}}L_{\hat{F}}t_{k} + \frac{L_{\hat{F}}^{2}t_{k}^{2}}{4}\leq\left(\varsigma - \frac{5L_{\hat{F}}t}{2}\right)^{2}\Rightarrow-6L_{\hat{F}}^{2}t_{k}^{2} + \left(3M_{\hat{F}}L_{\hat{F}} + 5L_{\hat{F}}\varsigma\right)t_{k} - \varsigma^{2}\leq0.
    \end{aligned}
\end{equation*}
Из квадратного уравнения выводится необходимый отрезок значений $t_{k}\geq0$, согласованный с полученным выше первым ограничением:
\begin{equation*}
    \begin{aligned}
        0&\leq t_{k}\leq\frac{1}{12L_{\hat{F}}}\left(\left(3M_{\hat{F}} + 5\varsigma\right) - \sqrt{\left(3M_{\hat{F}} + 5\varsigma\right)^{2} - 24\varsigma^{2}}\right)\Rightarrow\\
        &\Rightarrow\left\|x_{k} - x^{*}\right\|\leq\min\left\{\frac{2\varsigma}{5L_{\hat{F}}},~\frac{1}{12L_{\hat{F}}}\left(\left(3M_{\hat{F}} + 5\varsigma\right) - \sqrt{\left(3M_{\hat{F}} + 5\varsigma\right)^{2} - 24\varsigma^{2}}\right)\right\}.
    \end{aligned}
\end{equation*}
Таким образом, получена нижняя оценка на радиус сходимости, при котором будет суперлинейная сходимость с обозначенным выбором $\varepsilon_{k}$, $\tau_{k}$, $k\in\mathbb{Z}_{+}$.
\end{proof}

Теорема \ref{th:1_main} использует явную формулу проксимального отображения для вывода оценки на убывание оптимизируемого функционала в процессе, организованном по схеме \ref{alg:gen_det_gnm}.

\begin{re:theorem}\label{th:1}
    Пусть выполнены предположения \ref{as:det_hat_F_der_smooth} и \ref{as:det_hat_F_PL_condition}. Рассмотрим последовательность $\{x_{k}\}_{k\in\mathbb{Z}_{+}}$, вычисляемую по схеме \ref{alg:gen_det_gnm} с правилом \eqref{eq:det_expl_update_rule}, в котором $\tau_{k} > 0$, $\eta_{k}\in(0, 2)$, $k\in\mathbb{Z}_{+}$. Тогда для $k\in\mathbb{Z}_{+}$:
    \begin{equation}\label{eq:det_lin_conv_1}
        \hat{f}_{1}(x_{k + 1}) \leq \frac{\tau_{k}}{2} + \begin{sqcases}
        \frac{\hat{f}_{2}(x_{k})}{2\tau_{k}}\left(1 - \frac{\eta_{k}(2 - \eta_{k})\mu}{L_{k} \tau_{k} + \mu}\right),~\tau_{k} \geq \frac{\mu}{L_{k}};\\[10pt]
        \frac{\hat{f}_{2}(x_{k})(1 - \eta_{k})^{2}}{2\tau_{k}} + \frac{\eta_{k}(2 - \eta_{k})L_{k}\hat{f}_{2}(x_{k})}{2\mu} - \frac{\eta_{k}(2 - \eta_{k})L_{k}^{2}\hat{f}_{2}(x_{k})\tau_{k}}{2\mu^{2}(1 + \xi)^{3}}\text{, для}\\[5pt]
        \text{~~~~~некоторого }\xi\in(-1, 1]\text{ при }\tau_{k} < \frac{\mu}{L_{k}}.
        \end{sqcases}
    \end{equation}
    При этом для $\eta_{k} = 1,~k\in\mathbb{Z}_{+}$ верно:
    \begin{equation}\label{eq:det_lin_conv_2}
        \hat{f}_{1}(x_{k + 1}) \leq \frac{\tau_{k}}{2} + \frac{L_{\hat{F}}}{\mu}\hat{f}_{2}(x_{k}).
    \end{equation}
\end{re:theorem}
\begin{proof}
По определению $\psi_{x,L,\tau}(y)$ (лемма \ref{lm:aux_det_upper_model}):
\begin{equation}\label{eq:delta_f_1}
    \begin{aligned}
        \hat{f}_{1}(x_{k}) - \hat{f}_{1}(x_{k + 1})&\geq\hat{f}_{1}(x_{k}) - \psi_{x_{k},L_{k},\tau_{k}}(x_{k + 1}) = \hat{f}_{1}(x_{k}) - \frac{\tau_{k}}{2} - \frac{L_{k}}{2}\|x_{k + 1} - x_{k}\|^{2} -\\
        &- \frac{1}{2\tau_{k}}\left\|\hat{F}(x_{k}) + \hat{F}^{'}(x_{k})(x_{k + 1} - x_{k})\right\|^{2}.
    \end{aligned}
\end{equation}
Подставим выражение $x_{k + 1}$ в \eqref{eq:delta_f_1}:
\begin{equation*}
    \begin{aligned}
        &\hat{f}_{1}(x_{k}) - \hat{f}_{1}(x_{k + 1})\geq\hat{f}_{1}(x_{k}) - \frac{\tau_{k}}{2} - \frac{1}{2\tau_{k}}\left\|\hat{F}(x_{k}) + \hat{F}^{'}(x_{k})(x_{k + 1} - x_{k})\right\|^{2} - \frac{L_{k}}{2}\left\|x_{k + 1}- x_{k}\right\|^{2} = \hat{f}_{1}(x_{k}) - \frac{\tau_{k}}{2} -\\
        &- \frac{1}{2\tau_{k}}\left\|\hat{F}(x_{k}) - \eta_{k}\hat{F}^{'}(x_{k})\left(L_{k}\tau_{k}I_{n} + \hat{F}^{'}(x_{k})^{*}\hat{F}^{'}(x_{k})\right)^{-1}\hat{F}^{'}(x_{k})^{*}\hat{F}(x_{k})\right\|^{2} -\\
        &-\frac{L_{k}}{2}\left\|\eta_{k}\left(L_{k}\tau_{k}I_{n} + \hat{F}^{'}(x_{k})^{*}\hat{F}^{'}(x_{k})\right)^{-1}\hat{F}^{'}(x_{k})^{*}\hat{F}(x_{k})\right\|^{2} = \hat{f}_{1}(x_{k}) - \frac{\tau_{k}}{2} - \frac{\hat{f}_{2}(x_{k})}{2\tau_{k}} +\\
        &+ \frac{1}{2\tau_{k}}\left(2\left\langle \eta_{k}\left(L_{k}\tau_{k}I_{n} + \hat{F}^{'}(x_{k})^{*}\hat{F}^{'}(x_{k})\right)^{-1}\hat{F}^{'}(x_{k})^{*}\hat{F}(x_{k}), \hat{F}^{'}(x_{k})^{*}\hat{F}(x_{k})\right\rangle -\right.\\
        &\left.-\left\langle\eta_{k}\left(L_{k}\tau_{k}I_{n} + \hat{F}^{'}(x_{k})^{*}\hat{F}^{'}(x_{k})\right)^{-1}\hat{F}^{'}(x_{k})^{*}\hat{F}(x_{k}),\right.\right.\\
        &\left.\left.~~~~~~~~\hat{F}^{'}(x_{k})^{*}\hat{F}^{'}(x_{k})\eta_{k}\left(L_{k}\tau_{k}I_{n} + \hat{F}^{'}(x_{k})^{*}\hat{F}^{'}(x_{k})\right)^{-1}\hat{F}^{'}(x_{k})^{*}\hat{F}(x_{k})\right\rangle -\right.\\
    \end{aligned}
\end{equation*}
\begin{equation}\label{eq:delta_f_2}
    \begin{aligned}
        &\left.-L_{k}\tau_{k}\left\langle\eta_{k}\left(L_{k}\tau_{k}I_{n} + \hat{F}^{'}(x_{k})^{*}\hat{F}^{'}(x_{k})\right)^{-1}\hat{F}^{'}(x_{k})^{*}\hat{F}(x_{k}),~\eta_{k}\left(L_{k}\tau_{k}I_{n} + \hat{F}^{'}(x_{k})^{*}\hat{F}^{'}(x_{k})\right)^{-1}\hat{F}^{'}(x_{k})^{*}\hat{F}(x_{k})\right\rangle\right) =\\
        &= \hat{f}_{1}(x_{k}) - \frac{\tau_{k}}{2} - \frac{\hat{f}_{2}(x_{k})}{2\tau_{k}} + \frac{\eta_{k}(2 - \eta_{k})}{2\tau_{k}}\left\langle\left(L_{k}\tau_{k}I_{n} + \hat{F}^{'}(x_{k})^{*}\hat{F}^{'}(x_{k})\right)^{-1}\hat{F}^{'}(x_{k})^{*}\hat{F}(x_{k}),~\hat{F}^{'}(x_{k})^{*}\hat{F}(x_{k})\right\rangle =\\
        &= \hat{f}_{1}(x_{k}) - \frac{\tau_{k}}{2} - \frac{\hat{f}_{2}(x_{k})}{2\tau_{k}} + \frac{\eta_{k}(2 - \eta_{k})}{2\tau_{k}}\left\langle \hat{F}^{'}(x_{k})\left(L_{k}\tau_{k}I_{n} + \hat{F}^{'}(x_{k})^{*}\hat{F}^{'}(x_{k})\right)^{-1}\hat{F}^{'}(x_{k})^{*}\hat{F}(x_{k}),~\hat{F}(x_{k})\right\rangle.
    \end{aligned}
\end{equation}
Тогда $\hat{f}_{1}(x_{k + 1})$ имеет верхнюю оценку по \eqref{eq:delta_f_2}:
\begin{equation*}
    \begin{aligned}
        \hat{f}_{1}(x_{k + 1})&\leq\frac{\tau_{k}}{2} + \frac{\hat{f}_{2}(x_{k})}{2\tau_{k}} - \frac{\eta_{k}(2 - \eta_{k})}{2\tau_{k}}\left\langle \hat{F}^{'}(x_{k})\left(L_{k}\tau_{k}I_{n} + \hat{F}^{'}(x_{k})^{*}\hat{F}^{'}(x_{k})\right)^{-1}\hat{F}^{'}(x_{k})^{*}\hat{F}(x_{k}), \hat{F}(x_{k})\right\rangle\leq\\
        &\leq\left\{\text{лемма \ref{lm:aux_matrix_order}}\right\}\leq\frac{\tau_{k}}{2} + \frac{\hat{f}_{2}(x_{k})}{2\tau_{k}} - \frac{\eta_{k}(2 - \eta_{k})}{2\tau_{k}}\left(\frac{\|\hat{F}(x_{k})\|^{2}\mu}{L_{k}\tau_{k} + \mu}\right) = \frac{\tau_{k}}{2} + \frac{\hat{f}_{2}(x_{k})}{2\tau_{k}}\left(1 - \frac{\eta_{k}(2 - \eta_{k})\mu}{L_{k}\tau_{k} + \mu}\right).
    \end{aligned}
\end{equation*}
Первая часть \eqref{eq:det_lin_conv_1} выведена, рассмотрим  $\tau_{k} < \frac{\mu}{L_{k}}$:
\begin{equation*}
    \begin{aligned}
        \hat{f}_{1}(x_{k + 1})&\leq\frac{\tau_{k}}{2} + \frac{\hat{f}_{2}(x_{k})}{2\tau_{k}}\left(1 - \frac{\eta_{k}(2 - \eta_{k})\mu}{L_{k}\tau_{k} + \mu}\right) = \frac{\tau_{k}}{2} + \frac{\hat{f}_{2}(x_{k})}{2\tau_{k}}\left(1 - \eta_{k}(2 - \eta_{k})\left(1 + \frac{L_{k}\tau_{k}}{\mu}\right)^{-1}\right) =\\
        &=\left\{(1 + y)^{-1} = 1 - y + \frac{y^{2}}{(1 + \xi)^{3}},~|y| < 1,\text{ для некоторого }\xi\in(-1, 1]\right\} = \frac{\tau_{k}}{2} + \frac{\hat{f}_{2}(x_{k})(1 - \eta_{k})^{2}}{2\tau_{k}} +\\
        &+ \frac{\eta_{k}(2 - \eta_{k})L_{k}\hat{f}_{2}(x_{k})}{2\mu} - \frac{\eta_{k}(2 - \eta_{k})L_{k}^{2}\hat{f}_{2}(x_{k})\tau_{k}}{2\mu^{2}(1 + \xi)^{3}},
    \end{aligned}
\end{equation*}
вторая часть \eqref{eq:det_lin_conv_1} по формуле Тейлора с остаточным членом в форме Лагранжа выведена.
Докажем \eqref{eq:det_lin_conv_2}, согласно доказанному неравенству \eqref{eq:det_lin_conv_1}:
\begin{equation*}
    \begin{aligned}
        \hat{f}_{1}(x_{k + 1}) &\leq \frac{\tau_{k}}{2} + \begin{sqcases}
            \begin{aligned}
                \frac{\hat{f}_{2}(x_{k})}{2\tau_{k}}\underbrace{\left(1 - \frac{\eta_{k}(2 - \eta_{k})\mu}{L_{k} \tau_{k} + \mu}\right)}_{\in(0,~1)}&\leq\frac{\hat{f}_{2}(x_{k})}{2\tau_{k}}\leq\frac{\hat{f}_{2}(x_{k})L_{k}}{2\mu}\leq\left\{L_{k}\leq2L_{\hat{F}}\right\}\leq\\
                &\leq\frac{\hat{f}_{2}(x_{k})L_{\hat{F}}}{\mu},~\tau_{k} \geq \frac{\mu}{L_{k}};\\
                \vspace{1pt}
            \end{aligned}\\
            \begin{aligned}
                \frac{\hat{f}_{2}(x_{k})(1 - \eta_{k})^{2}}{2\tau_{k}} &+ \frac{\eta_{k}(2 - \eta_{k})L_{k}\hat{f}_{2}(x_{k})}{2\mu} - \underbrace{\frac{\eta_{k}(2 - \eta_{k})L_{k}^{2}\hat{f}_{2}(x_{k})\tau_{k}}{2\mu^{2}(1 + \xi)^{3}}}_{\geq0}\leq\left\{\eta_{k} = 1,~L_{k}2L_{\hat{F}}\right\}\leq\\
                &\leq\frac{\hat{f}_{2}(x_{k})L_{\hat{F}}}{\mu}\text{, иначе.}
            \end{aligned}
        \end{sqcases}
    \end{aligned}
\end{equation*}
\end{proof}

Теорема \ref{th:2_main} использует явную формулу проксимального отображения для вывода условий сходимости к стационарной точке.

\begin{re:theorem}\label{th:2}
    Пусть выполнено предположение \ref{as:det_hat_F_der_smooth}. Рассмотрим последовательность $\{x_{k}\}_{k\in\mathbb{Z}_{+}}$, вычисляемую по схеме \ref{alg:gen_det_gnm} c правилом \eqref{eq:det_expl_update_rule},  в котором $\tau_{k} > 0$ и $\eta_{k} > 0$, $\eta_{k}(2 - \eta_{k})\geq c > 0$, $k\in\mathbb{Z}_{+}$. Дополнительно предположим ограниченность нормы матрицы Якоби: существует $M_{\hat{F}} > 0$, для которого выполнено $\left\|\hat{F}^{'}(x)\right\|\leq M_{\hat{F}}$ при всех $x\in \mathcal{F}$. Тогда при $k\in\mathbb{N}$:
    \begin{equation*}
        \begin{aligned}
            \min\limits_{i\in\overline{0, k - 1}}\left\{\left\|\nabla\hat{f}_{2}(x_{i})\right\|^{2}\right\}&\leq\frac{8\left(2L_{\hat{F}}\max\limits_{i\in\overline{0, k - 1}}\left\{\tau_{i}\right\} + M_{\hat{F}}^{2}\right)}{\eta(2 - \eta)k}\sum\limits_{i = 0}^{k - 1}\left(\frac{1}{2}\left(\tau_{i} - \hat{f}_{1}(x_{i})\right)^{2} + \tau_{i}\left(\hat{f}_{1}(x_{i}) - \hat{f}_{1}(x_{i + 1})\right)\right),\\
            \eta&\in\Argmin\limits_{k\in\mathbb{Z}_{+}}\left\{\eta_{k}(2 - \eta_{k})\right\}.
        \end{aligned}
    \end{equation*}
\end{re:theorem}
\begin{proof}
Используя рассуждения из теоремы \ref{th:1} получаем следующее:
\begin{equation}\label{eq:der_f_1}
    \begin{aligned}
        \hat{f}_{1}(x_{k}) - \hat{f}_{1}(x_{k + 1})&\geq\hat{f}_{1}(x_{k}) - \frac{\tau_{k}}{2} - \frac{\hat{f}_{2}(x_{k})}{2\tau_{k}} +\\
        &+ \frac{\eta_{k}(2 - \eta_{k})}{2\tau_{k}}\left\langle\left(L_{k}\tau_{k}I_{n} + \hat{F}^{'}(x_{k})^{*}\hat{F}^{'}(x_{k})\right)^{-1}\hat{F}^{'}(x_{k})^{*}\hat{F}(x_{k}), \hat{F}^{'}(x_{k})^{*}\hat{F}(x_{k})\right\rangle\geq\\
        &\geq\left\{\text{ограничение спектра матрицы снизу},~\nabla\hat{f}_{2}(x_{k}) = 2\hat{F}^{'}(x_{k})^{*}\hat{F}(x_{x})\right\}\geq\\
        &\geq\hat{f}_{1}(x_{k}) - \frac{\tau_{k}}{2} - \frac{\hat{f}_{2}(x_{k})}{2\tau_{k}} + \frac{\eta_{k}(2 - \eta_{k})}{2\tau_{k}}\left(\frac{\left\|\nabla\hat{f}_{2}(x_{k})\right\|^{2}}{4(L_{k}\tau_{k} + M_{\hat{F}}^{2})}\right).
    \end{aligned}
\end{equation}
Из \eqref{eq:der_f_1} следует:
\begin{equation*}
    \hat{f}_{1}(x_{k + 1})\leq\frac{\tau_{k}}{2} + \frac{\hat{f}_{2}(x_{k})}{2\tau_{k}} - \frac{\eta_{k}(2 - \eta_{k})\left\|\nabla\hat{f}_{2}(x_{k})\right\|^{2}}{8\tau_{k}(L_{k}\tau_{k} + M_{\hat{F}}^{2})}\leq\frac{\tau_{k}}{2} + \frac{\hat{f}_{2}(x_{k})}{2\tau_{k}} - \frac{\eta(2 - \eta)\left\|\nabla\hat{f}_{2}(x_{k})\right\|^{2}}{8\tau_{k}\left(2L_{\hat{F}}\tau_{k} + M_{\hat{F}}^{2}\right)}.
\end{equation*}
Тогда, оценивая сверху минимальное значение квадрата нормы градиента $\nabla\hat{f}_{2}(x_{k})$ для $k\in\mathbb{N}$, выведем соотношение:
\begin{equation*}
    \begin{aligned}
        &\frac{\eta(2 - \eta)\left\|\nabla\hat{f}_{2}(x_{k})\right\|^{2}}{8\tau_{k}\left(2L_{\hat{F}}\max\limits_{i\in\overline{0, k - 1}}\left\{\tau_{i}\right\} + M_{\hat{F}}^{2}\right)}\leq\frac{\eta(2 - \eta)\left\|\nabla\hat{f}_{2}(x_{k})\right\|^{2}}{8\tau_{k}\left(2L_{\hat{F}}\tau_{k} + M_{\hat{F}}^{2}\right)}\leq\frac{\tau_{k}}{2} + \frac{\hat{f}_{2}(x_{k})}{2\tau_{k}} - \hat{f}_{1}(x_{k + 1})\Rightarrow\\
        &\Rightarrow\frac{\eta(2 - \eta)\left\|\nabla\hat{f}_{2}(x_{k})\right\|^{2}}{8\left(2L_{\hat{F}}\max\limits_{i\in\overline{0, k - 1}}\left\{\tau_{i}\right\} + M_{\hat{F}}^{2}\right)}\leq\frac{\tau_{k}^{2}}{2} + \frac{\hat{f}_{2}(x_{k})}{2} - \tau_{k}\hat{f}_{1}(x_{k + 1})\Rightarrow\\
        &\Rightarrow\frac{\eta(2 - \eta)k}{8\left(2L_{\hat{F}}\max\limits_{i\in\overline{0, k - 1}}\left\{\tau_{i}\right\} + M_{\hat{F}}^{2}\right)}\min\limits_{i\in\overline{0,k - 1}}\left\{\left\|\nabla\hat{f}_{2}(x_{i})\right\|^{2}\right\}\leq\\
        &\leq\frac{\eta(2 - \eta)}{8\left(2L_{\hat{F}}\max\limits_{i\in\overline{0, k - 1}}\left\{\tau_{i}\right\} + M_{\hat{F}}^{2}\right)}\sum\limits_{i = 0}^{k - 1}\left\|\nabla\hat{f}_{2}(x_{i})\right\|^{2}\leq\sum\limits_{i = 0}^{k - 1}\left(\frac{\tau_{i}^{2}}{2} + \frac{\hat{f}_{2}(x_{i})}{2} - \tau_{i}\hat{f}_{1}(x_{i + 1})\right) =\\
        &= \sum\limits_{i = 0}^{k - 1}\left(\frac{1}{2}\left(\tau_{i} - \hat{f}_{1}(x_{i})\right)^{2} + \tau_{i}\left(\hat{f}_{1}(x_{i}) - \hat{f}_{1}(x_{i + 1})\right)\right)\Rightarrow\min\limits_{i\in\overline{0, k - 1}}\left\{\left\|\nabla\hat{f}_{2}(x_{i})\right\|^{2}\right\}\leq\\
        &\leq\frac{8\left(2L_{\hat{F}}\max\limits_{i\in\overline{0, k - 1}}\left\{\tau_{i}\right\} + M_{\hat{F}}^{2}\right)}{\eta(2 - \eta)k}\sum\limits_{i = 0}^{k - 1}\left(\frac{1}{2}\left(\tau_{i} - \hat{f}_{1}(x_{i})\right)^{2} + \tau_{i}\left(\hat{f}_{1}(x_{i}) - \hat{f}_{1}(x_{i + 1})\right)\right).
    \end{aligned}
\end{equation*}
\end{proof}

В теореме \ref{lm:1_main} используется стратегия выбора $\tau_{k} = \operatorname{O}\left(\hat{f}_{1}(x_{k})\right)$ для вывода линейной сходимости к решению системы \eqref{eq:smooth_system} при выполнении условия Поляка--Лоясиевича.

\begin{re:theorem}\label{lm:1}
    Пусть условия теоремы \ref{th:1} верны, определим на каждой итерации метода нормализованных квадратов максимальный коэффициент линейной сходимости $\alpha_{k}$:
    $$\alpha_{k}\in\left[\frac{1}{2} + \frac{L_{k}
        \hat{f}_{1}(x_{k}) + (1 - \eta_{k})^{2}\mu}{2\left(L_{k}\hat{f}_{1}(x_{k}) + \mu\right)},~1\right),~k\in\mathbb{Z}_{+}.$$ Дополнительно предположим, что на каждой итерации $\tau_{k} = c_{k}\hat{f}_{1}(x_{k})$, $k\in\mathbb{Z}_{+}$:
    \begin{equation}\label{eq:det_const_cases_l1}
        \begin{aligned}
            c_{k} \in &\left[\frac{L_{k}\hat{f}_{1}(x_{k}) + (1 - \eta_{k})^{2}\mu}{(2\alpha_{k} - 1)(L_{k}\hat{f}_{1}(x_{k}) + \mu)},~\alpha_{k} - \frac{1}{2} - \frac{\mu}{2L_{k}\hat{f}_{1}(x_{k})} +\right.\\
            &\left.~~+ \sqrt{\alpha_{k}^{2} + \alpha_{k}\left(\frac{\mu}{L_{k}\hat{f}_{1}(x_{k})} - 1\right) + \frac{1}{4}\left(\frac{\mu}{L_{k}\hat{f}_{1}(x_{k})} + 1\right)^{2} - \frac{(1 - \eta_{k})^{2}\mu}{L_{k}\hat{f}_{1}(x_{k})}}~\right].
        \end{aligned}
    \end{equation}
    Тогда метод нормализованных квадратов с вычислением $x_{k + 1}$ по правилу \eqref{eq:det_expl_update_rule} глобально сходится не хуже, чем линейно к решению задачи \eqref{eq:main_opt_problem} $\lim\limits_{k\rightarrow+\infty}x_{k} = x^{*}: \hat{F}(x^{*}) = \mathbf{0}_{m}$ со следующей оценкой:
    $$\hat{f}_{1}(x_{k})\leq\hat{f}_{1}(x_{0})\prod_{i = 0}^{k - 1}\alpha_{i},~k\in\mathbb{Z}_{+},~\prod\limits_{i = 0}^{-1}\alpha_{i} \overset{\operatorname{def}}{=} 1.$$
\end{re:theorem}
\begin{proof}
Для удобства рассмотрим следующую декомпозицию задачи поиска области определения $c_{k},~k\in\mathbb{Z}_{+}$:
\begin{equation*}
    \begin{sqcases}
    1.~~\tau_{k}\in\left[\frac{\hat{f}_{1}(x_{k})}{c^{-1}_{k}},~\hat{f}_{1}(x_{k})\right],~c^{-1}_{k}\geq 1;\\[10pt]
    2.~~\tau_{k}\in\left[\hat{f}_{1}(x_{k}),~c_{k}\hat{f}_{1}(x_{k})\right],~c_{k}\geq 1.
    \end{sqcases}
\end{equation*}
Выведем границы значений $c_{k}$ для первого случая. Воспользуемся результатом теоремы \ref{th:1}:
\begin{equation*}
    \begin{aligned}
        \hat{f}_{1}(x_{k + 1})&\leq\frac{\tau_{k}}{2} + \frac{\hat{f}_{2}(x_{k})}{2\tau_{k}}\left(1 - \frac{\eta_{k}(2 - \eta_{k})\mu}{L_{k} \tau_{k} + \mu}\right)\leq\frac{\hat{f}_{1}(x_{k})}{2} + \frac{c^{-1}_{k}\hat{f}_{1}(x_{k})}{2}\left(1 - \frac{\eta_{k}(2 - \eta_{k})\mu}{L_{k} \hat{f}_{1}(x_{k}) + \mu}\right) =\\
        &= \hat{f}_{1}(x_{k})\left(\frac{c^{-1}_{k} + 1}{2} - \frac{\eta_{k}(2 - \eta_{k})c^{-1}_{k}\mu}{2\left(L_{k}\hat{f}_{1}(x_{k}) + \mu\right)}\right)
    \end{aligned}
\end{equation*}
По определению линейной сходимости:
$$\left(\frac{c^{-1}_{k} + 1}{2} - \frac{\eta_{k}(2 - \eta_{k})c^{-1}_{k}\mu}{2\left(L_{k}\hat{f}_{1}(x_{k}) + \mu\right)}\right)\in\left[0,~\alpha_{k}\right],~\alpha_{k}\in(0, 1).$$
Тогда:
\begin{equation*}
    \begin{aligned}
        &0\leq \frac{c^{-1}_{k} + 1}{2} - \frac{\eta_{k}(2 - \eta_{k})c^{-1}_{k}\mu}{2\left(L_{k}\hat{f}_{1}(x_{k}) + \mu\right)} \leq \alpha_{k} < 1\Rightarrow-1\leq c^{-1}_{k} - \frac{\eta_{k}(2 - \eta_{k})c^{-1}_{k}\mu}{L_{k}\hat{f}_{1}(x_{k}) + \mu} \leq 2\alpha_{k} - 1\Rightarrow\\
        &\Rightarrow-1\leq c^{-1}_{k}\left(1 - \frac{\eta_{k}(2 - \eta_{k})\mu}{L_{k}\hat{f}_{1}(x_{k}) + \mu}\right) \leq 2\alpha_{k} - 1\Rightarrow\\
        &\Rightarrow-\left(1 - \frac{\eta_{k}(2 - \eta_{k})\mu}{L_{k}\hat{f}_{1}(x_{k}) + \mu}\right)^{-1} \leq c^{-1}_{k} \leq (2\alpha_{k} - 1)\left(1 - \frac{\eta_{k}(2 - \eta_{k})\mu}{L_{k}\hat{f}_{1}(x_{k}) + \mu}\right)^{-1}\Rightarrow\\
        &\Rightarrow0 \leq c^{-1}_{k} \leq (2\alpha_{k} - 1)\left(1 - \frac{\eta_{k}(2 - \eta_{k})\mu}{L_{k}\hat{f}_{1}(x_{k}) + \mu}\right)^{-1} = (2\alpha_{k} - 1)\left(\frac{L_{k}\hat{f}_{1}(x_{k}) + \mu}{L_{k}\hat{f}_{1}(x_{k}) + (1 - \eta_{k}(2 - \eta_{k}))\mu}\right).
    \end{aligned}
\end{equation*}
Применим ограничения на $c_{k}^{-1}$ из условия:
$$1 \leq c^{-1}_{k} \leq (2\alpha_{k} - 1)\left(\frac{L_{k}\hat{f}_{1}(x_{k}) + \mu}{L_{k}\hat{f}_{1}(x_{k}) + (1 - \eta_{k})^{2}\mu}\right).$$
Следовательно,
$$c_{k}\in\left[\frac{L_{k}\hat{f}_{1}(x_{k}) + (1 - \eta_{k})^{2}\mu}{(2\alpha_{k} - 1)(L_{k}\hat{f}_{1}(x_{k}) + \mu)},~1\right].$$
Определим для таких $c_{k}$ допустимые значения $\alpha_{k}$:
\begin{equation}\label{eq:th1_first_conditions}
    \begin{aligned}
        &c^{-1}_{k} = (2\alpha_{k} - 1)\left(\frac{L_{k}\hat{f}_{1}(x_{k}) + \mu}{L_{k}\hat{f}_{1}(x_{k}) + (1 - \eta_{k})^{2}\mu}\right)\geq 1\Rightarrow\alpha_{k}\in\left[\frac{1}{2} + \frac{L_{k}
        \hat{f}_{1}(x_{k}) + (1 - \eta_{k})^{2}\mu}{2\left(L_{k}\hat{f}_{1}(x_{k}) + \mu\right)},~1\right).
    \end{aligned}
\end{equation}
Теперь рассмотрим второй случай:
$$\tau_{k}\in\left[\hat{f}_{1}(x_{k}),~c_{k}\hat{f}_{1}(x_{k})\right],~c_{k}\geq 1.$$
По теореме \ref{th:1}:
\begin{equation*}
    \begin{aligned}
        \hat{f}_{1}(x_{k + 1})&\leq\frac{\tau_{k}}{2} + \frac{\hat{f}_{2}(x_{k})}{2\tau_{k}}\left(1 - \frac{\eta_{k}(2 - \eta_{k})\mu}{L_{k} \tau_{k} + \mu}\right)\leq\frac{c_{k}\hat{f}_{1}(x_{k})}{2} + \frac{\hat{f}_{1}(x_{k})}{2}\left(1 - \frac{\eta_{k}(2 - \eta_{k})\mu}{L_{k} c_{k}\hat{f}_{1}(x_{k}) + \mu}\right) =\\
        &= \hat{f}_{1}(x_{k})\left(\frac{c_{k} + 1}{2} - \frac{\eta_{k}(2 - \eta_{k})\mu}{2\left(L_{k}\hat{f}_{1}(x_{k})c_{k} + \mu\right)}\right).
    \end{aligned}
\end{equation*}
По определению линейной сходимости:
$$\left(\frac{c_{k} + 1}{2} - \frac{\eta_{k}(2 - \eta_{k})\mu}{2\left(L_{k}\hat{f}_{1}(x_{k})c_{k} + \mu\right)}\right)\in\left[0,~\alpha_{k}\right],~\alpha_{k}\in(0, 1).$$
Распишем подробнее данные неравенства:
\begin{equation*}
    \begin{aligned}
        &0\leq \frac{c_{k} + 1}{2} - \frac{\eta_{k}(2 - \eta_{k})\mu}{2\left(L_{k}\hat{f}_{1}(x_{k})c_{k} + \mu\right)} \leq \alpha_{k} < 1\Rightarrow
        -1\leq c_{k} - \frac{\eta_{k}(2 - \eta_{k})\mu}{L_{k}\hat{f}_{1}(x_{k})c_{k} + \mu} \leq 2\alpha_{k} - 1\Rightarrow\\
        &\Rightarrow-\mu - L_{k}\hat{f}_{1}(x_{k})c_{k}\leq (c_{k})^{2}L_{k}\hat{f}_{1}(x_{k}) + c_{k}\mu - \eta_{k}(2 - \eta_{k})\mu \leq (2\alpha_{k} - 1)\left(L_{k}\hat{f}_{1}(x_{k})c_{k} + \mu\right)\Rightarrow\\
        &\Rightarrow0\leq (c_{k})^{2}L_{k}\hat{f}_{1}(x_{k}) + c_{k}\left(\mu + L_{k}\hat{f}_{1}(x_{k})\right) + (1 - \eta_{k})^{2}\mu\leq 2\alpha_{k}\left(L_{k}\hat{f}_{1}(x_{k})c_{k} + \mu\right).
    \end{aligned}
\end{equation*}
Рассмотрим левое неравенство:
$$0\leq (c_{k})^{2}L_{k}\hat{f}_{1}(x_{k}) + c_{k}\left(\mu + L_{k}\hat{f}_{1}(x_{k})\right) + (1 - \eta_{k})^{2}\mu.$$
Введём замену переменных $b_{k} = \frac{L_{k}\hat{f}_{1}(x_{k})}{\mu}$, подставим в неравенства:
$$0\leq b_{k}(c_{k})^{2} + c_{k}(1 + b_{k}) + (1 - \eta_{k})^{2}\leq 2\alpha_{k}(b_{k}c_{k} + 1).$$
Рассмотрим дискриминант полинома второй степени, записанного в центральном неравенстве выше:
$$(1 + b_{k})^{2} - 4b_{k}(1 - \eta_{k})^{2} = b_{k}^{2} + b_{k}(2 - 4(1 - \eta_{k})^{2}) + 1.$$
При тех $\eta_{k}$, при которых дискриминант отрицателен, новые ограничения на $c_{k}$ в левом неравенстве не возникают. При остальных $\eta_{k}$ получаются следующие корни:
$$c_{k} = -\frac{1}{2} - \frac{1}{2b_{k}}\pm\sqrt{\frac{1}{4} + \frac{1}{2b_{k}} - \frac{(1 - \eta_{k})^{2}}{b_{k}} + \frac{1}{4b_{k}^{2}}}.$$
Левый корень не накладывает дополнительные ограничения, так как он отрицателен. Правый корень не накладывает ограничений, потому что он левее $c_{k} = 1$:
\begin{equation*}
    \begin{aligned}
        &-\frac{1}{2} - \frac{1}{2b_{k}} + \sqrt{\frac{1}{4} + \frac{1}{2b_{k}} - \frac{(1 - \eta_{k})^{2}}{b_{k}} + \frac{1}{4b_{k}^{2}}}~\vee~1\Rightarrow
        \sqrt{\frac{1}{4} + \frac{1}{2b_{k}} - \frac{(1 - \eta_{k})^{2}}{b_{k}} + \frac{1}{4b_{k}^{2}}}~\vee~\frac{3}{2} + \frac{1}{2b_{k}}\Rightarrow\\
        &\Rightarrow\frac{1}{4} + \frac{1}{2b_{k}} - \frac{(1 - \eta_{k})^{2}}{b_{k}} + \frac{1}{4b_{k}^{2}}~\vee~\frac{9}{4} + \frac{3}{2b_{k}} + \frac{1}{4b_{k}}\Rightarrow-\frac{(1 - \eta_{k})^{2}}{b_{k}}\leq 2 + \frac{1}{b_{k}};
    \end{aligned}
\end{equation*}
то есть дополнительные ограничения не накладываются:
$$c_{k}\geq\max\left\{1,~-\frac{1}{2} - \frac{1}{2b_{k}} + \sqrt{\frac{1}{4} + \frac{1}{2b_{k}} - \frac{(1 - \eta_{k})^{2}}{b_{k}} + \frac{1}{4b_{k}^{2}}}\right\} = 1.$$
Перейдём к правому неравенству сразу с подстановкой $b_{k} = \frac{L_{k}\hat{f}_{1}(x_{k})}{\mu}$:
\begin{equation*}
    \begin{aligned}
        &b_{k}(c_{k})^{2} + c_{k}(1 + b_{k}) + (1 - \eta_{k})^{2}\leq 2\alpha_{k}(b_{k}c_{k} + 1)\Rightarrow b_{k}(c_{k})^{2} + c_{k}(1 + b_{k} - 2\alpha_{k}b_{k}) + (1 - \eta_{k})^{2} - 2\alpha_{k}\leq 0.
    \end{aligned}
\end{equation*}
Ограничения задаются полиномом второй степени, дискриминант которого равен:
\begin{equation*}
    \begin{aligned}
        &(1 + b_{k} - 2\alpha_{k}b_{k})^{2} - 4b_{k}((1 - \eta_{k})^{2} - 2\alpha_{k}) = \alpha_{k}^{2}(4b_{k}^{2}) + \alpha_{k}(4b_{k}(1 - b_{k})) + ((b_{k} + 1)^{2} - 4b_{k}(1-\eta_{k})^{2}).
    \end{aligned}
\end{equation*}
Дискриминант полинома является параболой по $\alpha_{k}$ с положительным коэффициентом при старшем члене. Рассмотрим значение данной параболы в точке вершины $\alpha_{k} = \frac{b_{k} - 1}{2b_{k}}$:
\begin{equation*}
    \begin{aligned}
        &4b_{k}^{2}\left(\frac{b_{k} - 1}{2b_{k}}\right)^{2} + 4b_{k}\left(\frac{b_{k} - 1}{2b_{k}}\right)(1 - b_{k}) + (b_{k} + 1)^{2} - 4b_{k}(1 - \eta_{k})^{2} = 4b_{k}\eta_{k}(2 - \eta_{k})\geq0,~\eta_{k}\in(0, 2).
    \end{aligned}
\end{equation*}
То есть множество допустимых значений $c_{k}$ непусто. Тогда граничные точки $c_{k}$ равны:
$$c_{k} = \alpha_{k} - \frac{b_{k} + 1}{2b_{k}} \pm \sqrt{\alpha_{k}^{2} + \frac{\alpha_{k}(1 - b_{k})}{b_{k}} + \left(\frac{1 + b_{k}}{2b_{k}}\right)^{2} - \frac{(1 - \eta_{k})^{2}}{b_{k}}}.$$
У параболы свободный член неположителен и равен $(1 - \eta_{k})^{2} - 2\alpha_{k}$, так как по \eqref{eq:th1_first_conditions} $\alpha_{k}\geq\frac{1}{2}$ и $\eta_{k}\in(0, 2)$. Поэтому по теореме Виета некратные ненулевые корни многочлена второй степени обладают разными знаками, сама парабола имеет положительный коэффициент при старшем члене, что по ограничению $c_{k}\geq 1$ для множества допустимых значений составляет отрезок:
\begin{equation*}
    \begin{aligned}
        c_{k}&\in\left[1,~\alpha_{k} - \frac{b_{k} + 1}{2b_{k}} + \sqrt{\alpha_{k}^{2} + \frac{\alpha_{k}(1 - b_{k})}{b_{k}} + \left(\frac{1 + b_{k}}{2b_{k}}\right)^{2} - \frac{(1 - \eta_{k})^{2}}{b_{k}}}\right] =\\
        &= \left[1,~\alpha_{k} - \frac{1}{2} - \frac{\mu}{2L_{k}\hat{f}_{1}(x_{k})} + \sqrt{\alpha_{k}^{2} + \alpha_{k}\left(\frac{\mu}{L_{k}\hat{f}_{1}(x_{k})} - 1\right) + \frac{1}{4}\left(\frac{\mu}{L_{k}\hat{f}_{1}(x_{k})} + 1\right)^{2} - \frac{(1 - \eta_{k})^{2}\mu}{L_{k}\hat{f}_{1}(x_{k})}}\right].
    \end{aligned}
\end{equation*}
Определим для таких $c_{k}$ допустимые значения $\alpha_{k}$:
\begin{equation*}
    \begin{aligned}
        &c_{k} = \alpha_{k} - \frac{b_{k} + 1}{2b_{k}} + \sqrt{\alpha_{k}^{2} + \frac{\alpha_{k}(1 - b_{k})}{b_{k}} + \left(\frac{1 + b_{k}}{2b_{k}}\right)^{2} - \frac{(1 - \eta_{k})^{2}}{b_{k}}}\geq 1\Rightarrow\\
        &\Rightarrow\alpha_{k}^{2} + \frac{\alpha_{k}(1 - b_{k})}{b_{k}} + \left(\frac{1 + b_{k}}{2b_{k}}\right)^{2} - \frac{(1 - \eta_{k})^{2}}{b_{k}}\geq\left(1 + \frac{1 + b_{k}}{2b_{k}} - \alpha_{k}\right)^{2}\Rightarrow\\
        &\Rightarrow\alpha_{k}(1 - b_{k}) - (1 - \eta_{k})^{2}\geq 1 + 2b_{k} - \alpha_{k}(1 + b_{k}) - 2\alpha_{k}b_{k}\Rightarrow2\alpha_{k}(b_{k} + 1)\geq 2b_{k} + 1 + (1 - \eta_{k})^{2}\Rightarrow\\
        &\Rightarrow\alpha_{k}\geq\frac{2b_{k} + 1 + (1- \eta_{k})^{2}}{2(b_{k} + 1)} = \frac{1}{2} + \frac{b_{k} + (1 - \eta_{k})^{2}}{2(b_{k} + 1)}\Rightarrow\alpha_{k}\in\left[\frac{1}{2} + \frac{L_{k}\hat{f}_{1}(x_{k}) + (1 - \eta_{k})^{2}\mu}{2\left(L_{k}\hat{f}_{1}(x_{k}) + \mu\right)},~1\right).
    \end{aligned}
\end{equation*}
Объединяя рассмотренные случаи, получаем искомое \eqref{eq:det_const_cases_l1}:
\begin{equation*}
    \begin{cases}
        \begin{aligned}
            c_{k} \in &\left[\frac{L_{k}\hat{f}_{1}(x_{k}) + (1 - \eta_{k})^{2}\mu}{(2\alpha_{k} - 1)(L_{k}\hat{f}_{1}(x_{k}) + \mu)},~\alpha_{k} - \frac{1}{2} - \frac{\mu}{2L_{k}\hat{f}_{1}(x_{k})} +\right.\\
            &\left.+ \sqrt{\alpha_{k}^{2} + \alpha_{k}\left(\frac{\mu}{L_{k}\hat{f}_{1}(x_{k})} - 1\right) + \frac{1}{4}\left(\frac{\mu}{L_{k}\hat{f}_{1}(x_{k})} + 1\right)^{2} - \frac{(1 - \eta_{k})^{2}\mu}{L_{k}\hat{f}_{1}(x_{k})}}~~\right].
        \end{aligned}\\
        \alpha_{k}\in\left[\frac{1}{2} + \frac{L_{k}
        \hat{f}_{1}(x_{k}) + (1 - \eta_{k})^{2}\mu}{2\left(L_{k}\hat{f}_{1}(x_{k}) + \mu\right)},~1\right).
    \end{cases}
\end{equation*}
Установлена линейная сходимость метода трёх квадратов, что по свойствам минимизируемой нормы приводит итерационный процесс к решению $\lim\limits_{k\rightarrow+\infty}x_{k} = x^{*}: \hat{F}(x^{*}) = \mathbf{0}_{m}$. Данная сходимость носит глобальный характер, так как для $x_{0}\in\mathcal{F}:~\mathcal{L}(\hat{f}_{1}(x_{0}))\subseteq\mathcal{F}$, верно вложение $$x_{k}\in\mathcal{L}(\hat{f}_{1}(x_{k}))\subseteq\mathcal{L}(\hat{f}_{1}(x_{0}))\subseteq\mathcal{F},~k\in\mathbb{Z}_{+}.$$
Из доказательства непосредственно выводится оценка сходимости:
$$\hat{f}_{1}(x_{k})\leq\hat{f}_{1}(x_{0})\prod_{i = 0}^{k - 1}\alpha_{i},~k\in\mathbb{N}.$$
\end{proof}

В теореме \ref{th:DetFlexSublinConvMain} представлены условия сходимости к стационарной точке в случае адаптивного подбора значения $\tau_{k}$ на каждой итерации.

\begin{re:theorem}\label{th:DetFlexSublinConv}
    Пусть выполнено предположение \ref{as:det_hat_F_der_smooth}, $k\in\mathbb{N},~r > 0$. Рассмотрим функции
    $$\varkappa(t) = \frac{t^{2}}{2}\mathds{1}_{\left\{t\in[0, 1]\right\}} + \left(t - \frac{1}{2}\right)\mathds{1}_{\left\{t > 1\right\}}\text{ и }\tilde{\Delta}_{r}(x) \overset{\operatorname{def}}{=} \hat{f}_{1}(x) - \min\limits_{y\in E_{1}}\left\{\phi(x, y):~\|y - x\|\leq r\right\}.$$
    Тогда для метода Гаусса--Ньютона, реализованного по схеме \ref{alg:gen_det_flex_gnm} с $\varepsilon_{k} = \varepsilon \geq 0$, верны следующие оценки:
    \begin{equation*}
        \begin{cases}
            &\frac{8L_{\hat{F}}^{2}}{L}\left(\varepsilon + \frac{\left(\hat{f}_{1}(x_{0}) - \hat{f}_{1}(x_{k})\right)}{k}\right)\geq\min\limits_{i\in\overline{0, k - 1}}\left\{\left\|2L_{\hat{F}}\left(T_{2L_{\hat{F}}, \mathcal{T}_{2L_{\hat{F}}}(x_{i})}(x_{i}) - x_{i}\right)\right\|^{2}\right\};\\[10pt]
            &L_{\hat{F}}\left(\varepsilon + \frac{\left(\hat{f}_{1}(x_{0}) - \hat{f}_{1}(x_{k})\right)}{k}\right)\geq\min\limits_{i\in\overline{0, k - 1}}\left\{2\left(L_{\hat{F}}r\right)^{2}\varkappa\left(\frac{\tilde{\Delta}_{r}(x_{i})}{2L_{\hat{F}}r^{2}}\right)\right\}.
        \end{cases}
    \end{equation*}
\end{re:theorem}
\begin{proof}
Согласно Lemma 2.4 \cite{Nesterov2007} выполнено следующее соотношение:
\begin{equation*}
    \begin{aligned}
        &\hat{f}_{1}(x_{k}) - \psi_{x_{k}, L_{k}, \mathcal{T}_{L_{k}}(x_{k})}(T_{L_{k}, \mathcal{T}_{L_{k}}(x_{k})}(x_{k}))\geq L_{k}r^{2}\varkappa\left(\frac{\tilde{\Delta}_{r}(x_{k})}{L_{k}r^{2}}\right),~k\in\mathbb{Z}_{+},~r > 0.
    \end{aligned}
\end{equation*}
Lemma 2.3 \cite{Nesterov2007} устанавливает другое соотношение:
\begin{equation*}
    \begin{aligned}
        &\hat{f}_{1}(x_{k}) - \psi_{x_{k}, L_{k}, \mathcal{T}_{L_{k}}(x_{k})}(T_{L_{k}, \mathcal{T}_{L_{k}}(x_{k})}(x_{k}))\geq\frac{L_{k}}{2}\left\|T_{L_{k}, \mathcal{T}_{L_{k}}(x_{k})}(x_{k}) - x_{k}\right\|^{2},~k\in\mathbb{Z}_{+}.
    \end{aligned}
\end{equation*}
Используя аргументацию из теоремы \ref{th:DetSublinConv}, добавим и вычтем значение $\psi_{x_{k}, L_{k}, \tau_{k}^{*}}(x_{k + 1})$:
\begin{equation*}
    \begin{cases}
        \begin{aligned}
            \hat{f}_{1}(x_{k}) - \hat{f}_{1}(x_{k + 1}) + \varepsilon&\geq\hat{f}_{1}(x_{k}) + \left(\psi_{x_{k}, L_{k}, \tau_{k}^{*}}(x_{k + 1}) - \psi_{x_{k}, L_{k}, \mathcal{T}_{L_{k}}(x_{k})}(T_{L_{k}, \mathcal{T}_{L_{k}}(x_{k})}(x_{k}))\right) - \psi_{x_{k}, L_{k}, \tau_{k}^{*}}(x_{k + 1})\geq\\
            &\geq\frac{L_{k}}{2}\left\|T_{L_{k}, \mathcal{T}_{L_{k}}(x_{k})}(x_{k}) - x_{k}\right\|^{2};
        \end{aligned}\\[10pt]
        \begin{aligned}
            \hat{f}_{1}(x_{k}) - \hat{f}_{1}(x_{k + 1}) + \varepsilon&\geq\hat{f}_{1}(x_{k}) + \left(\psi_{x_{k}, L_{k}, \tau_{k}^{*}}(x_{k + 1}) - \psi_{x_{k}, L_{k}, \mathcal{T}_{L_{k}}(x_{k})}(T_{L_{k}, \mathcal{T}_{L_{k}}(x_{k})}(x_{k}))\right) - \psi_{x_{k}, L_{k}, \tau_{k}^{*}}(x_{k + 1})\geq\\
            &\geq L_{k}r^{2}\varkappa\left(\frac{\tilde{\Delta}_{r}(x_{k})}{L_{k}r^{2}}\right).
        \end{aligned}
    \end{cases}
\end{equation*}
Усредним данные соотношения для $0,~\dots, k - 1$ итераций и воспользуемся тем, что $L_{k}\geq L$, функции $\left\|T_{L_{i}, \mathcal{T}_{L_{i}}(x_{i})}(x_{i}) - x_{i}\right\|^{2}$ и $L_{i}r^{2}\varkappa\left(\frac{\tilde{\Delta}_{r}(x_{i})}{L_{i}r^{2}}\right)$ монотонно убывают по $L_{i}$ \cite{Nesterov2007}:
\begin{equation*}
    \begin{cases}
        \begin{aligned}
            \varepsilon + \frac{\hat{f}_{1}(x_{0}) - \hat{f}_{1}(x_{k})}{k}&\geq\frac{1}{k}\sum\limits_{i = 0}^{k - 1}\frac{L_{i}}{2}\left\|T_{L_{i}, \mathcal{T}_{L_{i}}(x_{i})}(x_{i}) - x_{i}\right\|^{2}\geq\frac{1}{k}\sum\limits_{i = 0}^{k - 1}\frac{L}{2}\left\|T_{2L_{\hat{F}}, \mathcal{T}_{2L_{\hat{F}}}(x_{i})}(x_{i}) - x_{i}\right\|^{2}\geq\\
            &\geq\min\limits_{i\in\overline{0, k - 1}}\left\{\frac{L}{2}\left\|T_{2L_{\hat{F}}, \mathcal{T}_{2L_{\hat{F}}}(x_{i})}(x_{i}) - x_{i}\right\|^{2}\right\};
        \end{aligned}\\[10pt]
        \begin{aligned}
            \varepsilon + \frac{\hat{f}_{1}(x_{0}) - \hat{f}_{1}(x_{k})}{k}&\geq\frac{1}{k}\sum\limits_{i = 0}^{k - 1}L_{i}r^{2}\varkappa\left(\frac{\tilde{\Delta}_{r}(x_{i})}{L_{i}r^{2}}\right)\geq\frac{1}{k}\sum\limits_{i = 0}^{k - 1}2L_{\hat{F}}r^{2}\varkappa\left(\frac{\tilde{\Delta}_{r}(x_{i})}{2L_{\hat{F}}r^{2}}\right)\geq\\
            &\geq\min\limits_{i\in\overline{0, k - 1}}\left\{2L_{\hat{F}}r^{2}\varkappa\left(\frac{\tilde{\Delta}_{r}(x_{i})}{2L_{\hat{F}}r^{2}}\right)\right\}.
        \end{aligned}
    \end{cases}
\end{equation*}
Домножением приводим оценки к искомому виду.
\end{proof}
\begin{re:th:corollary}\label{th:DetFlexSublinConvCor1}
По аналогии со следствием \ref{th:SubLinConvCor1} существует возможность найти стационарную точку с произвольной точностью при монотонном увеличении точности вычисления $(\tau_{k}^{*}, x_{k + 1})$:
\begin{equation*}
    \begin{cases}
        \begin{aligned}
            \frac{8L_{\hat{F}}^{2}}{kL}&\left((1 + 2\varepsilon)\hat{f}_{1}(x_{0}) -\varepsilon\hat{f}_{1}(x_{k - 1}) - \hat{f}_{1}(x_{k})\right)\geq\min\limits_{i\in\overline{0, k - 1}}\left\{\left\|2L_{\hat{F}}\left(T_{2L_{\hat{F}}, \mathcal{T}_{2L_{\hat{F}}}(x_{i})}(x_{i}) - x_{i}\right)\right\|^{2}\right\};
        \end{aligned}\\[10pt]
        \begin{aligned}
            \frac{L_{\hat{F}}}{k}&\left((1 + 2\varepsilon)\hat{f}_{1}(x_{0}) - \varepsilon\hat{f}_{1}(x_{k - 1}) - \hat{f}_{1}(x_{k})\right)\geq\min\limits_{i\in\overline{0, k - 1}}\left\{2\left(L_{\hat{F}}r\right)^{2}\varkappa\left(\frac{\tilde{\Delta}_{r}(x_{i})}{2L_{\hat{F}}r^{2}}\right)\right\}.
        \end{aligned}
    \end{cases}
\end{equation*}
\end{re:th:corollary}
\begin{re:th:corollary}\label{th:DetFlexSublinConvCor2}
Воспользовавшись аргументацией из следствия \ref{th:SubLinConvCor2}, получаем оценки на вариацию последовательности $\left\{x_{k}\right\}_{k\in\mathbb{Z}_{+}}$:
\begin{equation*}
    \begin{cases}
        \hat{f}_{1}(x_{k}) - \hat{f}_{1}^{*} + \varepsilon\left(\hat{f}_{1}(x_{k - 1}) - \hat{f}_{1}^{*}\right)\geq\frac{L}{2}\left\|T_{2L_{\hat{F}}, \mathcal{T}_{2L_{\hat{F}}}(x_{k})}(x_{k}) - x_{k}\right\|^{2};\\[10pt]
        \hat{f}_{1}(x_{k}) - \hat{f}_{1}^{*} + \varepsilon\left(\hat{f}_{1}(x_{k - 1}) - \hat{f}_{1}^{*}\right)\geq2L_{\hat{F}}r^{2}\varkappa\left(\frac{\tilde{\Delta}_{r}(x_{k})}{2L_{\hat{F}}r^{2}}\right).
    \end{cases}
\end{equation*}
Неравенства выше означают $\lim\limits_{k\rightarrow+\infty}x_{k + 1} = \lim\limits_{k\rightarrow+\infty}T_{2L_{\hat{F}}, \mathcal{T}_{2L_{k}}(x_{k})}(x_{k}) = x^{*}$ и
\begin{equation*}
    \begin{cases}
        \lim\limits_{k\rightarrow+\infty}\left\|x_{k + 1} - x_{k}\right\| = 0;\\
        \lim\limits_{k\rightarrow+\infty}\tilde{\Delta}_{r}(x_{k}) = 0;
    \end{cases}
\end{equation*}
указывая на связность множества стационарных точек $\left\{x^{*}:~x^{*}\in E_{1},~\tilde{\Delta}_{r}(x^{*}) = 0\right\}$ для последовательности $\left\{x_{k}\right\}_{k\in\mathbb{Z}_{+}}$.
\end{re:th:corollary}

В теореме \ref{th:DetFlexlinConvMain} представлены улучшенные оценки сходимости для областей сублинейной и линейной сходимости в случае адаптивного подбора $\tau_{k}$ по сравнению с фиксацией $\tau_{k} = \hat{f}_{1}(x_{k})$.

\begin{re:theorem}\label{th:DetFlexlinConv}
    Допустим выполнение предположений \ref{as:det_hat_F_der_smooth} и \ref{as:det_hat_F_PL_condition} для метода Гаусса--Ньютона со схемой реализации \ref{alg:gen_det_flex_gnm}. Тогда для последовательности $\left\{x_{k}\right\}_{k\in\mathbb{Z}_{+}}$ выполняются следующие соотношения:
    \begin{equation*}
        \hat{f}_{1}(x_{k + 1})\leq\varepsilon_{k} + \begin{sqcases}
            \frac{L_{\hat{F}}}{\mu}\hat{f}_{2}(x_{k})\leq\frac{1}{2}\hat{f}_{1}(x_{k})\text{, если }\hat{f}_{1}(x_{k})\leq\frac{\mu}{2L_{\hat{F}}};\\[5pt]
            \hat{f}_{1}(x_{k}) - \frac{\mu}{4L_{\hat{F}}}\text{, иначе}.
        \end{sqcases}
    \end{equation*}
    Если при генерации последовательности $\left\{x_{k}\right\}_{k\in\mathbb{Z}_{+}}$ была зафиксирована $L_{k} = L_{\hat{F}}$, то данные соотношения выражаются по--другому:
    \begin{equation*}
        \hat{f}_{1}(x_{k + 1})\leq\varepsilon_{k} + \begin{sqcases}
            \frac{L_{\hat{F}}}{2\mu}\hat{f}_{2}(x_{k})\leq\frac{1}{2}\hat{f}_{1}(x_{k})\text{, если }\hat{f}_{1}(x_{k})\leq\frac{\mu}{L_{\hat{F}}};\\[5pt]
            \hat{f}_{1}(x_{k}) - \frac{\mu}{2L_{\hat{F}}}\text{, иначе}.
        \end{sqcases}
    \end{equation*}
\end{re:theorem}
\begin{proof}
Воспользуемся цепочкой рассуждений из теоремы \ref{th:glob_sub_lin_and_lin_conv}:
\begin{equation*}
    \begin{aligned}
        \hat{f}_{1}(x_{k + 1})&\leq\psi_{x_{k}, L_{k}, \tau_{k}^{*}}(x_{k + 1}) = \psi_{x_{k}, L_{k}, \mathcal{T}_{L_{k}}(x_{k})}(T_{L_{k}, \mathcal{T}_{L_{k}}(x_{k})}(x_{k})) + \left(\psi_{x_{k}, L_{k}, \tau_{k}^{*}}(x_{k + 1}) -\right.\\
        &\left.- \psi_{x_{k}, L_{k}, \mathcal{T}_{L_{k}}(x_{k})}(T_{L_{k}, \mathcal{T}_{L_{k}}(x_{k})}(x_{k}))\right)\leq\varepsilon_{k} +\\
        &+ \psi_{x_{k}, L_{k}, \mathcal{T}_{L_{k}}(x_{k})}(T_{L_{k}, \mathcal{T}_{L_{k}}(x_{k})}(x_{k})) = \varepsilon_{k} + \min\limits_{y\in E_{1}}\left\{\phi(x_{k}, x_{k} + y) + \frac{L_{k}}{2}\|y\|^{2}\right\}\leq\\
        &\leq\left\{\text{вместо $y$ подставим }th_{k} = -t\hat{F}^{'}(x_{k})^{*}\left(\hat{F}^{'}(x_{k})\hat{F}^{'}(x_{k})^{*}\right)^{-1}\hat{F}(x_{k}),~t\in[0,~1]\right\}\leq\varepsilon_{k} +\\
        &+ \min\limits_{t\in [0,~1]}\left\{\left\|\hat{F}(x_{k}) + t\hat{F}^{'}(x_{k})h_{k}\right\| + \frac{t^{2}L_{k}}{2}\|h_{k}\|^{2}\right\}\leq\left\{\text{неравенство \eqref{eq:th_glob_lin_conv_eq_1}}\right\}\leq\varepsilon_{k} +\\
        &+ \min\limits_{t\in[0,~1]}\left\{\left\|(1 - t)\hat{F}(x_{k})\right\| + \frac{t^{2}L_{k}}{2\mu}\hat{f}_{2}(x_{k})\right\}\leq\left\{\|\cdot\|\text{ --- выпуклая}\right\}\leq\varepsilon_{k} +\\
        &+ \min\limits_{t\in[0,~1]}\left\{(1 - t)\hat{f}_{1}(x_{k}) + \frac{t^{2}L_{k}}{2\mu}\hat{f}_{2}(x_{k})\right\} = \varepsilon_{k} + \hat{f}_{1}(x_{k}) + \frac{\hat{f}_{2}(x_{k})L_{k}}{\mu}\min\limits_{t\in[0,~1]}\left\{\frac{-t\mu}{\hat{f}_{1}(x_{k})L_{k}} + \frac{t^{2}}{2}\right\} =\\
    \end{aligned}
\end{equation*}
\begin{equation*}
    \begin{aligned}
        &= \varepsilon_{k} + \hat{f}_{1}(x_{k}) - \frac{\hat{f}_{2}(x_{k})L_{k}}{\mu}\max\limits_{t\in[0,~1]}\left\{\frac{t\mu}{\hat{f}_{1}(x_{k})L_{k}} - \frac{t^{2}}{2}\right\} =\\
        &= \left\{\text{\eqref{eq:aux_det_local_decrease_2_eq_1}, лемма \ref{lm:aux_det_local_decrease_2}}\right\} = \varepsilon_{k} + \hat{f}_{1}(x_{k}) - \frac{\hat{f}_{2}(x_{k})L_{k}}{\mu}\varkappa\left(\frac{\mu}{\hat{f}_{1}(x_{k})L_{k}}\right)\leq\left\{\text{монотонное убывание по }L_{k}\right\}\leq\\
        &\leq\varepsilon_{k} + \hat{f}_{1}(x_{k}) - \frac{2\hat{f}_{2}(x_{k})L_{\hat{F}}}{\mu}\varkappa\left(\frac{\mu}{2\hat{f}_{1}(x_{k})L_{\hat{F}}}\right).
    \end{aligned}
\end{equation*}
Явно запишем получившееся неравенство в зависимости от $\varkappa(\cdot)$:
\begin{equation*}
    \begin{aligned}
        &\hat{f}_{1}(x_{k + 1})\leq\varepsilon_{k} + \begin{sqcases}
            \hat{f}_{1}(x_{k}) - \frac{\mu}{4L_{\hat{F}}}\text{, если }\hat{f}_{1}(x_{k})\geq\frac{\mu}{2L_{\hat{F}}};\\[5pt]
            \frac{\hat{f}_{2}(x_{k})L_{\hat{F}}}{\mu}\leq\frac{1}{2}\hat{f}_{1}(x_{k})\text{, если }\hat{f}_{1}(x_{k})\leq\frac{\mu}{2L_{\hat{F}}}.
        \end{sqcases}
    \end{aligned}
\end{equation*}
Для ограничения на $\hat{f}_{1}(x_{k + 1})$ при $L_{k}\equiv L_{\hat{F}}$ представление в зависимости от $\varkappa(\cdot)$ задаётся иначе:
$$\hat{f}_{1}(x_{k + 1})\leq\varepsilon_{k} + \hat{f}_{1}(x_{k}) - \frac{\hat{f}_{2}(x_{k})L_{\hat{F}}}{\mu}\varkappa\left(\frac{\mu}{\hat{f}_{1}(x_{k})L_{\hat{F}}}\right).$$
В явном виде это означает следующее:
\begin{equation*}
    \begin{aligned}
        &\hat{f}_{1}(x_{k + 1})\leq\varepsilon_{k} + \begin{sqcases}
            \hat{f}_{1}(x_{k}) - \frac{\mu}{2L_{\hat{F}}}\text{, если }\hat{f}_{1}(x_{k})\geq\frac{\mu}{L_{\hat{F}}};\\[5pt]
            \frac{\hat{f}_{2}(x_{k})L_{\hat{F}}}{2\mu}\leq\frac{1}{2}\hat{f}_{1}(x_{k})\text{, если }\hat{f}_{1}(x_{k})\leq\frac{\mu}{L_{\hat{F}}}.
        \end{sqcases}
    \end{aligned}
\end{equation*}
\end{proof}
\begin{re:th:corollary}\label{th:DetFlexlinConvCor1}
По аналогии со следствием \ref{th:glob_sub_lin_and_lin_conv_cor} выбор $\varepsilon_{k}\geq 0$ позволяет сколь угодно точно решить задачу \eqref{eq:main_opt_problem} с помощью подбора последовательности величин $\left\{\delta_{k}\right\}_{k\in\mathbb{Z}_{+}}:~\frac{1}{2}\delta_{k} > \delta_{k + 1} > 0$, $\delta_{-1} = 4\delta_{0}$, $\lim\limits_{k\rightarrow+\infty}\delta_{k} = 0$. Введём:
\begin{equation*}
    \begin{aligned}
        &\begin{sqcases}
            1.~\hat{f}_{1}(x_{-1}) \overset{\operatorname{def}}{=} \frac{\mu}{2L_{\hat{F}}},~d = 4\text{, для }L_{k}\in[L,~L_{\hat{F}}];\\[5pt]
            2.~\hat{f}_{1}(x_{-1}) \overset{\operatorname{def}}{=} \frac{\mu}{L_{\hat{F}}},~d = 2\text{, для }L_{k} = L_{\hat{F}};
        \end{sqcases}
    \end{aligned}
\end{equation*}
Обозначив за $N\in\mathbb{Z}_{+}\cup\left\{-1\right\}$ минимальный номер итерации, на которой выполнена одна из двух цепочек неравенств (положим $N = -1$ в случае отсутствия такой итерации):
\begin{equation*}
    \begin{aligned}
        &\begin{sqcases}
            1.~\hat{f}_{1}(x_{N})\geq\frac{\mu}{2L_{\hat{F}}}\geq\hat{f}_{1}(x_{N + 1})\text{, для }L_{k}\in[L,~L_{\hat{F}}];\\[5pt]
            2.~\hat{f}_{1}(x_{N})\geq\frac{\mu}{L_{\hat{F}}}\geq\hat{f}_{1}(x_{N + 1})\text{, для }L_{k} = L_{\hat{F}};
        \end{sqcases}
    \end{aligned}
\end{equation*}
зададим $\varepsilon_{k}$:
\begin{equation*}
    \begin{aligned}
        &\varepsilon_{k} = \begin{sqcases}
            \delta_{0}:~\delta_{0} < \frac{\mu}{dL_{\hat{F}}}\text{, при }k = 0;\\
            \delta_{k - 1} - \delta_{k}\text{, если }0 < k\leq N + 1;\\
            \frac{1}{2}\delta_{k - 1} - \delta_{k}\text{, если }k > N + 1.
        \end{sqcases}
    \end{aligned}
\end{equation*}
Получаем с увеличением номера итерации убывание погрешности поиска $(\tau_{k}^{*}, x_{k + 1})$:
\begin{equation*}
    \begin{aligned}
        &\begin{sqcases}
            \hat{f}_{1}(x_{k})\leq2\delta_{0} - \delta_{k - 1} + \hat{f}_{1}(x_{0}) - \frac{k\mu}{dL_{\hat{F}}},\text{ если }0 < k\leq N + 1;\\[15pt]
            \hat{f}_{1}(x_{k})\leq\left(\frac{1}{2}\right)^{k - N - 1}\hat{f}_{1}(x_{N + 1}) + \delta_{N}\left(\frac{1}{2}\right)^{k - N - 1} - \delta_{k - 1},\text{ если }k > N + 1.
        \end{sqcases}
    \end{aligned}
\end{equation*}
\end{re:th:corollary}
\begin{re:th:corollary}\label{th:DetFlexlinConvCor2}
Применяя рассуждения из следствия \ref{th:glob_sub_lin_and_lin_conv_cor_2}, устанавливаются необходимое количество итераций и максимальное значение погрешности поиска $(\tau_{k}^{*}, x_{k + 1})$ в случае постоянной погрешности $\varepsilon_{k} = \varepsilon > 0$ для достижения уровня функции $\hat{f}_{1}(x_{k})\leq\epsilon$. Для $L_{k}\in[L,~2L_{\hat{F}}]$ условия следующие:
\begin{itemize}
    \item если $\epsilon\geq\frac{\mu}{2L_{\hat{F}}}$, то $k\geq\left\lceil\left(\frac{\mu}{4L_{\hat{F}}} - \varepsilon\right)^{-1}\left(\hat{f}_{1}(x_{0}) - \epsilon\right)\mathds{1}_{\left\{\hat{f}_{1}(x_{0}) > \epsilon\right\}}\right\rceil$, $\varepsilon < \frac{\mu}{4L_{\hat{F}}}$;
    \item если $\epsilon < \frac{\mu}{2L_{\hat{F}}}$, то $k\geq\left\lceil\left(\frac{\mu}{4L_{\hat{F}}} - \varepsilon\right)^{-1}\left(\hat{f}_{1}(x_{0}) - \frac{\mu}{2L_{\hat{F}}}\right)\mathds{1}_{\left\{\hat{f}_{1}(x_{0}) > \frac{\mu}{2L_{\hat{F}}}\right\}} + \log_{2}\left(\frac{\mu}{2r\epsilon L_{\hat{F}}}\right)\right\rceil$, $\varepsilon\leq\frac{(1 - r)\epsilon}{2}$,\\$r\in(0, 1)$.
\end{itemize}
Для точно известного значения $L_{k} = L_{\hat{F}}$ количество необходимых итераций меньше и допустимая погрешность больше:
\begin{itemize}
    \item если $\epsilon\geq\frac{\mu}{L_{\hat{F}}}$, то $k\geq\left\lceil\left(\frac{\mu}{2L_{\hat{F}}} - \varepsilon\right)^{-1}\left(\hat{f}_{1}(x_{0}) - \epsilon\right)\mathds{1}_{\left\{\hat{f}_{1}(x_{0}) > \epsilon\right\}}\right\rceil$, $\varepsilon < \frac{\mu}{2L_{\hat{F}}}$;
    \item если $\epsilon < \frac{\mu}{L_{\hat{F}}}$, то $k\geq\left\lceil\left(\frac{\mu}{2L_{\hat{F}}} - \varepsilon\right)^{-1}\left(\hat{f}_{1}(x_{0}) - \frac{\mu}{L_{\hat{F}}}\right)\mathds{1}_{\left\{\hat{f}_{1}(x_{0}) > \frac{\mu}{L_{\hat{F}}}\right\}} + \log_{2}\left(\frac{\mu}{r\epsilon L_{\hat{F}}}\right)\right\rceil$, $\varepsilon\leq\frac{(1 - r)\epsilon}{2}$, $r\in(0, 1)$.
\end{itemize}
\end{re:th:corollary}

\subsection*{Стохастическая модификация метода Гаусса--Ньютона}

\subsubsection*{Вспомогательные утверждения}

В лемме ниже выводятся основные и часто используемые в данной работе отношения частичного порядка, связанные со спектром симметричных матриц.

\begin{lemma}\label{lm:aux_matrix_power_order}
    Предположим выполнение предположений \ref{as:2} и \ref{as:5}. Тогда для любых $t\geq 0$, $x\in E_{1}$, $B\subseteq\mathcal{B}$, $|B| = b\in\overline{1, \min\{m, n\}}$ выполнены следующие соотношения:
    \begin{equation*}
        \begin{cases}
        &\tau^{t} I_{n}\preceq\left(\hat{G}^{'}(x, B)^{*}\hat{G}^{'}(x, B) + \tau I_{n}\right)^{t}\preceq \left(M_{\hat{G}}^{2} + \tau\right)^{t}I_{n},~\tau\geq 0;\\
        &\frac{1}{\left(M_{\hat{G}}^{2} + \tau\right)^{t}} I_{n}\preceq\left(\hat{G}^{'}(x, B)^{*}\hat{G}^{'}(x, B) + \tau I_{n}\right)^{-t}\preceq \frac{1}{\tau^{t}} I_{n},~\tau > 0;\\
        &\left(\mu + \tau\right)^{t} I_{b}\preceq\left(\hat{G}^{'}(x, B)\hat{G}^{'}(x, B)^{*} + \tau I_{b}\right)^{t}\preceq\left(M_{\hat{G}}^{2} + \tau\right)^{t}I_{b},~\tau\geq 0;\\
        &\frac{1}{\left(M_{\hat{G}}^{2} + \tau\right)^{t}} I_{b}\preceq\left(\hat{G}^{'}(x, B)\hat{G}^{'}(x, B)^{*} + \tau I_{b}\right)^{-t}\preceq \frac{1}{\left(\mu + \tau\right)^{t}} I_{b},~\tau\geq 0.
    \end{cases}
    \end{equation*}
    Отношение порядка <<$\preceq$>> выполнено на конусе неотрицательно определённых матриц.
\end{lemma}
\begin{proof}
Предположение \ref{as:2} сверху ограничивает максимальное сингулярное число матрицы\\$\hat{G}^{'}(x, B)$:
\begin{equation*}
    \begin{aligned}
        \left\|\hat{G}^{'}(x, B)\right\| = \sigma_{\max}(\hat{G}^{'}(x, B))\leq M_{\hat{G}}\Leftrightarrow&\hat{G}^{'}(x, B)^{*}\hat{G}^{'}(x, B)\preceq M_{\hat{G}}^{2}I_{n},\\
        &\hat{G}^{'}(x, B)\hat{G}^{'}(x, B)^{*}\preceq M_{\hat{G}}^{2}I_{b}.
    \end{aligned}
\end{equation*}
Предположение \ref{as:5} ограничивает снизу минимальное сингулярное число матрицы $\hat{G}^{'}(x, B)^{*}$:
\begin{equation*}
    \begin{aligned}
        \hat{G}^{'}(x, B)\hat{G}^{'}(x, B)^{*}\succeq\mu I_{b}\Leftrightarrow\sigma_{\min}(\hat{G}^{'}(x, B)^{*})\geq\sqrt{\mu}.
    \end{aligned}
\end{equation*}
Симметричные матрицы $\left(\hat{G}^{'}(x, B)^{*}\hat{G}^{'}(x, B) + \tau I_{n}\right)^{t}$ и $\left(\hat{G}^{'}(x, B)\hat{G}^{'}(x, B)^{*} + \tau I_{b}\right)^{t}$ обладают спектральным разложением с соответствующими диагональными матрицами собственных значений $\Lambda_{1}^{t}$ и $\Lambda_{2}^{t}$, соответствующими ортогональными матрицами собственных векторов $Q_{1}$ и $Q_{2}$. Для произвольного $v\in E_{1}$:
\begin{equation*}
    \begin{aligned}
        &\left\langle\left(\hat{G}^{'}(x, B)^{*}\hat{G}^{'}(x, B) + \tau I_{n}\right)^{t}v,~v\right\rangle=\left\langle Q_{1}\left(\Lambda_{1} + \tau I_{n}\right)^{t}\underbrace{Q_{1}^{*}v}_{\overset{\operatorname{def}}{=}v_{1}},~v\right\rangle=\\
        &=\begin{sqcases}
            \left\langle\underbrace{\left(\Lambda_{1} + \tau I_{n}\right)^{t}}_{\sigma_{\max}(\Lambda_{1})\leq M_{\hat{G}}^{2}}v_{1},~v_{1}\right\rangle\leq\underbrace{\left(M_{\hat{G}}^{2} + \tau\right)^{t}\left\|v_{1}\right\|^{2},~\forall v_{1}\in E_{1}}_{\text{ограничение соотношений Релея}};\\
            \left\langle\underbrace{\left(\Lambda_{1} + \tau I_{n}\right)^{t}}_{\sigma_{\min}(\Lambda_{1})\geq0}v_{1},~v_{1}\right\rangle\geq\underbrace{\tau^{t}\left\|v_{1}\right\|^{2},~\forall v_{1}\in E_{1}}_{\substack{\text{ограничение}\\\text{соотношений}\\\text{Релея}}}.
        \end{sqcases}
    \end{aligned}
\end{equation*}
Аналогичным образом для произвольного $w\in E_{3}^{*},~\dim(E_{3}^{*}) = b$:
\begin{equation*}
    \begin{aligned}
        &\left\langle w,~\left(\hat{G}^{'}(x, B)\hat{G}^{'}(x, B)^{*} + \tau I_{b}\right)^{t}w\right\rangle=\left\langle w,~Q_{2}\left(\Lambda_{2} + \tau I_{b}\right)^{t}\underbrace{Q_{2}^{*}w}_{\overset{\operatorname{def}}{=}w_{1}}\right\rangle=\\
        &=\begin{sqcases}
            \left\langle w_{1},~\underbrace{\left(\Lambda_{2} + \tau I_{b}\right)^{t}}_{\sigma_{\max}(\Lambda_{2})\leq M_{\hat{G}}^{2}}w_{1}\right\rangle\leq\underbrace{\left(M_{\hat{G}}^{2} + \tau\right)^{t}\left\|w_{1}\right\|^{2},~\forall w_{1}\in E_{3}^{*}}_{\text{ограничение соотношений Релея}};\\
            \left\langle w_{1},~\underbrace{\left(\Lambda_{2} + \tau I_{b}\right)^{t}}_{\sigma_{\min}(\Lambda_{2})\geq\mu}w_{1}\right\rangle\geq\underbrace{\left(\mu + \tau\right)^{t}\left\|w_{1}\right\|^{2},~\forall w_{1}\in E_{3}^{*}}_{\substack{\text{ограничение соотношений}\\\text{Релея}}}.
        \end{sqcases}
    \end{aligned}
\end{equation*}
В обоих случаях по принципу минимакса Куранта--Фишера--Вейля при замене $t$ на $-t$ происходит обращение спектра, оценка сверху становится оценкой снизу и наоборот.
Это означает выполнение следующего отношения частичного порядка для матриц со сдвинутым на $\tau$ спектром при $\mu = 0$ (с учётом обращения спектра собственных значений при обращении матриц и возведения в степень собственных значений при возведении в степень матриц):
\begin{equation}\label{eq:as2_matrix_order}
    \begin{cases}
        &\tau^{t} I_{n}\preceq\left(\hat{G}^{'}(x, B)^{*}\hat{G}^{'}(x, B) + \tau I_{n}\right)^{t}\preceq \left(M_{\hat{G}}^{2} + \tau\right)^{t}I_{n},~\tau\geq 0;\\
        &\frac{1}{\left(M_{\hat{G}}^{2} + \tau\right)^{t}} I_{n}\preceq\left(\hat{G}^{'}(x, B)^{*}\hat{G}^{'}(x, B) + \tau I_{n}\right)^{-t}\preceq \frac{1}{\tau^{t}} I_{n},~\tau > 0;\\
        &\tau^{t} I_{b}\preceq\left(\hat{G}^{'}(x, B)\hat{G}^{'}(x, B)^{*} + \tau I_{b}\right)^{t}\preceq\left(M_{\hat{G}}^{2} + \tau\right)^{t}I_{b},~\tau\geq 0;\\
        &\frac{1}{\left(M_{\hat{G}}^{2} + \tau\right)^{t}} I_{b}\preceq\left(\hat{G}^{'}(x, B)\hat{G}^{'}(x, B)^{*} + \tau I_{b}\right)^{-t}\preceq \frac{1}{\tau^{t}} I_{b},~\tau > 0.
    \end{cases}
\end{equation}
\end{proof}

В следующем утверждении выводится липшицевость матриц Якоби $\hat{G}^{'}$ и $\hat{F}^{'}$.

\begin{lemma}\label{lm:aux_lipschitz}
    Пусть выполнено предположение \ref{as:1}. Тогда $\hat{F}^{'}$ и $\hat{G}^{'}$ почти наверно липшиц--непрерывны:
    \begin{equation*}
        \begin{cases}
            \left\|\hat{F}^{'}(x) - \hat{F}^{'}(y)\right\|\leq L_{\hat{F}}\left\|x - y\right\|,~\forall (x, y)\in E_{1}^{2};\\[5pt]
            \left\|\hat{G}^{'}(x, B) - \hat{G}^{'}(y, B)\right\|\leq L_{\hat{F}}\left\|x - y\right\|,~\forall (x, y)\in E_{1}^{2},~\forall B\subseteq\mathcal{B}, |B| = b.
        \end{cases}
    \end{equation*}
    Аналогично функции $\hat{f}_{2}$ и $\hat{g}_{2}$ почти наверно липшиц--непрерывны:
    \begin{equation*}
        \begin{cases}
            \left|\hat{f}_{2}(x) - \hat{f}_{2}(y)\right|\leq l_{\hat{F}}\left\|x - y\right\|,~\forall (x, y)\in E_{1}^{2};\\
            \left|\hat{g}_{2}(x, B) - \hat{g}_{2}(y, B)\right|\leq l_{\hat{F}}\left\|x - y\right\|,~\forall (x, y)\in E_{1}^{2},~\forall B\subseteq\mathcal{B}, |B| = b.
        \end{cases}
    \end{equation*}
\end{lemma}
\begin{proof}
    Рассмотрим батч из функций $\hat{G}$, для произвольных $(x, y)\in E_{1}^{2}$ выпишем:
    \begin{equation*}
        \begin{aligned}
            \left\|\hat{G}^{'}(x, B) - \hat{G}^{'}(y, B)\right\| &\leq \left\|\hat{G}^{'}(x, B) - \hat{G}^{'}(y, B)\right\|_{F} = \sqrt{\frac{1}{b}\sum\limits_{j = 1}^{b}\left\|\nabla F_{i_{j}}(x) - \nabla F_{i_{j}}(y)\right\|^{2}}\leq\left\{\text{предположение \ref{as:1}}\right\}\leq\\
            &\leq\sqrt{\frac{1}{b}\sum\limits_{j = 1}^{b}L_{\hat{F}}^{2}\left\|x - y\right\|^{2}} = L_{\hat{F}}\left\|x - y\right\|;\\
            \left\|\hat{F}^{'}(x) - \hat{F}^{'}(y)\right\| &\leq \left\|\hat{F}^{'}(x) - \hat{F}^{'}(y)\right\|_{F} = \sqrt{\mathbb{E}_{q}\left[\left\|\nabla F_{\xi}(x) - \nabla F_{\xi}(y)\right\|^{2}\right]}\leq\left\{\text{предположение \ref{as:1}}\right\}\leq\\
            &\leq\sqrt{\mathbb{E}_{q}\left[L_{\hat{F}}^{2}\left\|x - y\right\|^{2}\right]} = L_{\hat{F}}\left\|x - y\right\|.
        \end{aligned}
    \end{equation*}
    Аналогичным образом поступим с $\hat{g}_{2}$ и $\hat{f}_{2}$, для произвольных $(x, y)\in E_{1}^{2}$ выпишем:
    \begin{equation*}
        \begin{aligned}
            \left|\hat{g}_{2}(x, B) - \hat{g}_{2}(y, B)\right| &= \left|\frac{1}{b}\sum\limits_{j = 1}^{b}\left(F_{i_{j}}(x)\right)^{2} - \frac{1}{b}\sum\limits_{j = 1}^{b}\left(F_{i_{j}}(y)\right)^{2}\right|\leq\frac{1}{b}\sum\limits_{j = 1}^{b}\left|\left(F_{i_{j}}(x)\right)^{2} - \left(F_{i_{j}}(y)\right)^{2}\right|\leq\\
            &\leq\left\{\text{предположение \ref{as:1}}\right\}\leq\frac{1}{b}\sum\limits_{j = 1}^{b}l_{\hat{F}}\left\|x - y\right\| = l_{\hat{F}}\left\|x - y\right\|;\\
        \end{aligned}
    \end{equation*}
    \begin{equation*}
        \begin{aligned}
            \left|\hat{f}_{2}(x) - \hat{f}_{2}(y)\right| &= \left|\mathbb{E}_{q}\left[\left(F_{\xi}(x)\right)^{2} - \left(F_{\xi}(y)\right)^{2}\right]\right|\leq\mathbb{E}_{q}\left[\left|\left(F_{\xi}(x)\right)^{2} - \left(F_{\xi}(y)\right)^{2}\right|\right]\leq\left\{\text{предположение \ref{as:1}}\right\}\leq\\
            &\leq\mathbb{E}_{q}\left[l_{\hat{F}}\left\|x - y\right\|\right] = l_{\hat{F}}\left\|x - y\right\|.
        \end{aligned}
    \end{equation*}
\end{proof}
\begin{lm:corollary}
Утверждение верно и в случае бесконечной генеральной совокупности $\mathcal{B}$.
\end{lm:corollary}

Выведем модель верхней оценки для функции $\hat{g}_{1}(x, B)$.

\begin{lemma}\label{lm:aux_upper_model}
    Пусть $(x, y)\in E_{1}^{2},~B\subseteq\mathcal{B},~L\geq L_{\hat{F}},~\tau > 0,~\hat{g}_{1}(x, B) > 0$ почти наверно и выполнено предположение \ref{as:1}. Тогда
    \begin{equation}\label{eq:aux_upper_model_formula}
        \begin{cases}
            \hat{g}_{1}(y, B)\leq\hat{\psi}_{x,L,\tau}(y, B) = \frac{\tau}{2} + \frac{L}{2}\left\|y - x\right\|^{2} + \frac{1}{2\tau}\left\|\hat{G}(x, B) + \hat{G}^{'}(x, B)(y - x)\right\|^{2};\\[10pt]
            \hat{f}_{1}(y)\leq\psi_{x,L,\tau}(y) \overset{\operatorname{def}}{=} \hat{\psi}_{x, L, \tau}(y, \mathcal{B}) = \frac{\tau}{2} + \frac{L}{2}\left\|y - x\right\|^{2} + \frac{1}{2\tau}\left\|\hat{F}(x) + \hat{F}^{'}(x)(y - x)\right\|^{2}.
        \end{cases}
    \end{equation}
\end{lemma}
\begin{proof}
    Доказательство структурно повторяет лемму \ref{lm:aux_det_upper_model} с $\hat{F} := \hat{G}$ в рамках одного батча $B\subseteq\mathcal{B}$ для произвольных $(x, y)\in E_{1}^{2}$, $L\geq L_{\hat{F}}$, $\tau > 0$.
\end{proof}

Для того, чтобы определить метод трёх стохастических квадратов, необходимо задать правило обновления искомого параметра $x_{k + 1}$. В отличие от \cite{Nesterov2020}, где это правило имеет вид $x_{k + 1} = x_{k} - v_{k},~v_{k}\in E_{1}^{*}$, в данном случае рассматривается масштабированное обновление:
$$x_{k + 1} = x_{k} - \eta_{k}v_{k},~\eta_{k} > 0,~v_{k}\in E_{1}^{*}.$$
Подобное масштабирование может быть полезным в случае, когда на каждой итерации метода оптимизации вспомогательная задача поиска $x_{k + 1}$ должна решаться не слишком точно, чтобы, например, уменьшить переобучение настраиваемой модели в задаче машинного обучения. Также такого вида обновление может быть использовано в задаче мета--обучения, в которой $\eta_{k}$ будет играть роль гиперпараметра, настраиваемого на отдельной валидационной выборке. Ниже определим основные свойства данного обновления $x_{k}$.

\begin{lemma}\label{lm:aux_update}
    Пусть выполнено предположение \ref{as:1}, $x_{k}\in E_{1}$, $\tau_{k} > 0$, $L_{k}\geq L_{\hat{F}}$, $B_{k}\subseteq\mathcal{B}$, $\eta_{k} > 0$, начальное приближение $x_{0}$ выбирается случайно и независимо от $B_{k},~k\in\mathbb{Z}_{+}$. Тогда
    \begin{equation}\label{eq:stoch_general_update_rule}
        \begin{cases}
            x_{k + 1} = x_{k} - \eta_{k}\left(\hat{G}^{'}(x_{k}, B_{k})^{*}\hat{G}^{'}(x_{k}, B_{k}) + \tau_{k} L_{k} I_{n}\right)^{-1}\hat{G}^{'}(x_{k}, B_{k})^{*}\hat{G}(x_{k}, B_{k});\\
            \begin{aligned}
                \hat{g}_{1}(x_{k}&, B_{k}) - \hat{g}_{1}(x_{k + 1}, B_{k})\geq\hat{g}_{1}(x_{k}, B_{k}) - \frac{\tau_{k}}{2} - \frac{\hat{g}_{2}(x_{k}, B_{k})}{2\tau_{k}} +\\
                &+ \frac{\eta_{k}(2 - \eta_{k})}{2\tau_{k}}\left\langle\left(\hat{G}^{'}(x_{k}, B_{k})^{*}\hat{G}^{'}(x_{k}, B_{k}) + \tau_{k} L_{k} I_{n}\right)^{-1}\hat{G}^{'}(x_{k}, B_{k})^{*}\hat{G}(x_{k}, B_{k}),~\hat{G}^{'}(x_{k}, B_{k})^{*}\hat{G}(x_{k}, B_{k})\right\rangle.
            \end{aligned}
        \end{cases}
    \end{equation}
\end{lemma}
\begin{proof}
    Доказательство заключается в использовании цепочки утверждений из теоремы \ref{th:1}, в частности, выражения \eqref{eq:delta_f_2} с подстановкой $\hat{F} := \hat{G}$ в рамках одного батча $B_{k}\subseteq\mathcal{B}$ для произвольных $x_{k}\in E_{1}$, $\tau_{k} > 0$, $L_{k}\geq L_{\hat{F}}$, $\eta_{k} > 0$.
\end{proof}
\begin{lm:corollary}\label{lm:aux_update_rule}
    Если взять $\eta_{k}\in (0, 2),~\tau_{k} = \hat{g}_{1}(x_{k}, B_{k})$, вычисленную на том же батче, формирующем $\hat{G}$, то получится
    \begin{equation*}
        \begin{aligned}
            \hat{g}_{1}(x_{k}, &B_{k}) - \hat{g}_{1}(x_{k + 1}, B_{k})\geq\frac{\eta_{k}(2 - \eta_{k})}{2\hat{g}_{1}(x_{k}, B_{k})}\left\langle\left(\hat{G}^{'}(x_{k}, B_{k})^{*}\hat{G}^{'}(x_{k}, B_{k}) + \hat{g}_{1}(x_{k}, B_{k}) L_{k} I_{n}\right)^{-1}\hat{G}^{'}(x_{k}, B_{k})^{*}\hat{G}(x_{k}, B_{k}),\right.\\
            &\left.~~~~~~~~~~~~~~~~~~~~~~~~~~~~~~~~~~~~~~~~~~~~~~\hat{G}^{'}(x_{k}, B_{k})^{*}\hat{G}(x_{k}, B_{k})\right\rangle\geq 0,
        \end{aligned}
    \end{equation*}
    так как матрица $\left(\hat{G}^{'}(x_{k}, B_{k})^{*}\hat{G}^{'}(x_{k}, B_{k}) + \hat{g}_{1}(x_{k}, B_{k}) L_{k} I_{n}\right)^{-1}$ положительно определённая.
\end{lm:corollary}
\begin{lm:corollary}\label{lm:aux_mean_decrease}
    Из следствия \ref{lm:aux_update_rule} выводится уменьшение значения $\hat{g}_{2}$ в среднем:
    \begin{equation*}
        \begin{aligned}
            \hat{g}_{1}(&x_{k}, B_{k}) - \hat{g}_{1}(x_{k + 1}, B_{k})\geq\frac{\eta_{k}(2 - \eta_{k})}{2\hat{g}_{1}(x_{k}, B_{k})}\left\langle\left(\hat{G}^{'}(x_{k}, B_{k})^{*}\hat{G}^{'}(x_{k}, B_{k}) + \hat{g}_{1}(x_{k}, B_{k}) L_{k} I_{n}\right)^{-1}\hat{G}^{'}(x_{k}, B_{k})^{*}\hat{G}(x_{k}, B_{k}),\right.\\
            &\left.\hat{G}^{'}(x_{k}, B_{k})^{*}\hat{G}(x_{k}, B_{k})\right\rangle\geq 0\Rightarrow\hat{g}_{2}(x_{k}, B_{k}) - \hat{g}_{2}(x_{k + 1}, B_{k})\geq\hat{g}_{2}(x_{k}, B_{k}) - \hat{g}_{1}(x_{k}, B_{k})\hat{g}_{1}(x_{k + 1}, B_{k})\geq\\
            &\geq\frac{\eta_{k}(2 - \eta_{k})}{2}\left\langle\left(\hat{G}^{'}(x_{k}, B_{k})^{*}\hat{G}^{'}(x_{k}, B_{k}) + \hat{g}_{1}(x_{k}, B_{k}) L_{k} I_{n}\right)^{-1}\hat{G}^{'}(x_{k}, B_{k})^{*}\hat{G}(x_{k}, B_{k}),\right.\\
            &\left.~~~~~~~~~~~~~~~~~~~~~\hat{G}^{'}(x_{k}, B_{k})^{*}\hat{G}(x_{k}, B_{k})\right\rangle\geq 0.
        \end{aligned}
    \end{equation*}
    Усреднив по всей случайности при оптимизации (выбор батча на каждой итерации и выбор начального приближения), получаем:
    \begin{equation*}
        \begin{aligned}
            \mathbb{E}&\left[\hat{g}_{2}(x_{k}, B_{k}) - \hat{g}_{2}(x_{k + 1}, B_{k})\right] = \mathbb{E}\left[\hat{f}_{2}(x_{k}) - \hat{g}_{2}(x_{k + 1}, B_{k})\right] = \mathbb{E}\left[\hat{f}_{2}(x_{k})\right] - \mathbb{E}\left[\hat{g}_{2}(x_{k + 1}, B_{k})\right]\geq\\
            &\geq\mathbb{E}\left[\frac{\eta_{k}(2 - \eta_{k})}{2}\left\langle\left(\hat{G}^{'}(x_{k}, B_{k})^{*}\hat{G}^{'}(x_{k}, B_{k}) + \hat{g}_{1}(x_{k}, B_{k}) L_{k} I_{n}\right)^{-1}\hat{G}^{'}(x_{k}, B_{k})^{*}\hat{G}(x_{k}, B_{k}),\right.\right.\\
            &\left.\left.~~~~~~~~~~~~~~~~~~~~~~~~~\hat{G}^{'}(x_{k}, B_{k})^{*}\hat{G}(x_{k}, B_{k})\right\rangle\right]\geq 0.
        \end{aligned}
    \end{equation*}
    Если же провести усреднение только по батчам на итерации получения $x_{k + 1}$, по которым вычисляется  $x_{k + 1}$ (обозначим соответствующий оператор как $\mathbb{E}_{B_{k}}\left[\cdot\right]$), не усредняя по стохастичности с предыдущих итераций, то получатся следующие неравенства между зависимыми случайными величинами:
    \begin{equation*}
        \begin{aligned}
            \sqrt{\mathbb{E}_{B_{k}}\left[\hat{g}_{2}(x_{k}, B_{k})\right]} &= \sqrt{\hat{f}_{2}(x_{k})} = \hat{f}_{1}(x_{k})\geq\sqrt{\mathbb{E}_{B_{k}}\left[\hat{g}_{2}(x_{k + 1}, B_{k})\right]}\geq\left\{\text{неравенство Йенсена}\right\}\geq\\
            &\geq\mathbb{E}_{B_{k}}\left[\sqrt{\hat{g}_{2}(x_{k + 1}, B_{k})}\right]\Rightarrow\hat{f}_{1}(x_{k})\geq\mathbb{E}_{B_{k}}\left[\hat{g}_{1}(x_{k + 1}, B_{k})\right],
        \end{aligned}
    \end{equation*}
    так как значение $x_{k + 1}$ зависит от $B_{k}$.
\end{lm:corollary}
\begin{lm:corollary}[\cite{Nesterov2020}]\label{lm:aux_det_update_rule}
    При замене в условии леммы $\hat{G}$ на $\hat{F}$ и $\hat{g}_{1}$ на $\hat{f}_{1}$ получается
    \begin{equation*}
        \begin{cases}
            x_{k + 1} = x_{k} - \eta_{k}\left(\hat{F}^{'}(x_{k})^{*}\hat{F}^{'}(x_{k}) + \tau_{k} L_{k} I_{n}\right)^{-1}\hat{F}^{'}(x_{k})^{*}\hat{F}(x_{k});\\
            \begin{aligned}
                \hat{f}_{1}(x_{k}) - \hat{f}_{1}(x_{k + 1})&\geq\hat{f}_{1}(x_{k}) - \frac{\tau_{k}}{2} - \frac{\hat{f}_{2}(x_{k})}{2\tau_{k}} +\\
                &+ \frac{\eta_{k}(2 - \eta_{k})}{2\tau_{k}}\left\langle\left(\hat{F}^{'}(x_{k})^{*}\hat{F}^{'}(x_{k}) + \tau_{k} L_{k} I_{n}\right)^{-1}\hat{F}^{'}(x_{k})^{*}\hat{F}(x_{k}),~\hat{F}^{'}(x_{k})^{*}\hat{F}(x_{k})\right\rangle.
            \end{aligned}
        \end{cases}
    \end{equation*}
    При $\eta_{k}\in (0, 2),~\tau_{k} = \hat{f}_{1}(x_{k})$:
    \begin{equation*}
        \begin{aligned}
            &\hat{f}_{1}(x_{k}) - \hat{f}_{1}(x_{k + 1})\geq\frac{\eta_{k}(2 - \eta_{k})}{2\hat{f}_{1}(x_{k})}\left\langle\left(\hat{F}^{'}(x_{k})^{*}\hat{F}^{'}(x_{k}) + \hat{f}_{1}(x_{k}) L_{k} I_{n}\right)^{-1}\hat{F}^{'}(x_{k})^{*}\hat{F}(x_{k}),~\hat{F}^{'}(x_{k})^{*}\hat{F}(x_{k})\right\rangle\geq 0,
        \end{aligned}
    \end{equation*}
    так как матрица $\left(\hat{F}^{'}(x_{k})^{*}\hat{F}^{'}(x_{k}) + \hat{f}_{1}(x_{k}) L_{k} I_{n}\right)^{-1}$ так же положительно определённая.
\end{lm:corollary}
Лемма \ref{lm:aux_update} по сути в своём условии содержит описание метода трёх стохастических квадратов. Для анализа влияния оценки $\hat{f}_{1}$ по батчу из функций $B$ важно понять, насколько из--за введённой стохастики $g_{1}$ неточно оценивает оптимизируемый функционал $\hat{f}_{1}$ в зависимости от итерации метода оптимизации, что как раз рассмотрено в леммах \ref{lm:aux_finite_population_variance} и \ref{lm:aux_bounded_deviation}. 

\begin{lemma}\label{lm:aux_finite_population_variance}
    Пусть выполнено предположение \ref{as:4}. При сэмплировании без возвращения батчей $B\subseteq\mathcal{B}$, $|B| = b$ из равномерного распределения $q$ на подмножества $B$ имеет место следующее неравенство:
    \begin{equation*}
        \mathbb{E}_{q}\left[\left|\hat{g}_{2}(x, B) - \hat{f}_{2}(x)\right|^{2}\right]\leq\frac{\tilde{\sigma}^{2}}{b}\left(1 - \frac{b}{m}\right),~\forall x\in E_{1},
    \end{equation*}
    для некоторого конечного $\tilde{\sigma}\geq\sigma$.
\end{lemma}
\begin{proof}
Математическое ожидание $\hat{g}_{2}(x, B)$ по $q$ от сэмпла батча $B$ можно представить с помощью зависимых Бернулли случайных величин $Z_{i}\in\{0, 1\}$, у которых значение 0 обозначает отсутствие $F_{i}$ в батче $B$, а значение 1 --- наличие $F_{i}$ в батче $B$:
\begin{equation*}
    \begin{aligned}
        \mathbb{E}_{q}\left[\hat{g}_{2}(x, B)\right] &= \mathbb{E}_{q}\left[\frac{1}{b}\sum\limits_{j = 1}^{b}\left(F_{i_{j}}(x)\right)^{2}\right] = \frac{1}{b}\mathbb{E}\left[\sum\limits_{i = 1}^{m}\left(F_{i}(x)\right)^{2}Z_{i}\right] = \frac{1}{b}\sum\limits_{i = 1}^{m}\left(F_{i}(x)\right)^{2}\mathbb{E}\left[Z_{i}\right] =\\
        &= \frac{1}{m}\sum\limits_{i = 1}^{m}\left(F_{i}(x)\right)^{2} = \hat{f}_{2}(x),
    \end{aligned}
\end{equation*}
так как вероятность для $F_{i}$ оказаться в сэмпле $B$ согласно определению $q$ составляет
$$\Prob\left(Z_{i} = 1\right) = \frac{C_{m - 1}^{b - 1}}{C_{m}^{b}} = \frac{(m - 1)!}{(m - b)!(b - 1)!}\frac{(m - b)!b!}{m!} = \frac{b}{m},~i\in\overline{1, m}.$$
По определению дисперсии случайной величины, заданной на конечной генеральной совокупности:
\begin{equation*}
    \mathbb{V}_{q}\left[\left(F_{\xi}(x)\right)^{2}\right] = \frac{1}{m}\sum\limits_{i = 1}^{m}\left(\left(F_{i}(x)\right)^{2} - \hat{f}_{2}(x)\right)^{2} = \frac{m - 1}{m}(\sigma(x))^{2}\leq\sigma^{2},
\end{equation*}
$\sigma(x)$ --- квази--дисперсия для сэмпла $B$ с $|B| = 1$ при фиксированном $x\in E_{1}$. По предположению \ref{as:4}: $\sigma(x)\leq\sigma\sqrt{\frac{m}{m - 1}}$.
Дисперсия значения функции $g_{2}$ при фиксированном $x$ равна:
\begin{equation*}
    \begin{aligned}
        \mathbb{V}_{q}\left[\hat{g}_{2}(x, B)\right] &= \mathbb{E}_{q}\left[\left|\hat{g}_{2}(x, B) - \hat{f}_{2}(x)\right|^{2}\right] = \mathbb{V}\left[\frac{1}{b}\sum\limits_{i = 1}^{m}(F_{i}(x))^{2}Z_{i}\right] =\\
        &= \frac{1}{b^{2}}\left(\sum\limits_{i = 1}^{m}(F_{i}(x))^{4}\mathbb{V}\left[Z_{i}\right] + 2\sum\limits_{i = 1}^{m}\sum\limits_{j = i + 1}^{m}(F_{i}(x)F_{j}(x))^{2}\operatorname{Cov}(Z_{i}, Z_{j})\right),
    \end{aligned}
\end{equation*}
а суммирование по пустому множеству индексов приравнено к нулю. $Z_{i},~i\in\overline{1, m}$ распределены по Бернулли, поэтому $\mathbb{V}$ и $\operatorname{Cov}$ определены следующим образом:
\begin{equation*}
    \begin{aligned}
        &\begin{cases}
            \mathbb{V}\left[Z_{i}\right] = \frac{b}{m}\left(1 - \frac{b}{m}\right);\\
            \operatorname{Cov}(Z_{i}, Z_{j}) = \mathbb{E}\left[Z_{i}Z_{j}\right] - \mathbb{E}\left[Z_{i}\right]\mathbb{E}\left[Z_{j}\right] = \frac{C_{m - 2}^{b - 2}}{C_{m}^{b}} - \left(\frac{b}{m}\right)^{2} = \frac{b(b - 1)}{m(m - 1)} - \left(\frac{b}{m}\right)^{2}.
        \end{cases}
    \end{aligned}
\end{equation*}
Подставим полученные значения в выражение $\mathbb{V}_{q}\left[\hat{g}_{2}(x, B)\right]$:
\begin{equation*}
    \begin{aligned}
        \mathbb{V}_{q}\left[\hat{g}_{2}(x, B)\right] &= \frac{1}{b^{2}}\left(\sum\limits_{i = 1}^{m}(F_{i}(x))^{4}\mathbb{V}\left[Z_{i}\right] + 2\sum\limits_{i = 1}^{m}\sum\limits_{j = i + 1}^{m}(F_{i}(x)F_{j}(x))^{2}\operatorname{Cov}(Z_{i}, Z_{j})\right)=\\
        &= \frac{1}{b^{2}}\left(\sum\limits_{i = 1}^{m}(F_{i}(x))^{4}\frac{b}{m}\left(1 - \frac{b}{m}\right) + 2\sum\limits_{i = 1}^{m}\sum\limits_{j = i + 1}^{m}(F_{i}(x)F_{j}(x))^{2}\left(\frac{b(b - 1)}{m(m - 1)} - \left(\frac{b}{m}\right)^{2}\right)\right) =\\
        &= \frac{(m - b)}{mb}\left(\frac{1}{m}\sum\limits_{i = 1}^{m}(F_{i}(x))^{4} - \frac{2}{m(m - 1)}\sum\limits_{i = 1}^{m}\sum\limits_{j = i + 1}^{m}(F_{i}(x)F_{j}(x))^{2}\right) =
    \end{aligned}
\end{equation*}
\begin{equation*}
    \begin{aligned}
        &= \frac{1}{b}\left(1 - \frac{b}{m}\right)\left(\frac{1}{m - 1}\left(\sum\limits_{i = 1}^{m}\left(F_{i}(x)\right)^{2} - \hat{f}_{2}(x)\right)^{2}\right) = \frac{(\sigma(x))^{2}}{b}\left(1 - \frac{b}{m}\right).
    \end{aligned}
\end{equation*}
Введём $\tilde{\sigma} \overset{\operatorname{def}}{=} \sigma\sqrt{\frac{m}{m - 1}}$, тогда получается
\begin{equation*}
    \begin{aligned}
        \mathbb{E}_{q}\left[\left|\hat{g}_{2}(x, B) - \hat{f}_{2}(x)\right|^{2}\right] &= \frac{(\sigma(x))^{2}}{b}\left(1 - \frac{b}{m}\right)\leq\frac{m\sigma^{2}}{b(m - 1)}\left(1 - \frac{b}{m}\right) = \frac{\tilde{\sigma}^{2}}{b}\left(1 - \frac{b}{m}\right),~\forall x\in E_{1}.
    \end{aligned}
\end{equation*}
\end{proof}
\begin{lm:corollary}\label{lm:co:aux_finite_population_variance}
Выведенная оценка также обобщается на случай бесконечной генеральной совокупности, соответствующая этому случаю оценка совпадает с оценкой, получаемой при независимом сэмплировании с возвращением из генеральной совокупности произвольного размера:
\begin{equation*}
    \begin{aligned}
        &\mathbb{E}_{q}\left[\left|\hat{g}_{2}(x, B) - \hat{f}_{2}(x)\right|^{2}\right]\leq\lim\limits_{m\rightarrow+\infty}\left[\frac{\tilde{\sigma}^{2}}{b}\left(1 - \frac{b}{m}\right)\right] = \frac{\tilde{\sigma}^{2}}{b},~\forall x\in E_{1}.
    \end{aligned}
\end{equation*}
Эту оценку можно вывести непосредственно:
\begin{equation*}
    \begin{aligned}
        &\mathbb{V}_{q}\left[\hat{g}_{2}(x, B)\right] = \mathbb{V}_{q}\left[\frac{1}{b}\sum\limits_{j = 1}^{b}\left(F_{i_{j}}(x)\right)^{2}\right] = \frac{1}{b^{2}}\sum\limits_{j = 1}^{b}\mathbb{V}_{q}\left[\left(F_{i_{j}}(x)\right)^{2}\right] \leq \frac{\sigma^{2}}{b}\leq\frac{\tilde{\sigma}^{2}}{b},
    \end{aligned}
\end{equation*}
при этом $\lim\limits_{m\rightarrow+\infty}[\tilde{\sigma}] = \sigma$.
\end{lm:corollary}
\begin{lm:corollary}\label{lm:aux_finite_population_deviation}
Из условий леммы следует ограниченность математического ожидания модуля\\$\left|\hat{g}_{2}(x, B) - \hat{f}_{2}(x)\right|$ при всех $x\in E_{1}$:
\begin{equation*}
    \begin{aligned}
        \mathbb{E}_{q}\left[\left|\hat{g}_{2}(x, B) - \hat{f}_{2}(x)\right|\right] &= \mathbb{E}_{q}\left[\sqrt{\left|\hat{g}_{2}(x, B) - \hat{f}_{2}(x)\right|^{2}}\right]\leq\sqrt{\mathbb{E}_{q}\left[\left|\hat{g}_{2}(x, B) - \hat{f}_{2}(x)\right|^{2}\right]}\leq\tilde{\sigma}\sqrt{\frac{1}{b} - \frac{1}{m}}.
    \end{aligned}
\end{equation*}
\end{lm:corollary}

\begin{lemma}\label{lm:aux_bounded_deviation}
    Пусть выполнены предположения \ref{as:1} и \ref{as:4} и дана последовательность $\{x_{k - 1}\}_{k\in\mathbb{N}}$, $x_{k - 1}\in E_{1}$, вычисляемая по одному из правил: \eqref{eq:stoch_direct_update_rule}, \eqref{eq:double_stoch_direct_update_rule}, \eqref{eq:stoch_approx_general_update_rule}. При независимом сэмплировании без возвращения батчей $B_{k - 1}\subseteq\mathcal{B}$, $|B_{k - 1}| = b$ из равномерного распределения $q$ на подмножества $B_{k - 1}$ для каждого $k\in\mathbb{N}$ имеет место следующее неравенство:
    \begin{equation*}
        \mathbb{E}\left[\left|\hat{f}_{2}(x_{k}) - \hat{g}_{2}(x_{k}, B_{k - 1})\right|\right]\leq2l_{\hat{F}}\mathbb{E}\left[\left\|x_{k} - x_{k - 1}\right\|\right]\mathds{1}_{\left\{b < m\right\}} + \tilde{\sigma}\sqrt{\frac{1}{b} - \frac{1}{m}},
    \end{equation*}
    для некоторого конечного $\tilde{\sigma}\geq\sigma$. Усреднение производится по всем сэмплам $B_{k - 1},~k\in\mathbb{N}$.
\end{lemma}
\begin{proof}
Выпишем оценку сверху:
\begin{equation}\label{eq:finite_population_deviation}
    \begin{aligned}
        \mathbb{E}&\left[\left|\hat{f}_{2}(x_{k}) - \hat{g}_{2}(x_{k}, B_{k - 1})\right|\right] = \mathbb{E}\left[\left|\hat{f}_{2}(x_{k}) - \hat{f}_{2}(x_{k - 1}) + \hat{f}_{2}(x_{k - 1}) - \hat{g}_{2}(x_{k - 1}, B_{k - 1}) + \hat{g}_{2}(x_{k - 1}, B_{k - 1}) -\right.\right.\\
        &\left.\left.- \hat{g}_{2}(x_{k}, B_{k - 1})\right|\right]\leq\mathbb{E}\left[\left|\hat{f}_{2}(x_{k}) - \hat{f}_{2}(x_{k - 1})\right|\right] + \mathbb{E}\left[\left|\hat{f}_{2}(x_{k - 1}) - \hat{g}_{2}(x_{k - 1}, B_{k - 1})\right|\right] +\\
        &+ \mathbb{E}\left[\left|\hat{g}_{2}(x_{k - 1}, B_{k - 1}) - \hat{g}_{2}(x_{k}, B_{k - 1})\right|\right]\leq\left\{\text{липшицевость $\hat{g}_{2}$ и $\hat{f}_{2}$}\right\}\leq 2l_{\hat{F}}\mathbb{E}\left[\left\|x_{k} - x_{k - 1}\right\|\right] +\\
        &+ \mathbb{E}\left[\sqrt{\left|\hat{f}_{2}(x_{k - 1}) - \hat{g}_{2}(x_{k - 1}, B_{k - 1})\right|^{2}}\right]\leq 2l_{\hat{F}}\mathbb{E}\left[\left\|x_{k} - x_{k - 1}\right\|\right] + \sqrt{\mathbb{E}\left[\left|\hat{f}_{2}(x_{k - 1}) - \hat{g}_{2}(x_{k - 1}, B_{k - 1})\right|^{2}\right]}\leq\\
        &\leq\left\{\text{лемма \ref{lm:aux_finite_population_variance}}\right\}\leq 2l_{\hat{F}}\mathbb{E}\left[\left\|x_{k} - x_{k - 1}\right\|\right] + \sqrt{\frac{\tilde{\sigma}^{2}}{b}\left(1 - \frac{b}{m}\right)} = 2l_{\hat{F}}\mathbb{E}\left[\left\|x_{k} - x_{k - 1}\right\|\right] + \tilde{\sigma}\sqrt{\frac{1}{b} - \frac{1}{m}} =\\
        &= 2l_{\hat{F}}\mathbb{E}\left[\left\|x_{k} - x_{k - 1}\right\|\right]\mathds{1}_{\left\{b < m\right\}} + \tilde{\sigma}\sqrt{\frac{1}{b} - \frac{1}{m}},
    \end{aligned}
\end{equation}
так как при $b = m$ выполнено $\hat{f}_{2}(x_{k}) = \hat{g}_{2}(x_{k}, \mathcal{B}) = \hat{g}_{2}(x_{k}, B_{k - 1})$.
Для конечной генеральной совокупности оценку в \eqref{eq:finite_population_deviation} иногда удобно представить следующим образом:
\begin{equation}\label{eq:deviation_upper_bound}
    \begin{aligned}
        \mathbb{E}\left[\left|\hat{f}_{2}(x_{k}) - \hat{g}_{2}(x_{k}, B_{k - 1})\right|\right]&\leq 2l_{\hat{F}}\mathbb{E}\left[\left\|x_{k} - x_{k - 1}\right\|\right]\mathds{1}_{\left\{b < m\right\}} + \tilde{\sigma}\sqrt{\frac{1}{b} - \frac{1}{m}}\leq\\
        &\leq\sqrt{\frac{1}{b} - \frac{1}{m}}\left(2l_{\hat{F}}\sqrt{m(m - 1)}\mathbb{E}\left[\left\|x_{k} - x_{k - 1}\right\|\right]\mathds{1}_{\left\{b < m\right\}} + \tilde{\sigma}\right).
    \end{aligned}
\end{equation}
\end{proof}
\begin{lm:corollary}
Как и в лемме \ref{lm:aux_finite_population_variance}, выведенная оценка также обобщается на случай бесконечной генеральной совокупности, и так же соответствующая этому случаю оценка совпадает с оценкой, получаемой при независимом сэмплировании с возвращением из генеральной совокупности произвольного размера:
\begin{equation*}
    \begin{aligned}
        \mathbb{E}\left[\left|\hat{f}_{2}(x_{k}) - \hat{g}_{2}(x_{k}, B_{k - 1})\right|\right]&\leq\lim\limits_{m\rightarrow+\infty}\left[2l_{\hat{F}}\mathbb{E}\left[\left\|x_{k} - x_{k - 1}\right\|\right] + \tilde{\sigma}\sqrt{\frac{1}{b} - \frac{1}{m}}\right] =\\
        &= 2l_{\hat{F}}\mathbb{E}\left[\left\|x_{k} - x_{k - 1}\right\|\right] + \frac{\tilde{\sigma}}{\sqrt{b}},~\forall k\in\mathbb{N}.
    \end{aligned}
\end{equation*}
Аналогично с помощью следствия \ref{lm:co:aux_finite_population_variance} можно вывести эту оценку:
\begin{equation*}
    \begin{aligned}
        \mathbb{E}\left[\left|\hat{f}_{2}(x_{k}) - \hat{g}_{2}(x_{k}, B_{k - 1})\right|\right] &\leq\left\{\text{\eqref{eq:finite_population_deviation}}\right\}\leq 2l_{\hat{F}}\mathbb{E}\left[\left\|x_{k} - x_{k - 1}\right\|\right] + \sqrt{\mathbb{E}\left[\left|\hat{f}_{2}(x_{k - 1}) - \hat{g}_{2}(x_{k - 1}, B_{k - 1})\right|^{2}\right]}\leq\\
        &\leq\left\{\text{лемма \ref{lm:aux_finite_population_variance}, следствие \ref{lm:co:aux_finite_population_variance}}\right\}\leq2l_{\hat{F}}\mathbb{E}\left[\left\|x_{k} - x_{k - 1}\right\|\right] + \frac{\tilde{\sigma}}{\sqrt{b}},
    \end{aligned}
\end{equation*}
при этом $\lim\limits_{m\rightarrow+\infty}[\tilde{\sigma}] = \sigma$.
\end{lm:corollary}
\begin{lm:corollary}
При $|B_{k - 1}| = m,~k\in\mathbb{N}$ или при полностью детерминированном вычислении $x_{k}$ относительно (при условии) $(x_{k - 1}, B_{k - 1})$ значение $\mathbb{E}\left[\left|\hat{f}_{2}(x_{k}) - \hat{g}_{2}(x_{k}, B_{k - 1})\right|\right] = 0$
\end{lm:corollary}
Стоит также заметить, что согласно леммам \ref{lm:aux_bounded_deviation}  и \ref{lm:aux_finite_population_variance} оценка $\hat{f}_{2}(x_{k})$, полученная на $(k - 1)$--ой итерации в виде $\hat{g}_{2}(x_{k}, B_{k - 1})$ менее точна, чем $\hat{g}_{2}(x_{k}, B_{k})$, и явно зависит от расстояния между $x_{k}$ и $x_{k - 1}$. Это наблюдение перекликается с оценками неопределённости в теории случайных процессов, что наводит на необходимость проверки гипотезы о пользе масштабирования шага на каждой итерации с помощью $\eta_{k}$ для получения более точного и быстрого решения задачи \eqref{eq:main_opt_problem}.

В следующей лемме выведено представление $\hat{\psi}_{x_{k}, L_{k}, \tau_{k}}(y, B_{k})$ относительно точки $x_{k + 1}$, которая при $\eta_{k} = 1$ является точкой глобального минимума функции $\hat{\psi}_{x_{k}, L_{k}, \tau_{k}}(\cdot, B_{k})$, это представление непосредственно можно использовать для анализа сходимости метода трёх квадратов в случае поиска $x_{k + 1}$ не просто в рамках стохастической оптимизации, но ещё и при использовании \textit{неточного проксимального отображения} (следствие \ref{lm:aux_approx_solution}).
\begin{lemma}\label{lm:aux_value_tolerance}
    Пусть выполнено предположение \ref{as:1} и дана последовательность $\{x_{k}\}_{k\in\mathbb{Z}_{+}}$, $x_{k}\in E_{1}$, вычисляемая по правилу \eqref{eq:stoch_direct_update_rule} с $\tau_{k} > 0$, $L_{k} > 0$, $B_{k}\subseteq\mathcal{B}$, $\eta_{k}\in(0, 2)$. Тогда для произвольного $y\in E_{1}$ верно
    \begin{equation*}
        \begin{aligned}
            \hat{\psi}_{x_{k}, L_{k}, \tau_{k}}(y, B_{k}) &= \hat{\psi}_{x_{k}, L_{k}, \tau_{k}}(x_{k + 1}, B_{k}) + \frac{L_{k}}{2}\left\|y - x_{k + 1}\right\|^{2} +\\
            &+ \frac{1}{2\tau_{k}}\left\|\hat{G}^{'}(x_{k}, B_{k})(y - x_{k + 1})\right\|^{2} + \frac{1 - \eta_{k}}{2\tau_{k}}\left\langle y - x_{k + 1},~\nabla_{x_{k}}\hat{g}_{2}(x_{k}, B_{k})\right\rangle.
        \end{aligned}
    \end{equation*}
\end{lemma}
\begin{proof}
Перепишем $\hat{\psi}_{x_{k},L_{k},\tau_{k}}(y, B_{k})$:
\begin{equation*}
    \begin{aligned}
        &\hat{\psi}_{x_{k},L_{k},\tau_{k}}(y, B_{k}) = \frac{\tau_{k}}{2} + \frac{L_{k}}{2}\left\|y - x_{k}\right\|^{2} + \frac{1}{2\tau_{k}}\left\|\hat{G}(x_{k}, B_{k}) + \hat{G}^{'}(x_{k}, B_{k})(y - x_{k})\right\|^{2} = \frac{\tau_{k}}{2} +\\
        &+ \frac{L_{k}}{2}\left\|(y - x_{k + 1}) + (x_{k + 1} - x_{k})\right\|^{2} + \frac{1}{2\tau_{k}}\left\|\hat{G}(x_{k}, B_{k}) + \hat{G}^{'}(x_{k}, B_{k})((y - x_{k + 1}) + (x_{k + 1} - x_{k}))\right\|^{2} =\\
        &= \frac{\tau_{k}}{2} + \frac{L_{k}}{2}\left\|y - x_{k + 1}\right\|^{2} + L_{k}\left\langle y - x_{k + 1},~x_{k + 1} - x_{k}\right\rangle + \frac{L_{k}}{2}\left\|x_{k + 1} - x_{k}\right\|^{2} +\\
        &+ \frac{1}{2\tau_{k}}\left\|\left(\hat{G}(x_{k}, B_{k}) + \hat{G}^{'}(x_{k}, B_{k})(x_{k + 1} - x_{k})\right) + \hat{G}^{'}(x_{k}, B_{k})(y - x_{k + 1})\right\|^{2} =\\
        &= \left(\frac{\tau_{k}}{2} + \frac{L_{k}}{2}\left\|x_{k + 1} - x_{k}\right\|^{2} + \frac{1}{2\tau_{k}}\left\|\hat{G}(x_{k}, B_{k}) + \hat{G}^{'}(x_{k}, B_{k})(x_{k + 1} - x_{k})\right\|^{2}\right) + \frac{L_{k}}{2}\left\|y - x_{k + 1}\right\|^{2} +\\
        &+ \left\langle y - x_{k + 1},~L_{k}(x_{k + 1} - x_{k})\right\rangle + \frac{1}{\tau_{k}}\left\langle \hat{G}^{'}(x_{k}, B_{k})(y - x_{k + 1}),~\hat{G}(x_{k}, B_{k}) + \hat{G}^{'}(x_{k}, B_{k})(x_{k + 1} - x_{k})\right\rangle +\\
        &+ \frac{1}{2\tau_{k}}\left\|\hat{G}^{'}(x_{k}, B_{k})(y - x_{k + 1})\right\|^{2} = \hat{\psi}_{x_{k}, L_{k}, \tau_{k}}(x_{k + 1}, B_{k}) + \frac{L_{k}}{2}\left\|y - x_{k + 1}\right\|^{2} + \frac{1}{2\tau_{k}}\left\|\hat{G}^{'}(x_{k}, B_{k})(y - x_{k + 1})\right\|^{2} +\\
    \end{aligned}
\end{equation*}
\begin{equation*}
    \begin{aligned}
        &+ \left\langle y - x_{k + 1},\underbrace{L_{k}(x_{k + 1} - x_{k}) + \frac{1}{\tau_{k}}\hat{G}^{'}(x_{k}, B_{k})^{*}\left(\hat{G}(x_{k}, B_{k}) + \hat{G}^{'}(x_{k}, B_{k})(x_{k + 1} - x_{k})\right)}_{ = \nabla_{x_{k + 1}}\hat{\psi}_{x_{k}, L_{k}, \tau_{k}}(x_{k + 1}, B_{k})}\right\rangle = \hat{\psi}_{x_{k}, L_{k}, \tau_{k}}(x_{k + 1}, B_{k}) +\\
        &+ \frac{L_{k}}{2}\left\|y - x_{k + 1}\right\|^{2} + \frac{1}{2\tau_{k}}\left\|\hat{G}^{'}(x_{k}, B_{k})(y - x_{k + 1})\right\|^{2} +\\
        &+ \frac{1}{2\tau_{k}}\left\langle y - x_{k + 1},~2\left(\left(\tau_{k}L_{k}I_{n} + \hat{G}^{'}(x_{k}, B_{k})^{*}\hat{G}^{'}(x_{k}, B_{k})\right)(x_{k + 1} - x_{k}) + \hat{G}^{'}(x_{k}, B_{k})^{*}\hat{G}(x_{k}, B_{k})\right)\right\rangle =\\
        &=\left\{x_{k + 1} - x_{k} = - \eta_{k}\left(\hat{G}^{'}(x_{k}, B_{k})^{*}\hat{G}^{'}(x_{k}, B_{k}) + \tau_{k} L_{k} I_{n}\right)^{-1}\hat{G}^{'}(x_{k}, B_{k})^{*}\hat{G}(x_{k}, B_{k})\right\} = \hat{\psi}_{x_{k}, L_{k}, \tau_{k}}(x_{k + 1}, B_{k}) +\\
        &+ \frac{L_{k}}{2}\left\|y - x_{k + 1}\right\|^{2} + \frac{1}{2\tau_{k}}\left\|\hat{G}^{'}(x_{k}, B_{k})(y - x_{k + 1})\right\|^{2} + \frac{1}{2\tau_{k}}\left\langle y - x_{k + 1},~(1 - \eta_{k})\left(2\hat{G}^{'}(x_{k}, B_{k})^{*}\hat{G}(x_{k}, B_{k})\right)\right\rangle =\\
        &= \hat{\psi}_{x_{k}, L_{k}, \tau_{k}}(x_{k + 1}, B_{k}) + \frac{L_{k}}{2}\left\|y - x_{k + 1}\right\|^{2} + \frac{1}{2\tau_{k}}\left\|\hat{G}^{'}(x_{k}, B_{k})(y - x_{k + 1})\right\|^{2} + \frac{1 - \eta_{k}}{2\tau_{k}}\left\langle y - x_{k + 1},~\nabla_{x_{k}}\hat{g}_{2}(x_{k}, B_{k})\right\rangle.
    \end{aligned}
\end{equation*}
\end{proof}
\begin{lm:corollary}\label{lm:aux_approx_solution}
При $\eta_{k} = 1$ полученное представление $\hat{\psi}_{x_{k}, L_{k}, \tau_{k}}(y, B_{k})$ относительно точки глобального минимума $x_{k + 1}$ позволяет оценить, насколько $y\in E_{1}$ близок к глобальному минимуму (достигается на $x_{k + 1}$ при $\eta_{k} = 1$). А разность $\hat{\psi}_{x_{k}, L_{k}, \tau_{k}}(y, B_{k}) - \hat{\psi}_{x_{k}, L_{k}, \tau_{k}}(x_{k + 1}, B_{k})$ позволяет оценить точность приближённого поиска глобального минимума на $k$--ой итерации, если $\hat{x}_{k + 1}\in E_{1}$ обозначить за приближённое значение $x_{k + 1}$, полученное с погрешностью $\varepsilon_{k} > 0$ по значению локальной модели:
\begin{equation*}
    \begin{aligned}
        0&\leq\hat{\psi}_{x_{k}, L_{k}, \tau_{k}}(\hat{x}_{k + 1}, B_{k}) - \hat{\psi}_{x_{k}, L_{k}, \tau_{k}}(x_{k + 1}, B_{k}) = \frac{L_{k}}{2}\left\|\hat{x}_{k + 1} - x_{k + 1}\right\|^{2} + \frac{1}{2\tau_{k}}\left\|\hat{G}^{'}(x_{k}, B_{k})(\hat{x}_{k + 1} - x_{k + 1})\right\|^{2}\leq\varepsilon_{k},\\
        &x_{k + 1} = \hat{T}_{L_{k}, \tau_{k}}(x_{k}, B_{k}).
    \end{aligned}
\end{equation*}
\end{lm:corollary}
\begin{lm:corollary}
В детерминированном случае при выполнении условий леммы ($\hat{G}(x, \mathcal{B}) = \hat{F}(x),~x\in E_{1}$) представление $\psi_{x_{k}, L_{k}, \tau_{k}}(y),~y\in E_{1}$ выглядит аналогично:
\begin{equation*}
    \begin{aligned}
        \psi_{x_{k}, L_{k}, \tau_{k}}(y) &= \psi_{x_{k}, L_{k}, \tau_{k}}(x_{k + 1}) + \frac{L_{k}}{2}\left\|y - x_{k + 1}\right\|^{2} + \frac{1}{2\tau_{k}}\left\|\hat{F}^{'}(x_{k})(y - x_{k + 1})\right\|^{2} +\\
        &+ \frac{1 - \eta_{k}}{2\tau_{k}}\left\langle y - x_{k + 1},~\nabla\hat{f}_{2}(x_{k})\right\rangle.
    \end{aligned}
\end{equation*}
\end{lm:corollary}

В оценках из лемм \ref{lm:aux_bounded_deviation} и \ref{lm:aux_value_tolerance} явно фигурируют выражения вида норма разности значений $x_{k}$, взятых на соседних итерациях. Для понимания процесса оптимизации в классе методов Гаусса--Ньютона важно иметь представление о множестве значений этого выражения, и следующая лемма предоставляет необходимую информацию об отрезке значений данного выражения.

\begin{lemma}\label{lm:aux_bounded_variation}
    При выполнении предположений \ref{as:1} и \ref{as:2} для последовательности $\{x_{k}\}_{k\in\mathbb{Z}_{+}}$, вычисляемой по правилу \eqref{eq:stoch_direct_update_rule}, $\tau_{k} = \hat{g}_{1}(x_{k}, B_{k})$, $\eta_{k}\in(0, 1]$, $L_{k} > 0$, верна ограниченность вариации членов:
    \begin{equation*}
        \begin{aligned}
            \left\|x_{k + 1} - x_{k}\right\|&\in\left[\frac{\eta_{k}\left\|\nabla_{x_{k}}\hat{g}_{2}(x_{k}, B_{k})\right\|}{2\left(M_{\hat{G}}^{2} + \hat{g}_{1}(x_{k}, B_{k})L_{k}\right)},~\min\left\{\sqrt{\frac{2\hat{g}_{1}(x_{k}, B_{k})}{L_{k}}},~\frac{\eta_{k}M_{\hat{G}}}{L_{k}}\right\}\right],~k\in\mathbb{Z}_{+}.
        \end{aligned}
    \end{equation*}
    В случае использования правила \eqref{eq:double_stoch_direct_update_rule} отрезок значений $\|x_{k + 1} - x_{k}\|$ выглядит следующим образом:
    \begin{equation*}
        \begin{aligned}
            \left\|x_{k + 1} - x_{k}\right\|&\in\left[\frac{\eta_{k}\left\|\nabla_{x_{k}}\hat{g}_{2}(x_{k}, B_{k})\right\|}{2\left(M_{\hat{G}}^{2} + \tilde{\tau}_{k}L_{k}\right)},~\frac{\eta_{k}M_{\hat{G}}\hat{g}_{1}(x_{k}, B_{k})}{\tilde{\tau}_{k}L_{k}}\right],~k\in\mathbb{Z}_{+}.
        \end{aligned}
    \end{equation*}
\end{lemma}
\begin{proof}
Согласно правилу построения последовательности \eqref{eq:stoch_direct_update_rule}, положим $\tilde{B}_{k} = B_{k}$, $\tilde{\tau}_{k} = \tau_{k}$:
\begin{equation*}
    \begin{aligned}
        &\left\|x_{k + 1} - x_{k}\right\| = \left\|- \eta_{k}\left(\hat{G}^{'}(x_{k}, \tilde{B}_{k})^{*}\hat{G}^{'}(x_{k}, \tilde{B}_{k}) + \tilde{\tau}_{k} L_{k} I_{n}\right)^{-1}\hat{G}^{'}(x_{k}, B_{k})^{*}\hat{G}(x_{k}, B_{k})\right\| =\\
        &= \eta_{k}\left\|\underbrace{\left(\hat{G}^{'}(x_{k}, \tilde{B}_{k})^{*}\hat{G}^{'}(x_{k}, \tilde{B}_{k}) + \tilde{\tau}_{k} L_{k} I_{n}\right)^{-1}}_{\text{симметричная матрица}}\hat{G}^{'}(x_{k}, B_{k})^{*}\hat{G}(x_{k}, B_{k})\right\| =\\
    \end{aligned}
\end{equation*}
\begin{equation}\label{eq:aux_bounded_variation_lower_bound}
    \begin{aligned}
        &= \eta_{k}\left(\left\langle\left(\hat{G}^{'}(x_{k}, \tilde{B}_{k})^{*}\hat{G}^{'}(x_{k}, \tilde{B}_{k}) + \tilde{\tau}_{k} L_{k} I_{n}\right)^{-2}\hat{G}^{'}(x_{k}, B_{k})^{*}\hat{G}(x_{k}, B_{k}),~\hat{G}^{'}(x_{k}, B_{k})^{*}\hat{G}(x_{k}, B_{k})\right\rangle\right)^{\frac{1}{2}} =\\
        &= \left\{\nabla_{x_{k}}\hat{g}_{2}(x_{k}, B_{k}) = 2\hat{G}^{'}(x_{k}, B_{k})^{*}\hat{G}(x_{k}, B_{k})\right\} =\\
        &= \frac{\eta_{k}}{2}\sqrt{\left\langle\left(\hat{G}^{'}(x_{k}, \tilde{B}_{k})^{*}\hat{G}^{'}(x_{k}, \tilde{B}_{k}) + \tilde{\tau}_{k} L_{k} I_{n}\right)^{-2}\nabla_{x_{k}}\hat{g}_{2}(x_{k}, B_{k}),~\nabla_{x_{k}}\hat{g}_{2}(x_{k}, B_{k})\right\rangle}\geq\left\{\text{предположение \ref{as:2}, \eqref{eq:as2_matrix_order}}\right\}\geq\\
        &\geq\frac{\eta_{k}\left\|\nabla_{x_{k}}\hat{g}_{2}(x_{k}, B_{k})\right\|}{2\left(M_{\hat{G}}^{2} + \tilde{\tau}_{k}L_{k}\right)}\geq\left\{\tilde{\tau}_{k} = \hat{g}_{1}(x_{k}, B_{k})\right\}\geq\frac{\eta_{k}\left\|\nabla_{x_{k}}\hat{g}_{2}(x_{k}, B_{k})\right\|}{2\left(M_{\hat{G}}^{2} + \hat{g}_{1}(x_{k}, B_{k})L_{k}\right)}.
    \end{aligned}
\end{equation}
Полученные формулы выше применимы и для правила \eqref{eq:double_stoch_direct_update_rule}, что позволяет установить нижнюю границу значений $\|x_{k + 1} - x_{k}\|$ при построении $\{x_{k}\}_{k\in\mathbb{Z}_{+}}$ с помощью правила \eqref{eq:double_stoch_direct_update_rule}. Теперь выведем оценку сверху:
\begin{equation}\label{eq:aux_bounded_variation_upper_bound_1}
    \begin{aligned}
        &\hat{\psi}_{x_{k}, L_{k}, \hat{g}_{1}(x_{k}, B_{k})}(x_{k}, B_{k}) = \hat{g}_{1}(x_{k}, B_{k}) = \left\{\text{лемма \ref{lm:aux_value_tolerance}}\right\} = \hat{\psi}_{x_{k}, L_{k}, \hat{g}_{1}(x_{k}, B_{k})}(x_{k + 1}, B_{k}) + \frac{L_{k}}{2}\left\|x_{k} - x_{k + 1}\right\|^{2} +\\
        &+ \frac{1}{2\hat{g}_{1}(x_{k}, B_{k})}\left\|\hat{G}^{'}(x_{k}, B_{k})(x_{k} - x_{k + 1})\right\|^{2} +
        \frac{1 - \eta_{k}}{2\hat{g}_{1}(x_{k}, B_{k})}\left\langle x_{k} - x_{k + 1},~\nabla_{x_{k}}\hat{g}_{2}(x_{k}, B_{k})\right\rangle =\\
        &= \hat{\psi}_{x_{k}, L_{k}, \hat{g}_{1}(x_{k}, B_{k})}(x_{k + 1}, B_{k}) + \frac{L_{k}}{2}\left\|x_{k} - x_{k + 1}\right\|^{2} + \frac{1}{2\hat{g}_{1}(x_{k}, B_{k})}\left\|\hat{G}^{'}(x_{k}, B_{k})(x_{k} - x_{k + 1})\right\|^{2} +\\
        &+ \frac{(1 - \eta_{k})\eta_{k}}{4\hat{g}_{1}(x_{k}, B_{k})}\left\langle \left(\hat{G}^{'}(x_{k}, B_{k})^{*}\hat{G}^{'}(x_{k}, B_{k}) + \tau_{k} L_{k} I_{n}\right)^{-1}\nabla_{x_{k}}\hat{g}_{2}(x_{k}, B_{k}),~\nabla_{x_{k}}\hat{g}_{2}(x_{k}, B_{k})\right\rangle\geq 0.
    \end{aligned}
\end{equation}
Из выражения выше следует неравенство:
\begin{equation*}
    \begin{aligned}
        &\hat{g}_{1}(x_{k}, B_{k})\geq\hat{g}_{1}(x_{k}, B_{k}) - \hat{\psi}_{x_{k}, L_{k}, \hat{g}_{1}(x_{k}, B_{k})}(x_{k + 1}, B_{k})\geq\frac{L_{k}}{2}\|x_{k + 1} - x_{k}\|^{2}\Rightarrow\|x_{k + 1} - x_{k}\|\leq\sqrt{\frac{2\hat{g}_{1}(x_{k}, B_{k})}{L_{k}}}.
    \end{aligned}
\end{equation*}
Также существует другая оценка сверху, положим $\tilde{B}_{k} = B_{k}$, $\tilde{\tau}_{k} = \tau_{k}$:
\begin{equation}\label{eq:aux_bounded_variation_upper_bound_2}
    \begin{aligned}
        &\left\|x_{k + 1} - x_{k}\right\| = \left\|- \eta_{k}\left(\hat{G}^{'}(x_{k}, \tilde{B}_{k})^{*}\hat{G}^{'}(x_{k}, \tilde{B}_{k}) + \tilde{\tau}_{k} L_{k} I_{n}\right)^{-1}\hat{G}^{'}(x_{k}, B_{k})^{*}\hat{G}(x_{k}, B_{k})\right\| =\\
        &=\eta_{k}\left(\left\langle\underbrace{\left(\hat{G}^{'}(x_{k}, \tilde{B}_{k})^{*}\hat{G}^{'}(x_{k}, \tilde{B}_{k}) + \tilde{\tau}_{k} L_{k} I_{n}\right)^{-2}}_{\text{здесь $\tilde{B}_{k}$ можно независимо сэмплировать от $B_{k}$}}\hat{G}^{'}(x_{k}, B_{k})^{*}\hat{G}(x_{k}, B_{k}),~\hat{G}^{'}(x_{k}, B_{k})^{*}\hat{G}(x_{k}, B_{k})\right\rangle\right)^{\frac{1}{2}}\leq\\
        &\leq\frac{\eta_{k}}{\tilde{\tau}_{k}L_{k}}\left\|\hat{G}^{'}(x_{k}, B_{k})^{*}\hat{G}(x_{k}, B_{k})\right\|\leq\frac{\eta_{k}}{\tilde{\tau}_{k}L_{k}}\left\|\hat{G}^{'}(x_{k}, B_{k})^{*}\right\|\left\|\hat{G}(x_{k}, B_{k})\right\|\leq\frac{\eta_{k}M_{\hat{G}}\hat{g}_{1}(x_{k}, B_{k})}{\tilde{\tau}_{k}L_{k}} = \frac{\eta_{k}M_{\hat{G}}}{L_{k}}.
    \end{aligned}
\end{equation}
Выражения в \eqref{eq:aux_bounded_variation_upper_bound_2} применимы и для правила \eqref{eq:double_stoch_direct_update_rule}, благодаря чему устанавливается верхняя граница отрезка $\|x_{k + 1} - x_{k}\|$ при использовании правила \eqref{eq:double_stoch_direct_update_rule}.
\end{proof}
\begin{lm:corollary}\label{lm:aux_double_stoch_variation_upper_bound}
В случае $\tau_{k}\in\left[\tilde{\tau},~\tilde{\mathcal{T}}\right]$, $\tilde{\tau}\in\left(0,~\tilde{\mathcal{T}}\right]$, $L_{k}\in\left[L,~\tilde{\gamma}L_{\hat{F}}\right]$, $L\in\left(0,~\tilde{\gamma}L_{\hat{F}}\right]$, $\tilde{\gamma}\geq1$ и при выполнении предположения \ref{as:3} отрезок значений $\left\|x_{k + 1} - x_{k}\right\|$ принимает следующие значения:
\begin{equation*}
    \begin{aligned}
        &\|x_{k + 1} - x_{k}\|\in\left[\frac{\eta_{k}\left\|\nabla_{x_{k}}\hat{g}_{2}(x_{k}, B_{k})\right\|}{2\left(M_{\hat{G}}^{2} + \tilde{\gamma}\tilde{\mathcal{T}}L_{\hat{F}}\right)},~\min\left\{\sqrt{\frac{1}{L}\left(\tilde{\mathcal{T}} + \frac{P_{\hat{g}_{1}}^{2}}{\tilde{\tau}}\right)},~\frac{\eta_{k}M_{\hat{G}}P_{\hat{g}_{1}}}{\tau_{k}L_{k}}\right\}\right].
    \end{aligned}
\end{equation*}
Нижняя граница получается из \eqref{eq:aux_bounded_variation_lower_bound} в силу монотонного убывания по $\tau_{k}L_{k}$. Верхняя граница $\frac{\eta_{k}M_{\hat{G}}P_{\hat{g}_{1}}}{\tau_{k} L_{k}}$ выводится как верхняя оценка выражения \eqref{eq:aux_bounded_variation_upper_bound_2} по предположениям \ref{as:2} и \ref{as:3}. Верхняя граница
$$\sqrt{\frac{1}{L}\left(\tilde{\mathcal{T}} + \frac{P_{\hat{g}_{1}}^{2}}{\tilde{\tau}}\right)}$$
выводится из представления \eqref{eq:aux_bounded_variation_upper_bound_1} для локальной модели $\hat{\psi}_{x_{k}, L_{k}, \tau_{k}}(\cdot, B_{k})$ в рамках предположения \ref{as:3}:
\begin{equation*}
    \begin{aligned}
        \frac{\tilde{\mathcal{T}}}{2} + \frac{P_{\hat{g}_{1}}^{2}}{2\tilde{\tau}}\geq\hat{\psi}_{x_{k}, L_{k}, \tau_{k}}(x_{k}, B_{k}) &= \frac{\tau_{k}}{2} + \frac{\hat{g}_{2}(x_{k}, B_{k})}{2\tau_{k}}\geq\frac{L_{k}}{2}\left\|x_{k + 1} - x_{k}\right\|^{2}\geq\frac{L}{2}\left\|x_{k + 1} - x_{k}\right\|^{2}.
    \end{aligned}
\end{equation*}
\end{lm:corollary}
\begin{lm:corollary}\label{lm:aux_variation_upper_bound}
При дополнительном выполнении предположения \ref{as:3} и $L_{k}\in[L,~\gamma L_{\hat{F}}]$,~$L\in(0,~\gamma L_{\hat{F}}]$,~$\gamma\geq1$ ограничение на возможные значения $\left\|x_{k + 1} - x_{k}\right\|$ преобразуется следующим образом:
\begin{equation*}
    \left\|x_{k + 1} - x_{k}\right\|\in\left[\frac{\eta_{k}\left\|\nabla_{x_{k}}\hat{g}_{2}(x_{k}, B_{k})\right\|}{2\left(M_{\hat{G}}^{2} + \gamma P_{\hat{g}_{1}}L_{\hat{F}}\right)},~\min\left\{\sqrt{\frac{2P_{\hat{g}_{1}}}{L}},~\frac{\eta_{k}M_{\hat{G}}}{L}\right\}\right],~k\in\mathbb{Z}_{+}.
\end{equation*}
\end{lm:corollary}
\begin{lm:corollary}
В детерминированном случае при $\hat{G}(x_{k}, B_{k}) = \hat{F}(x_{k})$, $B_{k} = \mathcal{B}$ и $\tau_{k} = \hat{f}_{1}(x_{k})$ оценка на вариацию построенной с помощью правила из следствия \ref{lm:aux_det_update_rule} последовательности $\{x_{k}\}_{k\in\mathbb{Z}_{+}}$ следующая:
\begin{equation*}
    \left\|x_{k + 1} - x_{k}\right\|\in\left[\frac{\eta_{k}\left\|\nabla\hat{f}_{2}(x_{k})\right\|}{2\left(M_{\hat{F}}^{2} + \hat{f}_{1}(x_{k})L_{k}\right)},~\min\left\{\sqrt{\frac{2\hat{f}_{1}(x_{k})}{L_{k}}},~\frac{\eta_{k}M_{\hat{G}}}{L_{k}}\right\}\right],~k\in\mathbb{Z}_{+}.
\end{equation*}
\end{lm:corollary}
\begin{lm:corollary}\label{lm:aux_bounded_variation_term_cond}
При постоянном $\eta_{k} = \eta = \operatorname{const}$, $0 < L_{k}\leq \gamma L_{\hat{F}}$, $\gamma\geq1$, $k\in\mathbb{Z}_{+}$ и в рамках предположений \ref{as:1}, \ref{as:2} и \ref{as:3} нижняя граница $\left\|x_{k + 1} - x_{k}\right\|$ пропорциональна норме градиента и может быть использована в качестве критерия останова для достижения нормой градиента уровня $\epsilon > 0$:
\begin{equation*}
    \begin{aligned}
        \mathbb{E}\left[\left\|\nabla\hat{f}_{2}(x_{k})\right\|\right]&\leq\sqrt{\mathbb{E}\left[\left\|\nabla\hat{f}_{2}(x_{k})\right\|^{2}\right]}\leq\sqrt{\mathbb{E}\left[\mathbb{E}\left[\left\|\nabla_{x_{k}}\hat{g}_{2}(x_{k}, B_{k})\right\|^{2}\right]\right]}= \sqrt{\mathbb{E}\left[\left\|\nabla_{x_{k}}\hat{g}_{2}(x_{k}, B_{k})\right\|^{2}\right]} \leq\\
        &\leq\frac{2\left(M_{\hat{G}}^{2} + \gamma P_{\hat{g}_{1}}L_{\hat{F}}\right)}{\eta}\sqrt{\mathbb{E}\left[\left\|x_{k + 1} - x_{k}\right\|^{2}\right]}\leq\epsilon,
    \end{aligned}
\end{equation*}
то есть
$$\mathbb{E}\left[\left\|x_{k + 1} - x_{k}\right\|\right]\leq\sqrt{\mathbb{E}\left[\left\|x_{k + 1} - x_{k}\right\|^{2}\right]}\leq\frac{\epsilon\eta}{2\left(M_{\hat{G}}^{2} + \gamma P_{\hat{g}_{1}}L_{\hat{F}}\right)},$$
где усреднение ведётся по всей стохастике метода (выбор батчей на каждой итерации и начального приближения).
\end{lm:corollary}

В следующей лемме устанавливается липшицевость градиентов ограниченных функций $\hat{f}_{2}$ и $\hat{g}_{2}$.

\begin{lemma}\label{lm:aux_g2_lipschitz_gradient}
Пусть выполнены предположения \ref{as:2}, \ref{as:3}, \ref{as:jacob_smoothness}. Тогда функция $\hat{g}_{2}$ обладает липшицевым градиентом с верхней оценкой постоянной Липшица $l_{\hat{g}_{2}} = 2\left(M_{\hat{G}}^{2} + L_{\hat{F}}P_{\hat{g}_{1}}\right)$.
\end{lemma}
\begin{proof}
Вычислим $l_{\hat{g}_{2}}$ --- верхнюю оценку на лучшую (наименьшую) постоянную Липшица для произвольных $(x, y)\in E_{1}^{2}$ и $B\subseteq\mathcal{B}$:
\begin{equation}\label{eq:upper_lipschitz_property_bound}
    \begin{aligned}
        &\left\|\nabla_{y}\hat{g}_{2}(y, B) - \nabla_{x}\hat{g}_{2}(x, B)\right\| = \left\|2\hat{G}^{'}(y, B)^{*}\hat{G}(y, B) - 2\hat{G}^{'}(x, B)^{*}\hat{G}(x, B)\right\|=\\
        &= 2\left\|\left(\hat{G}^{'}(y, B)^{*}\hat{G}(y, B) - \hat{G}^{'}(x, B)^{*}\hat{G}(y, B)\right) + \left(\hat{G}^{'}(x, B)^{*}\hat{G}(y, B) - \hat{G}^{'}(x, B)^{*}\hat{G}(x, B)\right)\right\|\leq\\
        &\leq2\left(\left\|\hat{G}^{'}(y, B)^{*}\hat{G}(y, B) - \hat{G}^{'}(x, B)^{*}\hat{G}(y, B)\right\| + \left\|\hat{G}^{'}(x, B)^{*}\hat{G}(y, B) - \hat{G}^{'}(x, B)^{*}\hat{G}(x, B)\right\|\right)\leq\\
        &\leq2\left(\left\|\hat{G}^{'}(y, B)^{*} - \hat{G}^{'}(x, B)^{*}\right\|\left\|\hat{G}(y, B)\right\| + \left\|\hat{G}^{'}(x, B)^{*}\right\|\left\|\hat{G}(y, B) - \hat{G}(x, B)\right\|\right)\leq\\
        &\leq2\left(L_{\hat{F}}\|y - x\|P_{\hat{g}_{1}} + M_{\hat{G}}^{2}\|y - x\|\right)\leq\left(2\left(L_{\hat{F}}P_{\hat{g}_{1}} + M_{\hat{G}}^{2}\right)\right)\|y - x\|\Rightarrow l_{\hat{g}_{2}} = 2\left(L_{\hat{F}}P_{\hat{g}_{1}} + M_{\hat{G}}^{2}\right).
    \end{aligned}
\end{equation}
В доказательстве использована липшицевость отображения $\hat{G}$:
\begin{equation*}
    \begin{aligned}
        \left\|\hat{G}(y, B) - \hat{G}(x, B)\right\| &= \left\|\int\limits_{0}^{1}\hat{G}^{'}(x + t(y - x), B)(y - x)\operatorname{d}t\right\|\leq\int\limits_{0}^{1}\left\|\hat{G}^{'}(x + t(y - x), B)\right\|\left\|y - x\right\|\operatorname{d}t\leq\\
        &\leq M_{\hat{G}}\left\|y - x\right\|.
    \end{aligned}
\end{equation*}
\end{proof}
\begin{lm:corollary}\label{lm:aux_f2_lipschitz_gradient}
В случае $B = \mathcal{B}$ функция $\hat{f}_{2}$ обладает липшицевым градиентом и значение оценки постоянной Липшица равно $l_{\hat{f}_{2}} \overset{\operatorname{def}}{=} 2\left(M_{\hat{F}}^{2} + L_{\hat{F}}P_{\hat{f}_{1}}\right)$.
\end{lm:corollary}

В лемме ниже выводится локальная модель квадрата оптимизируемого функционала, используемая в схеме с сэмплированием двух батчей на каждом шаге метода.

\begin{lemma}\label{lm:aux_g2_local_model}
Пусть выполнены предположения \ref{as:2}, \ref{as:3}, \ref{as:jacob_smoothness}. Тогда существует следующая стохастическая локальная модель для функции $\hat{g}_{2}$:
\begin{equation*}
    \begin{aligned}
        &\hat{g}_{2}(y, B)\leq\hat{\varphi}_{x, l}(y, B) = \hat{g}_{2}(x, B) + \left\langle\nabla_{x}\hat{g}_{2}(x, B),~y - x\right\rangle + \frac{l}{2}\left\|y - x\right\|^{2},~\forall l\geq l_{\hat{g}_{2}},~\forall (x, y)\in E_{1}^{2},~\forall B\subseteq\mathcal{B}.
    \end{aligned}
\end{equation*}
\end{lemma}
\begin{proof}
Рассмотрим верхнюю оценку на $\hat{g}_{2}(x, B)$ для произвольных $(x, y)\in E_{1}^{2}$ и $B\subseteq\mathcal{B}$:
\begin{equation*}
    \begin{aligned}
        \hat{g}_{2}(y, B) &= \hat{g}_{2}(y, B) - \hat{g}_{2}(x, B) - \left\langle\nabla_{x}\hat{g}_{2}(x, B),~y - x\right\rangle + \hat{g}_{2}(x, B) + \left\langle\nabla_{x}\hat{g}_{2}(x, B),~y - x\right\rangle\leq\\
        &\leq\hat{g}_{2}(x, B) + \left\langle\nabla_{x}\hat{g}_{2}(x, B),~y - x\right\rangle + \left|\hat{g}_{2}(y, B) - \hat{g}_{2}(x, B) - \left\langle\nabla_{x}\hat{g}_{2}(x, B),~y - x\right\rangle\right|=\\
        &=\hat{g}_{2}(x, B) + \left\langle\nabla_{x}\hat{g}_{2}(x, B),~y - x\right\rangle +\\
        &+ \left|\int\limits_{0}^{1}\left\langle\nabla_{x + t(y - x)}\hat{g}_{2}(x + t(y - x), B),~y - x\right\rangle\operatorname{d}t - \left\langle\nabla_{x}\hat{g}_{2}(x, B),~y - x\right\rangle\right| =\\
        &= \hat{g}_{2}(x, B) + \left\langle\nabla_{x}\hat{g}_{2}(x, B),~y - x\right\rangle + \left|\int\limits_{0}^{1}\left\langle\nabla_{x + t(y - x)}\hat{g}_{2}(x + t(y - x), B) - \nabla_{x}\hat{g}_{2}(x, B),~y - x\right\rangle\operatorname{d}t\right|\leq\\
        &\leq\hat{g}_{2}(x, B) + \left\langle\nabla_{x}\hat{g}_{2}(x, B),~y - x\right\rangle + \int\limits_{0}^{1}\left|\left\langle\nabla_{x + t(y - x)}\hat{g}_{2}(x + t(y - x), B) - \nabla_{x}\hat{g}_{2}(x, B),~y - x\right\rangle\right|\operatorname{d}t\leq\\
        &\leq\hat{g}_{2}(x, B) + \left\langle\nabla_{x}\hat{g}_{2}(x, B),~y - x\right\rangle + \int\limits_{0}^{1}\left\|\nabla_{x + t(y - x)}\hat{g}_{2}(x + t(y - x), B) - \nabla_{x}\hat{g}_{2}(x, B)\right\|\left\|y - x\right\|\operatorname{d}t\leq\\
        &\leq\hat{g}_{2}(x, B) + \left\langle\nabla_{x}\hat{g}_{2}(x, B),~y - x\right\rangle + \int\limits_{0}^{1}tl_{\hat{g}_{2}}\left\|y - x\right\|^{2}\operatorname{d}t =\\
        &= \underbrace{\hat{g}_{2}(x, B) + \left\langle\nabla_{x}\hat{g}_{2}(x, B),~y - x\right\rangle + \frac{l_{\hat{g}_{2}}}{2}\left\|y - x\right\|^{2}}_{=\hat{\varphi}_{x, l_{\hat{g}_{2}}}(y, B)}\leq\\
        &\leq\hat{g}_{2}(x, B) + \left\langle\nabla_{x}\hat{g}_{2}(x, B),~y - x\right\rangle + \frac{l}{2}\left\|y - x\right\|^{2},~l\geq l_{\hat{g}_{2}}.
    \end{aligned}
\end{equation*}
\end{proof}
\begin{lm:corollary}\label{lm:aux_f2_local_model}
В случае $B = \mathcal{B}$ локальная модель $\hat{g}_{2}$ становится локальной моделью функции $\hat{f}_{2}$:
\begin{equation*}
    \begin{aligned}
        &\hat{f}_{2}(y)\leq\varphi_{x, l}(y) \overset{\operatorname{def}}{=} \hat{f}_{2}(x) + \left\langle\nabla\hat{f}_{2}(x),~y - x\right\rangle + \frac{l}{2}\left\|y - x\right\|^{2},~\forall l\geq l_{\hat{f}_{2}},~\forall (x, y)\in E_{1}^{2}.
    \end{aligned}
\end{equation*}
\end{lm:corollary}

Следующая лемма описывает количественно убывание стохастической локальной модели при использовании правила \eqref{eq:double_stoch_direct_update_rule} вычисления приближения решения $x_{k + 1}$ задачи \eqref{eq:main_opt_problem}.

\begin{lemma}\label{lm:aux_double_stoch_update}
    Пусть выполнены предположения \ref{as:2}, \ref{as:3} и \ref{as:jacob_smoothness}. Рассмотрим правило вычисления $x_{k + 1}$ \eqref{eq:double_stoch_direct_update_rule} в рамках схемы \ref{alg:gen_double_stoch_gnm}. Тогда на каждом шаге метода Гаусса--Ньютона существует фактор шага $\eta_{k}\geq0$, для которого верно следующее соотношение:
    \begin{equation*}
        \begin{aligned}
            \hat{g}_{2}(x_{k + 1}, B_{k})&\leq\hat{g}_{2}(x_{k}, B_{k}) - \frac{\left(\left\langle\nabla_{x}\hat{g}_{2}(x_{k}, B_{k}),~\left(\hat{G}^{'}(x_{k}, \tilde{B}_{k})^{*}\hat{G}^{'}(x_{k}, \tilde{B}_{k}) + \tilde{\tau}_{k}L_{k}I_{n}\right)^{-1}\nabla_{x}\hat{g}_{2}(x_{k}, B_{k})\right\rangle\right)^{2}}{2l_{k}\left\langle\nabla_{x}\hat{g}_{2}(x_{k}, B_{k}),~\left(\hat{G}^{'}(x_{k}, \tilde{B}_{k})^{*}\hat{G}^{'}(x_{k}, \tilde{B}_{k}) + \tilde{\tau}_{k}L_{k}I_{n}\right)^{-2}\nabla_{x}\hat{g}_{2}(x_{k}, B_{k})\right\rangle},\\
            &k\in\mathbb{Z}_{+}.
        \end{aligned}
    \end{equation*}
\end{lemma}
\begin{proof}
Выпишем стохастическую локальную модель функции $\hat{g}_{2}$ в точке $x_{k + 1}$ относительно $x_{k}$ (лемма \ref{lm:aux_g2_local_model}) и прооптимизируем её по $\eta_{k}\geq0$:
\begin{equation*}
    \begin{aligned}
        \hat{g}_{2}(x_{k + 1}, B_{k})&\leq\hat{\varphi}_{x_{k}, l_{k}}(x_{k + 1}, B_{k}) = \hat{g}_{2}(x_{k}, B_{k}) + \left\langle\nabla_{x_{k}}\hat{g}_{2}(x_{k}, B_{k}),~x_{k + 1} - x_{k}\right\rangle + \frac{l_{k}}{2}\left\|x_{k + 1} - x_{k}\right\|^{2} =\\
        &= \left\{\text{правило \eqref{eq:double_stoch_direct_update_rule}}\right\}= \hat{g}_{2}(x_{k}, B_{k}) -\\
        &- 2\eta_{k}\left\langle\hat{G}^{'}(x_{k}, B_{k})^{*}\hat{G}(x_{k}, B_{k}),~\left(\hat{G}^{'}(x_{k}, \tilde{B}_{k})^{*}\hat{G}^{'}(x_{k}, \tilde{B}_{k}) + \tilde{\tau}_{k}L_{k}I_{n}\right)^{-1}\hat{G}^{'}(x_{k}, B_{k})^{*}\hat{G}(x_{k}, B_{k})\right\rangle +\\
        &+ \frac{\eta_{k}^{2}l_{k}}{2}\left\|\left(\hat{G}^{'}(x_{k}, \tilde{B}_{k})^{*}\hat{G}^{'}(x_{k}, \tilde{B}_{k}) + \tilde{\tau}_{k}L_{k}I_{n}\right)^{-1}\hat{G}^{'}(x_{k}, B_{k})^{*}\hat{G}(x_{k}, B_{k})\right\|^{2}\rightarrow\min\limits_{\eta_{k}\geq0}.
    \end{aligned}
\end{equation*}
Правая часть неравенства выше является параболой по $\eta_{k}$ с ветвями, направленными вверх, у которой минимум достигается в точке вершины:
\begin{equation*}
    \begin{aligned}
        &\eta_{k} = \frac{2\left\langle\hat{G}^{'}(x_{k}, B_{k})^{*}\hat{G}(x_{k}, B_{k}),~\left(\hat{G}^{'}(x_{k}, \tilde{B}_{k})^{*}\hat{G}^{'}(x_{k}, \tilde{B}_{k}) + \tilde{\tau}_{k}L_{k}I_{n}\right)^{-1}\hat{G}^{'}(x_{k}, B_{k})^{*}\hat{G}(x_{k}, B_{k})\right\rangle}{l_{k}\left\langle\hat{G}^{'}(x_{k}, B_{k})^{*}\hat{G}(x_{k}, B_{k}),~\left(\hat{G}^{'}(x_{k}, \tilde{B}_{k})^{*}\hat{G}^{'}(x_{k}, \tilde{B}_{k}) + \tilde{\tau}_{k}L_{k}I_{n}\right)^{-2}\hat{G}^{'}(x_{k}, B_{k})^{*}\hat{G}(x_{k}, B_{k})\right\rangle}\geq0,~k\in\mathbb{Z}_{+}.
    \end{aligned}
\end{equation*}
Подставляя это значение в стохастическую локальную модель, получаем искомую оценку:
\begin{equation*}
    \begin{aligned}
        \hat{g}_{2}(x_{k + 1}, B_{k})&\leq\hat{g}_{2}(x_{k}, B_{k}) -\\
        &- \underbrace{\frac{2\left(\left\langle\hat{G}^{'}(x_{k}, B_{k})^{*}\hat{G}(x_{k}, B_{k}),~\left(\hat{G}^{'}(x_{k}, \tilde{B}_{k})^{*}\hat{G}^{'}(x_{k}, \tilde{B}_{k}) + \tilde{\tau}_{k}L_{k}I_{n}\right)^{-1}\hat{G}^{'}(x_{k}, B_{k})^{*}\hat{G}(x_{k}, B_{k})\right\rangle\right)^{2}}{l_{k}\left\langle\hat{G}^{'}(x_{k}, B_{k})^{*}\hat{G}(x_{k}, B_{k}),~\left(\hat{G}^{'}(x_{k}, \tilde{B}_{k})^{*}\hat{G}^{'}(x_{k}, \tilde{B}_{k}) + \tilde{\tau}_{k}L_{k}I_{n}\right)^{-2}\hat{G}^{'}(x_{k}, B_{k})^{*}\hat{G}(x_{k}, B_{k})\right\rangle}}_{\geq 0}\Rightarrow
    \end{aligned}
\end{equation*}
\begin{equation*}
    \begin{aligned}
        \Rightarrow\hat{g}_{2}(x_{k + 1}, B_{k})&\leq\hat{g}_{2}(x_{k}, B_{k}) -\\
        &- \frac{\left(\left\langle\nabla_{x}\hat{g}_{2}(x_{k}, B_{k}),~\left(\hat{G}^{'}(x_{k}, \tilde{B}_{k})^{*}\hat{G}^{'}(x_{k}, \tilde{B}_{k}) + \tilde{\tau}_{k}L_{k}I_{n}\right)^{-1}\nabla_{x}\hat{g}_{2}(x_{k}, B_{k})\right\rangle\right)^{2}}{2l_{k}\left\langle\nabla_{x}\hat{g}_{2}(x_{k}, B_{k}),~\left(\hat{G}^{'}(x_{k}, \tilde{B}_{k})^{*}\hat{G}^{'}(x_{k}, \tilde{B}_{k}) + \tilde{\tau}_{k}L_{k}I_{n}\right)^{-2}\nabla_{x}\hat{g}_{2}(x_{k}, B_{k})\right\rangle},~k\in\mathbb{Z}_{+}.
    \end{aligned}
\end{equation*}
\end{proof}

В следующей лемме используются условие слабого роста и условие Поляка--Лоясиевича для вывода границ значений квадрата нормы градиента.

\begin{lemma}\label{lm:aux_grad_g2_PL_bounds}
    Пусть выполнены предположения \ref{as:2} и \ref{as:5}. Тогда квадрат нормы градиента функции $\hat{g}_{2}$ ограничен с двух сторон значением функции $\hat{g}_{2}$:
    \begin{equation*}
        \begin{aligned}
            &4\mu\hat{g}_{2}(x, B)\leq\left\|\nabla_{x}\hat{g}_{2}(x, B)\right\|^{2}\leq4M_{\hat{G}}^{2}\hat{g}_{2}(x, B),~\forall x\in E_{1},~\forall B\subseteq\mathcal{B}.
        \end{aligned}
    \end{equation*}
\end{lemma}
\begin{proof}
Условия \ref{as:2} и \ref{as:5} задают следующие неравенства на норму градиента:
\begin{equation}\label{eq:grad_norm_bounds}
    \begin{aligned}
        &4\mu\hat{g}_{2}(x, B)\leq\left\{\text{предположение \ref{as:5}}\right\}\leq4\left\|\hat{G}^{'}(x, B)^{*}\hat{G}(x, B)\right\|^{2}= \left\|\nabla_{x}\hat{g}_{2}(x, B)\right\|^{2}\leq\\
        &\leq4\left\|\hat{G}^{'}(x, B)\right\|^{2}\left\|\hat{G}(x, B)\right\|^{2}\leq\left\{\text{предположение \ref{as:2}}\right\}\leq\\
        &\leq4M_{\hat{G}}^{2}\hat{g}_{2}(x, B),~\forall x\in E_{1}, \forall B\subseteq\mathcal{B}\Rightarrow\mu\leq M_{\hat{G}}^{2},~4\mu\hat{g}_{2}(x, B)\leq\left\|\nabla_{x}\hat{g}_{2}(x, B)\right\|^{2}.
    \end{aligned}
\end{equation}
\end{proof}
\begin{lm:corollary}\label{lm:aux_mean_grad_g2_PL_bounds}
При усреднении по батчам $B$ квадрата нормы градиента $\hat{g}_{2}$ выполнено следующее неравенство:
\begin{equation*}
    \begin{aligned}
        &4\mu\hat{f}_{2}(x)\leq\mathbb{E}_{B}\left[\left\|\nabla_{x}\hat{g}_{2}(x, B)\right\|^{2}\right]\leq4M_{\hat{G}}^{2}\hat{f}_{2}(x),~\forall x\in E_{1}.
    \end{aligned}
\end{equation*}
\end{lm:corollary}

\subsubsection*{Основные утверждения}

Теорема \ref{th:3_main} устанавливает условия сходимости к окрестности стационарной точки относительно среднего минимального квадрата нормы градиента функции $\hat{f}_{2}$.

\begin{re:theorem}\label{th:3}
    Пусть выполнены предположения \ref{as:1}, \ref{as:2}, \ref{as:3}, \ref{as:4}. Рассмотрим метод Гаусса--Ньютона со схемой реализации \ref{alg:gen_stoch_gnm}, в которой последовательность $\{x_{k}\}_{k\in\mathbb{Z}_{+}}$ вычисляется по правилу \eqref{eq:stoch_direct_update_rule} с $\tau_{k} = \hat{g}_{1}(x_{k}, B_{k})$,\\$\eta_{k}\in[\eta, 1]$, $\eta\in(0, 1]$. Тогда:
    \begin{equation}\label{eq:stoch_sub_lin_conv_1}
        \begin{aligned}
            \mathbb{E}\left[\min\limits_{i\in\overline{0, k - 1}}\left\|\nabla\hat{f}_{2}(x_{i})\right\|^{2}\right]&\leq\frac{8\left(M_{\hat{G}}^{2} + \gamma P_{\hat{g}_{1}}L_{\hat{F}}\right)}{\eta(2 - \eta)}\left(\frac{\mathbb{E}\left[\hat{f}_{2}(x_{0})\right]}{k} + 2l_{\hat{F}}\min\left\{\sqrt{\frac{2P_{\hat{g}_{1}}}{L}},~\frac{M_{\hat{G}}}{L}\right\}\mathds{1}_{\left\{b < m\right\}} + \tilde{\sigma}\sqrt{\frac{1}{b} - \frac{1}{m}}\right),\\
            k&\in\mathbb{N}.
        \end{aligned}
    \end{equation}
    Оператор математического ожидания $\mathbb{E}\left[\cdot\right]$ усредняет по всей случайности процесса оптимизации.
\end{re:theorem}
\begin{proof}
Согласно правилу вычисления $x_{k + 1}$ \eqref{eq:stoch_direct_update_rule}, \eqref{eq:stoch_general_update_rule}:
\begin{equation}\label{eq:th3_iter_diff}
    \begin{aligned}
        &\hat{g}_{1}(x_{k}, B_{k}) - \hat{g}_{1}(x_{k + 1}, B_{k})\geq\hat{g}_{1}(x_{k}, B_{k}) - \frac{\tau_{k}}{2} - \frac{\hat{g}_{2}(x_{k}, B_{k})}{2\tau_{k}} +\\
        &+ \frac{\eta_{k}(2 - \eta_{k})}{2\tau_{k}}\left\langle\left(\hat{G}^{'}(x_{k}, B_{k})^{*}\hat{G}^{'}(x_{k}, B_{k}) + \tau_{k} L_{k} I_{n}\right)^{-1}\hat{G}^{'}(x_{k}, B_{k})^{*}\hat{G}(x_{k}, B_{k}),~\hat{G}^{'}(x_{k}, B_{k})^{*}\hat{G}(x_{k}, B_{k})\right\rangle\geq\\
        &\geq\left\{\tau_{k} = \hat{g}_{1}(x_{k}, B_{k}),~\eta_{k}\geq\eta,~\nabla_{x_{k}}\hat{g}_{2}(x_{k}, B_{k}) = 2\hat{G}^{'}(x_{k}, B_{k})^{*}\hat{G}(x_{k}, B_{k}),~\text{следствие \ref{lm:aux_update_rule}}\right\}\geq\\
        &\geq\frac{\eta(2 - \eta)}{8\hat{g}_{1}(x_{k}, B_{k})}\left\langle\left(\hat{G}^{'}(x_{k}, B_{k})^{*}\hat{G}^{'}(x_{k}, B_{k}) + \hat{g}_{1}(x_{k}, B_{k})L_{k}I_{n}\right)^{-1}\nabla_{x_{k}}\hat{g}_{2}(x_{k}, B_{k}),~\nabla_{x_{k}}\hat{g}_{2}(x_{k}, B_{k})\right\rangle\geq 0\Rightarrow\\
        &\Rightarrow\hat{g}_{2}(x_{k}, B_{k}) - \hat{g}_{2}(x_{k + 1}, B_{k})\geq\hat{g}_{2}(x_{k}, B_{g}) - \hat{g}_{1}(x_{k}, B_{k})\hat{g}_{1}(x_{k + 1}, B_{k})\geq\\
        &\geq\frac{\eta(2 - \eta)}{8}\left\langle\left(\hat{G}^{'}(x_{k}, B_{k})^{*}\hat{G}^{'}(x_{k}, B_{k}) + \hat{g}_{1}(x_{k}, B_{k})L_{k}I_{n}\right)^{-1}\nabla_{x_{k}}\hat{g}_{2}(x_{k}, B_{k}),~\nabla_{x_{k}}\hat{g}_{2}(x_{k}, B_{k})\right\rangle\geq\\
        &\geq\left\{\text{по предположениям \ref{as:2} и \ref{as:3}, \eqref{eq:as2_matrix_order}}\right\}\geq\frac{\eta(2 - \eta)\left\|\nabla_{x_{k}}\hat{g}_{2}(x_{k}, B_{k})\right\|^{2}}{8\left(M_{\hat{G}}^{2} + P_{\hat{g}_{1}}L_{k}\right)}\geq\left\{L_{k}\leq \gamma L_{\hat{F}}\right\}\geq\\
        &\geq\frac{\eta(2 - \eta)\left\|\nabla_{x_{k}}\hat{g}_{2}(x_{k}, B_{k})\right\|^{2}}{8\left(M_{\hat{G}}^{2} + \gamma P_{\hat{g}_{1}}L_{\hat{F}}\right)}.
    \end{aligned}
\end{equation}
Сложим неравенства выше при различных номерах итерации от $0$ до $k$ и усредним с помощью $\mathbb{E}\left[\cdot\right]$:
\begin{equation*}
    \begin{aligned}
        \mathbb{E}&\left[\sum\limits_{i = 0}^{k - 1}\left(\hat{g}_{2}(x_{i}, B_{i}) - \hat{g}_{2}(x_{i + 1}, B_{i})\right)\right]=\mathbb{E}\left[\sum\limits_{i = 0}^{k - 1}\left(\hat{f}_{2}(x_{i}) - \hat{g}_{2}(x_{i + 1}, B_{i})\right)\right]\geq\mathbb{E}\left[\sum\limits_{i = 0}^{k - 1}\frac{\eta(2 - \eta)\left\|\nabla_{x_{i}}\hat{g}_{2}(x_{i}, B_{i})\right\|^{2}}{8\left(M_{\hat{G}}^{2} + \gamma P_{\hat{g}_{1}}L_{\hat{F}}\right)}\right] =\\
        &= \frac{\eta(2- \eta)}{8\left(M_{\hat{G}}^{2} + \gamma P_{\hat{g}_{1}}L_{\hat{F}}\right)}\sum\limits_{i = 0}^{k - 1}\mathbb{E}\left[\mathbb{E}\left[\left\|\nabla_{x_{i}}\hat{g}_{2}(x_{i}, B_{i})\right\|^{2}\right]\right]\geq\frac{\eta(2 - \eta)}{8\left(M_{\hat{G}}^{2} + \gamma P_{\hat{g}_{1}}L_{\hat{F}}\right)}\sum\limits_{i = 0}^{k - 1}\mathbb{E}\left[\left\|\nabla\hat{f}_{2}(x_{i})\right\|^{2}\right]\geq\\
        &\geq\frac{k\eta(2 - \eta)}{8\left(M_{\hat{G}}^{2} + \gamma P_{\hat{g}_{1}}L_{\hat{F}}\right)}\min\limits_{i\in\overline{0, k - 1}}\left(\mathbb{E}\left[\left\|\nabla\hat{f}_{2}(x_{i})\right\|^{2}\right]\right)\geq\frac{k\eta(2 - \eta)}{8\left(M_{\hat{G}}^{2} + \gamma P_{\hat{g}_{1}}L_{\hat{F}}\right)}\mathbb{E}\left[\min\limits_{i\in\overline{0, k - 1}}\left\|\nabla\hat{f}_{2}(x_{i})\right\|^{2}\right].
    \end{aligned}
\end{equation*}
Перепишем получившееся неравенство следующим образом:
\begin{equation*}
    \begin{aligned}
        &\frac{k\eta(2 - \eta)}{8\left(M_{\hat{G}}^{2} + \gamma P_{\hat{g}_{1}}L_{\hat{F}}\right)}\mathbb{E}\left[\min\limits_{i\in\overline{0, k - 1}}\left\|\nabla\hat{f}_{2}(x_{i})\right\|^{2}\right]\leq\mathbb{E}\left[\hat{f}_{2}(x_{0}) + \sum\limits_{i = 1}^{k - 1}\left(\hat{f}_{2}(x_{i}) - \hat{g}_{2}(x_{i}, B_{i - 1})\right) - \hat{g}_{2}(x_{k}, B_{k - 1})\right]\leq\\
        &\leq\mathbb{E}\left[\hat{f}_{2}(x_{0})\right] + \sum\limits_{i = 1}^{k - 1}\mathbb{E}\left[\hat{f}_{2}(x_{i}) - \hat{g}_{2}(x_{i}, B_{i - 1})\right] = \left|\mathbb{E}\left[\hat{f}_{2}(x_{0})\right] + \sum\limits_{i = 1}^{k - 1}\mathbb{E}\left[\hat{f}_{2}(x_{i}) - \hat{g}_{2}(x_{i}, B_{i - 1})\right]\right|\leq\mathbb{E}\left[\hat{f}_{2}(x_{0})\right] +\\
        &+ \sum\limits_{i = 1}^{k - 1}\mathbb{E}\left[\left|\hat{f}_{2}(x_{i}) - \hat{g}_{2}(x_{i}, B_{i - 1})\right|\right].
    \end{aligned}
\end{equation*}
Согласно лемме \ref{lm:aux_bounded_deviation} выражение выше можно ограничить сверху:
\begin{equation*}
    \begin{aligned}
        &\frac{k\eta(2 - \eta)}{8\left(M_{\hat{G}}^{2} + \gamma P_{\hat{g}_{1}}L_{\hat{F}}\right)}\mathbb{E}\left[\min\limits_{i\in\overline{0, k - 1}}\left\|\nabla\hat{f}_{2}(x_{i})\right\|^{2}\right]\leq\mathbb{E}\left[\hat{f}_{2}(x_{0})\right] + \sum\limits_{i = 1}^{k - 1}\mathbb{E}\left[\left|\hat{f}_{2}(x_{i}) - \hat{g}_{2}(x_{i}, B_{i - 1})\right|\right]\leq\mathbb{E}\left[\hat{f}_{2}(x_{0})\right] +\\
        &+ \sum\limits_{i = 1}^{k - 1}\left(2l_{\hat{F}}\mathbb{E}\left[\left\|x_{i} - x_{i - 1}\right\|\right]\mathds{1}_{\left\{b < m\right\}} + \tilde{\sigma}\sqrt{\frac{1}{b} - \frac{1}{m}}\right)\leq\left\{\text{лемма \ref{lm:aux_bounded_variation}, следствие \ref{lm:aux_variation_upper_bound}},~L_{k}\geq L,~\eta_{k}\leq1\right\}\leq\\
        &\leq\mathbb{E}\left[\hat{f}_{2}(x_{0})\right] + (k - 1)\left(2l_{\hat{F}}\min\left\{\sqrt{\frac{2P_{\hat{g}_{1}}}{L}},~\frac{M_{\hat{G}}}{L}\right\}\mathds{1}_{\left\{b < m\right\}} + \tilde{\sigma}\sqrt{\frac{1}{b} - \frac{1}{m}}\right) = \mathbb{E}\left[\hat{g}_{2}(x_{0}, B_{0})\right] +\\
        &+ (k - 1)\left(2l_{\hat{F}}\min\left\{\sqrt{\frac{2P_{\hat{g}_{1}}}{L}},~\frac{M_{\hat{G}}}{L}\right\}\mathds{1}_{\left\{b < m\right\}} + \tilde{\sigma}\sqrt{\frac{1}{b} - \frac{1}{m}}\right)\leq\\
        &\leq\mathbb{E}\left[\hat{f}_{2}(x_{0})\right] + k\left(2l_{\hat{F}}\min\left\{\sqrt{\frac{2P_{\hat{g}_{1}}}{L}},~\frac{M_{\hat{G}}}{L}\right\}\mathds{1}_{\left\{b < m\right\}} + \tilde{\sigma}\sqrt{\frac{1}{b} - \frac{1}{m}}\right).
    \end{aligned}
\end{equation*}
Таким образом, делением на $\frac{k\eta(2 - \eta)}{8\left(M_{\hat{G}}^{2} + \gamma P_{\hat{g}_{1}}L_{\hat{F}}\right)}$ выводится искомая оценка \eqref{eq:stoch_sub_lin_conv_1}:
\begin{equation*}
    \begin{aligned}
        \mathbb{E}\left[\min\limits_{i\in\overline{0, k - 1}}\left\|\nabla\hat{f}_{2}(x_{i})\right\|^{2}\right]&\leq\frac{8\left(M_{\hat{G}}^{2} + \gamma P_{\hat{g}_{1}}L_{\hat{F}}\right)}{\eta(2 - \eta)}\left(\frac{\mathbb{E}\left[\hat{f}_{2}(x_{0})\right]}{k} + 2l_{\hat{F}}\min\left\{\sqrt{\frac{2P_{\hat{g}_{1}}}{L}},~\frac{M_{\hat{G}}}{L}\right\}\mathds{1}_{\left\{b < m\right\}} + \tilde{\sigma}\sqrt{\frac{1}{b} - \frac{1}{m}}\right),\\
        k&\in\mathbb{N}.
    \end{aligned}
\end{equation*}
\end{proof}

Теорема \ref{th:4_main} в предположении \ref{as:5} показывает сходимость не просто к некоторой окрестности стационарной точки, а уже к области точки решения задачи \eqref{eq:main_opt_problem}, в среднем.

\begin{re:theorem}\label{th:4}
    Пусть выполнены предположения \ref{as:1}, \ref{as:2}, \ref{as:3}, \ref{as:4}, \ref{as:5}. Рассмотрим  метод Гаусса--Ньютона со схемой реализации \ref{alg:gen_stoch_gnm}, в котором последовательность $\{x_{k}\}_{k\in\mathbb{Z}_{+}}$ вычисляется по правилу \eqref{eq:stoch_direct_update_rule} с\\
    $\tau_{k}~=~\hat{g}_{1}(x_{k}, B_{k})$, $\eta_{k}\in[\eta, 1]$, $\eta\in(0, 1]$. Тогда:
    \begin{equation}\label{eq:stoch_lin_conv_1}
        \begin{cases}
            \begin{aligned}
                \mathbb{E}&\left[\left\|\nabla\hat{f}_{2}(x_{k})\right\|^{2}\right]\leq4M_{\hat{G}}^{2}\Delta_{k,b};\\[5pt]
                \mathbb{E}&\left[\hat{f}_{2}(x_{k})\right]\leq\hat{f}_{2}^{*} + \Delta_{k,b};\\
                &\Delta_{k,b} \overset{\operatorname{def}}{=} \mathbb{E}\left[\hat{f}_{2}(x_{0})\right]\exp\left(-\frac{k\eta(2 - \eta)\mu}{2\left(\gamma L_{\hat{F}}P_{\hat{g}_{1}} + \mu\right)}\right) +\\
                &+ 4\left(l_{\hat{F}}\min\left\{\sqrt{\frac{2P_{\hat{g}_{1}}}{L}},~\frac{M_{\hat{G}}}{L}\right\}\mathds{1}_{\left\{b < m\right\}} + \tilde{\sigma}\sqrt{\frac{1}{b} - \frac{1}{m}}\right)\left(\frac{\gamma L_{\hat{F}}P_{\hat{g}_{1}} + \mu}{\eta(2 - \eta)\mu}\right),~k\in\mathbb{Z}_{+},~b\in\overline{1,~\min\{m, n\}}.
            \end{aligned}
        \end{cases}
    \end{equation}
    Оператор математического ожидания $\mathbb{E}\left[\cdot\right]$ усредняет по всей случайности процесса оптимизации.
\end{re:theorem}
\begin{proof}
Согласно правилу обновления $x_{k}$ \eqref{eq:stoch_direct_update_rule}, \eqref{eq:stoch_general_update_rule} выполнено соотношение (следствие \ref{lm:aux_update_rule}):
\begin{equation}\label{eq:th4_prelim_inequality}
    \begin{aligned}
        &\hat{g}_{1}(x_{k}, B_{k}) - \hat{g}_{1}(x_{k + 1}, B_{k})\geq\\
        &\geq\frac{\eta_{k}(2 - \eta_{k})}{2\hat{g}_{1}(x_{k}, B_{k})}\left\langle\left(\hat{G}^{'}(x_{k}, B_{k})^{*}\hat{G}^{'}(x_{k}, B_{k}) + \hat{g}_{1}(x_{k}, B_{k}) L_{k} I_{n}\right)^{-1}\hat{G}^{'}(x_{k}, B_{k})^{*}\hat{G}(x_{k}, B_{k}),\right.\\
        &\left.\hat{G}^{'}(x_{k}, B_{k})^{*}\hat{G}(x_{k}, B_{k})\right\rangle\geq 0\Rightarrow\hat{g}_{2}(x_{k}, B_{k}) - \hat{g}_{2}(x_{k + 1}, B_{k})\geq\hat{g}_{2}(x_{k}, B_{k}) - \hat{g}_{1}(x_{k}, B_{k})\hat{g}_{1}(x_{k + 1}, B_{k})\geq\\
        &\geq\left\{\eta_{k}\geq\eta,~L_{k}\leq \gamma L_{\hat{F}}\right\}\geq\\
        &\geq\frac{\eta(2 - \eta)}{2}\left\langle\left(\hat{G}^{'}(x_{k}, B_{k})^{*}\hat{G}^{'}(x_{k}, B_{k}) + \gamma\hat{g}_{1}(x_{k}, B_{k})L_{\hat{F}} I_{n}\right)^{-1}\hat{G}^{'}(x_{k}, B_{k})^{*}\hat{G}(x_{k}, B_{k}),~\hat{G}^{'}(x_{k}, B_{k})^{*}\hat{G}(x_{k}, B_{k})\right\rangle =\\
        &= \frac{\eta(2 - \eta)}{2}\left\langle\hat{G}^{'}(x_{k}, B_{k})\left(\hat{G}^{'}(x_{k}, B_{k})^{*}\hat{G}^{'}(x_{k}, B_{k}) + \gamma\hat{g}_{1}(x_{k}, B_{k})L_{\hat{F}} I_{n}\right)^{-1}\hat{G}^{'}(x_{k}, B_{k})^{*}\hat{G}(x_{k}, B_{k}), \hat{G}(x_{k}, B_{k})\right\rangle\geq\\
    \end{aligned}
\end{equation}
\begin{equation}
    \begin{aligned}
        &\geq\left\{\text{лемма \ref{lm:aux_matrix_order}}\right\}\geq\frac{\eta(2 - \eta)\left\|\hat{G}(x_{k}, B_{k})\right\|^{2}\mu}{2\left(\gamma L_{\hat{F}}\hat{g}_{1}(x_{k}, B_{k}) + \mu\right)} = \hat{g}_{2}(x_{k}, B_{k})\frac{\eta(2 - \eta)\mu}{2\left(\gamma L_{\hat{F}}\hat{g}_{1}(x_{k}, B_{k}) + \mu\right)}\Rightarrow\\
        &\Rightarrow\hat{g}_{2}(x_{k + 1}, B_{k})\leq\hat{g}_{2}(x_{k}, B_{k})\left(1 - \frac{\eta(2 - \eta)\mu}{2\left(\gamma L_{\hat{F}}\hat{g}_{1}(x_{k}, B_{k}) + \mu\right)}\right).
    \end{aligned}
\end{equation}
В полученном неравенстве прибавим слева и справа $-\hat{f}_{2}^{*}\leq 0$:
\begin{equation}\label{eq:th4_prelim_inference}
    \begin{aligned}
        &\hat{g}_{2}(x_{k + 1}, B_{k}) - \hat{f}_{2}^{*}\leq\left(\hat{g}_{2}(x_{k}, B_{k}) - \hat{f}_{2}^{*}\right) - \hat{g}_{2}(x_{k}, B_{k})\frac{\eta(2 - \eta)\mu}{2\left(\gamma L_{\hat{F}}\hat{g}_{1}(x_{k}, B_{k}) + \mu\right)}\leq\\
        &\leq\left(\hat{g}_{2}(x_{k}, B_{k}) - \hat{f}_{2}^{*}\right)\left(1 - \frac{\eta(2 - \eta)\mu}{2\left(\gamma L_{\hat{F}}\hat{g}_{1}(x_{k}, B_{k}) + \mu\right)}\right) = \left(\hat{g}_{2}(x_{k}, B_{k}) - \hat{f}_{2}(x_{k}) + \hat{f}_{2}(x_{k}) - \hat{g}_{2}(x_{k}, B_{k - 1}) +\right.\\
        &\left.+ \hat{g}_{2}(x_{k}, B_{k - 1}) .- \hat{f}_{2}^{*}\right)\left(1 - \frac{\eta(2 - \eta)\mu}{2\left(\gamma L_{\hat{F}}\hat{g}_{1}(x_{k}, B_{k}) + \mu\right)}\right)\leq\left(\left|\hat{g}_{2}(x_{k}, B_{k}) - \hat{f}_{2}(x_{k}) + \hat{f}_{2}(x_{k}) - \hat{g}_{2}(x_{k}, B_{k - 1})\right| +\right.\\
        &\left.+ \hat{g}_{2}(x_{k}, B_{k - 1}) - \hat{f}_{2}^{*}\right)\left(1 - \frac{\eta(2 - \eta)\mu}{2\left(\gamma L_{\hat{F}}\hat{g}_{1}(x_{k}, B_{k}) + \mu\right)}\right)\leq\left(\left|\hat{g}_{2}(x_{k}, B_{k}) - \hat{f}_{2}(x_{k})\right| + \left|\hat{f}_{2}(x_{k}) - \hat{g}_{2}(x_{k}, B_{k - 1})\right| +\right.\\
        &\left.+ \left(\hat{g}_{2}(x_{k}, B_{k - 1}) - \hat{f}_{2}^{*}\right)\right)\left(1 - \frac{\eta(2 - \eta)\mu}{2\left(\gamma L_{\hat{F}}\hat{g}_{1}(x_{k}, B_{k}) + \mu\right)}\right).
    \end{aligned}
\end{equation}
Усредним полученное неравенство с помощью оператора $\mathbb{E}\left[\cdot\right]$:
\begin{equation*}
    \begin{aligned}
        &\mathbb{E}\left[\hat{g}_{2}(x_{k + 1}, B_{k}) - \hat{f}_{2}^{*}\right]\leq\mathbb{E}\left[\left(\left|\hat{g}_{2}(x_{k}, B_{k}) - \hat{f}_{2}(x_{k})\right| + \left|\hat{f}_{2}(x_{k}) - \hat{g}_{2}(x_{k}, B_{k - 1})\right| +\right.\right.\\
        &\left.\left.+ \left(\hat{g}_{2}(x_{k}, B_{k - 1}) - \hat{f}_{2}^{*}\right)\right)\left(1 - \frac{\eta(2 - \eta)\mu}{2\left(\gamma L_{\hat{F}}\hat{g}_{1}(x_{k}, B_{k}) + \mu\right)}\right)\right]\leq\left\{\text{по предположению \ref{as:3}}\right\}\leq\\
        &\leq\left(\underbrace{\mathbb{E}\left[\left|\hat{g}_{2}(x_{k}, B_{k}) - \hat{f}_{2}(x_{k})\right|\right]}_{\text{ограничено дисперсией по батчам}} + \underbrace{\mathbb{E}\left[\left|\hat{f}_{2}(x_{k}) - \hat{g}_{2}(x_{k}, B_{k - 1})\right|\right]}_{\text{ограничено по лемме \ref{lm:aux_bounded_deviation}}} +\right.\\
        &\left.+ \mathbb{E}\left[\hat{g}_{2}(x_{k}, B_{k - 1}) - \hat{f}_{2}^{*}\right]\right)\left(1 - \frac{\eta(2 - \eta)\mu}{2\left(\gamma L_{\hat{F}}P_{\hat{g}_{1}} + \mu\right)}\right).
    \end{aligned}
\end{equation*}
Выражение выше согласно леммам \ref{lm:aux_finite_population_variance} (следствие \ref{lm:aux_finite_population_deviation}) и \ref{lm:aux_bounded_deviation} ограничено:
\begin{equation}\label{eq:rec_geom_exp}
    \begin{aligned}
        &\mathbb{E}\left[\hat{g}_{2}(x_{k + 1}, B_{k}) - \hat{f}_{2}^{*}\right]\leq\left(\tilde{\sigma}\sqrt{\frac{1}{b} - \frac{1}{m}} + 2l_{\hat{F}}\mathbb{E}\left[\left\|x_{k} - x_{k - 1}\right\|\right]\mathds{1}_{\left\{b < m\right\}} + \tilde{\sigma}\sqrt{\frac{1}{b} - \frac{1}{m}} +\right.\\
        &\left.+ \mathbb{E}\left[\hat{g}_{2}(x_{k}, B_{k - 1}) - \hat{f}_{2}^{*}\right]\right)\left(1 - \frac{\eta(2 - \eta)\mu}{2\left(\gamma L_{\hat{F}}P_{\hat{g}_{1}} + \mu\right)}\right)\leq\left\{\text{следствие \ref{lm:aux_variation_upper_bound}},~\eta_{k}\leq1\right\}\leq\\
        &\leq\left(2\left(l_{\hat{F}}\min\left\{\sqrt{\frac{2P_{\hat{g}_{1}}}{L}},~\frac{M_{\hat{G}}}{L}\right\}\mathds{1}_{\left\{b < m\right\}} + \tilde{\sigma}\sqrt{\frac{1}{b} - \frac{1}{m}}\right) + \mathbb{E}\left[\hat{g}_{2}(x_{k}, B_{k - 1}) - \hat{f}_{2}^{*}\right]\right)\left(1 - \frac{\eta(2 - \eta)\mu}{2\left(\gamma L_{\hat{F}}P_{\hat{g}_{1}} + \mu\right)}\right).
    \end{aligned}
\end{equation}
Выражение \eqref{eq:rec_geom_exp} представляет собой рекуррентную зависимость по итерациям $k\in\mathbb{N}$:
\begin{equation}\label{eq:sub_rec_geom_exp}
    \begin{cases}
        a_{k} &\overset{\operatorname{def}}{=} \mathbb{E}\left[\hat{g}_{2}(x_{k}, B_{k - 1}) - \hat{f}_{2}^{*}\right];\\
        c_{k} &\overset{\operatorname{def}}{=} c = 2\left(l_{\hat{F}}\min\left\{\sqrt{\frac{2P_{\hat{g}_{1}}}{L}},~\frac{M_{\hat{G}}}{L}\right\}\mathds{1}_{\left\{b < m\right\}} + \tilde{\sigma}\sqrt{\frac{1}{b} - \frac{1}{m}}\right);\\
        q &\overset{\operatorname{def}}{=} \left(1 - \frac{\eta(2 - \eta)\mu}{2\left(\gamma L_{\hat{F}}P_{\hat{g}_{1}} + \mu\right)}\right)\leq\exp\left(-\frac{\eta(2 - \eta)\mu}{2\left(\gamma L_{\hat{F}}P_{\hat{g}_{1}} + \mu\right)}\right)\in(0, 1).
    \end{cases}
\end{equation}
В связи с этим для последовательности $\left\{a_{k}\right\}_{k\in\mathbb{N}}$, определённой  в \eqref{eq:sub_rec_geom_exp}, выполнены следующие соотношения:
\begin{equation*}
    \begin{cases}
        a_{0}&\overset{\operatorname{def}}{=}\mathbb{E}\left[\hat{g}_{2}(x_{0}, B_{0})\right];\\
        a_{1}&\leq a_{0}q\leq a_{0}q + cq;\\
        a_{k + 1}&\leq\left(a_{k} + c\right)q,~k\in\mathbb{N}.
    \end{cases}
\end{equation*}
Оценка на $a_{1}$ получена непосредственно из \eqref{eq:th4_prelim_inference}:
\begin{equation*}
    \begin{aligned}
        &\hat{g}_{2}(x_{k + 1}, B_{k}) - \hat{f}_{2}^{*}\leq\hat{g}_{2}(x_{k + 1}, B_{k})\leq\hat{g}_{2}(x_{k}, B_{k})\left(1 - \frac{\eta(2 - \eta)\mu}{2\left(\gamma L_{\hat{F}}\hat{g}_{1}(x_{k}, B_{k}) + \mu\right)}\right)\leq\left\{\text{предположение \ref{as:3}}\right\}\leq\\
        &\leq\hat{g}_{2}(x_{k}, B_{k})\left(1 - \frac{\eta(2 - \eta)\mu}{2\left(\gamma L_{\hat{F}}P_{\hat{g}_{1}} + \mu\right)}\right)\Rightarrow\left\{k = 0\right\}\Rightarrow\mathbb{E}\left[\hat{g}_{2}(x_{1}, B_{0}) - \hat{f}_{2}^{*}\right]\leq\\
        &\leq\mathbb{E}\left[\hat{g}_{2}(x_{0}, B_{0})\right]\left(1 - \frac{\eta(2 - \eta)\mu}{2\left(\gamma L_{\hat{F}}P_{\hat{g}_{1}} + \mu\right)}\right)\leq\mathbb{E}\left[\hat{g}_{2}(x_{0}, B_{0})\right]\exp\left(-\frac{\eta(2 - \eta)\mu}{2\left(\gamma L_{\hat{F}}P_{\hat{g}_{1}} + \mu\right)}\right).
    \end{aligned}
\end{equation*}
Значение $a_{k}$ оценивается сверху суммой:
\begin{equation}\label{eq:th4_geom_sum}
    \begin{aligned}
        &a_{k}\leq\underbrace{(a_{k - 1} + c)q\leq((a_{k - 2} + c)q + c)q}_{\text{частичная сумма геометрического ряда}}\leq\dots\leq a_{0}q^{k} + c\sum\limits_{i = 1}^{k - 1}q^{i} = a_{0}q^{k} + cq\left(\frac{1 - q^{k - 1}}{1 - q}\right)\mathds{1}_{\left\{k > 0\right\}},~k\in\mathbb{Z}_{+}.
    \end{aligned}
\end{equation}
Теперь свяжем оценку на $a_{k}$ с оценкой на $\hat{f}_{2}(x_{k})$:
\begin{equation*}
    \begin{aligned}
        &\hat{f}_{2}(x_{k}) - \hat{f}_{2}^{*} = \hat{f}_{2}(x_{k}) - \hat{g}_{2}(x_{k}, B_{k - 1}) + \hat{g}_{2}(x_{k}, B_{k - 1}) - \hat{f}_{2}^{*}\leq\left|\hat{f}_{2}(x_{k}) - \hat{g}_{2}(x_{k}, B_{k - 1})\right| + \hat{g}_{2}(x_{k}, B_{k - 1}) - \hat{f}_{2}^{*}\Rightarrow\\
        &\Rightarrow\mathbb{E}\left[\hat{f}_{2}(x_{k}) - \hat{f}_{2}^{*}\right]\leq\mathbb{E}\left[\left|\hat{f}_{2}(x_{k}) - \hat{g}_{2}(x_{k}, B_{k - 1})\right|\right] + \mathbb{E}\left[\hat{g}_{2}(x_{k}, B_{k - 1}) - \hat{f}_{2}^{*}\right]\leq\\
        &\leq\left\{\text{по лемме \ref{lm:aux_bounded_deviation} и следствию \ref{lm:aux_variation_upper_bound}},~\eta_{k}\leq1\right\}\leq\\
        &\leq2l_{\hat{F}}\min\left\{\sqrt{\frac{2P_{\hat{g}_{1}}}{L}},~\frac{M_{\hat{G}}}{L}\right\}\mathds{1}_{\left\{b < m\right\}} + \tilde{\sigma}\sqrt{\frac{1}{b} - \frac{1}{m}} + \mathbb{E}\left[\hat{g}_{2}(x_{k}, B_{k - 1}) - \hat{f}_{2}^{*}\right]=\left\{\text{\eqref{eq:sub_rec_geom_exp}}\right\}=\\
        &=2l_{\hat{F}}\min\left\{\sqrt{\frac{2P_{\hat{g}_{1}}}{L}},~\frac{M_{\hat{G}}}{L}\right\}\mathds{1}_{\left\{b < m\right\}} + \tilde{\sigma}\sqrt{\frac{1}{b} - \frac{1}{m}} + a_{k}\leq\left\{\text{\eqref{eq:th4_geom_sum}}\right\}\leq2l_{\hat{F}}\min\left\{\sqrt{\frac{2P_{\hat{g}_{1}}}{L}},~\frac{M_{\hat{G}}}{L}\right\}\mathds{1}_{\left\{b < m\right\}} +\\
        &+ \tilde{\sigma}\sqrt{\frac{1}{b} - \frac{1}{m}} + a_{0}q^{k} + cq\left(\frac{1 - q^{k - 1}}{1 - q}\right)\mathds{1}_{\left\{k > 0\right\}}\leq2\left(l_{\hat{F}}\min\left\{\sqrt{\frac{2P_{\hat{g}_{1}}}{L}},~\frac{M_{\hat{G}}}{L}\right\}\mathds{1}_{\left\{b < m\right\}} + \tilde{\sigma}\sqrt{\frac{1}{b} - \frac{1}{m}}\right) +\\
        &+ a_{0}q^{k} + cq\left(\frac{1 - q^{k - 1}}{1 - q}\right)\mathds{1}_{\left\{k > 0\right\}} = a_{0}q^{k} + c\left(1 + q\left(\frac{1 - q^{k - 1}}{1 - q}\right)\mathds{1}_{\left\{k > 0\right\}}\right),~k\in\mathbb{Z}_{+}.
    \end{aligned}
\end{equation*}
Из получившегося выражения выводится оценка на $\mathbb{E}\left[\hat{f}_{2}(x_{k})\right]$:
\begin{equation}\label{eq:th4_delta_value}
    \begin{aligned}
        \mathbb{E}\left[\hat{f}_{2}(x_{k})\right] \leq \hat{f}_{2}^{*} + a_{0}q^{k} + c\left(1 + q\left(\frac{1 - q^{k - 1}}{1 - q}\right)\mathds{1}_{\left\{k > 0\right\}}\right),~k\in\mathbb{Z}_{+}.
    \end{aligned}
\end{equation}
Выведем оценку на средний квадрат нормы градиента:
\begin{equation*}
    \begin{aligned}
        &\mathbb{E}\left[\left\|\nabla\hat{f}_{2}(x_{k})\right\|^{2}\right]\leq\mathbb{E}\left[\mathbb{E}\left[\left\|\nabla_{x_{k}}\hat{g}_{2}(x_{k}, B_{k})\right\|^{2}\right]\right] = \mathbb{E}\left[\left\|2\hat{G}^{'}(x_{k}, B_{k})^{*}\hat{G}(x_{k}, B_{k})\right\|^{2}\right]\leq\\
        &\leq 4\mathbb{E}\left[\left\|\hat{G}^{'}(x_{k}, B_{k})^{*}\right\|^{2}\left\|\hat{G}(x_{k}, B_{k})\right\|^{2}\right]\leq\left\{\text{предположение \ref{as:2}}\right\}\leq 4M_{\hat{G}}^{2}\mathbb{E}\left[\hat{g}_{2}(x_{k}, B_{k})\right].
    \end{aligned}
\end{equation*}
Рассмотрим следующее выражение:
\begin{equation*}
    \begin{aligned}
        &\hat{g}_{2}(x_{k}, B_{k}) - \hat{f}_{2}^{*} = \hat{g}_{2}(x_{k}, B_{k}) - \hat{f}_{2}(x_{k}) + \hat{f}_{2}(x_{k}) - \hat{g}_{2}(x_{k}, B_{k - 1}) + \hat{g}_{2}(x_{k}, B_{k - 1}) - \hat{f}_{2}^{*}\leq\left|\hat{g}_{2}(x_{k}, B_{k}) - \hat{f}_{2}(x_{k})\right| +\\
        &+ \left|\hat{f}_{2}(x_{k}) - \hat{g}_{2}(x_{k}, B_{k - 1})\right| + \left(\hat{g}_{2}(x_{k}, B_{k - 1}) - \hat{f}_{2}^{*}\right)\Rightarrow\mathbb{E}\left[\hat{g}_{2}(x_{k}, B_{k}) - \hat{f}_{2}^{*}\right]\leq\mathbb{E}\left[\left|\hat{g}_{2}(x_{k}, B_{k}) - \hat{f}_{2}(x_{k})\right|\right] +\\
    \end{aligned}
\end{equation*}
\begin{equation}\label{eq:th4_grad_bound_mult}
    \begin{aligned}
        &+ \mathbb{E}\left[\left|\hat{f}_{2}(x_{k}) - \hat{g}_{2}(x_{k}, B_{k - 1})\right|\right] + \mathbb{E}\left[\hat{g}_{2}(x_{k}, B_{k - 1}) - \hat{f}_{2}^{*}\right]\leq\left\{\text{следствие \ref{lm:aux_finite_population_deviation} и лемма \ref{lm:aux_bounded_deviation}}\right\}\leq\tilde{\sigma}\sqrt{\frac{1}{b} - \frac{1}{m}} +\\
        &+ 2l_{\hat{F}}\mathbb{E}\left[\left\|x_{k} - x_{k - 1}\right\|\right]\mathds{1}_{\left\{b < m\right\}} + \tilde{\sigma}\sqrt{\frac{1}{b} - \frac{1}{m}} + \mathbb{E}\left[\hat{g}_{2}(x_{k}, B_{k - 1}) - \hat{f}_{2}^{*}\right]\leq\left\{\text{следствие \ref{lm:aux_variation_upper_bound}},~\eta_{k}\leq1\right\}\leq\\
        &\leq2\left(l_{\hat{F}}\min\left\{\sqrt{\frac{2P_{\hat{g}_{1}}}{L}},~\frac{M_{\hat{G}}}{L}\right\}\mathds{1}_{\left\{b < m\right\}} + \tilde{\sigma}\sqrt{\frac{1}{b} - \frac{1}{m}}\right) + \mathbb{E}\left[\hat{g}_{2}(x_{k}, B_{k - 1}) - \hat{f}_{2}^{*}\right] = c + a_{k}\leq\left\{\text{\eqref{eq:th4_geom_sum}}\right\}\leq\\
        &\leq a_{0}q^{k} + c\left(1 + q\left(\frac{1 - q^{k - 1}}{1 - q}\right)\mathds{1}_{\left\{k > 0\right\}}\right),~k\in\mathbb{Z}_{+}.
    \end{aligned}
\end{equation}
В силу, вообще говоря, возможности использования произвольной неотрицательной величины вместо $\hat{f}_{2}^{*}$ в \eqref{eq:th4_prelim_inference} и \eqref{eq:th4_grad_bound_mult} вывод при $\hat{f}_{2}^{*} = 0$ в \eqref{eq:th4_prelim_inference} и \eqref{eq:th4_grad_bound_mult} соответствует верхней границе на $\mathbb{E}\left[\left\|\nabla\hat{f}_{2}(x_{k})\right\|^{2}\right]$:
\begin{equation}\label{eq:th4_gradient_bound}
    \begin{aligned}
        &\mathbb{E}\left[\left\|\nabla\hat{f}_{2}(x_{k})\right\|^{2}\right]\leq 4M_{\hat{G}}^{2}\left(a_{0}q^{k} + c\left(1 + q\left(\frac{1 - q^{k - 1}}{1 - q}\right)\mathds{1}_{\left\{k > 0\right\}}\right)\right),~k\in\mathbb{Z}_{+}.
    \end{aligned}
\end{equation}

Упростим вид выражения в \eqref{eq:th4_delta_value} и \eqref{eq:th4_gradient_bound}:
\begin{equation}\label{eq:th4_common_delta_term}
    \begin{aligned}
        &a_{0}q^{k} + c\left(1 + q\left(\frac{1 - q^{k - 1}}{1 - q}\right)\mathds{1}_{\left\{k > 0\right\}}\right)\leq a_{0}q^{k} + c\left(1 + \frac{q}{1 - q}\right) = \mathbb{E}\left[\hat{g}_{2}(x_{0}, B_{0})\right]\left(1 - \frac{\eta(2 - \eta)\mu}{2\left(\gamma L_{\hat{F}}P_{\hat{g}_{1}} + \mu\right)}\right)^{k} +\\
        &+ 2\left(l_{\hat{F}}\min\left\{\sqrt{\frac{2P_{\hat{g}_{1}}}{L}},~\frac{M_{\hat{G}}}{L}\right\}\mathds{1}_{\left\{b < m\right\}} + \tilde{\sigma}\sqrt{\frac{1}{b} - \frac{1}{m}}\right)\left(1 + \left(1 - \frac{\eta(2 - \eta)\mu}{2\left(\gamma L_{\hat{F}}P_{\hat{g}_{1}} + \mu\right)}\right)\frac{2\left(\gamma L_{\hat{F}}P_{\hat{g}_{1}} + \mu\right)}{\eta(2 - \eta)\mu}\right)\leq\\
        &\leq\mathbb{E}\left[\hat{f}_{2}(x_{0})\right]\exp\left(-\frac{k\eta(2 - \eta)\mu}{2\left(\gamma L_{\hat{F}}P_{\hat{g}_{1}} + \mu\right)}\right) + 4\left(l_{\hat{F}}\min\left\{\sqrt{\frac{2P_{\hat{g}_{1}}}{L}},~\frac{M_{\hat{G}}}{L}\right\}\mathds{1}_{\left\{b < m\right\}} + \tilde{\sigma}\sqrt{\frac{1}{b} - \frac{1}{m}}\right)\left(\frac{\gamma L_{\hat{F}}P_{\hat{g}_{1}} + \mu}{\eta(2 - \eta)\mu}\right) =\\
        &= \Delta_{k,b},~k\in\mathbb{Z}_{+},~b\in\overline{1, m}.
    \end{aligned}
\end{equation}
Таким образом, оценки \eqref{eq:th4_delta_value} и \eqref{eq:th4_gradient_bound} с помощью $\Delta_{k,b}$ приведены в соответствие с \eqref{eq:stoch_lin_conv_1}:
\begin{equation*}
    \begin{cases}
        \mathbb{E}\left[\hat{f}_{2}(x_{k})\right] \leq \hat{f}_{2}^{*} + \Delta_{k,b};\\[5pt]
        \mathbb{E}\left[\left\|\nabla\hat{f}_{2}(x_{k})\right\|^{2}\right]\leq 4M_{\hat{G}}^{2}\Delta_{k,b}.
    \end{cases}
\end{equation*}
\end{proof}

В теореме \ref{th:double_stoch_sublin_conv_main} продемонстрировано улучшение оценки сходимости относительно шума батча на каждой итерации при смене правила \eqref{eq:stoch_direct_update_rule} на правило \eqref{eq:double_stoch_direct_update_rule}.

\begin{re:theorem}\label{th:double_stoch_sublin_conv}
    Пусть выполнены предположения \ref{as:1}, \ref{as:2}, \ref{as:3} и \ref{as:4}. Рассмотрим метод Гаусса--Ньютона, реализованный по схеме \ref{alg:gen_double_stoch_gnm} со стратегией вычисления $x_{k + 1}$ \eqref{eq:double_stoch_direct_update_rule}, в которой
    \begin{equation*}
    \begin{aligned}
        \eta_{k} &= \frac{2\left\langle\hat{G}^{'}(x_{k}, B_{k})^{*}\hat{G}(x_{k}, B_{k}),~\left(\hat{G}^{'}(x_{k}, \tilde{B}_{k})^{*}\hat{G}^{'}(x_{k}, \tilde{B}_{k}) + \tilde{\tau}_{k}L_{k}I_{n}\right)^{-1}\hat{G}^{'}(x_{k}, B_{k})^{*}\hat{G}(x_{k}, B_{k})\right\rangle}{l_{k}\left\langle\hat{G}^{'}(x_{k}, B_{k})^{*}\hat{G}(x_{k}, B_{k}),~\left(\hat{G}^{'}(x_{k}, \tilde{B}_{k})^{*}\hat{G}^{'}(x_{k}, \tilde{B}_{k}) + \tilde{\tau}_{k}L_{k}I_{n}\right)^{-2}\hat{G}^{'}(x_{k}, B_{k})^{*}\hat{G}(x_{k}, B_{k})\right\rangle},\\
        &\tilde{\mathcal{T}}\geq\tilde{\tau}_{k}\geq\tilde{\tau} > 0,~L_{k}\geq L > 0,~k\in\mathbb{Z}_{+}.
    \end{aligned}
\end{equation*}
Тогда при независимом сэмплировании $B_{k}$ и $\tilde{B}_{k}$:
\begin{equation*}
    \begin{aligned}
        \mathbb{E}\left[\min\limits_{i\in\overline{0, k - 1}}\left\{\left\|\nabla\hat{f}_{2}(x_{i})\right\|^{2}\right\}\right]&\leq2\gamma l_{\hat{g}_{2}}\left(\frac{M_{\hat{G}}^{2}}{\tilde{\tau}L} + 1\right)^{2}\left(\frac{\mathbb{E}\left[\hat{f}_{2}(x_{0})\right]}{k} +\right.\\
        &\left.+ \frac{4l_{\hat{F}}M_{\hat{G}}P_{\hat{g}_{1}}}{l}\left(\frac{M_{\hat{G}}^{2}}{\tilde{\tau}L} + 1\right)^{2}\mathds{1}_{\left\{b < m\right\}} + \tilde{\sigma}\sqrt{\frac{1}{b} - \frac{1}{m}}\right),~k\in\mathbb{N}.
    \end{aligned}
\end{equation*}
В случае сэмплирования одного батча на каждом шаге ($B_{k}\equiv\tilde{B}_{k}$) оценка сходимости следующая:
\begin{equation*}
    \begin{aligned}
        \mathbb{E}\left[\min\limits_{i\in\overline{0, k - 1}}\left\{\left\|\nabla\hat{f}_{2}(x_{i})\right\|^{2}\right\}\right]&\leq2\gamma l_{\hat{g}_{2}}\left(\frac{M_{\hat{G}}^{2}}{\tilde{\tau}L} + 1\right)^{2}\left(\frac{\mathbb{E}\left[\hat{f}_{2}(x_{0})\right]}{k} +\right.\\
        &\left.+ 2l_{\hat{F}}\min\left\{\sqrt{\frac{1}{L}\left(\tilde{\mathcal{T}} + \frac{P_{\hat{g}_{1}}^{2}}{\tilde{\tau}}\right)},~\frac{2M_{\hat{G}}P_{\hat{g}_{1}}}{l}\left(\frac{M_{\hat{G}}^{2}}{\tilde{\tau}L} + 1\right)^{2}\right\}\mathds{1}_{\left\{b < m\right\}} + \tilde{\sigma}\sqrt{\frac{1}{b} - \frac{1}{m}}\right),\\
        &k\in\mathbb{N}.
    \end{aligned}
\end{equation*}
Оператор математического ожидания $\mathbb{E}\left[\cdot\right]$ усредняет по всей случайности процесса оптимизации.
\end{re:theorem}
\begin{proof}
Условия теоремы соответствуют следующему соотношению, согласно лемме \ref{lm:aux_double_stoch_update}:
\begin{equation*}
    \begin{aligned}
        \hat{g}_{2}(x_{k + 1}, B_{k})&\leq\hat{g}_{2}(x_{k}, B_{k}) - \frac{\left(\left\langle\nabla_{x_{k}}\hat{g}_{2}(x_{k}, B_{k}),~\left(\hat{G}^{'}(x_{k}, \tilde{B}_{k})^{*}\hat{G}^{'}(x_{k}, \tilde{B}_{k}) + \tilde{\tau}_{k}L_{k}I_{n}\right)^{-1}\nabla_{x_{k}}\hat{g}_{2}(x_{k}, B_{k})\right\rangle\right)^{2}}{2l_{k}\left\langle\nabla_{x_{k}}\hat{g}_{2}(x_{k}, B_{k}),~\left(\hat{G}^{'}(x_{k}, \tilde{B}_{k})^{*}\hat{G}^{'}(x_{k}, \tilde{B}_{k}) + \tilde{\tau}_{k}L_{k}I_{n}\right)^{-2}\nabla_{x_{k}}\hat{g}_{2}(x_{k}, B_{k})\right\rangle},\\
        &k\in\mathbb{Z}_{+}.
    \end{aligned}
\end{equation*}
Оценим сверху правую часть неравенства выше:
\begin{equation*}
    \begin{aligned}
        \hat{g}_{2}(x_{k + 1}, B_{k})&\leq\hat{g}_{2}(x_{k}, B_{k}) - \frac{\left(\left\langle\nabla_{x_{k}}\hat{g}_{2}(x_{k}, B_{k}),~\left(\hat{G}^{'}(x_{k}, \tilde{B}_{k})^{*}\hat{G}^{'}(x_{k}, \tilde{B}_{k}) + \tilde{\tau}_{k}L_{k}I_{n}\right)^{-1}\nabla_{x_{k}}\hat{g}_{2}(x_{k}, B_{k})\right\rangle\right)^{2}}{2l_{k}\left\langle\nabla_{x_{k}}\hat{g}_{2}(x_{k}, B_{k}),~\left(\hat{G}^{'}(x_{k}, \tilde{B}_{k})^{*}\hat{G}^{'}(x_{k}, \tilde{B}_{k}) + \tilde{\tau}_{k}L_{k}I_{n}\right)^{-2}\nabla_{x_{k}}\hat{g}_{2}(x_{k}, B_{k})\right\rangle}\leq\\
        &\leq\hat{g}_{2}(x_{k}, B_{k}) - \frac{\left\|\nabla_{x_{k}}\hat{g}_{2}(x_{k}, B_{k})\right\|^{2}\left(\tilde{\tau}_{k}L_{k}\right)^{2}}{2l_{k}\left(M_{\hat{G}}^{2} + \tilde{\tau}_{k}L_{k}\right)^{2}} = \hat{g}_{2}(x_{k}, B_{k}) - \frac{\left\|\nabla_{x_{k}}\hat{g}_{2}(x_{k}, B_{k})\right\|^{2}}{2l_{k}}\underbrace{\left(\frac{\tilde{\tau}_{k}L_{k}}{M_{\hat{G}}^{2} + \tilde{\tau}_{k}L_{k}}\right)^{2}}_{\text{возрастает по }\tilde{\tau}_{k}L_{k}}\leq\\
        &\leq\hat{g}_{2}(x_{k}, B_{k}) - \frac{\left\|\nabla_{x_{k}}\hat{g}_{2}(x_{k}, B_{k})\right\|^{2}}{2l_{k}}\left(\frac{M_{\hat{G}}^{2}}{\tilde{\tau}L} + 1\right)^{-2}\leq\hat{g}_{2}(x_{k}, B_{k}) -\\
        &- \frac{\left\|\nabla_{x_{k}}\hat{g}_{2}(x_{k}, B_{k})\right\|^{2}}{2\gamma l_{\hat{g}_{2}}}\left(\frac{M_{\hat{G}}^{2}}{\tilde{\tau}L} + 1\right)^{-2}\Rightarrow\\
        &\Rightarrow\frac{\left\|\nabla_{x_{k}}\hat{g}_{2}(x_{k}, B_{k})\right\|^{2}}{2\gamma l_{\hat{g}_{2}}}\left(\frac{M_{\hat{G}}^{2}}{\tilde{\tau}L} + 1\right)^{-2}\leq\hat{g}_{2}(x_{k}, B_{k}) - \hat{g}_{2}(x_{k + 1}, B_{k}).
    \end{aligned}
\end{equation*}
Усредним эти неравенства для итераций $0,~\dots, k - 1$ с помощью оператора $\mathbb{E}\left[\cdot\right]$ и просуммируем:
\begin{equation}\label{eq:double_stoch_sublin_conv_eq1}
    \begin{aligned}
        &\frac{1}{2\gamma l_{\hat{g}_{2}}}\left(\frac{M_{\hat{G}}^{2}}{\tilde{\tau}L} + 1\right)^{-2}\sum\limits_{i = 0}^{k - 1}\mathbb{E}\left[\left\|\nabla_{x_{i}}\hat{g}_{2}(x_{i}, B_{i})\right\|^{2}\right]\leq\sum\limits_{i = 0}^{k - 1}\mathbb{E}\left[\hat{g}_{2}(x_{i}, B_{i}) - \hat{g}_{2}(x_{i + 1}, B_{i})\right]\Rightarrow\\
        &\Rightarrow\frac{k}{2\gamma l_{\hat{g}_{2}}}\left(\frac{M_{\hat{G}}^{2}}{\tilde{\tau}L} + 1\right)^{-2}\mathbb{E}\left[\min\limits_{i\in\overline{0, k - 1}}\left\{\left\|\nabla\hat{f}_{2}(x_{i})\right\|^{2}\right\}\right]\leq\frac{1}{2\gamma l_{\hat{g}_{2}}}\left(\frac{M_{\hat{G}}^{2}}{\tilde{\tau}L} + 1\right)^{-2}\mathbb{E}\left[\sum\limits_{i = 0}^{k - 1}\left\{\left\|\nabla\hat{f}_{2}(x_{i})\right\|^{2}\right\}\right]\leq\\
        &\leq\frac{1}{2\gamma l_{\hat{g}_{2}}}\left(\frac{M_{\hat{G}}^{2}}{\tilde{\tau}L} + 1\right)^{-2}\mathbb{E}\left[\sum\limits_{i = 0}^{k - 1}\left\|\mathbb{E}\left[\nabla_{x_{i}}\hat{g}_{2}(x_{i}, B_{i})\right]\right\|^{2}\right]\leq\frac{1}{2\gamma l_{\hat{g}_{2}}}\left(\frac{M_{\hat{G}}^{2}}{\tilde{\tau}L} + 1\right)^{-2}\mathbb{E}\left[\sum\limits_{i = 0}^{k - 1}\left\|\nabla_{x_{i}}\hat{g}_{2}(x_{i}, B_{i})\right\|^{2}\right]\leq\\
        &\leq\mathbb{E}\left[\hat{g}_{2}(x_{0}, B_{0})\right] + \sum\limits_{i = 0}^{k - 2}\mathbb{E}\left[\hat{g}_{2}(x_{i + 1}, B_{i + 1}) - \hat{g}_{2}(x_{i + 1}, B_{i})\right] - \mathbb{E}\left[\hat{g}_{2}(x_{k}, B_{k - 1})\right]\leq\mathbb{E}\left[\hat{f}_{2}(x_{0})\right] +\\
        &+ \sum\limits_{i = 0}^{k - 1}\underbrace{\mathbb{E}\left[\left|\hat{f}_{2}(x_{i + 1}) - \hat{g}_{2}(x_{i + 1}, B_{i})\right|\right]}_{\text{ограничено по лемме \ref{lm:aux_bounded_deviation}}}\leq\mathbb{E}\left[\hat{f}_{2}(x_{0})\right] + \sum\limits_{i = 0}^{k - 1}\left(2l_{\hat{F}}\mathbb{E}\left[\left\|x_{i + 1} - x_{i}\right\|\right]\mathds{1}_{\left\{b < m\right\}} + \tilde{\sigma}\sqrt{\frac{1}{b} - \frac{1}{m}}\right).
    \end{aligned}
\end{equation}
Оценим сверху значение масштаба шага $\eta_{i}$:
\begin{equation}\label{eq:double_stoch_sublin_conv_eq2}
    \begin{aligned}
        \eta_{i} &= \frac{2\left\langle\hat{G}^{'}(x_{i}, B_{i})^{*}\hat{G}(x_{i}, B_{i}),~\left(\hat{G}^{'}(x_{i}, \tilde{B}_{i})^{*}\hat{G}^{'}(x_{i}, \tilde{B}_{i}) + \tilde{\tau}_{i}L_{i}I_{n}\right)^{-1}\hat{G}^{'}(x_{i}, B_{i})^{*}\hat{G}(x_{i}, B_{i})\right\rangle}{l_{i}\left\langle\hat{G}^{'}(x_{i}, B_{i})^{*}\hat{G}(x_{i}, B_{i}),~\left(\hat{G}^{'}(x_{i}, \tilde{B}_{i})^{*}\hat{G}^{'}(x_{i}, \tilde{B}_{i}) + \tilde{\tau}_{i}L_{i}I_{n}\right)^{-2}\hat{G}^{'}(x_{i}, B_{i})^{*}\hat{G}(x_{i}, B_{i})\right\rangle}\leq\frac{2}{l_{i}}\frac{\left(\tilde{\tau}_{i}L_{i} + M_{\hat{G}}^{2}\right)^{2}}{\tilde{\tau}_{i}L_{i}}.
    \end{aligned}
\end{equation}
Рассмотрим случай независимого сэмплирования $B_{k}$ и $\tilde{B}_{k}$:
\begin{equation*}
    \begin{aligned}
        &\frac{k}{2\gamma l_{\hat{g}_{2}}}\left(\frac{M_{\hat{G}}^{2}}{\tilde{\tau}L} + 1\right)^{-2}\mathbb{E}\left[\min\limits_{i\in\overline{0, k - 1}}\left\{\left\|\nabla\hat{f}_{2}(x_{i})\right\|^{2}\right\}\right]\leq\mathbb{E}\left[\hat{f}_{2}(x_{0})\right] +\\
        &+ \sum\limits_{i = 0}^{k - 1}\left(2l_{\hat{F}}\mathbb{E}\left[\left\|x_{i + 1} - x_{i}\right\|\right]\mathds{1}_{\left\{b < m\right\}} + \tilde{\sigma}\sqrt{\frac{1}{b} - \frac{1}{m}}\right)\leq\left\{\text{лемма \ref{lm:aux_bounded_variation}, оценка \eqref{eq:aux_bounded_variation_upper_bound_2}}\right\}\leq\mathbb{E}\left[\hat{f}_{2}(x_{0})\right] +\\
        &+ \sum\limits_{i = 0}^{k - 1}\left(2l_{\hat{F}}\frac{\eta_{i}M_{\hat{G}}P_{\hat{g}_{1}}}{\tilde{\tau}_{i}L_{i}}\mathds{1}_{\left\{b < m\right\}} + \tilde{\sigma}\sqrt{\frac{1}{b} - \frac{1}{m}}\right)\leq\left\{\text{\eqref{eq:double_stoch_sublin_conv_eq2}}\right\}\leq\mathbb{E}\left[\hat{f}_{2}(x_{0})\right] +\\
        &+ \sum\limits_{i = 0}^{k - 1}\left(\frac{4l_{\hat{F}}M_{\hat{G}}P_{\hat{g}_{1}}}{l_{i}}\left(\frac{M_{\hat{G}}^{2}}{\tilde{\tau}_{i}L_{i}} + 1\right)^{2}\mathds{1}_{\left\{b < m\right\}} + \tilde{\sigma}\sqrt{\frac{1}{b} - \frac{1}{m}}\right)\leq\mathbb{E}\left[\hat{f}_{2}(x_{0})\right] + k\left(\frac{4l_{\hat{F}}M_{\hat{G}}P_{\hat{g}_{1}}}{l}\left(\frac{M_{\hat{G}}^{2}}{\tilde{\tau}L} + 1\right)^{2}\mathds{1}_{\left\{b < m\right\}} +\right.\\
        &\left.+ \tilde{\sigma}\sqrt{\frac{1}{b} - \frac{1}{m}}\right)\Rightarrow\mathbb{E}\left[\min\limits_{i\in\overline{0, k - 1}}\left\{\left\|\nabla\hat{f}_{2}(x_{i})\right\|^{2}\right\}\right]\leq2\gamma l_{\hat{g}_{2}}\left(\frac{M_{\hat{G}}^{2}}{\tilde{\tau}L} + 1\right)^{2}\left(\frac{\mathbb{E}\left[\hat{f}_{2}(x_{0})\right]}{k} + \right.\\
        &\left.+ \frac{4l_{\hat{F}}M_{\hat{G}}P_{\hat{g}_{1}}}{l}\left(\frac{M_{\hat{G}}^{2}}{\tilde{\tau}L} + 1\right)^{2}\mathds{1}_{\left\{b < m\right\}} + \tilde{\sigma}\sqrt{\frac{1}{b} - \frac{1}{m}}\right).
    \end{aligned}
\end{equation*}
Теперь в \eqref{eq:double_stoch_sublin_conv_eq1} рассмотрим сэмплирование одного батча на каждом шаге ($B_{k}\equiv\tilde{B}_{k}$):
\begin{equation*}
    \begin{aligned}
        &\frac{k}{2\gamma l_{\hat{g}_{2}}}\left(\frac{M_{\hat{G}}^{2}}{\tilde{\tau}L} + 1\right)^{-2}\mathbb{E}\left[\min\limits_{i\in\overline{0, k - 1}}\left\{\left\|\nabla\hat{f}_{2}(x_{i})\right\|^{2}\right\}\right]\leq\mathbb{E}\left[\hat{f}_{2}(x_{0})\right] +\\
        &+ \sum\limits_{i = 0}^{k - 1}\left(2l_{\hat{F}}\mathbb{E}\left[\left\|x_{i + 1} - x_{i}\right\|\right]\mathds{1}_{\left\{b < m\right\}} + \tilde{\sigma}\sqrt{\frac{1}{b} - \frac{1}{m}}\right)\leq\left\{\text{следствие \ref{lm:aux_double_stoch_variation_upper_bound}},~\tilde{\tau}_{k}\in[\tilde{\tau},~\tilde{\mathcal{T}}]\right\}\leq\\
        &\leq\mathbb{E}\left[\hat{f}_{2}(x_{0})\right] + \sum\limits_{i = 0}^{k - 1}\left(2l_{\hat{F}}\min\left\{\sqrt{\frac{1}{L}\left(\tilde{\mathcal{T}} + \frac{P_{\hat{g}_{1}}^{2}}{\tilde{\tau}}\right)},~\frac{\eta_{i}M_{\hat{G}}P_{\hat{g}_{1}}}{\tilde{\tau}_{i}L_{i}}\right\}\mathds{1}_{\left\{b < m\right\}} + \tilde{\sigma}\sqrt{\frac{1}{b} - \frac{1}{m}}\right)\leq\left\{\text{\eqref{eq:double_stoch_sublin_conv_eq2}}\right\}\leq\\
        &\leq\mathbb{E}\left[\hat{f}_{2}(x_{0})\right] + \sum\limits_{i = 0}^{k - 1}\left(2l_{\hat{F}}\min\left\{\sqrt{\frac{1}{L}\left(\tilde{\mathcal{T}} + \frac{P_{\hat{g}_{1}}^{2}}{\tilde{\tau}}\right)},~\frac{2M_{\hat{G}}P_{\hat{g}_{1}}}{l_{i}}\left(\frac{M_{\hat{G}}^{2}}{\tilde{\tau}_{i}L_{i}} + 1\right)^{2}\right\}\mathds{1}_{\left\{b < m\right\}} + \tilde{\sigma}\sqrt{\frac{1}{b} - \frac{1}{m}}\right)\leq\\
        &\leq\mathbb{E}\left[\hat{f}_{2}(x_{0})\right] + k\left(2l_{\hat{F}}\min\left\{\sqrt{\frac{1}{L}\left(\tilde{\mathcal{T}} + \frac{P_{\hat{g}_{1}}^{2}}{\tilde{\tau}}\right)},~\frac{2M_{\hat{G}}P_{\hat{g}_{1}}}{l}\left(\frac{M_{\hat{G}}^{2}}{\tilde{\tau}L} + 1\right)^{2}\right\}\mathds{1}_{\left\{b < m\right\}} + \tilde{\sigma}\sqrt{\frac{1}{b} - \frac{1}{m}}\right)\Rightarrow\\
        &\Rightarrow\mathbb{E}\left[\min\limits_{i\in\overline{0, k - 1}}\left\{\left\|\nabla\hat{f}_{2}(x_{i})\right\|^{2}\right\}\right]\leq2\gamma l_{\hat{g}_{2}}\left(\frac{M_{\hat{G}}^{2}}{\tilde{\tau}L} + 1\right)^{2}\left(\frac{\mathbb{E}\left[\hat{f}_{2}(x_{0})\right]}{k} +\right.\\
        &\left.+ 2l_{\hat{F}}\min\left\{\sqrt{\frac{1}{L}\left(\tilde{\mathcal{T}} + \frac{P_{\hat{g}_{1}}^{2}}{\tilde{\tau}}\right)},~\frac{2M_{\hat{G}}P_{\hat{g}_{1}}}{l}\left(\frac{M_{\hat{G}}^{2}}{\tilde{\tau}L} + 1\right)^{2}\right\}\mathds{1}_{\left\{b < m\right\}} + \tilde{\sigma}\sqrt{\frac{1}{b} - \frac{1}{m}}\right).
    \end{aligned}
\end{equation*}
\end{proof}
\begin{re:th:corollary}\label{th:double_stoch_sublin_conv_cor1}
При безграничном увеличении $L\rightarrow+\infty$, сохраняя выполнение неравенства $L_{k}\geq L$, в случае сэмплирования одного батча на шаге метода оценка на средний минимальный квадрат нормы градиента уменьшается, сходясь к следующей форме:
\begin{equation*}
    \begin{aligned}
        \mathbb{E}\left[\min\limits_{i\in\overline{0, k - 1}}\left\{\left\|\nabla\hat{f}_{2}(x_{i})\right\|^{2}\right\}\right]&\leq2\gamma l_{\hat{g}_{2}}\left(\frac{\mathbb{E}\left[\hat{f}_{2}(x_{0})\right]}{k} + \tilde{\sigma}\sqrt{\frac{1}{b} - \frac{1}{m}}\right) =\\
        &= 4\gamma\left(M_{\hat{G}}^{2} + L_{\hat{F}}P_{\hat{g}_{1}}\right)\left(\frac{\mathbb{E}\left[\hat{f}_{2}(x_{0})\right]}{k} + \tilde{\sigma}\sqrt{\frac{1}{b} - \frac{1}{m}}\right),~k\in\mathbb{N}.
    \end{aligned}
\end{equation*}
При этом масштабированный шаг метода Гаусса--Ньютона преобразуется в шаг градиентного метода:
\begin{equation*}
    \begin{aligned}
        &x_{k + 1} = x_{k} - \eta_{k}\left(\hat{G}^{'}(x_{k}, \tilde{B}_{k})^{*}\hat{G}^{'}(x_{k}, \tilde{B}_{k}) + \tilde{\tau}_{k}L_{k}I_{n}\right)^{-1}\hat{G}^{'}(x_{k}, B_{k})^{*}\hat{G}(x_{k}, B_{k}) = x_{k} -\\
        &- \frac{\left\langle\hat{G}^{'}(x_{k}, B_{k})^{*}\hat{G}(x_{k}, B_{k}),~\left(\frac{\hat{G}^{'}(x_{k}, \tilde{B}_{k})^{*}\hat{G}^{'}(x_{k}, \tilde{B}_{k})}{\tilde{\tau}_{k}L_{k}} + I_{n}\right)^{-1}\hat{G}^{'}(x_{k}, B_{k})^{*}\hat{G}(x_{k}, B_{k})\right\rangle}{l_{k}\left\langle\hat{G}^{'}(x_{k}, B_{k})^{*}\hat{G}(x_{k}, B_{k}),~\left(\frac{\hat{G}^{'}(x_{k}, \tilde{B}_{k})^{*}\hat{G}^{'}(x_{k}, \tilde{B}_{k})}{\tilde{\tau}_{k}L_{k}} + I_{n}\right)^{-2}\hat{G}^{'}(x_{k}, B_{k})^{*}\hat{G}(x_{k}, B_{k})\right\rangle}\left(\frac{\hat{G}^{'}(x_{k}, \tilde{B}_{k})^{*}\hat{G}^{'}(x_{k}, \tilde{B}_{k})}{\tilde{\tau}_{k}L_{k}} +\right.\\
        &\left.+ I_{n}\right)^{-1}\left(2\hat{G}^{'}(x_{k}, B_{k})^{*}\hat{G}(x_{k}, B_{k})\right)\underset{L_{k}\rightarrow+\infty}{\longrightarrow}x_{k + 1} = x_{k} - \frac{1}{l_{k}}\nabla_{x_{k}}\hat{g}_{2}(x_{k}, B_{k}).
    \end{aligned}
\end{equation*}
\end{re:th:corollary}

Как и в теореме \ref{th:double_stoch_sublin_conv}, теорема \ref{th:double_stoch_lin_conv_main} содержит в себе гибкую относительно шума батчей оценку.

\begin{re:theorem}\label{th:double_stoch_lin_conv}
    Пусть выполнены предположения \ref{as:1}, \ref{as:2}, \ref{as:3}, \ref{as:4} и \ref{as:5}. Рассмотрим метод Гаусса--Ньютона, реализованный по схеме \ref{alg:gen_double_stoch_gnm} со стратегией вычисления $x_{k + 1}$ \eqref{eq:double_stoch_direct_update_rule}, в которой
    \begin{equation*}
    \begin{aligned}
        \eta_{k} &= \frac{2\left\langle\hat{G}^{'}(x_{k}, B_{k})^{*}\hat{G}(x_{k}, B_{k}),~\left(\hat{G}^{'}(x_{k}, \tilde{B}_{k})^{*}\hat{G}^{'}(x_{k}, \tilde{B}_{k}) + \tilde{\tau}_{k}L_{k}I_{n}\right)^{-1}\hat{G}^{'}(x_{k}, B_{k})^{*}\hat{G}(x_{k}, B_{k})\right\rangle}{l_{k}\left\langle\hat{G}^{'}(x_{k}, B_{k})^{*}\hat{G}(x_{k}, B_{k}),~\left(\hat{G}^{'}(x_{k}, \tilde{B}_{k})^{*}\hat{G}^{'}(x_{k}, \tilde{B}_{k}) + \tilde{\tau}_{k}L_{k}I_{n}\right)^{-2}\hat{G}^{'}(x_{k}, B_{k})^{*}\hat{G}(x_{k}, B_{k})\right\rangle},\\
        &\tilde{\mathcal{T}}\geq\tilde{\tau}_{k}\geq\tilde{\tau} > 0,~L_{k}\geq L > 0,~k\in\mathbb{Z}_{+}.
    \end{aligned}
\end{equation*}
Тогда:
\begin{equation*}
    \begin{aligned}
        &\begin{cases}
            \mathbb{E}\left[\left\|\nabla\hat{f}_{2}(x_{k})\right\|^{2}\right]\leq4M_{\hat{G}}^{2}\hat{\Delta}_{k,b};\\[10pt]
            \mathbb{E}\left[\hat{f}_{2}(x_{k})\right]\leq\hat{f}_{2}^{*} + \hat{\Delta}_{k,b};
        \end{cases}
    \end{aligned}
\end{equation*}
где при $k\in\mathbb{Z}_{+}$ и $b\in\overline{1,~\min\{m, n\}}$ оценка $\hat{\Delta}_{k, b}$ определяется следующим образом:
\begin{equation*}
    \begin{aligned}
        \hat{\Delta}_{k, b} &= \mathbb{E}\left[\hat{f}_{2}(x_{0})\right]\exp\left(-\frac{2\mu k}{\gamma l_{\hat{g}_{2}}}\left(\frac{\tilde{\tau}L}{M_{\hat{G}}^{2} + \tilde{\tau}L}\right)^{2}\right) +\\
        &+ \frac{\gamma l_{\hat{g}_{2}}}{\mu}\left(\tilde{\sigma}\sqrt{\frac{1}{b} - \frac{1}{m}} + \frac{2l_{\hat{F}}M_{\hat{G}}P_{\hat{g}_{1}}}{l}\left(\frac{M_{\hat{G}}^{2}}{\tilde{\tau}L} + 1\right)^{2}\mathds{1}_{\left\{b < m\right\}}\right)\left(\frac{M_{\hat{G}}^{2}}{\tilde{\tau}L} + 1\right)^{2}
    \end{aligned}
\end{equation*}
при независимом сэмплировании $B_{k}$ и $\tilde{B}_{k}$ и
\begin{equation*}
    \begin{aligned}
        \hat{\Delta}_{k, b} &= \mathbb{E}\left[\hat{f}_{2}(x_{0})\right]\exp\left(-\frac{2\mu k}{\gamma l_{\hat{g}_{2}}}\left(\frac{\tilde{\tau}L}{M_{\hat{G}}^{2} + \tilde{\tau}L}\right)^{2}\right) +\\
        &+ \frac{\gamma l_{\hat{g}_{2}}}{\mu}\left(\tilde{\sigma}\sqrt{\frac{1}{b} - \frac{1}{m}} + l_{\hat{F}}\min\left\{\sqrt{\frac{1}{L}\left(\tilde{\mathcal{T}} + \frac{P_{\hat{g}_{1}}^{2}}{\tilde{\tau}}\right)},~\frac{2M_{\hat{G}}P_{\hat{g}_{1}}}{l}\left(\frac{M_{\hat{G}}^{2}}{\tilde{\tau}L} + 1\right)^{2}\right\}\mathds{1}_{\left\{b < m\right\}}\right)\left(\frac{M_{\hat{G}}^{2}}{\tilde{\tau}L} + 1\right)^{2}
    \end{aligned}
\end{equation*}
в случае сэмплирования на каждом шаге одного батча ($B_{k}\equiv\tilde{B}_{k}$). Оператор математического ожидания $\mathbb{E}\left[\cdot\right]$ усредняет по всей случайности процесса оптимизации.
\end{re:theorem}
\begin{proof}
Условия теоремы позволяют применить схему доказательства из теоремы \ref{th:double_stoch_sublin_conv} для получения следующего неравенства:
\begin{equation*}
    \begin{aligned}
        &\frac{2\mu\hat{g}_{2}(x_{k}, B_{k})}{\gamma l_{\hat{g}_{2}}}\left(\frac{M_{\hat{G}}^{2}}{\tilde{\tau}L} + 1\right)^{-2}\leq\left\{\text{\eqref{eq:grad_norm_bounds}}\right\}\leq\frac{\left\|\nabla_{x_{k}}\hat{g}_{2}(x_{k}, B_{k})\right\|^{2}}{2\gamma l_{\hat{g}_{2}}}\left(\frac{M_{\hat{G}}^{2}}{\tilde{\tau}L} + 1\right)^{-2}\leq\hat{g}_{2}(x_{k}, B_{k}) - \hat{g}_{2}(x_{k + 1}, B_{k})\Rightarrow\\
        &\Rightarrow\hat{g}_{2}(x_{k + 1}, B_{k})\leq\hat{g}_{2}(x_{k}, B_{k})\left(1 - \frac{2\mu}{\gamma l_{\hat{g}_{2}}}\left(\frac{\tilde{\tau}L}{M_{\hat{G}}^{2} + \tilde{\tau}L}\right)^{2}\right),~k\in\mathbb{Z}_{+}.
    \end{aligned}
\end{equation*}
Полученное выше неравенство позволяет применить рассуждения из теоремы \ref{th:4} со следующей рекуррентной зависимостью:
\begin{equation*}
    \begin{cases}
        a_{k} &\overset{\operatorname{def}}{=} \mathbb{E}\left[\hat{g}_{2}(x_{k}, B_{k - 1}) - \hat{f}_{2}^{*}\right],~\hat{f}_{2}^{*}\geq0,~k\in\mathbb{N};\\
        a_{0}&\overset{\operatorname{def}}{=}\mathbb{E}\left[\hat{g}_{2}(x_{0}, B_{0})\right] = \mathbb{E}\left[\hat{f}_{2}(x_{0})\right];\\
        q &\overset{\operatorname{def}}{=} 1 - \frac{2\mu}{\gamma l_{\hat{g}_{2}}}\left(\frac{\tilde{\tau}L}{M_{\hat{G}}^{2} + \tilde{\tau}L}\right)^{2}\leq\exp\left(-\frac{2\mu}{\gamma l_{\hat{g}_{2}}}\left(\frac{\tilde{\tau}L}{M_{\hat{G}}^{2} + \tilde{\tau}L}\right)^{2}\right);\\
        a_{1}&\leq a_{0}q;\\
        a_{k + 1}&\leq\left(a_{k} + c_{k}\right)q,~k\in\mathbb{N}.
    \end{cases}
\end{equation*}
Условия данной теоремы определяют значение параметра  $c_{k}$ из теоремы \ref{th:4} в зависимости от способа сэмплирования $B_{k}$ и $\tilde{B}_{k}$, $k\in\mathbb{N}$:
\begin{equation*}
    \begin{cases}
        c_{k} &= 2\left(\left(\frac{l_{\hat{F}}\eta_{k}M_{\hat{G}}P_{\hat{g}_{1}}}{\tilde{\tau}_{k}L_{k}}\right)\mathds{1}_{\left\{b < m\right\}} + \tilde{\sigma}\sqrt{\frac{1}{b} - \frac{1}{m}}\right),~B_{k}\text{ и }\tilde{B}_{k}\text{ независимы};\\[10pt]
        c_{k} &= 2\left(l_{\hat{F}}\min\left\{\sqrt{\frac{1}{L}\left(\tilde{\mathcal{T}} + \frac{P_{\hat{g}_{1}}^{2}}{\tilde{\tau}}\right)},~\frac{\eta_{k}M_{\hat{G}}P_{\hat{g}_{1}}}{\tilde{\tau}_{k}L_{k}}\right\}\mathds{1}_{\left\{b < m\right\}} + \tilde{\sigma}\sqrt{\frac{1}{b} - \frac{1}{m}}\right),~B_{k}\equiv\tilde{B}_{k}.
    \end{cases}
\end{equation*}
По аналогии с рассуждениями из теоремы \ref{th:double_stoch_sublin_conv} ограничим сверху значения $c_{k},~k\in\mathbb{Z}_{+}$, используя \eqref{eq:double_stoch_sublin_conv_eq2}:
\begin{equation*}
    \begin{cases}
        c_{k} &\leq c = 2\left(\frac{2l_{\hat{F}}M_{\hat{G}}P_{\hat{g}_{1}}}{l}\left(\frac{M_{\hat{G}}^{2}}{\tilde{\tau}L} + 1\right)^{2}\mathds{1}_{\left\{b < m\right\}} + \tilde{\sigma}\sqrt{\frac{1}{b} - \frac{1}{m}}\right),~B_{k}\text{ и }\tilde{B}_{k}\text{ независимы};\\[10pt]
        c_{k} &\leq c = 2\left(l_{\hat{F}}\min\left\{\sqrt{\frac{1}{L}\left(\tilde{\mathcal{T}} + \frac{P_{\hat{g}_{1}}^{2}}{\tilde{\tau}}\right)},~\frac{2M_{\hat{G}}P_{\hat{g}_{1}}}{l}\left(\frac{M_{\hat{G}}^{2}}{\tilde{\tau}L} + 1\right)^{2}\right\}\mathds{1}_{\left\{b < m\right\}} + \tilde{\sigma}\sqrt{\frac{1}{b} - \frac{1}{m}}\right),~B_{k}\equiv\tilde{B}_{k}.
    \end{cases}
\end{equation*}
Из рекуррентной зависимости $a_{k + 1}\leq(a_{k} + c_{k})q\leq(a_{k} + c)q$ применением рассуждений из теоремы \ref{th:4} выводятся следующие величины:
\begin{equation*}
    \begin{cases}
        \begin{aligned}
            \hat{\Delta}_{k, b} &= \mathbb{E}\left[\hat{f}_{2}(x_{0})\right]\exp\left(-\frac{2\mu k}{\gamma l_{\hat{g}_{2}}}\left(\frac{\tilde{\tau}L}{M_{\hat{G}}^{2} + \tilde{\tau}L}\right)^{2}\right) + \frac{\gamma l_{\hat{g}_{2}}}{\mu}\left(\frac{2l_{\hat{F}}M_{\hat{G}}P_{\hat{g}_{1}}}{l}\left(\frac{M_{\hat{G}}^{2}}{\tilde{\tau}L} + 1\right)^{2}\mathds{1}_{\left\{b < m\right\}} +\right.\\
            &\left.+ \tilde{\sigma}\sqrt{\frac{1}{b} - \frac{1}{m}}\right)\left(\frac{M_{\hat{G}}^{2}}{\tilde{\tau}L} + 1\right)^{2},~B_{k}\text{ и }\tilde{B}_{k}\text{ независимы};
        \end{aligned}\\\\
        \begin{aligned}
            \hat{\Delta}_{k, b} &= \mathbb{E}\left[\hat{f}_{2}(x_{0})\right]\exp\left(-\frac{2\mu k}{\gamma l_{\hat{g}_{2}}}\left(\frac{\tilde{\tau}L}{M_{\hat{G}}^{2} + \tilde{\tau}L}\right)^{2}\right) +\\
            &+ \frac{\gamma l_{\hat{g}_{2}}}{\mu}\left(l_{\hat{F}}\min\left\{\sqrt{\frac{1}{L}\left(\tilde{\mathcal{T}} + \frac{P_{\hat{g}_{1}}^{2}}{\tilde{\tau}}\right)},~\frac{2M_{\hat{G}}P_{\hat{g}_{1}}}{l}\left(\frac{M_{\hat{G}}^{2}}{\tilde{\tau}L} + 1\right)^{2}\right\}\mathds{1}_{\left\{b < m\right\}} + \tilde{\sigma}\sqrt{\frac{1}{b} - \frac{1}{m}}\right)\left(\frac{M_{\hat{G}}^{2}}{\tilde{\tau}L} + 1\right)^{2},\\
            &B_{k}\equiv\tilde{B}_{k}.
        \end{aligned}
    \end{cases}
\end{equation*}
Таким образом, с помощью $\hat{\Delta}_{k,b}$ получены искомые оценки:
\begin{equation*}
    \begin{cases}
        \mathbb{E}\left[\hat{f}_{2}(x_{k})\right] \leq \hat{f}_{2}^{*} + \hat{\Delta}_{k,b};\\[5pt]
        \mathbb{E}\left[\left\|\nabla\hat{f}_{2}(x_{k})\right\|^{2}\right]\leq 4M_{\hat{G}}^{2}\hat{\Delta}_{k,b}.
    \end{cases}
\end{equation*}
Условия сходимости накладывают следующие ограничения:
\begin{equation*}
    \begin{aligned}
        0 &< \frac{2\mu}{\gamma l_{\hat{g}_{2}}}\left(\frac{\tilde{\tau}L}{M_{\hat{G}}^{2} + \tilde{\tau}L}\right)^{2} \leq 1\Rightarrow\gamma\geq\max\left\{1,~\frac{2\mu}{l_{\hat{g}_{2}}}\left(\frac{\tilde{\tau}L}{M_{\hat{G}}^{2} + \tilde{\tau}L}\right)^{2}\right\} =\\
        &= \max\left\{1,~\frac{\mu}{M_{\hat{G}}^{2} + P_{\hat{g}_{1}}L_{\hat{F}}}\left(\frac{\tilde{\tau}L}{M_{\hat{G}}^{2} + \tilde{\tau}L}\right)^{2}\right\} = 1,
    \end{aligned}
\end{equation*}
так как $\mu\leq M_{\hat{G}}^{2}$, согласно \eqref{eq:grad_norm_bounds}.
\end{proof}
\begin{re:th:corollary}\label{th:double_stoch_lin_conv_cor1}
При безграничном увеличении $L\rightarrow+\infty$, сохраняя выполнение неравенства $L\leq L_{k}$, в случае сэмплирования одного батча на каждом шаге метода величина $\hat{\Delta}_{k,b}$ уменьшается, сходясь к следующей форме:
\begin{equation*}
    \begin{aligned}
        \hat{\Delta}_{k, b} &= \mathbb{E}\left[\hat{f}_{2}(x_{0})\right]\exp\left(-\frac{2\mu k}{\gamma l_{\hat{g}_{2}}}\right) + \frac{\tilde{\sigma}\gamma l_{\hat{g}_{2}}}{\mu}\sqrt{\frac{1}{b} - \frac{1}{m}} =\\
        &= \mathbb{E}\left[\hat{f}_{2}(x_{0})\right]\exp\left(-\frac{\mu k}{\gamma\left(M_{\hat{G}}^{2} + L_{\hat{F}}P_{\hat{g}_{1}}\right)}\right) + \frac{2\tilde{\sigma}\gamma\left(M_{\hat{G}}^{2} + L_{\hat{F}}P_{\hat{g}_{1}}\right)}{\mu}\sqrt{\frac{1}{b} - \frac{1}{m}},~k\in\mathbb{Z}_{+},~b\in\overline{1,~\min\left\{m, n\right\}},\\
        &\gamma\geq\max\left\{1,~\frac{\mu}{M_{\hat{G}}^{2} + P_{\hat{g}_{1}}L_{\hat{F}}}\right\} = 1.
    \end{aligned}
\end{equation*}
При этом масштабированный шаг метода Гаусса--Ньютона преобразуется в шаг градиентного метода, как и в случае теоремы \ref{th:double_stoch_sublin_conv}:
\begin{equation*}
    \begin{aligned}
        &x_{k + 1} = x_{k} - \frac{1}{l_{k}}\nabla_{x_{k}}\hat{g}_{2}(x_{k}, B_{k}).
    \end{aligned}
\end{equation*}
\end{re:th:corollary}

В лемме ниже устанавливается верхняя оценка на количество внутренних итераций поиска подходящей локальной постоянной Липшица $L_{k}$ в худшем случае для схем \ref{alg:gen_stoch_unbounded_gnm} и \ref{alg:gen_stoch_flex_gnm} с переменными пределами поиска $L_{k}$.

\begin{lemma}\label{lm:aux_L_search_time_complexity}
    В схемах \ref{alg:gen_stoch_unbounded_gnm} и \ref{alg:gen_stoch_flex_gnm} метода Гаусса--Ньютона количество итераций подбора $L_{k}$ на переменном отрезке
    $$L_{k}\in\left[\max\left\{L,~\frac{L}{s_{k}}\right\},~\max\left\{\tilde{\gamma}L_{\hat{F}},~\frac{\gamma L_{\hat{F}}}{s_{k}}\right\}\right],~\gamma\geq\tilde{\gamma}\geq1,~s_{k} > 0,~L\in\left(0,~\tilde{\gamma}L_{\hat{F}}\right]$$
    сверху ограничено значением $\operatorname{O}\left(\left\lceil\log_{2}\left(\frac{\gamma L_{\hat{F}}}{L}\right)\right\rceil + 1\right)$ на каждом шаге $k\in\mathbb{Z}_{+}$  метода.
\end{lemma}
\begin{proof}
Рассмотрим потенциально возможные случаи для определения отрезка поиска $L_{k}$:
\begin{itemize}
    \item[1.] $L_{k}\in\left[L,~\tilde{\gamma}L_{\hat{F}}\right]$;
    \item[2.] $L_{k}\in\left[\frac{L}{s_{k}},~\frac{\gamma L_{\hat{F}}}{s_{k}}\right]$;
    \item[3.] $L_{k}\in\left[\frac{L}{s_{k}},~\tilde{\gamma}L_{\hat{F}}\right]$;
    \item[4.] $L_{k}\in\left[L,~\frac{\gamma L_{\hat{F}}}{s_{k}}\right]$.
\end{itemize}
Первый случай соответствует достаточно большому значению $s_{k}$: $$s_{k}\geq\frac{\gamma}{\tilde{\gamma}}\geq1.$$
Во втором случае значение $s_{k}$ достаточно малое:
$$0 < s_{k}\leq\min\left\{1,~\frac{\gamma}{\tilde{\gamma}}\right\}.$$
Третий случай на практике вложен в  первые два, так как в этом случае $$s_{k}\in\left[\frac{\gamma}{\tilde{\gamma}},~1\right]\Rightarrow\gamma\in\left(0,~\tilde{\gamma}\right]\cap\left[\tilde{\gamma},~+\infty\right) = \left\{\tilde{\gamma}\right\},~\tilde{\gamma}\geq1.$$
Четвёртый случай порождает следующие отрезки:
$$s_{k}\in\left[1,~\frac{\gamma}{\tilde{\gamma}}\right]\Rightarrow\frac{\gamma L_{\hat{F}}}{s_{k}L}\in\left[\frac{\tilde{\gamma}L_{\hat{F}}}{L},~\frac{\gamma L_{\hat{F}}}{L}\right].$$
Оценим сверху количество итераций подбора $L_{k}$. Наибольшее число неудачных попыток $(i - 1)\in\mathbb{Z}_{+}$ можно получить, если начинать с минимально возможного значения $L_{k}$:
\begin{itemize}
    \item[1.] $L_{k}\in\left[L,~\tilde{\gamma}L_{\hat{F}}\right]\Rightarrow\tilde{\gamma}L_{\hat{F}}\geq2^{i - 1}L,~\log_{2}\left(\frac{\tilde{\gamma}L_{\hat{F}}}{L}\right) + 1\geq i$;
    \item[2.] $L_{k}\in\left[\frac{L}{s_{k}},~\frac{\gamma L_{\hat{F}}}{s_{k}}\right]\Rightarrow\frac{\gamma L_{\hat{F}}}{s_{k}}\geq2^{i - 1}\left(\frac{L}{s_{k}}\right),~\log_{2}\left(\frac{\gamma L_{\hat{F}}}{L}\right) + 1\geq i$;
    \item[3.] $L_{k}\in\left[\frac{L}{s_{k}},~\tilde{\gamma}L_{\hat{F}}\right]\Rightarrow\tilde{\gamma}L_{\hat{F}}\geq2^{i - 1}L,~\log_{2}\left(\frac{\tilde{\gamma}L_{\hat{F}}}{L}\right) + 1\geq i$;
    \item[4.] $L_{k}\in\left[L,~\frac{\gamma L_{\hat{F}}}{s_{k}}\right]\Rightarrow\gamma L_{\hat{F}}\geq\tilde{\gamma}L_{\hat{F}}\geq2^{i - 1}L,~\log_{2}\left(\frac{\gamma L_{\hat{F}}}{L}\right) + 1\geq i$.
\end{itemize}
Верхняя оценка на искомое $i$ (количество итераций подбора $L_{k}$):
$$i\leq\left\lceil\max\left\{\log_{2}\left(\frac{\tilde{\gamma}L_{\hat{F}}}{L}\right),~\log_{2}\left(\frac{\gamma L_{\hat{F}}}{L}\right)\right\}\right\rceil + 1 = \left\lceil\log_{2}\left(\frac{\gamma L_{\hat{F}}}{L}\right)\right\rceil + 1.$$
Стоит заметить, что при $s_{k}\leq\frac{L}{L_{\hat{F}}}$ весь отрезок поиска $L_{k}$ удовлетворяет условиям поиска, что означает на практике не больше двух итераций подбора $L_{k}$, так как значение $\max\left\{\frac{L_{k - 1}}{2},~\max\left\{L,~\frac{L}{s_{k - 1}}\right\}\right\}$ может оказаться вне этого отрезка. А при уменьшении $s_{k}$ значение $x_{k + 1}$ будет меньше отличаться от $x_{k}$ в силу монотонного убывания по $L_{k}$ значения $\|T_{L_{k}, \tau_{k}}(x_{k}) - x_{k}\|$ (согласно следствию \ref{lm:cor_aux_det_local_decrease_1} при замене модели $\psi_{x, L, \tau}(y)$ на $\hat{\psi}_{x, L, \tau}(y, B)$).
\end{proof}

Теорема \ref{th:5_main} рассматривает правило вычисления $x_{k + 1}$ \eqref{eq:stoch_approx_general_update_rule} и выводит оценки сходимости к окрестности стационарной точки в терминах среднего, используя две стратегии определения погрешности проксимального отображения $\varepsilon_{k}$: в первой стратегии погрешность сверху всегда ограничена ненулевой величиной $\frac{\varepsilon}{\hat{g}_{1}(x_{k}, B_{k})}$, во второй стратегии предполагается возможность точного вычисления $x_{k + 1}$.

\begin{re:theorem}\label{th:5}
    Пусть выполнены предположения \ref{as:1}, \ref{as:2}, \ref{as:3}, \ref{as:4}. Рассмотрим метод Гаусса--Ньютона со схемой реализации \ref{alg:gen_stoch_unbounded_gnm}, в котором последовательность $\{x_{k}\}_{k\in\mathbb{Z}_{+}}$ вычисляется по правилу \eqref{eq:stoch_approx_general_update_rule} с $\tau_{k} = \hat{g}_{1}(x_{k}, B_{k})$. Если в схеме \ref{alg:gen_stoch_unbounded_gnm} выбрать следующий отрезок погрешностей $\varepsilon_{k}$:
    \begin{equation*}
        \begin{aligned}
            &0\leq\varepsilon_{k}\leq\frac{\varepsilon}{\hat{g}_{1}(x_{k}, B_{k})},
        \end{aligned}
    \end{equation*}
    то выполнено:
    \begin{equation}\label{eq:stoch_approx_sub_lin_conv}
        \begin{aligned}
            \mathbb{E}\left[\min\limits_{i\in\overline{0, k - 1}}\left\{\left\|\nabla\hat{f}_{2}(x_{i})\right\|^{2}\right\}\right]&\leq8\left(M_{\hat{G}}^{2} + \max\left\{\tilde{\gamma}P_{\hat{g}_{1}}L_{\hat{F}},~\gamma L_{\hat{F}}\right\}\right)\left(\frac{\mathbb{E}\left[\hat{f}_{2}(x_{0})\right]}{k} + \varepsilon +\right.\\
            &\left.+ 2l_{\hat{F}}\left(\sqrt{\frac{2\varepsilon}{L}} + \sqrt{\frac{2P_{\hat{g}_{1}}}{L}}\right)\mathds{1}_{\left\{b < m\right\}} + \tilde{\sigma}\sqrt{\frac{1}{b} - \frac{1}{m}}\right),~k\in\mathbb{N}.
        \end{aligned}
    \end{equation}
    Если в схеме \ref{alg:gen_stoch_unbounded_gnm} ограничения на $\varepsilon_{k}$ выбрать такие:
    \begin{equation*}
        \begin{aligned}
            &0\leq\varepsilon_{k}\leq\frac{\delta\left\|\nabla_{x_{k}}\hat{g}_{2}(x_{k}, B_{k})\right\|^{2}}{8\hat{g}_{1}(x_{k}, B_{k})\left(M_{\hat{G}}^{2} + \hat{g}_{1}(x_{k}, B_{k})L_{k}\right)},~\delta\in[0, 1),
        \end{aligned}
    \end{equation*}
    то оценка сходимости будет следующей:
    \begin{equation}\label{eq:stoch_approx_sub_lin_conv_1}
        \begin{aligned}
            \mathbb{E}\left[\min\limits_{i\in\overline{0, k - 1}}\left\|\nabla\hat{f}_{2}(x_{i})\right\|^{2}\right]&\leq\frac{8\left(M_{\hat{G}}^{2} + \max\left\{\tilde{\gamma}P_{\hat{g}_{1}}L_{\hat{F}},~\gamma L_{\hat{F}}\right\}\right)}{(1 - \delta)}\left(\frac{\mathbb{E}\left[\hat{f}_{2}(x_{0})\right]}{k} +\right.\\
            &\left.+ 2l_{\hat{F}}\left(\sqrt{\frac{\delta P_{\hat{g}_{1}}}{L}} + \sqrt{\frac{2P_{\hat{g}_{1}}}{L}}\right)\mathds{1}_{\left\{b < m\right\}} + \tilde{\sigma}\sqrt{\frac{1}{b} - \frac{1}{m}}\right),~k\in\mathbb{N}.
        \end{aligned}
    \end{equation}
    Оператор математического ожидания $\mathbb{E}\left[\cdot\right]$ усредняет по всей случайности процесса оптимизации.
\end{re:theorem}
\begin{proof}
Свяжем неточный шаг с минимумом стохастической локальной модели:
\begin{equation*}
    \begin{aligned}
        &\hat{g}_{1}(x_{k}, B_{k}) - \hat{g}_{1}(x_{k + 1}, B_{k}) + \varepsilon_{k}\geq\hat{g}_{1}(x_{k}, B_{k}) - \hat{\psi}_{x_{k}, L_{k}, \hat{g}_{1}(x_{k}, B_{k})}(x_{k + 1}, B_{k}) +\\
        &+ \varepsilon_{k}\geq \hat{g}_{1}(x_{k}, B_{k}) - \hat{\psi}_{x_{k}, L_{k}, \hat{g}_{1}(x_{k}, B_{k})}(x_{k + 1}, B_{k}) + \left(\hat{\psi}_{x_{k}, L_{k}, \hat{g}_{1}(x_{k}, B_{k})}(x_{k + 1}, B_{k}) -\right.\\
        &\left.- \hat{\psi}_{x_{k}, L_{k}, \hat{g}_{1}(x_{k}, B_{k})}(\hat{T}_{L_{k}, \hat{g}_{1}(x_{k}, B_{k})}(x_{k}, B_{k}), B_{k})\right) = \hat{g}_{1}(x_{k}, B_{k}) - \hat{\psi}_{x_{k}, L_{k}, \hat{g}_{1}(x_{k}, B_{k})}(\hat{T}_{L_{k}, \hat{g}_{1}(x_{k}, B_{k})}(x_{k}, B_{k}), B_{k}) =\\
        &= \left\{\text{согласно \eqref{eq:aux_upper_model_formula}},~\eta_{k} = 1,~\tau_{k} = \hat{g}_{1}(x_{k}, B_{k})\right\} =\\
        &= \frac{1}{2\hat{g}_{1}(x_{k}, B_{k})}\left\langle\left(\hat{G}^{'}(x_{k}, B_{k})^{*}\hat{G}^{'}(x_{k}, B_{k}) + \hat{g}_{1}(x_{k}, B_{k})L_{k}I_{n}\right)^{-1}\hat{G}^{'}(x_{k}, B_{k})^{*}\hat{G}(x_{k}, B_{k}),\right.\\
        &\left.\hat{G}^{'}(x_{k}, B_{k})^{*}\hat{G}(x_{k}, B_{k})\right\rangle\geq\left\{\text{предположение \ref{as:2},~\eqref{eq:as2_matrix_order}},~\nabla_{x_{k}}\hat{g}_{2}(x_{k}, B_{k}) = 2\hat{G}^{'}(x_{k}, B_{k})^{*}\hat{G}(x_{k}, B_{k})\right\}\geq\\
        &\geq\frac{\left\|\nabla_{x_{k}}\hat{g}_{2}(x_{k}, B_{k})\right\|^{2}}{8\hat{g}_{1}(x_{k}, B_{k})\left(M_{\hat{G}}^{2} + \hat{g}_{1}(x_{k}, B_{k})L_{k}\right)}\Rightarrow\hat{g}_{1}(x_{k}, B_{k})\varepsilon_{k} + \hat{g}_{2}(x_{k}, B_{k}) - \hat{g}_{2}(x_{k + 1}, B_{k})\geq\\
        &\geq\hat{g}_{1}(x_{k}, B_{k})\varepsilon_{k} + \hat{g}_{2}(x_{k}, B_{k}) - \hat{g}_{1}(x_{k}, B_{k})\hat{g}_{1}(x_{k + 1}, B_{k})\geq\\
        &\geq\frac{\left\|\nabla_{x_{k}}\hat{g}_{2}(x_{k}, B_{k})\right\|^{2}}{8\left(M_{\hat{G}}^{2} + \hat{g}_{1}(x_{k}, B_{k})L_{k}\right)}\geq\left\{\text{предположение \ref{as:3}},~\hat{g}_{1}(x_{k}, B_{k})L_{k}\leq \max\left\{\tilde{\gamma}\hat{g}_{1}(x_{k}, B_{k})L_{\hat{F}},~\gamma L_{\hat{F}}\right\}\right\}\geq\\
        &\geq\frac{\left\|\nabla_{x_{k}}\hat{g}_{2}(x_{k}, B_{k})\right\|^{2}}{8\left(M_{\hat{G}}^{2} + \max\left\{\tilde{\gamma}P_{\hat{g}_{1}}L_{\hat{F}},~\gamma L_{\hat{F}}\right\}\right)}.
    \end{aligned}
\end{equation*}
Просуммируем получившиеся неравенства для итераций $0,~\dots, k - 1$:
\begin{equation*}
    \begin{aligned}
        &\hat{g}_{2}(x_{0}, B_{0}) + \sum\limits_{i = 0}^{k - 1}\hat{g}_{1}(x_{i}, B_{i})\varepsilon_{i} + \sum\limits_{i = 0}^{k - 2}\left(\hat{g}_{2}(x_{i + 1}, B_{i + 1}) - \hat{g}_{2}(x_{i + 1}, B_{i})\right) -\\
        &- \hat{g}_{2}(x_{k}, B_{k - 1})\geq\sum\limits_{i = 0}^{k - 1}\frac{\left\|\nabla_{x_{i}}\hat{g}_{2}(x_{i}, B_{i})\right\|^{2}}{8\left(M_{\hat{G}}^{2} + \max\left\{\tilde{\gamma}P_{\hat{g}_{1}}L_{\hat{F}},~\gamma L_{\hat{F}}\right\}\right)} = \frac{\sum\limits_{i = 0}^{k - 1}\left\|\nabla_{x_{i}}\hat{g}_{2}(x_{i}, B_{i})\right\|^{2}}{8\left(M_{\hat{G}}^{2} + \max\left\{\tilde{\gamma}P_{\hat{g}_{1}}L_{\hat{F}},~\gamma L_{\hat{F}}\right\}\right)}\Rightarrow\\
        &\Rightarrow\frac{\sum\limits_{i = 0}^{k - 1}\left\|\nabla_{x_{i}}\hat{g}_{2}(x_{i}, B_{i})\right\|^{2}}{8\left(M_{\hat{G}}^{2} + \max\left\{\tilde{\gamma}P_{\hat{g}_{1}}L_{\hat{F}},~\gamma L_{\hat{F}}\right\}\right)}\leq\hat{g}_{2}(x_{0}, B_{0}) + k\max\limits_{i\in\overline{0, k - 1}}\left\{\hat{g}_{1}(x_{i}, B_{i})\varepsilon_{i}\right\} +\\
        &+ \sum\limits_{i = 0}^{k - 2}\left(\hat{g}_{2}(x_{i + 1}, B_{i + 1}) - \hat{g}_{2}(x_{i + 1}, B_{i})\right).
    \end{aligned}
\end{equation*}
Усредним получившиеся суммы с помощью оператора математического ожидания $\mathbb{E}\left[\cdot\right]$:
\begin{equation*}
    \begin{aligned}
        &\frac{\sum\limits_{i = 0}^{k - 1}\mathbb{E}\left[\left\|\nabla\hat{f}_{2}(x_{i})\right\|^{2}\right]}{8\left(M_{\hat{G}}^{2} + \max\left\{\tilde{\gamma}P_{\hat{g}_{1}}L_{\hat{F}},~\gamma L_{\hat{F}}\right\}\right)} = \frac{\sum\limits_{i = 0}^{k - 1}\mathbb{E}\left[\left\|\mathbb{E}\left[\nabla_{x_{i}}\hat{g}_{2}(x_{i}, B_{i})\right]\right\|^{2}\right]}{8\left(M_{\hat{G}}^{2} + \max\left\{\tilde{\gamma}P_{\hat{g}_{1}}L_{\hat{F}},~\gamma L_{\hat{F}}\right\}\right)}\leq\frac{\sum\limits_{i = 0}^{k - 1}\mathbb{E}\left[\left\|\nabla_{x_{i}}\hat{g}_{2}(x_{i}, B_{i})\right\|^{2}\right]}{8\left(M_{\hat{G}}^{2} + \max\left\{\tilde{\gamma}P_{\hat{g}_{1}}L_{\hat{F}},~\gamma L_{\hat{F}}\right\}\right)}\leq\\
        &\leq\mathbb{E}\left[\hat{f}_{2}(x_{0})\right] + k\mathbb{E}\left[\max\limits_{i\in\overline{0, k - 1}}\left\{\hat{g}_{1}(x_{i}, B_{i})\varepsilon_{i}\right\}\right] + \sum\limits_{i = 0}^{k - 2}\mathbb{E}\left[\hat{g}_{2}(x_{i + 1}, B_{i + 1}) - \hat{g}_{2}(x_{i + 1}, B_{i})\right]\leq\mathbb{E}\left[\hat{f}_{2}(x_{0})\right] +\\
    \end{aligned}
\end{equation*}
\begin{equation}\label{eq:local_approx_descent}
    \begin{aligned}
        &+ k\mathbb{E}\left[\max\limits_{i\in\overline{0, k - 1}}\left\{\hat{g}_{1}(x_{i}, B_{i})\varepsilon_{i}\right\}\right] + \sum\limits_{i = 0}^{k - 2}\underbrace{\mathbb{E}\left[\left|\hat{f}_{2}(x_{i + 1}) - \hat{g}_{2}(x_{i + 1}, B_{i})\right|\right]}_{\text{ограничено по лемме \ref{lm:aux_bounded_deviation}}}\leq\mathbb{E}\left[\hat{f}_{2}(x_{0})\right] + k\varepsilon +\\
        &+ \sum\limits_{i = 0}^{k - 2}\left(2l_{\hat{F}}\mathbb{E}\left[\left\|x_{i + 1} - x_{i}\right\|\right]\mathds{1}_{\{b < m\}} + \tilde{\sigma}\sqrt{\frac{1}{b} - \frac{1}{m}}\right).
    \end{aligned}
\end{equation}
Ограничим сверху значение $\left\|x_{i + 1} - x_{i}\right\|$:
\begin{equation}\label{eq:bounding_approx_x_variance}
    \begin{aligned}
        &\left\|x_{i + 1} - x_{i}\right\| = \left\|x_{i + 1} - \hat{T}_{L_{i}, \tau_{i}}(x_{i}, B_{i}) + \hat{T}_{L_{i}, \tau_{i}}(x_{i}, B_{i}) - x_{i}\right\|\leq\underbrace{\left\|x_{i + 1} - \hat{T}_{L_{i}, \tau_{i}}(x_{i}, B_{i})\right\|}_{\text{ограничено по следствию \ref{lm:aux_approx_solution}}} +\\
        &+ \underbrace{\left\|\hat{T}_{L_{i}, \tau_{i}}(x_{i}, B_{i}) - x_{i}\right\|}_{\text{ограничено по лемме \ref{lm:aux_bounded_variation}}}\leq\sqrt{\frac{2\varepsilon_{i}}{L_{i}}} + \sqrt{\frac{2\hat{g}_{1}(x_{i}, B_{i})}{L_{i}}}\leq\\
        &\leq\left\{\varepsilon_{i}\leq\frac{\varepsilon}{\hat{g}_{1}(x_{i}, B_{i})},~L_{i}\geq\max\left\{L,~\frac{L}{\hat{g}_{1}(x_{i}, B_{i})}\right\},~\hat{g}_{1}(x_{i}, B_{i})\leq P_{\hat{g}_{1}}\right\}\leq\sqrt{\frac{2\varepsilon}{L}} + \sqrt{\frac{2P_{\hat{g}_{1}}}{L}}.
    \end{aligned}
\end{equation}
Подставим результат из \eqref{eq:bounding_approx_x_variance} в \eqref{eq:local_approx_descent}:
\begin{equation*}
    \begin{aligned}
        &\frac{\sum\limits_{i = 0}^{k - 1}\mathbb{E}\left[\left\|\nabla\hat{f}_{2}(x_{i})\right\|^{2}\right]}{8\left(M_{\hat{G}}^{2} + \max\left\{\tilde{\gamma}P_{\hat{g}_{1}}L_{\hat{F}},~\gamma L_{\hat{F}}\right\}\right)}\leq\mathbb{E}\left[\hat{f}_{2}(x_{0})\right] + k\varepsilon + (k - 1)\left(2l_{\hat{F}}\left(\sqrt{\frac{2\varepsilon}{L}} + \sqrt{\frac{2P_{\hat{g}_{1}}}{L}}\right)\mathds{1}_{\{b < m\}} + \tilde{\sigma}\sqrt{\frac{1}{b} - \frac{1}{m}}\right)\leq\\
        &\leq\mathbb{E}\left[\hat{f}_{2}(x_{0})\right] + k\varepsilon + k\left(2l_{\hat{F}}\left(\sqrt{\frac{2\varepsilon}{L}} + \sqrt{\frac{2P_{\hat{g}_{1}}}{L}}\right)\mathds{1}_{\{b < m\}} + \tilde{\sigma}\sqrt{\frac{1}{b} - \frac{1}{m}}\right)\Rightarrow\left\{\text{выведена оценка \eqref{eq:stoch_approx_sub_lin_conv}}\right\}\Rightarrow\\
        &\Rightarrow\mathbb{E}\left[\min\limits_{i\in\overline{0, k - 1}}\left\{\left\|\nabla\hat{f}_{2}(x_{i})\right\|^{2}\right\}\right]\leq\frac{1}{k}\sum\limits_{i = 0}^{k - 1}\mathbb{E}\left[\left\|\nabla\hat{f}_{2}(x_{i})\right\|^{2}\right]\leq\\
        &\leq8\left(M_{\hat{G}}^{2} + \max\left\{\tilde{\gamma}P_{\hat{g}_{1}}L_{\hat{F}},~\gamma L_{\hat{F}}\right\}\right)\left(\frac{\mathbb{E}\left[\hat{f}_{2}(x_{0})\right]}{k} + \varepsilon + 2l_{\hat{F}}\left(\sqrt{\frac{2\varepsilon}{L}} + \sqrt{\frac{2P_{\hat{g}_{1}}}{L}}\right)\mathds{1}_{\left\{b < m\right\}} + \tilde{\sigma}\sqrt{\frac{1}{b} - \frac{1}{m}}\right).
    \end{aligned}
\end{equation*}
Докажем оставшуюся часть теоремы \eqref{eq:stoch_approx_sub_lin_conv_1}. По определению $\hat{\psi}_{x_{k}, L_{k}, \hat{g}_{1}(x_{k}, B_{k})}(\cdot, B_{k})$ \eqref{eq:aux_upper_model_formula}:
\begin{equation}\label{eq:th5_iter_diff}
    \begin{aligned}
        &\hat{g}_{1}(x_{k}, B_{k}) - \hat{g}_{1}(x_{k + 1}, B_{k})\geq\hat{g}_{1}(x_{k}, B_{k}) - \hat{\psi}_{x_{k}, L_{k}, \hat{g}_{1}(x_{k}, B_{k})}(x_{k + 1}, B_{k}) = \hat{g}_{1}(x_{k}, B_{k}) -\\
        &- \hat{\psi}_{x_{k}, L_{k}, \hat{g}_{1}(x_{k}, B_{k})}(\hat{T}_{L_{k}, \hat{g}_{1}(x_{k}, B_{k})}(x_{k}, B_{k}), B_{k}) + \left(\hat{\psi}_{x_{k}, L_{k}, \hat{g}_{1}(x_{k}, B_{k})}(\hat{T}_{L_{k}, \hat{g}_{1}(x_{k}, B_{k})}(x_{k}, B_{k}), B_{k}) -\right.\\
        &\left.- \hat{\psi}_{x_{k}, L_{k}, \hat{g}_{1}(x_{k}, B_{k})}(x_{k + 1}, B_{k})\right)\geq\hat{g}_{1}(x_{k}, B_{k}) - \hat{\psi}_{x_{k}, L_{k}, \hat{g}_{1}(x_{k}, B_{k})}(\hat{T}_{L_{k}, \hat{g}_{1}(x_{k}, B_{k})}(x_{k}, B_{k}), B_{k}) - \varepsilon_{k} =\\
        &= \left\{\text{согласно \eqref{eq:aux_upper_model_formula} с}~\eta_{k} = 1,~\tau_{k} = \hat{g}_{1}(x_{k}, B_{k})\right\} =\\
        &= \frac{1}{2\hat{g}_{1}(x_{k}, B_{k})}\left\langle\left(\hat{G}^{'}(x_{k}, B_{k})^{*}\hat{G}^{'}(x_{k}, B_{k}) + \hat{g}_{1}(x_{k}, B_{k})L_{k}I_{n}\right)^{-1}\hat{G}^{'}(x_{k}, B_{k})^{*}\hat{G}(x_{k}, B_{k}),\right.\\
        &\left.\hat{G}^{'}(x_{k}, B_{k})^{*}\hat{G}(x_{k}, B_{k})\right\rangle - \varepsilon_{k}\geq\left\{\text{предположение \ref{as:2}, \eqref{eq:as2_matrix_order}},~\nabla_{x_{k}}\hat{g}_{2}(x_{k}, B_{k}) = 2\hat{G}^{'}(x_{k}, B_{k})^{*}\hat{G}(x_{k}, B_{k})\right\}\geq\\
        &\geq\frac{\left\|\nabla_{x_{k}}\hat{g}_{2}(x_{k}, B_{k})\right\|^{2}}{8\hat{g}_{1}(x_{k}, B_{k})\left(M_{\hat{G}}^{2} + \hat{g}_{1}(x_{k}, B_{k})L_{k}\right)} - \varepsilon_{k}\geq\frac{(1 - \delta)\left\|\nabla_{x_{k}}\hat{g}_{2}(x_{k}, B_{k})\right\|^{2}}{8\hat{g}_{1}(x_{k}, B_{k})\left(M_{\hat{G}}^{2} + \hat{g}_{1}(x_{k}, B_{k})L_{k}\right)}\Rightarrow\\
        &\Rightarrow\hat{g}_{2}(x_{k}, B_{k}) - \hat{g}_{2}(x_{k + 1}, B_{k})\geq\hat{g}_{2}(x_{k}, B_{k}) - \hat{g}_{1}(x_{k}, B_{k})\hat{g}_{1}(x_{k + 1}, B_{k})\geq\\
        &\geq\frac{(1 - \delta)\left\|\nabla_{x_{k}}\hat{g}_{2}(x_{k}, B_{k})\right\|^{2}}{8\left(M_{\hat{G}}^{2} + \hat{g}_{1}(x_{k}, B_{k})L_{k}\right)}\geq\left\{\text{предположение \ref{as:3}},~\hat{g}_{1}(x_{k}, B_{k})L_{k}\leq\max\left\{\tilde{\gamma}P_{\hat{g}_{1}} L_{\hat{F}},~\gamma L_{\hat{F}}\right\}\right\}\geq\\
        &\geq\frac{(1 - \delta)\left\|\nabla_{x_{k}}\hat{g}_{2}(x_{k}, B_{k})\right\|^{2}}{8\left(M_{\hat{G}}^{2} + \max\left\{\tilde{\gamma}P_{\hat{g}_{1}}L_{\hat{F}},~\gamma L_{\hat{F}}\right\}\right)}.
    \end{aligned}
\end{equation}
Ограничим сверху значение $\left\|x_{i + 1} - x_{i}\right\|$:
\begin{equation}\label{eq:bounding_approx_x_delta_variance}
    \begin{aligned}
        &\varepsilon_{i}\leq\frac{\delta\left\|\nabla_{x_{i}}\hat{g}_{2}(x_{i}, B_{i})\right\|^{2}}{8\hat{g}_{1}(x_{i}, B_{i})\left(M_{\hat{G}}^{2} + \hat{g}_{1}(x_{i}, B_{i})L_{i}\right)}\leq\frac{4\delta\left\|\hat{G}^{'}(x_{i}, B_{i})\right\|^{2}\left\|\hat{G}(x_{i}, B_{i})\right\|^{2}}{8\hat{g}_{1}(x_{i}, B_{i})\left(M_{\hat{G}}^{2} + \hat{g}_{1}(x_{i}, B_{i})L_{i}\right)}\leq\frac{\delta M_{\hat{G}}^{2}\hat{g}_{1}(x_{i}, B_{i})}{2\left(M_{\hat{G}}^{2} + \hat{g}_{1}(x_{i}, B_{i})L_{i}\right)}\leq\\
        &\leq\frac{\delta P_{\hat{g}_{1}}}{2}\Rightarrow\left\|x_{i + 1} - x_{i}\right\|\leq\underbrace{\left\|x_{i + 1} - \hat{T}_{L_{i}, \tau_{i}}(x_{i}, B_{i})\right\|}_{\text{ограничено по следствию \ref{lm:aux_approx_solution}}} + \underbrace{\left\|\hat{T}_{L_{i}, \tau_{i}}(x_{i}, B_{i}) - x_{i}\right\|}_{\text{ограничено по лемме \ref{lm:aux_bounded_variation}}}\leq\sqrt{\frac{2\varepsilon_{i}}{L_{i}}} + \sqrt{\frac{2\hat{g}_{1}(x_{i}, B_{i})}{L_{i}}}\leq\\
        &\leq\left\{\varepsilon_{i}\leq\frac{\delta P_{\hat{g}_{1}}}{2},~L_{i}\geq\max\left\{L,~\frac{L}{\hat{g}_{1}(x_{i}, B_{i})}\right\},~\hat{g}_{1}(x_{i}, B_{i})\leq P_{\hat{g}_{1}}\right\}\leq\sqrt{\frac{\delta P_{\hat{g}_{1}}}{L}} + \sqrt{\frac{2P_{\hat{g}_{1}}}{L}}.
    \end{aligned}
\end{equation}
Выражение в \eqref{eq:th5_iter_diff} соответствует неравенству из \eqref{eq:th3_iter_diff} с заменой $\eta(2 - \eta)$ на $(1 - \delta)$ и $\gamma P_{\hat{g}_{1}}L_{\hat{F}}$ на\\$\max\left\{\tilde{\gamma}P_{\hat{g}_{1}}L_{\hat{F}},~\gamma L_{\hat{F}}\right\}$. Поэтому для выражения \eqref{eq:th5_iter_diff} применима цепочка рассуждений из теоремы \ref{th:3}, использованная для неравенства из \eqref{eq:th3_iter_diff} с учётом оценки из \eqref{eq:bounding_approx_x_delta_variance}, приводящая к искомой оценке:
\begin{equation*}
    \begin{aligned}
        \mathbb{E}\left[\min\limits_{i\in\overline{0, k - 1}}\left\|\nabla\hat{f}_{2}(x_{i})\right\|^{2}\right]&\leq\frac{8\left(M_{\hat{G}}^{2} + \max\left\{\tilde{\gamma}P_{\hat{g}_{1}}L_{\hat{F}},~\gamma L_{\hat{F}}\right\}\right)}{(1 - \delta)}\left(\frac{\mathbb{E}\left[\hat{f}_{2}(x_{0})\right]}{k} +\right.\\
        &\left.+ 2l_{\hat{F}}\left(\sqrt{\frac{\delta P_{\hat{g}_{1}}}{L}} + \sqrt{\frac{2P_{\hat{g}_{1}}}{L}}\right)\mathds{1}_{\left\{b < m\right\}} + \tilde{\sigma}\sqrt{\frac{1}{b} - \frac{1}{m}}\right),~k\in\mathbb{N}.
    \end{aligned}
\end{equation*}
\end{proof}

В теореме \ref{th:6_main} по сравнению с теоремой \ref{th:4} вводится неточное вычисление приближения решения $x_{k + 1}$ по аналогии с предыдущей теоремой и без изменения линейного характера сходимости к области точки минимума в терминах среднего.

\begin{re:theorem}\label{th:6}
    Пусть выполнены предположения \ref{as:1}, \ref{as:2}, \ref{as:3}, \ref{as:4}, \ref{as:5}. Рассмотрим метод Гаусса--Ньютона со схемой реализации \ref{alg:gen_stoch_unbounded_gnm}, в котором последовательность $\{x_{k}\}_{k\in\mathbb{Z}_{+}}$ вычисляется по правилу \eqref{eq:stoch_approx_general_update_rule} с\\$\tau_{k} = \hat{g}_{1}(x_{k}, B_{k})$. Тогда:
    \begin{equation}\label{eq:stoch_approx_lin_conv_1}
        \begin{cases}
            \begin{aligned}
                &\mathbb{E}\left[\left\|\nabla\hat{f}_{2}(x_{k})\right\|^{2}\right]\leq4M_{\hat{G}}^{2}\tilde{\Delta}_{k,b};\\
                &\mathbb{E}\left[\hat{f}_{2}(x_{k})\right]\leq\hat{f}_{2}^{*} + \tilde{\Delta}_{k,b};
            \end{aligned}
        \end{cases}
    \end{equation}
    где 
    \begin{equation*}
        \begin{aligned}
            \tilde{\Delta}_{k,b} &= \mathbb{E}\left[\hat{f}_{2}(x_{0})\right]\exp\left(-\frac{k\mu}{2\left(\max\left\{\tilde{\gamma}L_{\hat{F}}P_{\hat{g}_{1}},~\gamma L_{\hat{F}}\right\} + \mu\right)}\right) + 2\left(\varepsilon + 2\left(l_{\hat{F}}\left(\sqrt{\frac{2\varepsilon}{L}} + \sqrt{\frac{2P_{\hat{g}_{1}}}{L}}\right)\mathds{1}_{\left\{b < m\right\}} +\right.\right.\\
            &\left.\left.+ \tilde{\sigma}\sqrt{\frac{1}{b} - \frac{1}{m}}\right)\right)\left(\frac{\max\left\{\tilde{\gamma}L_{\hat{F}}P_{\hat{g}_{1}},~\gamma L_{\hat{F}}\right\}}{\mu} + 1\right),~k\in\mathbb{Z}_{+},~b\in\overline{1,~\min\{m, n\}}
        \end{aligned}
    \end{equation*}
    в случае 
    \begin{equation*}
        \begin{aligned}
            &0\leq\varepsilon_{k}\leq\frac{\varepsilon}{\hat{g}_{1}(x_{k}, B_{k})},
        \end{aligned}
    \end{equation*}
    и
    \begin{equation*}
        \begin{aligned}
            \tilde{\Delta}_{k,b} &= \mathbb{E}\left[\hat{f}_{2}(x_{0})\right]\exp\left(-\frac{k(1 - \delta)\mu}{2\left(\max\left\{\tilde{\gamma}L_{\hat{F}}P_{\hat{g}_{1}},~\gamma L_{\hat{F}}\right\} + \mu\right)}\right) + 4\left(l_{\hat{F}}\left(\sqrt{\frac{\delta P_{\hat{g}_{1}}}{L}} + \sqrt{\frac{2P_{\hat{g}_{1}}}{L}}\right)\mathds{1}_{\left\{b < m\right\}} +\right.\\
            &\left.+ \tilde{\sigma}\sqrt{\frac{1}{b} - \frac{1}{m}}\right)\left(\frac{\max\left\{\tilde{\gamma}L_{\hat{F}}P_{\hat{g}_{1}},~\gamma L_{\hat{F}}\right\} + \mu}{(1 - \delta)\mu}\right),~k\in\mathbb{Z}_{+},~b\in\overline{1,~\min\{m, n\}}
        \end{aligned}
    \end{equation*}
    в случае
    \begin{equation*}
        \begin{aligned}
            &0\leq\varepsilon_{k}\leq\frac{\delta\hat{g}_{1}(x_{k}, B_{k})\mu}{2\left(L_{k}\hat{g}_{1}(x_{k}, B_{k}) + \mu\right)},~\delta\in[0, 1).
        \end{aligned}
    \end{equation*}
    Оператор математического ожидания $\mathbb{E}\left[\cdot\right]$ усредняет по всей случайности процесса оптимизации.
\end{re:theorem}
\begin{proof}
Докажем теорему для первой стратегии подбора $\varepsilon_{k}$. Выведем нижнюю оценку на приращение локальной модели в результате неточного шага:
\begin{equation*}
    \begin{aligned}
        &\hat{g}_{1}(x_{k}, B_{k}) - \hat{g}_{1}(x_{k + 1}, B_{k}) + \varepsilon_{k}\geq\hat{g}_{1}(x_{k}, B_{k}) - \hat{\psi}_{x_{k}, L_{k}, \hat{g}_{1}(x_{k}, B_{k})}(x_{k + 1}, B_{k}) + \varepsilon_{k} \geq \hat{g}_{1}(x_{k}, B_{k}) -\\
        &- \hat{\psi}_{x_{k}, L_{k}, \hat{g}_{1}(x_{k}, B_{k})}(x_{k + 1}, B_{k}) + \left(\hat{\psi}_{x_{k}, L_{k}, \hat{g}_{1}(x_{k}, B_{k})}(x_{k + 1}, B_{k}) - \hat{\psi}_{x_{k}, L_{k}, \hat{g}_{1}(x_{k}, B_{k})}(\hat{T}_{L_{k}, \hat{g}_{1}(x_{k}, B_{k})}(x_{k}, B_{k}), B_{k})\right) =\\
        &= \hat{g}_{1}(x_{k}, B_{k}) - \hat{\psi}_{x_{k}, L_{k}, \hat{g}_{1}(x_{k}, B_{k})}(\hat{T}_{L_{k}, \hat{g}_{1}(x_{k}, B_{k})}(x_{k}, B_{k}), B_{k}) = \left\{\text{согласно \eqref{eq:aux_upper_model_formula} с}~\eta_{k} = 1,~\tau_{k} = \hat{g}_{1}(x_{k}, B_{k})\right\} =\\
        &=\frac{1}{2\hat{g}_{1}(x_{k}, B_{k})}\left\langle\left(\hat{G}^{'}(x_{k}, B_{k})^{*}\hat{G}^{'}(x_{k}, B_{k}) + \hat{g}_{1}(x_{k}, B_{k})L_{k}I_{n}\right)^{-1}\hat{G}^{'}(x_{k}, B_{k})^{*}\hat{G}(x_{k}, B_{k}),\right.\\
        &\left.\hat{G}^{'}(x_{k}, B_{k})^{*}\hat{G}(x_{k}, B_{k})\right\rangle =\\
        &= \frac{1}{2\hat{g}_{1}(x_{k}, B_{k})}\left\langle\hat{G}^{'}(x_{k}, B_{k})\left(\hat{G}^{'}(x_{k}, B_{k})^{*}\hat{G}^{'}(x_{k}, B_{k}) + \hat{g}_{1}(x_{k}, B_{k})L_{k}I_{n}\right)^{-1}\hat{G}^{'}(x_{k}, B_{k})^{*}\hat{G}(x_{k}, B_{k}),\right.\\
        &\left.\hat{G}(x_{k}, B_{k})\right\rangle\geq\left\{\text{лемма \ref{lm:aux_matrix_order}},~\hat{g}_{1}(x_{k}, B_{k}) = \left\|\hat{G}(x_{k}, B_{k})\right\|\right\}\geq\frac{\hat{g}_{1}(x_{k}, B_{k})\mu}{2\left(L_{k}\hat{g}_{1}(x_{k}, B_{k}) + \mu\right)}\geq\\
        &\geq\left\{\hat{g}_{1}(x_{k}, B_{k})L_{k}\leq\max\left\{\tilde{\gamma}L_{\hat{F}}P_{\hat{g}_{1}},~\gamma L_{\hat{F}}\right\}\right\}\geq\\
        &\geq\frac{\hat{g}_{1}(x_{k}, B_{k})\mu}{2\left(\max\left\{\tilde{\gamma}L_{\hat{F}}P_{\hat{g}_{1}},~\gamma L_{\hat{F}}\right\} + \mu\right)}\Rightarrow\hat{g}_{1}(x_{k}, B_{k})\varepsilon_{k} + \hat{g}_{2}(x_{k}, B_{k}) - \hat{g}_{2}(x_{k + 1}, B_{k})\geq\hat{g}_{1}(x_{k}, B_{k})\varepsilon_{k} + \hat{g}_{2}(x_{k}, B_{k}) -\\
        &- \hat{g}_{1}(x_{k}, B_{k})\hat{g}_{1}(x_{k + 1}, B_{k})\geq\frac{\hat{g}_{2}(x_{k}, B_{k})\mu}{2\left(\max\left\{\tilde{\gamma}L_{\hat{F}}P_{\hat{g}_{1}},~\gamma L_{\hat{F}}\right\} + \mu\right)}\Rightarrow\\
        &\Rightarrow\hat{g}_{2}(x_{k + 1}, B_{k})\leq\hat{g}_{2}(x_{k}, B_{k})\left(1 - \frac{\mu}{2\left(\max\left\{\tilde{\gamma}L_{\hat{F}}P_{\hat{g}_{1}},~\gamma L_{\hat{F}}\right\} + \mu\right)}\right) + \hat{g}_{1}(x_{k}, B_{k})\varepsilon_{k}\leq\\
        &\leq\hat{g}_{2}(x_{k}, B_{k})\left(1 - \frac{\mu}{2\left(\max\left\{\tilde{\gamma}L_{\hat{F}}P_{\hat{g}_{1}},~\gamma L_{\hat{F}}\right\} + \mu\right)}\right) + \varepsilon.
    \end{aligned}
\end{equation*}
Полученное неравенство позволяет применить рассуждения из теоремы \ref{th:4}, положив $\eta_{k} = 1,~k\in\mathbb{Z}_{+}$. В частности, определённая в \eqref{eq:sub_rec_geom_exp} линейная рекуррентная зависимость в данной теореме выглядит следующим образом:
\begin{equation*}
    \begin{cases}
        a_{0}&=\mathbb{E}\left[\hat{g}_{2}(x_{0}, B_{0})\right];\\
        a_{1}&\leq a_{0}q + \varepsilon;\\
        a_{k + 1}&\leq\left(a_{k} + c\right)q + \varepsilon,~k\in\mathbb{N};\\
        q &= 1 - \frac{\mu}{2\left(\max\left\{\tilde{\gamma}L_{\hat{F}}P_{\hat{g}_{1}},~\gamma L_{\hat{F}}\right\} + \mu\right)}.
    \end{cases}
\end{equation*}
Оставшееся непереопределённое значение оценивается с помощью \eqref{eq:bounding_approx_x_variance}:
\begin{equation*}
    \begin{aligned}
        c &= 2\left(l_{\hat{F}}\left(\sqrt{\frac{2\varepsilon}{L}} + \sqrt{\frac{2P_{\hat{g}_{1}}}{L}}\right)\mathds{1}_{\left\{b < m\right\}} + \tilde{\sigma}\sqrt{\frac{1}{b} - \frac{1}{m}}\right).
    \end{aligned}
\end{equation*}
Раскроем рекурсию:
\begin{equation*}
    \begin{aligned}
        &a_{k}\leq\underbrace{(a_{k - 1} + c)q + \varepsilon\leq((a_{k - 2} + c)q + \varepsilon + c)q + \varepsilon}_{\text{частичная сумма геометрического ряда}}\leq\dots\leq a_{0}q^{k} + c\sum\limits_{i = 1}^{k - 1}q^{i} +\\
        &+ \varepsilon\sum\limits_{i = 0}^{k - 1}q^{i}= a_{0}q^{k} + cq\left(\frac{1 - q^{k - 1}}{1 - q}\right)\mathds{1}_{\left\{k > 0\right\}} + \varepsilon\left(\frac{1 - q^{k}}{1 - q}\right)\mathds{1}_{\left\{k > 0\right\}},~k\in\mathbb{Z}_{+}.
    \end{aligned}
\end{equation*}
Поэтому в условиях данной теоремы выражения \eqref{eq:th4_delta_value} и \eqref{eq:th4_gradient_bound} будут содержать следующую общую часть:
\begin{equation*}
    \begin{aligned}
        &a_{0}q^{k} + cq\left(\frac{1 - q^{k - 1}}{1 - q}\right)\mathds{1}_{\left\{k > 0\right\}} + \left(c + \varepsilon\left(\frac{1 - q^{k}}{1 - q}\right)\right)\mathds{1}_{\left\{k > 0\right\}},~k\in\mathbb{Z}_{+}.
    \end{aligned}
\end{equation*}
У полученной общей части в упрощённом виде добавляется слагаемое с $\varepsilon$ по сравнению с \eqref{eq:th4_common_delta_term}:
\begin{equation*}
    \begin{aligned}
        &a_{k}\leq a_{0}q^{k} + c + \frac{cq + \varepsilon}{1 - q},~k\in\mathbb{Z}_{+}.
    \end{aligned}
\end{equation*}
Неравенство выше определяет $\tilde{\Delta}_{k,b}$ в случае $\varepsilon_{k}\leq\frac{\varepsilon}{\hat{g}_{1}(x_{k}, B_{k})}$ аналогично \eqref{eq:th4_common_delta_term}:
\begin{equation*}
    \begin{aligned}
        \tilde{\Delta}_{k,b} &= \mathbb{E}\left[\hat{f}_{2}(x_{0})\right]\exp\left(-\frac{k\mu}{2\left(\max\left\{\tilde{\gamma}L_{\hat{F}}P_{\hat{g}_{1}},~\gamma L_{\hat{F}}\right\} + \mu\right)}\right)
        + 2\left(\varepsilon + 2\left(l_{\hat{F}}\left(\sqrt{\frac{2\varepsilon}{L}} + \sqrt{\frac{2P_{\hat{g}_{1}}}{L}}\right)\mathds{1}_{\left\{b < m\right\}} +\right.\right.\\
        &\left.\left.+ \tilde{\sigma}\sqrt{\frac{1}{b} - \frac{1}{m}}\right)\right)\left(\frac{\max\left\{\tilde{\gamma}L_{\hat{F}}P_{\hat{g}_{1}},~\gamma L_{\hat{F}}\right\} + \mu}{\mu}\right),~k\in\mathbb{Z}_{+},~b\in\overline{1, m}.
    \end{aligned}
\end{equation*}
В данном случае учтено неравенство \eqref{eq:bounding_approx_x_variance}, так как $x_{k + 1}$ вычисляется неточно и напрямую лемма \ref{lm:aux_bounded_variation} не применима. Теперь докажем теорему для другой стратегии подбора $\varepsilon_{k}$. По аналогии с доказательством теоремы \ref{th:5} используем определение $\hat{\psi}_{x_{k}, L_{k}, \hat{g}_{1}(x_{k}, B_{k})}(\cdot, B_{k})$ \eqref{eq:aux_upper_model_formula}:
\begin{equation}\label{eq:th6_iter_diff}
    \begin{aligned}
        &\hat{g}_{1}(x_{k}, B_{k}) - \hat{g}_{1}(x_{k + 1}, B_{k})\geq\hat{g}_{1}(x_{k}, B_{k}) - \hat{\psi}_{x_{k}, L_{k}, \hat{g}_{1}(x_{k}, B_{k})}(x_{k + 1}, B_{k}) = \hat{g}_{1}(x_{k}, B_{k}) -\\
        &- \hat{\psi}_{x_{k}, L_{k}, \hat{g}_{1}(x_{k}, B_{k})}(\hat{T}_{L_{k}, \hat{g}_{1}(x_{k}, B_{k})}(x_{k}, B_{k}), B_{k}) + \left(\hat{\psi}_{x_{k}, L_{k}, \hat{g}_{1}(x_{k}, B_{k})}(\hat{T}_{L_{k}, \hat{g}_{1}(x_{k}, B_{k})}(x_{k}, B_{k}), B_{k}) -\right.\\
        &\left.- \hat{\psi}_{x_{k}, L_{k}, \hat{g}_{1}(x_{k}, B_{k})}(x_{k + 1}, B_{k})\right)\geq\hat{g}_{1}(x_{k}, B_{k}) - \hat{\psi}_{x_{k}, L_{k}, \hat{g}_{1}(x_{k}, B_{k})}(\hat{T}_{L_{k}, \hat{g}_{1}(x_{k}, B_{k})}(x_{k}, B_{k}), B_{k}) - \varepsilon_{k} =\\
        &= \left\{\text{согласно \eqref{eq:aux_upper_model_formula} с}~\eta_{k} = 1,~\tau_{k} = \hat{g}_{1}(x_{k}, B_{k})\right\} =\\
        &=\frac{1}{2\hat{g}_{1}(x_{k}, B_{k})}\left\langle\left(\hat{G}^{'}(x_{k}, B_{k})^{*}\hat{G}^{'}(x_{k}, B_{k}) + \hat{g}_{1}(x_{k}, B_{k})L_{k}I_{n}\right)^{-1}\hat{G}^{'}(x_{k}, B_{k})^{*}\hat{G}(x_{k}, B_{k}),\right.\\
        &\left.\hat{G}^{'}(x_{k}, B_{k})^{*}\hat{G}(x_{k}, B_{k})\right\rangle - \varepsilon_{k}=\\
        &=\frac{1}{2\hat{g}_{1}(x_{k}, B_{k})}\left\langle\hat{G}^{'}(x_{k}, B_{k})\left(\hat{G}^{'}(x_{k}, B_{k})^{*}\hat{G}^{'}(x_{k}, B_{k})
        + \hat{g}_{1}(x_{k}, B_{k})L_{k}I_{n}\right)^{-1}\hat{G}^{'}(x_{k}, B_{k})^{*}\hat{G}(x_{k}, B_{k}),\right.\\
        &\left.\hat{G}(x_{k}, B_{k})\right\rangle - \varepsilon_{k}\geq\left\{\text{лемма \ref{lm:aux_matrix_order}},~\hat{g}_{1}(x_{k}, B_{k}) = \left\|\hat{G}(x_{k}, B_{k})\right\|\right\}\geq\frac{\hat{g}_{1}(x_{k}, B_{k})\mu}{2\left(L_{k}\hat{g}_{1}(x_{k}, B_{k}) + \mu\right)} - \varepsilon_{k}\geq\\
        &\geq\frac{(1 - \delta)\hat{g}_{1}(x_{k}, B_{k})\mu}{2\left(L_{k}\hat{g}_{1}(x_{k}, B_{k}) + \mu\right)}\geq\left\{\hat{g}_{1}(x_{k}, L_{k})L_{k}\leq \max\left\{\tilde{\gamma}L_{\hat{F}}P_{\hat{g}_{1}},~\gamma L_{\hat{F}}\right\}\right\}\geq\frac{(1 - \delta)\hat{g}_{1}(x_{k}, B_{k})\mu}{2\left(\max\left\{\tilde{\gamma}L_{\hat{F}}P_{\hat{g}_{1}},~\gamma L_{\hat{F}}\right\} + \mu\right)}\Rightarrow\\
        &\Rightarrow\hat{g}_{2}(x_{k}, B_{k}) - \hat{g}_{2}(x_{k + 1}, B_{k})\geq\hat{g}_{2}(x_{k}, B_{k}) - \hat{g}_{1}(x_{k}, B_{k})\hat{g}_{1}(x_{k + 1}, B_{k})\geq\frac{(1 - \delta)\hat{g}_{2}(x_{k}, B_{k})\mu}{2\left(\max\left\{\tilde{\gamma}L_{\hat{F}}P_{\hat{g}_{1}},~\gamma L_{\hat{F}}\right\} + \mu\right)}\Rightarrow\\
        &\Rightarrow\hat{g}_{2}(x_{k + 1}, B_{k})\leq\hat{g}_{2}(x_{k}, B_{k})\left(1 - \frac{(1 - \delta)\mu}{2\left(\max\left\{\tilde{\gamma}L_{\hat{F}}P_{\hat{g}_{1}},~\gamma L_{\hat{F}}\right\} + \mu\right)}\right).
    \end{aligned}
\end{equation}
В \eqref{eq:th6_iter_diff} получено неравенство, отличающееся от неравенства в \eqref{eq:th4_prelim_inequality} заменой $\eta(2 - \eta)$ на $(1 - \delta)$ и $\gamma\hat{g}_{1}(x_{k}, B_{k})L_{\hat{F}}$ на $\max\left\{\tilde{\gamma}L_{\hat{F}}P_{\hat{g}_{1}},~\gamma L_{\hat{F}}\right\}$. Поэтому аналогичным образом для выражения \eqref{eq:th6_iter_diff} применима цепочка рассуждений из теоремы \ref{th:4} с использованием \eqref{eq:bounding_approx_x_delta_variance}, применённая для неравенства из \eqref{eq:th4_prelim_inequality}, приводящая к искомой оценке \eqref{eq:stoch_approx_lin_conv_1}:
\begin{equation*}
    \begin{cases}
        \begin{aligned}
            \mathbb{E}\left[\left\|\nabla\hat{f}_{2}(x_{k})\right\|^{2}\right]&\leq4M_{\hat{G}}^{2}\tilde{\Delta}_{k,b};\\
            \mathbb{E}\left[\hat{f}_{2}(x_{k})\right]&\leq\hat{f}_{2}^{*} + \tilde{\Delta}_{k,b}.
        \end{aligned}
    \end{cases}
\end{equation*}
\end{proof}

Теорема \ref{th:gen_stoch_sublin_conv_main} устанавливает общие условия сходимости к окрестности стационарной точки относительно среднего минимума квадрата нормы проксимального градиента при оптимизации неограниченных функционалов $\hat{f}_{1}$.

\begin{re:theorem}\label{th:gen_stoch_sublin_conv}
Предположим выполнение условий \ref{as:jacob_smoothness} и \ref{as:bounded_variance_growth}. Рассмотрим метод Гаусса--Ньютона, реализованный по схеме \ref{alg:gen_stoch_unbounded_gnm} со стратегией обновления приближения решения \eqref{eq:stoch_approx_general_update_rule}, $\tau_{k} = \hat{g}_{1}(x_{k}, B_{k})$, $k\in\mathbb{Z}_{+}$. Тогда выполнена следующая оценка:
\begin{equation*}
    \begin{aligned}
        \left(\gamma L_{\hat{F}}\right)^{2}&\left(\frac{L}{2} - c_{3}\sqrt{\frac{1}{b} - \frac{1}{m}}\right)^{-1}\left(\frac{\mathbb{E}\left[\hat{f}_{2}(x_{0})\right]}{k} + \frac{1}{k}\sum\limits_{i = 0}^{k - 1}\mathbb{E}\left[\varepsilon_{i}\hat{g}_{1}(x_{i}, B_{i})\right]\left(1 + c_{2}\sqrt{\frac{1}{b} - \frac{1}{m}}\right) + c_{1}\sqrt{\frac{1}{b} - \frac{1}{m}}\right)\geq\\
        &\geq\mathbb{E}\left[\min\limits_{i\in\overline{0, k - 1}}\left\{\left\|\gamma L_{\hat{F}}\left(\hat{T}_{\max\left\{\tilde{\gamma}L_{\hat{F}},\frac{\gamma L_{\hat{F}}}{\hat{g}_{1}(x_{i}, B_{i})}\right\}, \hat{g}_{1}(x_{i}, B_{i})}(x_{i}, B_{i}) - x_{i}\right)\right\|^{2}\right\}\right],\\[10pt]
    \end{aligned}
\end{equation*}
\begin{equation*}
    \begin{aligned}
        b&\in\overline{\min\left\{m, \left\lceil\frac{4c_{3}^{2}m}{L^{2}m + 4c_{3}^{2}}\right\rceil + 1\right\},~m},~\tilde{\gamma}L_{\hat{F}}\geq L > 2c_{3}\sqrt{\frac{1}{b} - \frac{1}{m}},~\tilde{\gamma}>\max\left\{1,~\frac{2c_{3}}{L_{\hat{F}}}\sqrt{\frac{1}{b} - \frac{1}{m}}\right\},~\gamma\geq\tilde{\gamma}.
    \end{aligned}
\end{equation*}
\end{re:theorem}
\begin{proof}
Лемма \ref{lm:aux_upper_model} задаёт в условиях данной теоремы локальную модель $\hat{\psi}_{x_{k}, L_{k}, \hat{g}_{1}(x_{k}, B_{k})}$, для которой применима лемма \ref{lm:aux_det_local_decrease_1} с заменой $\hat{f}_{1} := \hat{g}_{1}$ в рамках одного батча, что позволяет применить рассуждения из теоремы \ref{th:DetSublinConv} и получить оценки на убывание $\hat{g}_{1}$ на $k$--ой итерации:
\begin{equation*}
    \begin{aligned}
        \hat{g}_{1}(x_{k}, B_{k}) + \varepsilon_{k} - \hat{g}_{1}(x_{k + 1}, B_{k})\geq\frac{L_{k}}{2}\left\|\hat{T}_{L_{k}, \hat{g}_{1}(x_{k}, B_{k})}(x_{k}, B_{k}) - x_{k}\right\|^{2}.
    \end{aligned}
\end{equation*}
Приведём правую часть к форме проксимального градиента и оценим снизу правую часть неравенства, пользуясь монотонным убыванием по $L_{k}$ нормы в правой части (следствие \ref{lm:cor_aux_det_local_decrease_1} в случае стохастической локальной модели) и тем, что $L_{k}\geq\frac{L}{\hat{g}_{1}(x_{k}, B_{k})}$:
\begin{equation*}
    \begin{aligned}
        &\hat{g}_{1}(x_{k}, B_{k}) + \varepsilon_{k} - \hat{g}_{1}(x_{k + 1}, B_{k})\geq\frac{L_{k}}{2}\left\|\hat{T}_{L_{k}, \hat{g}_{1}(x_{k}, B_{k})}(x_{k}, B_{k}) - x_{k}\right\|^{2}\geq\\
        &\geq\frac{L}{2\hat{g}_{1}(x_{k}, B_{k})\left(\gamma L_{\hat{F}}\right)^{2}}\left\|\gamma L_{\hat{F}}\left(\hat{T}_{\max\left\{\tilde{\gamma}L_{\hat{F}},\frac{\gamma L_{\hat{F}}}{\hat{g}_{1}(x_{k}, B_{k})}\right\}, \hat{g}_{1}(x_{k}, B_{k})}(x_{k}, B_{k}) - x_{k}\right)\right\|^{2}.
    \end{aligned}
\end{equation*}
Домножим неравенство на $\hat{g}_{1}(x_{k}, B_{k})$:
\begin{equation*}
    \begin{aligned}
        \hat{g}_{2}(x_{k}, B_{k}) &+ \varepsilon_{k}\hat{g}_{1}(x_{k}, B_{k}) - \hat{g}_{2}(x_{k + 1}, B_{k})\geq\hat{g}_{2}(x_{k}, B_{k}) + \varepsilon_{k}\hat{g}_{1}(x_{k}, B_{k}) - \hat{g}_{1}(x_{k}, B_{k})\hat{g}_{1}(x_{k + 1}, B_{k})\geq\\
        &\geq\frac{L}{2\left(\gamma L_{\hat{F}}\right)^{2}}\left\|\gamma L_{\hat{F}}\left(\hat{T}_{\max\left\{\tilde{\gamma}L_{\hat{F}},\frac{\gamma L_{\hat{F}}}{\hat{g}_{1}(x_{k}, B_{k})}\right\}, \hat{g}_{1}(x_{k}, B_{k})}(x_{k}, B_{k}) - x_{k}\right)\right\|^{2}.
    \end{aligned}
\end{equation*}
Сложим неравенства для итераций $0,~\dots, k - 1$ и усредним оператором математического ожидания, увеличив левую часть на $\left|\hat{f}_{2}(x_{k}) - \hat{g}_{2}(x_{k}, B_{k - 1})\right|$:
\begin{equation*}
    \begin{aligned}
        &\mathbb{E}\left[\hat{f}_{2}(x_{0})\right] + \sum\limits_{i = 0}^{k - 1}\mathbb{E}\left[\varepsilon_{i}\hat{g}_{1}(x_{i}, B_{i})\right] + \sum\limits_{i = 1}^{k}\mathbb{E}\left[\left|\hat{f}_{2}(x_{i}) - \hat{g}_{2}(x_{i}, B_{i - 1})\right|\right]\geq\mathbb{E}\left[\hat{f}_{2}(x_{0})\right] + \sum\limits_{i = 0}^{k - 1}\mathbb{E}\left[\varepsilon_{i}\hat{g}_{1}(x_{i}, B_{i})\right] +\\
        &+ \sum\limits_{i = 1}^{k - 1}\mathbb{E}\left[\hat{f}_{2}(x_{i}) - \hat{g}_{2}(x_{i}, B_{i - 1})\right] - \mathbb{E}\left[\hat{g}_{2}(x_{k}, B_{k - 1})\right]\geq\\
        &\geq\frac{L}{2\left(\gamma L_{\hat{F}}\right)^{2}}\sum\limits_{i = 0}^{k - 1}\mathbb{E}\left[\left\|\gamma L_{\hat{F}}\left(\hat{T}_{\max\left\{\tilde{\gamma}L_{\hat{F}},\frac{\gamma L_{\hat{F}}}{\hat{g}_{1}(x_{i}, B_{i})}\right\}, \hat{g}_{1}(x_{i}, B_{i})}(x_{i}, B_{i}) - x_{i}\right)\right\|^{2}\right].
    \end{aligned}
\end{equation*}
Применим предположение \ref{as:bounded_variance_growth} к неравенствам выше:
\begin{equation*}
    \begin{aligned}
        &\mathbb{E}\left[\hat{f}_{2}(x_{0})\right] + \sum\limits_{i = 0}^{k - 1}\mathbb{E}\left[\varepsilon_{i}\hat{g}_{1}(x_{i}, B_{i})\right] + \sum\limits_{i = 0}^{k - 1}\left(c_{1} + c_{2}\mathbb{E}\left[\hat{g}_{1}(x_{i}, B_{i})\varepsilon_{i}\right] +\right.
    \end{aligned}
\end{equation*}
\begin{equation*}
    \begin{aligned}
        &\left.+ c_{3}\mathbb{E}\left[\left\|\hat{T}_{\max\left\{\tilde{\gamma}L_{\hat{F}},\frac{\gamma L_{\hat{F}}}{\hat{g}_{1}(x_{i}, B_{i})}\right\}, \hat{g}_{1}(x_{i}, B_{i})}(x_{i}, B_{i}) - x_{i}\right\|^{2}\right]\right)\sqrt{\frac{1}{b} - \frac{1}{m}}\geq\\
        &\geq\frac{L}{2\left(\gamma L_{\hat{F}}\right)^{2}}\sum\limits_{i = 0}^{k - 1}\mathbb{E}\left[\left\|\gamma L_{\hat{F}}\left(\hat{T}_{\max\left\{\tilde{\gamma}L_{\hat{F}},\frac{\gamma L_{\hat{F}}}{\hat{g}_{1}(x_{i}, B_{i})}\right\}, \hat{g}_{1}(x_{i}, B_{i})}(x_{i}, B_{i}) - x_{i}\right)\right\|^{2}\right].
    \end{aligned}
\end{equation*}
Приведём подобные слагаемые:
\begin{equation*}
    \begin{aligned}
        \mathbb{E}&\left[\hat{f}_{2}(x_{0})\right] + \sum\limits_{i = 0}^{k - 1}\mathbb{E}\left[\varepsilon_{i}\hat{g}_{1}(x_{i}, B_{i})\right]\left(1 + c_{2}\sqrt{\frac{1}{b} - \frac{1}{m}}\right) + c_{1}k\sqrt{\frac{1}{b} - \frac{1}{m}}\geq
    \end{aligned}
\end{equation*}
\begin{equation*}
    \begin{aligned}
        &\geq\sum\limits_{i = 0}^{k - 1}\mathbb{E}\left[\frac{1}{\left(\gamma L_{\hat{F}}\right)^{2}}\left(\frac{L}{2} - c_{3}\sqrt{\frac{1}{b} - \frac{1}{m}}\right)\left\|\gamma L_{\hat{F}}\left(\hat{T}_{\max\left\{\tilde{\gamma}L_{\hat{F}},\frac{\gamma L_{\hat{F}}}{\hat{g}_{1}(x_{i}, B_{i})}\right\}, \hat{g}_{1}(x_{i}, B_{i})}(x_{i}, B_{i}) - x_{i}\right)\right\|^{2}\right]\geq\\
        &\geq \frac{k}{\left(\gamma L_{\hat{F}}\right)^{2}}\left(\frac{L}{2} - c_{3}\sqrt{\frac{1}{b} - \frac{1}{m}}\right)\mathbb{E}\left[\min\limits_{i\in\overline{0, k - 1}}\left\{\left\|\gamma L_{\hat{F}}\left(\hat{T}_{\max\left\{\tilde{\gamma}L_{\hat{F}},\frac{\gamma L_{\hat{F}}}{\hat{g}_{1}(x_{i}, B_{i})}\right\}, \hat{g}_{1}(x_{i}, B_{i})}(x_{i}, B_{i}) - x_{i}\right)\right\|^{2}\right\}\right].
    \end{aligned}
\end{equation*}
Домножением на $\frac{\left(\gamma L_{\hat{F}}\right)^{2}}{k}\left(\frac{L}{2} - c_{3}\sqrt{\frac{1}{b} - \frac{1}{m}}\right)^{-1}$ получаем искомую оценку сходимости:
\begin{equation*}
    \begin{aligned}
        \left(\gamma L_{\hat{F}}\right)^{2}\left(\frac{L}{2} - c_{3}\sqrt{\frac{1}{b} - \frac{1}{m}}\right)^{-1}&\left(\frac{\mathbb{E}\left[\hat{f}_{2}(x_{0})\right]}{k} + \frac{1}{k}\sum\limits_{i = 0}^{k - 1}\mathbb{E}\left[\varepsilon_{i}\hat{g}_{1}(x_{i}, B_{i})\right]\left(1 + c_{2}\sqrt{\frac{1}{b} - \frac{1}{m}}\right) + c_{1}\sqrt{\frac{1}{b} - \frac{1}{m}}\right)\geq\\
        &\geq\mathbb{E}\left[\min\limits_{i\in\overline{0, k - 1}}\left\{\left\|\gamma L_{\hat{F}}\left(\hat{T}_{\max\left\{\tilde{\gamma}L_{\hat{F}},\frac{\gamma L_{\hat{F}}}{\hat{g}_{1}(x_{i}, B_{i})}\right\}, \hat{g}_{1}(x_{i}, B_{i})}(x_{i}, B_{i}) - x_{i}\right)\right\|^{2}\right\}\right].
    \end{aligned}
\end{equation*}
Необходимое ограничение $\frac{L}{2} - c_{3}\sqrt{\frac{1}{b} - \frac{1}{m}} > 0$ задаёт минимальный размер батча и нижнюю оценку на $L$:
\begin{equation*}
    \begin{cases}
        b\in\overline{\min\left\{m, \left\lceil\frac{4c_{3}^{2}m}{L^{2}m + 4c_{3}^{2}}\right\rceil + 1\right\},~m};\\[5pt]
        \tilde{\gamma}L_{\hat{F}}\geq L > 2c_{3}\sqrt{\frac{1}{b} - \frac{1}{m}}\Rightarrow\tilde{\gamma}>\max\left\{1,~\frac{2c_{3}}{L_{\hat{F}}}\sqrt{\frac{1}{b} - \frac{1}{m}}\right\}.
    \end{cases}
\end{equation*}
\end{proof}
\begin{re:th:corollary}
В случае  $\varepsilon_{k} = \frac{\delta_{k - 1} - \delta_{k}}{\hat{g}_{1}(x_{k}, B_{k})},~k\in\mathbb{N}$ и $\varepsilon_{0} = \frac{\delta_{0}}{\hat{g}_{1}(x_{0}, B_{0})}$ с $0 < \delta_{k} < \delta_{k - 1},~\lim\limits_{k\rightarrow\infty}\delta_{k} = 0$ оценка сходимости следующая:
\begin{equation*}
    \begin{aligned}
        \left(\gamma L_{\hat{F}}\right)^{2}\left(\frac{L}{2} - c_{3}\sqrt{\frac{1}{b} - \frac{1}{m}}\right)^{-1}&\left(\frac{\mathbb{E}\left[\hat{f}_{2}(x_{0})\right]}{k} + \frac{2\delta_{0} - \delta_{k - 1}}{k}\left(1 + c_{2}\sqrt{\frac{1}{b} - \frac{1}{m}}\right) + c_{1}\sqrt{\frac{1}{b} - \frac{1}{m}}\right)\geq\\
        &\geq\mathbb{E}\left[\min\limits_{i\in\overline{0, k - 1}}\left\{\left\|\gamma L_{\hat{F}}\left(\hat{T}_{\max\left\{\tilde{\gamma}L_{\hat{F}},\frac{\gamma L_{\hat{F}}}{\hat{g}_{1}(x_{i}, B_{i})}\right\}, \hat{g}_{1}(x_{i}, B_{i})}(x_{i}, B_{i}) - x_{i}\right)\right\|^{2}\right\}\right].
    \end{aligned}
\end{equation*}
\end{re:th:corollary}
\begin{re:th:corollary}\label{th:gen_stoch_sublin_conv_cor2}
Если в предположении \ref{as:bounded_variance_growth} $c_{1} = 0$, то оценка сходимости из предыдущего следствия указывает на гипотетическую возможность решить задачу с любой наперёд заданной точностью, лишь увеличивая количество итераций, при этом малый размер батча можно скомпенсировать увеличением $L\geq L_{\hat{F}}$ и установкой $L := \tilde{\gamma} L_{\hat{F}},~\tilde{\gamma}\geq1$.
\end{re:th:corollary}

Теорема \ref{th:gen_stoch_lin_conv_main} использует проксимальное условие Поляка--Лоясиевича (предположене \ref{as:prox_PL_condition}) для вывода общих условий линейной сходимости к области решения задачи \eqref{eq:main_opt_problem} в среднем.

\begin{re:theorem}\label{th:gen_stoch_lin_conv}
Предположим выполнение условий \ref{as:jacob_smoothness}, \ref{as:bounded_variance_growth} и \ref{as:prox_PL_condition}. Рассмотрим метод Гаусса--Ньютона, реализованный по схеме \ref{alg:gen_stoch_unbounded_gnm} со стратегией обновления приближения решения \eqref{eq:stoch_approx_general_update_rule}, $\tau_{k} = \hat{g}_{1}(x_{k}, B_{k})$, $k\in\mathbb{Z}_{+}$. Тогда выполнена следующая оценка
\begin{equation*}
    \begin{aligned}
        \frac{\left(\gamma L_{\hat{F}}\right)^{2}}{\nu}&\left(\frac{L}{2} - c_{3}\sqrt{\frac{1}{b} - \frac{1}{m}}\right)^{-1}\left(c_{1}\sqrt{\frac{1}{b} - \frac{1}{m}} + \left(1 + c_{2}\sqrt{\frac{1}{b} - \frac{1}{m}}\right)\mathbb{E}\left[\max\limits_{i\in\overline{0, k - 1}}\left\{\varepsilon_{i}\hat{g}_{1}(x_{i}, B_{i})\right\}\right]\right) +\\
        &+ \exp\left(-\frac{k\nu}{\left(\gamma L_{\hat{F}}\right)^{2}}\left(\frac{L}{2} - c_{3}\sqrt{\frac{1}{b} - \frac{1}{m}}\right)\right)\mathbb{E}\left[\hat{f}_{2}(x_{0}) - \hat{f}_{2}^{*}\right]\geq\mathbb{E}\left[\hat{f}_{2}(x_{k}) - \hat{f}_{2}^{*}\right],\\[10pt]
        b&\in\overline{\min\left\{m,~\left\lceil\frac{4c_{3}^{2}m}{L^{2}m + 4c_{3}^{2}}\right\rceil + 1\right\},~m},~\tilde{\gamma}L_{\hat{F}}\geq L>2c_{3}\sqrt{\frac{1}{b} - \frac{1}{m}},\\
        \gamma&\geq\max\left\{\tilde{\gamma},~\frac{1}{L_{\hat{F}}}\sqrt{\nu\left(\frac{L}{2} - c_{3}\sqrt{\frac{1}{b} - \frac{1}{m}}\right)}\right\},~\tilde{\gamma} > \max\left\{1,~\frac{2c_{3}}{L_{\hat{F}}}\sqrt{\frac{1}{b} - \frac{1}{m}}\right\}.
    \end{aligned}
\end{equation*}
\end{re:theorem}
\begin{proof}
Применяя рассуждения из теоремы \ref{th:gen_stoch_sublin_conv}, получаем оценки на убывание $\hat{g}_{1}$ на $k$--ой итерации:
\begin{equation*}
    \begin{aligned}
        \hat{g}_{1}(x_{k}, B_{k}) + \varepsilon_{k} - \hat{g}_{1}(x_{k + 1}, B_{k})\geq\frac{L_{k}}{2}\left\|\hat{T}_{L_{k}, \hat{g}_{1}(x_{k}, B_{k})}(x_{k}, B_{k}) - x_{k}\right\|^{2}.
    \end{aligned}
\end{equation*}
По аналогии с доказательством теоремы \ref{th:gen_stoch_sublin_conv} приведём правую часть к форме проксимального градиента и оценим её снизу:
\begin{equation*}
    \begin{aligned}
        &\hat{g}_{1}(x_{k}, B_{k}) + \varepsilon_{k} - \hat{g}_{1}(x_{k + 1}, B_{k})\geq\frac{L_{k}}{2}\left\|\hat{T}_{L_{k}, \hat{g}_{1}(x_{k}, B_{k})}(x_{k}, B_{k}) - x_{k}\right\|^{2}\geq\\
        &\geq\frac{L}{2\hat{g}_{1}(x_{k}, B_{k})\left(\gamma L_{\hat{F}}\right)^{2}}\left\|\gamma L_{\hat{F}}\left(\hat{T}_{\max\left\{\tilde{\gamma}L_{\hat{F}},\frac{\gamma L_{\hat{F}}}{\hat{g}_{1}(x_{k}, B_{k})}\right\}, \hat{g}_{1}(x_{k}, B_{k})}(x_{k}, B_{k}) - x_{k}\right)\right\|^{2}.
    \end{aligned}
\end{equation*}
Домножим неравенство на $\hat{g}_{1}(x_{k}, B_{k})$:
\begin{equation*}
    \begin{aligned}
        \hat{g}_{2}(x_{k}, B_{k}) &+ \varepsilon_{k}\hat{g}_{1}(x_{k}, B_{k}) - \hat{g}_{2}(x_{k + 1}, B_{k})\geq\hat{g}_{2}(x_{k}, B_{k}) + \varepsilon_{k}\hat{g}_{1}(x_{k}, B_{k}) - \hat{g}_{1}(x_{k}, B_{k})\hat{g}_{1}(x_{k + 1}, B_{k})\geq\\
        &\geq\frac{L}{2\left(\gamma L_{\hat{F}}\right)^{2}}\left\|\gamma L_{\hat{F}}\left(\hat{T}_{\max\left\{\tilde{\gamma}L_{\hat{F}},\frac{\gamma L_{\hat{F}}}{\hat{g}_{1}(x_{k}, B_{k})}\right\}, \hat{g}_{1}(x_{k}, B_{k})}(x_{k}, B_{k}) - x_{k}\right)\right\|^{2}.
    \end{aligned}
\end{equation*}
Усредним неравенство с помощью оператора математического ожидания, добавив к обеим частям\\$\hat{f}_{2}(x_{k + 1}) - \hat{f}_{2}^{*}$:
\begin{equation*}
    \begin{aligned}
        &\mathbb{E}\left[\hat{f}_{2}(x_{k}) - \hat{f}_{2}^{*}\right] + \mathbb{E}\left[\varepsilon_{k}\hat{g}_{1}(x_{k}, B_{k})\right] + \mathbb{E}\left[\left|\hat{f}_{2}(x_{k + 1}) - \hat{g}_{2}(x_{k + 1}, B_{k})\right|\right]\geq\mathbb{E}\left[\hat{f}_{2}(x_{k}) - \hat{f}_{2}^{*}\right] + \mathbb{E}\left[\varepsilon_{k}\hat{g}_{1}(x_{k}, B_{k})\right] +\\
        &+ \mathbb{E}\left[\hat{f}_{2}(x_{k + 1}) - \hat{g}_{2}(x_{k + 1}, B_{k})\right]\geq\mathbb{E}\left[\hat{f}_{2}(x_{k + 1}) - \hat{f}_{2}^{*}\right] +\\
        &+ \frac{L}{2\left(\gamma L_{\hat{F}}\right)^{2}}\mathbb{E}\left[\left\|\gamma L_{\hat{F}}\left(\hat{T}_{\max\left\{\tilde{\gamma}L_{\hat{F}},\frac{\gamma L_{\hat{F}}}{\hat{g}_{1}(x_{k}, B_{k})}\right\}, \hat{g}_{1}(x_{k}, B_{k})}(x_{k}, B_{k}) - x_{k}\right)\right\|^{2}\right].
    \end{aligned}
\end{equation*}
Применим предположение \ref{as:bounded_variance_growth}:
\begin{equation*}
    \begin{aligned}
        &\left(c_{1} + c_{2}\mathbb{E}\left[\hat{g}_{1}(x_{k}, B_{k})\varepsilon_{k}\right] + c_{3}\mathbb{E}\left[\left\|\gamma L_{\hat{F}}\left(\hat{T}_{\max\left\{\tilde{\gamma}L_{\hat{F}},\frac{\gamma L_{\hat{F}}}{\hat{g}_{1}(x_{k}, B_{k})}\right\}, \hat{g}_{1}(x_{k}, B_{k})}(x_{k}, B_{k}) - x_{k}\right)\right\|^{2}\right]\right)\sqrt{\frac{1}{b} - \frac{1}{m}} +\\
        &+ \mathbb{E}\left[\varepsilon_{k}\hat{g}_{1}(x_{k}, B_{k})\right] + \mathbb{E}\left[\hat{f}_{2}(x_{k}) - \hat{f}_{2}^{*}\right]\geq\mathbb{E}\left[\hat{f}_{2}(x_{k}) - \hat{f}_{2}^{*}\right] + \mathbb{E}\left[\varepsilon_{k}\hat{g}_{1}(x_{k}, B_{k})\right] + \mathbb{E}\left[\left|\hat{f}_{2}(x_{k + 1}) - \hat{g}_{2}(x_{k + 1}, B_{k})\right|\right]\geq\\
        &\geq\mathbb{E}\left[\hat{f}_{2}(x_{k + 1}) - \hat{f}_{2}^{*}\right] + \frac{L}{2\left(\gamma L_{\hat{F}}\right)^{2}}\mathbb{E}\left[\left\|\gamma L_{\hat{F}}\left(\hat{T}_{\max\left\{\tilde{\gamma}L_{\hat{F}},\frac{\gamma L_{\hat{F}}}{\hat{g}_{1}(x_{k}, B_{k})}\right\}, \hat{g}_{1}(x_{k}, B_{k})}(x_{k}, B_{k}) - x_{k}\right)\right\|^{2}\right]\Rightarrow\\
        &\Rightarrow c_{1}\sqrt{\frac{1}{b} - \frac{1}{m}} + \left(1 + c_{2}\sqrt{\frac{1}{b} - \frac{1}{m}}\right)\mathbb{E}\left[\varepsilon_{k}\hat{g}_{1}(x_{k}, B_{k})\right] + \mathbb{E}\left[\hat{f}_{2}(x_{k}) - \hat{f}_{2}^{*}\right]\geq\mathbb{E}\left[\hat{f}_{2}(x_{k + 1}) - \hat{f}_{2}^{*}\right] +\\
        &+ \frac{1}{\left(\gamma L_{\hat{F}}\right)^{2}}\left(\frac{L}{2} - c_{3}\sqrt{\frac{1}{b} - \frac{1}{m}}\right)\mathbb{E}\left[\left\|\gamma L_{\hat{F}}\left(\hat{T}_{\max\left\{\tilde{\gamma}L_{\hat{F}},\frac{\gamma L_{\hat{F}}}{\hat{g}_{1}(x_{k}, B_{k})}\right\}, \hat{g}_{1}(x_{k}, B_{k})}(x_{k}, B_{k}) - x_{k}\right)\right\|^{2}\right].
    \end{aligned}
\end{equation*}
Применим предположение \ref{as:prox_PL_condition} к неравенствам выше:
\begin{equation*}
    \begin{aligned}
        &c_{1}\sqrt{\frac{1}{b} - \frac{1}{m}} + \left(1 + c_{2}\sqrt{\frac{1}{b} - \frac{1}{m}}\right)\mathbb{E}\left[\varepsilon_{k}\hat{g}_{1}(x_{k}, B_{k})\right] + \mathbb{E}\left[\hat{f}_{2}(x_{k}) - \hat{f}_{2}^{*}\right]\geq\\
        &\geq\mathbb{E}\left[\hat{f}_{2}(x_{k + 1}) - \hat{f}_{2}^{*}\right] + \frac{\nu}{\left(\gamma L_{\hat{F}}\right)^{2}}\left(\frac{L}{2} - c_{3}\sqrt{\frac{1}{b} - \frac{1}{m}}\right)\mathbb{E}\left[\hat{f}_{2}(x_{k}) - \hat{f}_{2}^{*}\right].
    \end{aligned}
\end{equation*}
Приведём подобные слагаемые, использовав свойство
$$\varepsilon_{j}\hat{g}_{1}(x_{j}, B_{j})\leq\max\limits_{i\in\overline{0, k}}\left\{\varepsilon_{i}\hat{g}_{1}(x_{i}, B_{i})\right\},~j\in\overline{0, k}:$$
\begin{equation*}
    \begin{aligned}
        &c_{1}\sqrt{\frac{1}{b} - \frac{1}{m}} + \left(1 + c_{2}\sqrt{\frac{1}{b} - \frac{1}{m}}\right)\mathbb{E}\left[\max\limits_{i\in\overline{0, k}}\left\{\varepsilon_{i}\hat{g}_{1}(x_{i}, B_{i})\right\}\right] +\\
        &+ \left(1 - \frac{\nu}{\left(\gamma L_{\hat{F}}\right)^{2}}\left(\frac{L}{2} - c_{3}\sqrt{\frac{1}{b} - \frac{1}{m}}\right)\right)\mathbb{E}\left[\hat{f}_{2}(x_{k}) - \hat{f}_{2}^{*}\right]\geq\mathbb{E}\left[\hat{f}_{2}(x_{k + 1}) - \hat{f}_{2}^{*}\right].
    \end{aligned}
\end{equation*}
Раскрывая рекуррентное соотношение по аналогии с \eqref{eq:th4_geom_sum} и применяя неравенство $1 + t\leq\exp(t)$, получаем искомую оценку:
\begin{equation*}
    \begin{aligned}
        \frac{\left(\gamma L_{\hat{F}}\right)^{2}}{\nu}\left(\frac{L}{2} - c_{3}\sqrt{\frac{1}{b} - \frac{1}{m}}\right)^{-1}&\left(c_{1}\sqrt{\frac{1}{b} - \frac{1}{m}} + \left(1 + c_{2}\sqrt{\frac{1}{b} - \frac{1}{m}}\right)\mathbb{E}\left[\max\limits_{i\in\overline{0, k - 1}}\left\{\varepsilon_{i}\hat{g}_{1}(x_{i}, B_{i})\right\}\right]\right) +\\
        &+ \exp\left(-\frac{k\nu}{\left(\gamma L_{\hat{F}}\right)^{2}}\left(\frac{L}{2} - c_{3}\sqrt{\frac{1}{b} - \frac{1}{m}}\right)\right)\mathbb{E}\left[\hat{f}_{2}(x_{0}) - \hat{f}_{2}^{*}\right]\geq\mathbb{E}\left[\hat{f}_{2}(x_{k}) - \hat{f}_{2}^{*}\right].
    \end{aligned}
\end{equation*}
Для наличия ожидаемой сходимости необходимо выполнение неравенства $\frac{L}{2} - c_{3}\sqrt{\frac{1}{b} - \frac{1}{m}} > 0$, которое задаёт допустимый размер батча и нижнюю границу на $\tilde{\gamma}$:
\begin{equation*}
    \begin{cases}
        &b\in\overline{\min\left\{m,~\left\lceil\frac{4c_{3}^{2}m}{L^{2}m + 4c_{3}^{2}}\right\rceil + 1\right\},~m};\\[5pt]
        &\tilde{\gamma} > \max\left\{1,~\frac{2c_{3}}{L_{\hat{F}}}\sqrt{\frac{1}{b} - \frac{1}{m}}\right\}.
    \end{cases}
\end{equation*}
Данные неравенства накладывают ограничения и на $L$:
$$\tilde{\gamma}L_{\hat{F}}\geq L>2c_{3}\sqrt{\frac{1}{b} - \frac{1}{m}}.$$
Также для наличия линейной сходимости важно соблюсти следующее условие:
\begin{equation*}
    \begin{aligned}
        &0<\frac{\nu}{\left(\gamma L_{\hat{F}}\right)^{2}}\left(\frac{L}{2} - c_{3}\sqrt{\frac{1}{b} - \frac{1}{m}}\right)\leq1\Rightarrow\gamma\geq\max\left\{\tilde{\gamma},~\frac{1}{L_{\hat{F}}}\sqrt{\nu\left(\frac{L}{2} - c_{3}\sqrt{\frac{1}{b} - \frac{1}{m}}\right)}\right\}.
    \end{aligned}
\end{equation*}
\end{proof}
\begin{re:th:corollary}
В случае  $\varepsilon_{k} = \frac{q\delta_{k - 1} - \delta_{k}}{\hat{g}_{1}(x_{k}, B_{k})},~k\in\mathbb{N}$ и $\varepsilon_{0} = \frac{\delta_{0}}{\hat{g}_{1}(x_{0}, B_{0})}$ с $0 < \delta_{k} < q\delta_{k - 1},~\lim\limits_{k\rightarrow+\infty}\delta_{k} = 0$,$$q = \exp\left(-\frac{\nu}{\left(\gamma L_{\hat{F}}\right)^{2}}\left(\frac{L}{2} - c_{3}\sqrt{\frac{1}{b} - \frac{1}{m}}\right)\right)$$ оценка сходимости следующая:
\begin{equation*}
    \begin{aligned}
        \frac{\left(\gamma L_{\hat{F}}\right)^{2}}{\nu}\left(\frac{L}{2} - c_{3}\sqrt{\frac{1}{b} - \frac{1}{m}}\right)^{-1}&\left(c_{1}\sqrt{\frac{1}{b} - \frac{1}{m}} + \right.\\
        &\left.+ 2\delta_{0}\left(1 + c_{2}\sqrt{\frac{1}{b} - \frac{1}{m}}\right)\exp\left(-\frac{\left(k - 1\right)\nu}{\left(\gamma L_{\hat{F}}\right)^{2}}\left(\frac{L}{2} - c_{3}\sqrt{\frac{1}{b} - \frac{1}{m}}\right)\right)\right) +\\
        &+ \exp\left(-\frac{k\nu}{\left(\gamma L_{\hat{F}}\right)^{2}}\left(\frac{L}{2} - c_{3}\sqrt{\frac{1}{b} - \frac{1}{m}}\right)\right)\mathbb{E}\left[\hat{f}_{2}(x_{0}) - \hat{f}_{2}^{*}\right]\geq\mathbb{E}\left[\hat{f}_{2}(x_{k}) - \hat{f}_{2}^{*}\right].
    \end{aligned}
\end{equation*}
\end{re:th:corollary}
\begin{re:th:corollary}\label{th:gen_stoch_lin_conv_cor2}
Если в предположении \ref{as:bounded_variance_growth} $c_{1} = 0$, то оценка сходимости из предыдущего следствия указывает на гипотетическую возможность решить задачу с любой наперёд заданной точностью, лишь увеличивая количество итераций, при этом малый размер батча можно скомпенсировать увеличением $L\geq L_{\hat{F}}$ и установкой $L := \tilde{\gamma} L_{\hat{F}},~\tilde{\gamma}\geq 1$.
\end{re:th:corollary}

\subsubsection*{Условие слабого роста в методе Гаусса--Ньютона}

Теорема \ref{th:weak_growth_condition_main} использует наличие условия слабого роста для ограниченного оптимизируемого функционала и предположение \ref{as:5} для вывода совместности системы нелинейных уравнений \eqref{eq:smooth_system} и линейной сходимости к её решению с произвольно заданной точностью и с произвольным размером батча при использовании правила \eqref{eq:double_stoch_direct_update_rule}.

\begin{re:theorem}\label{th:weak_growth_condition}
Пусть выполнены предположения \ref{as:2}, \ref{as:3}, \ref{as:5}, \ref{as:jacob_smoothness}. Рассмотрим метод Гаусса--Ньютона со схемой реализации \ref{alg:gen_double_stoch_gnm}, в котором последовательность $\{x_{k}\}_{k\in\mathbb{Z}_{+}}$ вычисляется по правилу \eqref{eq:double_stoch_direct_update_rule} с $\tilde{\tau}_{k}\geq\tilde{\tau} > 0$,\\$L_{k} \geq L > 0$ и $l_{k} \equiv l_{\hat{f}_{2}} = 2\left(L_{\hat{F}}P_{\hat{f}_{1}} + M_{\hat{F}}^{2}\right)$. Тогда для последовательности
\begin{equation*}
    \begin{aligned}
        &\eta_{k} = \frac{\mu\left(\tilde{\tau}_{k}L_{k}\right)^{2}}{\left(M_{\hat{G}}^{2} + \tilde{\tau}_{k}L_{k}\right)\left(L_{\hat{F}}P_{\hat{f}_{1}} + M_{\hat{F}}^{2}\right)M_{\hat{G}}^{2}},~k\in\mathbb{Z}_{+}
    \end{aligned}
\end{equation*}
    верна следующая оценка
    \begin{equation*}
        \begin{aligned}
            &\mathbb{E}\left[\hat{f}_{2}(x_{k})\right]\leq\mathbb{E}\left[\hat{f}_{2}(x_{0})\right]\exp\left(-\frac{k}{\left(L_{\hat{F}}P_{\hat{f}_{1}} + M_{\hat{F}}^{2}\right)M_{\hat{G}}^{2}}\left(\frac{\mu\tilde{\tau}L}{M_{\hat{G}}^{2} + \tilde{\tau}L}\right)^{2}\right),~k\in\mathbb{Z}_{+}.
        \end{aligned}
    \end{equation*}
    В случае $\eta_{k} = 1,~k\in\mathbb{Z}_{+}$ оценка сходимости при использовании правила \eqref{eq:double_stoch_direct_update_rule} в условиях данной теоремы не лучше
    \begin{equation}\label{eq:sgn_approx_conv}
        \begin{cases}
            \mathbb{E}\left[\hat{f}_{2}(x_{k})\right]&\leq\mathbb{E}\left[\hat{f}_{2}(x_{0})\right]\exp\left(-\frac{k\mu^{2}}{M_{\hat{G}}^{2}}\left(\frac{2}{\mu + \left(L_{\hat{F}}P_{\hat{f}_{1}} + M_{\hat{F}}^{2}\right)c} - \frac{1}{\left(L_{\hat{F}}P_{\hat{f}_{1}} + M_{\hat{F}}^{2}\right)c^{2}}\right)\right);\\[10pt]
            &c \overset{\operatorname{def}}{=} \frac{1}{3}\left(1 + 7\sqrt[3]{\frac{2}{47 + 3\sqrt{93}}} + \sqrt[3]{\frac{47 + 3\sqrt{93}}{2}}\right),~k\in\mathbb{Z}_{+}.
        \end{cases}
    \end{equation}
    Оператор математического ожидания $\mathbb{E}\left[\cdot\right]$ усредняет по всей случайности процесса оптимизации.    
\end{re:theorem}
\begin{proof}
Функция $\hat{f}_{2}$ обладает липшицевым градиентом с верхней оценкой на постоянную Липшица $l_{\hat{f}_{2}} = 2\left(L_{\hat{F}}P_{\hat{f}_{1}} + M_{\hat{F}}^{2}\right)$ (следствие \ref{lm:aux_f2_lipschitz_gradient}). Выпишем локальную модель для функции $\hat{f}_{2}$ на $k$--ой итерации в точке $x_{k + 1},~k\in\mathbb{Z}_{+}$ (следствие \ref{lm:aux_f2_local_model}):
\begin{equation*}
    \hat{f}_{2}(x_{k + 1})\leq\hat{f}_{2}(x_{k}) + \left\langle\nabla\hat{f}_{2}(x_{k}),~x_{k + 1} - x_{k}\right\rangle + \frac{l_{\hat{f}_{2}}}{2}\|x_{k + 1} - x_{k}\|^{2}.
\end{equation*}
Правило обновления $x_{k + 1}$ задаётся следующим образом:
\begin{equation*}
    \begin{aligned}
        x_{k + 1} &= x_{k} - \eta_{k}\underbrace{\left(\hat{G}^{'}(x_{k}, \tilde{B}_{k})^{*}\hat{G}^{'}(x_{k}, \tilde{B}_{k}) + \tilde{\tau}_{k}L_{k}I_{n}\right)^{-1}}_{\overset{\operatorname{def}}{=}2H_{k}}\underbrace{\hat{G}^{'}(x_{k}, B_{k})^{*}\hat{G}(x_{k}, B_{k})}_{\overset{\operatorname{def}}{=}\frac{1}{2}\nabla_{x_{k}}\hat{g}_{2}(x_{k}, B_{k})} = x_{k} - \eta_{k}H_{k}\nabla_{x_{k}}\hat{g}_{2}(x_{k}, B_{k}),\\
        &\eta_{k} > 0.
    \end{aligned}
\end{equation*}
$B_{k}, \tilde{B}_{k}\subseteq\mathcal{B}$ --- независимо сэмплированные батчи на $k$--ой итерации. Подставим данное правило в локальную модель:
\begin{equation*}
    \begin{aligned}
        \hat{f}_{2}(x_{k + 1})&\leq\hat{f}_{2}(x_{k}) - \eta_{k}\left\langle\nabla\hat{f}_{2}(x_{k}),~H_{k}\nabla_{x_{k}}\hat{g}_{2}(x_{k}, B_{k})\right\rangle + \frac{\eta_{k}^{2}l_{\hat{f}_{2}}}{2}\|H_{k}\nabla_{x_{k}}\hat{g}_{2}(x_{k}, B_{k})\|^{2}=\\
        &=\hat{f}_{2}(x_{k}) - \eta_{k}\left\langle\nabla\hat{f}_{2}(x_{k}),H_{k}\nabla_{x_{k}}\hat{g}_{2}(x_{k}, B_{k})\right\rangle + \frac{\eta_{k}^{2}l_{\hat{f}_{2}}}{2}\left\langle H_{k}^{2}\nabla_{x_{k}}\hat{g}_{2}(x_{k}, B_{k}),\nabla_{x_{k}}\hat{g}_{2}(x_{k}, B_{k})\right\rangle.
    \end{aligned}
\end{equation*}
Для матрицы $H_{k}$ выполнены следующие соотношения (лемма \ref{lm:aux_matrix_power_order}):
$$\frac{1}{\left(2\left(M_{\hat{G}}^{2} + \tilde{\tau}_{k}L_{k}\right)\right)^{t}}I_{n}\preceq H_{k}^{t}\preceq\frac{1}{\left(2\tilde{\tau}_{k}L_{k}\right)^{t}}I_{n},~k\in\mathbb{Z}_{+},~t\geq0.$$
Усредним неравенство с помощью оператора математического ожидания:
\begin{equation*}
    \begin{aligned}
        \mathbb{E}\left[\hat{f}_{2}(x_{k + 1})\right]&\leq\mathbb{E}\left[\hat{f}_{2}(x_{k})\right] - \eta_{k}\mathbb{E}\left[\left\langle\nabla\hat{f}_{2}(x_{k}),H_{k}\nabla\hat{f}_{2}(x_{k})\right\rangle\right] + \frac{\eta_{k}^{2}l_{\hat{f}_{2}}}{2}\mathbb{E}\left[\left\langle H_{k}^{2}\nabla_{x_{k}}\hat{g}_{2}(x_{k}, B_{k}),\nabla_{x_{k}}\hat{g}_{2}(x_{k}, B_{k})\right\rangle\right]\leq\\
        &\leq\mathbb{E}\left[\hat{f}_{2}(x_{k})\right] - \frac{\eta_{k}\mathbb{E}\left[\left\|\nabla\hat{f}_{2}(x_{k})\right\|^{2}\right]}{2\left(M_{\hat{G}}^{2} + \tilde{\tau}_{k}L_{k}\right)} + \frac{\eta_{k}^{2}l_{\hat{f}_{2}}}{8\left(\tilde{\tau}_{k}L_{k}\right)^{2}}\mathbb{E}\left[\left\|\nabla_{x_{k}}\hat{g}_{2}(x_{k}, B_{k})\right\|^{2}\right].
    \end{aligned}
\end{equation*}
Применим ограничения на норму градиента к локальной модели \eqref{eq:grad_norm_bounds}:
\begin{equation*}
    \begin{aligned}
        \mathbb{E}\left[\hat{f}_{2}(x_{k + 1})\right]&\leq\mathbb{E}\left[\hat{f}_{2}(x_{k})\right] - \frac{\eta_{k}\mathbb{E}\left[\left\|\nabla\hat{f}_{2}(x_{k})\right\|^{2}\right]}{2\left(M_{\hat{G}}^{2} + \tilde{\tau}_{k}L_{k}\right)} + \frac{\eta_{k}^{2}l_{\hat{f}_{2}}}{8\left(\tilde{\tau}_{k}L_{k}\right)^{2}}\mathbb{E}\left[\left\|\nabla_{x_{k}}\hat{g}_{2}(x_{k}, B_{k})\right\|^{2}\right]
        \leq\mathbb{E}\left[\hat{f}_{2}(x_{k})\right] -\\
        &- \frac{2\eta_{k}\mu\mathbb{E}\left[\hat{f}_{2}(x_{k})\right]}{M_{\hat{G}}^{2} + \tilde{\tau}_{k}L_{k}} + \frac{\eta_{k}^{2}M_{\hat{G}}^{2}l_{\hat{f}_{2}}}{2\left(\tilde{\tau}_{k}L_{k}\right)^{2}}\mathbb{E}\left[\hat{f}_{2}(x_{k})\right] = \mathbb{E}\left[\hat{f}_{2}(x_{k})\right]\left(1 - \frac{2\eta_{k}\mu}{M_{\hat{G}}^{2} + \tilde{\tau}_{k}L_{k}} + \frac{l_{\hat{f}_{2}}}{2}\left(\frac{\eta_{k}M_{\hat{G}}}{\tilde{\tau}_{k}L_{k}}\right)^{2}\right) =\\
        &= \mathbb{E}\left[\hat{f}_{2}(x_{k})\right]\left(1 - \frac{2\eta_{k}\mu}{M_{\hat{G}}^{2} + \tilde{\tau}_{k}L_{k}} + \left(L_{\hat{F}}P_{\hat{f}_{1}} + M_{\hat{F}}^{2}\right)\left(\frac{\eta_{k}M_{\hat{G}}}{\tilde{\tau}_{k}L_{k}}\right)^{2}\right).
    \end{aligned}
\end{equation*}
Вычислим оптимальный шаг для каждой итерации:
\begin{equation*}
    \begin{aligned}
        &\eta_{k}^{2}\left(\left(L_{\hat{F}}P_{\hat{f}_{1}} + M_{\hat{F}}^{2}\right)\left(\frac{M_{\hat{G}}}{\tilde{\tau}_{k}L_{k}}\right)^{2}\right) - \eta_{k}\left(\frac{2\mu}{M_{\hat{G}}^{2} + \tilde{\tau}_{k}L_{k}}\right)\rightarrow\min\limits_{\eta_{k} > 0}\Rightarrow\\
        &\Rightarrow\eta_{k} = \frac{\mu\left(\tilde{\tau}_{k}L_{k}\right)^{2}}{\left(M_{\hat{G}}^{2} + \tilde{\tau}_{k}L_{k}\right)\left(L_{\hat{F}}P_{\hat{f}_{1}} + M_{\hat{F}}^{2}\right)M_{\hat{G}}^{2}},~\tilde{\tau}_{k} > 0,~L_{k} > 0,~k\in\mathbb{Z}_{+}.
    \end{aligned}
\end{equation*}
При таком шаге получается линейная скорость сходимости с произвольным размером батча:
\begin{equation*}
    \begin{aligned}
        \mathbb{E}\left[\hat{f}_{2}(x_{k + 1})\right]&\leq\mathbb{E}\left[\hat{f}_{2}(x_{k})\right]\underbrace{\left(1 - \left(\frac{\mu\tilde{\tau}_{k}L_{k}}{M_{\hat{G}}^{2} + \tilde{\tau}_{k}L_{k}}\right)^{2}\frac{1}{\left(L_{\hat{F}}P_{\hat{f}_{1}} + M_{\hat{F}}^{2}\right)M_{\hat{G}}^{2}}\right)}_{\in(0,~1)\text{ так как }0 < \mu\leq\min\{M_{\hat{F}}^{2},~M_{\hat{G}}^{2}\}\text{ и }\tilde{\tau}_{k}L_{k} > 0}\leq\\
        &\leq\mathbb{E}\left[\hat{f}_{2}(x_{0})\right]\exp\left(\frac{-(k + 1)}{\left(L_{\hat{F}}P_{\hat{f}_{1}} + M_{\hat{F}}^{2}\right)M_{\hat{G}}^{2}}\left(\frac{\mu\tilde{\tau}L}{M_{\hat{G}}^{2} + \tilde{\tau}L}\right)^{2}\right),~k\in\mathbb{Z}_{+}.
    \end{aligned}
\end{equation*}
Введём функцию
$$\alpha(t) \overset{\operatorname{def}}{=} 1 - \left(\frac{\mu t}{M_{\hat{G}}^{2} + t}\right)^{2}\frac{1}{\left(L_{\hat{F}}P_{\hat{f}_{1}} + M_{\hat{F}}^{2}\right)M_{\hat{G}}^{2}},$$
для неё найдём минимальное значение и множество значений $t = \tilde{\tau}_{k}L_{k}$, на которых достигается минимум, чтобы гарантировать наискорейшее убывание $\mathbb{E}\left[\hat{f}_{2}(x_{k + 1})\right]$:
$$\mathbb{E}\left[\hat{f}_{2}(x_{k + 1})\right]\leq\alpha(\tilde{\tau}_{k}L_{k})\mathbb{E}\left[\hat{f}_{2}(x_{k})\right].$$
Поиск точек минимума $\alpha(t)$ равносилен поиску точек максимума функции:
$$\beta(t) \overset{\operatorname{def}}{=} \frac{\mu t}{M_{\hat{G}}^{2} + t}.$$
Функция $\beta(t)$ на $\mathbb{R}_{+}$ обладает неотрицательной первой производной и неположительной второй производной:
\begin{equation*}
    \begin{aligned}
        &\beta^{'}(t) = \frac{\mu}{M_{\hat{G}}^{2} + t}\left(1 - \frac{t}{M_{\hat{G}}^{2} + t}\right)\geq 0;\\
        &\beta^{''}(t) = \frac{2\mu}{\left(M_{\hat{G}}^{2} + t\right)^{2}}\left(\frac{t}{M_{\hat{G}}^{2} + t} - 1\right)\leq 0.
    \end{aligned}
\end{equation*}
Значит, чем больше $t$, тем меньше значение $\alpha(t)$, что задаёт диапазон значений:
\begin{equation*}
    1 = \alpha(0)\geq\alpha(t)\geq\lim\limits_{t\rightarrow+\infty}\alpha(t) = 1 - \frac{\mu^{2}}{\left(L_{\hat{F}}P_{\hat{f}_{1}} + M_{\hat{F}}^{2}\right)M_{\hat{G}}^{2}} > 0,~t\in\mathbb{R}_{+},
\end{equation*}
так как $\mu\leq\min\{M_{\hat{G}}^{2},~M_{\hat{F}}^{2}\}$ (предположение \ref{as:2}).
Рассмотрим изменение шага метода в зависимости от $t = \tilde{\tau}_{k}L_{k}$:
\begin{equation*}
    \begin{aligned}
        &\eta_{k}H_{k}\nabla_{x_{k}}\hat{g}_{2}(x_{k}, B_{k}) = \frac{\mu\left(\tilde{\tau}_{k}L_{k}\right)^{2}\left(\hat{G}^{'}(x_{k}, \tilde{B}_{k})^{*}\hat{G}^{'}(x_{k}, \tilde{B}_{k}) + \tilde{\tau}_{k}L_{k}I_{n}\right)^{-1}\hat{G}^{'}(x_{k}, B_{k})^{*}\hat{G}(x_{k}, B_{k})}{\left(M_{\hat{G}}^{2} + \tilde{\tau}_{k}L_{k}\right)\left(L_{\hat{F}}P_{\hat{f}_{1}} + M_{\hat{F}}^{2}\right)M_{\hat{G}}^{2}}\Rightarrow\\
        &\Rightarrow\lim\limits_{t\rightarrow+\infty}\frac{\mu t}{M_{\hat{G}}^{2} + t}\frac{\left(\frac{\hat{G}^{'}(x_{k}, \tilde{B}_{k})^{*}\hat{G}^{'}(x_{k}, \tilde{B}_{k})}{t} + I_{n}\right)^{-1}\hat{G}^{'}(x_{k}, B_{k})^{*}\hat{G}(x_{k}, B_{k})}{\left(L_{\hat{F}}P_{\hat{f}_{1}} + M_{\hat{F}}^{2}\right)M_{\hat{G}}^{2}} =\\
        &= \left(\frac{\mu}{2\left(L_{\hat{F}}P_{\hat{f}_{1}} + M_{\hat{F}}^{2}\right)M_{\hat{G}}^{2}}\right)\nabla_{x_{k}}\hat{g}_{2}(x_{k}, B_{k}).
    \end{aligned}
\end{equation*}
То есть, чем ближе шаг метода Гаусса--Ньютона в данных условиях интерполяции к градиентному, тем быстрее метод сходится к решению системы \eqref{eq:smooth_system}. Примечательно, что нам для сходимости при таком обновлении параметров достаточно потребовать $\tilde{\tau}_{k} > 0,~L_{k} > 0$, грамотно выбирая шаг $\eta_{k}$.

Если же зафиксировать значение $\eta_{k} = 1$, то возникает единственное конечное оптимальное значение произведения $t = \tilde{\tau}_{k}L_{k},~k\in\mathbb{Z}_{+}$, выражающееся из неравенства
$$\mathbb{E}\left[\hat{f}_{2}(x_{k + 1})\right]\leq\mathbb{E}\left[\hat{f}_{2}(x_{k})\right]\left(1 - \frac{2\mu}{M_{\hat{G}}^{2} + t} + \left(L_{\hat{F}}P_{\hat{f}_{1}} + M_{\hat{F}}^{2}\right)\left(\frac{M_{\hat{G}}}{t}\right)^{2}\right).$$
Найдём оптимальный коэффициент линейной сходимости:
\begin{equation*}
    \begin{aligned}
        &\underbrace{\frac{\left(L_{\hat{F}}P_{\hat{f}_{1}} + M_{\hat{F}}^{2}\right)M_{\hat{G}}^{2}}{t^{2}} - \frac{2\mu}{M_{\hat{G}}^{2} + t}}_{\overset{\operatorname{def}}{=} \zeta(t)}\rightarrow\min\limits_{t > 0}.
    \end{aligned}
\end{equation*}
Выпишем условия оптимальности второго порядка для введённой функции $\zeta(t)$:
\begin{equation*}
    \begin{cases}
        \zeta^{'}(t) = \frac{-2\left(L_{\hat{F}}P_{\hat{f}_{1}} + M_{\hat{F}}^{2}\right)M_{\hat{G}}^{2}}{t^{3}} + \frac{2\mu}{\left(M_{\hat{G}}^{2} + t\right)^{2}} = 0;\\
        \zeta^{''}(t) = \frac{6\left(L_{\hat{F}}P_{\hat{f}_{1}} + M_{\hat{F}}^{2}\right)M_{\hat{G}}^{2}}{t^{4}} - \frac{4\mu}{\left(M_{\hat{G}}^{2} + t\right)^{3}} > 0.
    \end{cases}
\end{equation*}
Условие $\zeta^{'}(t) = 0$ порождает кубическое уравнение
$$\mu t^{3} - \left(L_{\hat{F}}P_{\hat{f}_{1}} + M_{\hat{F}}^{2}\right)M_{\hat{G}}^{2}t^{2} - 2\left(L_{\hat{F}}P_{\hat{f}_{1}} + M_{\hat{F}}^{2}\right)M_{\hat{G}}^{4}t - \left(L_{\hat{F}}P_{\hat{f}_{1}} + M_{\hat{F}}^{2}\right)M_{\hat{G}}^{6} = 0,$$
с единственным действительным корнем, имеющим следующее представление по общей формуле вычисления корней кубического уравнения:
$$t^{*} = \frac{\left(L_{\hat{F}}P_{\hat{f}_{1}} + M_{\hat{F}}^{2}\right)M_{\hat{G}}^{2}}{3\mu}\left(1 + 7\sqrt[3]{\frac{2}{47 + 3\sqrt{93}}} + \sqrt[3]{\frac{47 + 3\sqrt{93}}{2}}\right),$$
при данном значении вторая производная $\zeta^{''}(t^{*})$ положительна. Более того, при $t^{*}$ имеет место линейная сходимость метода Гаусса--Ньютона с оценкой \eqref{eq:sgn_approx_conv} и коэффициентом линейной сходимости, принадлежащим интервалу $(0, 1)$:
\begin{equation*}
    \begin{cases}
        0 < 1 - \frac{\mu^{2}}{M_{\hat{G}}^{2}}\left(\frac{2}{\mu + \left(L_{\hat{F}}P_{\hat{f}_{1}} + M_{\hat{F}}^{2}\right)c} - \frac{1}{\left(L_{\hat{F}}P_{\hat{f}_{1}} + M_{\hat{F}}^{2}\right)c^{2}}\right) < 1;\\[10pt]
        c = \frac{1}{3}\left(1 + 7\sqrt[3]{\frac{2}{47 + 3\sqrt{93}}} + \sqrt[3]{\frac{47 + 3\sqrt{93}}{2}}\right) \in (2,~3).
    \end{cases}
\end{equation*}
Неравенства $\mu\leq\min\left\{M_{\hat{G}}^{2},~M_{\hat{F}}^{2}\right\}$ и $3 > c > 2$ обеспечивают нахождение коэффициента линейной сходимости в интервале $(0,~1)$. Если же не фиксировать $\eta_{k} = 1$, а искать оптимальный $\eta_{k}$ для произвольного $t$, то в случае $t\rightarrow+\infty$ коэффициент линейной сходимости будет меньше, чем в случае $\eta_{k} = 1$ при $t = t^{*}$.
\end{proof}

\end{document}